\documentclass[a4paper]{amsart}
\usepackage[utf8]{inputenc}
\usepackage[T1]{fontenc}
\usepackage{lmodern}
\usepackage{amssymb}
\usepackage{mathtools}
\usepackage[inline,shortlabels]{enumitem}

\usepackage{tikz}
\usetikzlibrary{matrix,arrows}
\tikzset{cd/.style=matrix of math nodes,row sep=2em,column sep=2em, text height=1.5ex, text depth=0.5ex}
\tikzset{cdar/.style=->,auto}
\tikzset{dar/.style={double,double equal sign distance,-implies}}
\tikzset{mid/.style={anchor=mid}} 
\tikzset{narrowfill/.style={inner sep=0pt, fill=white}}
\usepackage{tikz-cd}

\newcommand*{\meascor}[5][1]{\draw[cdar] (#2) --
  node[inner sep=#1pt] {\(\scriptstyle #3\)}
  node[inner sep=#1pt,swap] {\(\scriptstyle #4\)} (#5);}

\setlist[enumerate,1]{label=\textup{(\arabic*)}}
\setlist[enumerate,2]{label=\textup{(\alph*)}}

\usepackage{microtype}
\usepackage[pdftitle={A universal property for groupoid C*-algebras.  II.  Fell bundles},
pdfauthor={Alcides Buss, Rohit Holkar and Ralf Meyer},
pdfsubject={Mathematics}
]{hyperref}
\usepackage[lite]{amsrefs}

\renewcommand{\PrintDOI}[1]{\href{http://dx.doi.org/#1}{DOI #1}}

\BibSpec{book}{%
  +{}  {\PrintPrimary}                {transition}
  +{,} { \textit}                     {title}
  +{.} { }                            {part}
  +{:} { \textit}                     {subtitle}
  +{,} { \PrintEdition}               {edition}
  +{}  { \PrintEditorsB}              {editor}
  +{,} { \PrintTranslatorsC}          {translator}
  +{,} { \PrintContributions}         {contribution}
  +{,} { }                            {series}
  +{,} { \voltext}                    {volume}
  +{,} { }                            {publisher}
  +{,} { }                            {organization}
  +{,} { }                            {address}
  +{,} { \PrintDateB}                 {date}
  +{,} { }                            {status}
  +{}  { \parenthesize}               {language}
  +{}  { \PrintTranslation}           {translation}
  +{;} { \PrintReprint}               {reprint}
  +{.} { }                            {note}
  +{.} {}                             {transition}
  +{,} { \PrintDOI}                   {doi}
  +{}  {\SentenceSpace \PrintReviews} {review}
}

\numberwithin{equation}{section}

\theoremstyle{plain}
\newtheorem{theorem}[equation]{Theorem}
\newtheorem{lemma}[equation]{Lemma}
\newtheorem{proposition}[equation]{Proposition}
\newtheorem{deflem}[equation]{Definition and Lemma}
\newtheorem{corollary}[equation]{Corollary}

\theoremstyle{definition}
\newtheorem{definition}[equation]{Definition}

\theoremstyle{remark}
\newtheorem{remark}[equation]{Remark}
\newtheorem{example}[equation]{Example}

\DeclarePairedDelimiter{\abs}{\lvert}{\rvert}
\DeclarePairedDelimiter{\norm}{\lVert}{\rVert}
\DeclarePairedDelimiter{\ket}{\lvert}{\rangle}
\DeclarePairedDelimiter{\bra}{\langle}{\rvert}
\DeclarePairedDelimiterX{\braket}[2]{\langle}{\rangle}{#1\,\delimsize\vert\,\mathopen{}#2}
\DeclarePairedDelimiterX{\BRAKET}[2]{\langle}{\rangle}{\!\delimsize\langle#1\,\delimsize\vert\,\mathopen{}#2\delimsize\rangle\!}
\DeclarePairedDelimiterX{\setgiven}[2]{\{}{\}}{#1\,{:}\,\mathopen{}#2}

\newcommand*{\alb}{\hspace{0pt}} 

\newcommand*{\dual}[1]{\widehat{#1}}
\newcommand*{\Ideals}{\mathbb I}
\newcommand*{\Open}{\mathbb O}

\newcommand*{\Mult}{\mathcal M}
\newcommand*{\Q}{\mathcal Q}
\newcommand*{\U}{\mathcal U}
\newcommand*{\UM}{\mathcal U}
\newcommand*{\Lt}{\mathcal L}
\newcommand*{\Dom}{\mathcal D}

\newcommand{\idealin}{\mathrel{\triangleleft}} 

\newcommand{\mor}{\tau}
\newcommand*{\s}{s} 
\newcommand*{\rg}{r}
\newcommand*{\inv}{\mathrm{inv}}
\newcommand*{\ev}{\mathrm{ev}}

\newcommand{\C}{\mathbb C}
\newcommand{\N}{\mathbb N}
\newcommand{\Z}{\mathbb Z}
\newcommand{\T}{\mathbb T}

\newcommand*{\E}{\mathcal E}
\newcommand*{\F}{\mathcal F}
\newcommand*{\Hils}[1][H]{\mathcal{#1}}
\newcommand*{\Bun}[1][A]{\mathcal{#1}}
\newcommand*{\A}{\mathcal{A}}
\newcommand*{\FU}[1][A]{\mathsf{#1}}
\newcommand*{\SUF}[1][A]{#1}
\newcommand*{\B}{\mathcal{B}}
\newcommand*{\I}{\mathcal{I}}

\newcommand*{\univ}{\mathrm{univ}}
\newcommand*{\diff}{\mathrm d}
\newcommand*{\red}{\mathrm r}
\newcommand*{\Cont}{\mathrm C}
\newcommand*{\Contc}{\mathrm{C_c}}

\newcommand*{\Star}{$^*$\nobreakdash-\hspace{0pt}}
\newcommand*{\nb}{\nobreakdash}
\newcommand*{\Cst}{\mathrm C^*}
\newcommand*{\Cred}{\mathrm C^*_\mathrm r}

\newcommand*{\congto}{\xrightarrow\sim}
\newcommand*{\defeq}{\mathrel{\vcentcolon=}}

\newcommand{\Comp}{\mathbb K}
\newcommand{\Bound}{\mathbb B}

\newcommand{\blank}{{\llcorner\!\lrcorner}}

\newcommand*{\Id}{\mathrm{id}}

\newcommand*{\into}{\rightarrowtail}
\newcommand*{\injto}{\hookrightarrow}
\newcommand*{\onto}{\twoheadrightarrow}

\newcommand{\act}{\mathfrak{a}}
\newcommand{\actB}{\mathfrak{b}}

\DeclareMathOperator{\Endo}{End} 
\DeclareMathOperator{\supp}{supp}
\DeclareMathOperator{\Ad}{Ad}
\DeclareMathOperator{\pr}{pr}
\DeclareMathOperator{\Ind}{Ind} 
\DeclareMathOperator{\IM}{im}

\begin{document}
\title{A universal property for groupoid C*-algebras.  II.  Fell bundles}

\author{Alcides Buss} \email{alcides.buss@ufsc.br}

\author{Rohit Holkar} \email{rohit.d.holkar@gmail.com}

\address{Departamento de Matem\'atica\\
  Universidade Federal de Santa Catarina\\
  88.040-900 Florian\'opolis-SC\\
  Brazil}

\author{Ralf Meyer} \email{rmeyer2@uni-goettingen.de}
\address{Mathematisches Institut\\
  Georg-August Universität Göttingen\\
  Bunsenstraße 3--5\\
  37073 Göttingen\\
  Germany}

\thanks{Partially supported by CNPq and FAPESC (Brazil), by Humboldt Foundation (Germany) and by EU grant
  HORIZON-MSCA-SE-2021, Project 101086394.}

\begin{abstract}
  We define possibly unsaturated, upper semicontinuous Fell bundles
  over Hausdorff, locally compact groupoids and establish a
  universal property for representations of their full section
  \(\Cst\)\nb-algebras on Hilbert modules over arbitrary
  \(\Cst\)\nb-algebras.  Based on this, we prove that the full
  section \(\Cst\)\nb-algebra is functorial and exact, and we define a
  quasi-orbit space and a quasi-orbit map.  We deduce and extend Renault's
  Integration and Disintegration Theorems to general Fell bundles
  using our universal property.
\end{abstract}
\maketitle

\tableofcontents

\section{Introduction}
\label{sec:intro}

Let~\(G\) be a Hausdorff, locally compact groupoid with a Haar
system~\(\alpha\).  The groupoid \(\Cst\)\nb-algebra
\(\Cst(G,\alpha)\) is characterised
in~\cite{Buss-Holkar-Meyer:Universal} by a universal property.
It describes the representations of \(\Cst(G,\alpha)\) on Hilbert
modules over arbitrary \(\Cst\)\nb-algebras and thus determines
\(\Cst(G,\alpha)\) uniquely.
When combined with the representation theory of commutative
\(\Cst\)\nb-algebras, it implies the Integration and Disintegration
Theorems of Renault for representations of \(\Cst(G,\alpha)\) on a
separable Hilbert space.
Here we extend this universal property to section \(\Cst\)\nb-algebras
of Fell bundles over~\(G\).
In particular, our universal property covers crossed products for
actions and partial actions of~\(G\).
While the basic ideas are the same as
in~\cite{Buss-Holkar-Meyer:Universal}, the Fell bundles cause some new
difficulties.

Fell bundles over~\(G\) are defined in Section~\ref{sec:Fell_bundles}.
Our technical assumptions are as mild as conceivable at this time: we
allow upper semicontinuous fields of \(\Cst\)\nb-algebras and
unsaturated Fell bundles.  Thus our Fell bundles are more general
than in previous works.
Renault~\cite{Renault:Representations} only considers
continuous fields of \(\Cst\)\nb-algebras as opposed to upper
semicontinuous fields, and he
considers only twisted actions, not Fell bundles; in particular, the
Fell bundles associated to twisted actions are always saturated.
Yamagami~\cite{Yamagami:IdealStructure} does allow unsaturated Fell
bundles, but does not give details on the basic theory of Fell bundles
and their \(\Cst\)\nb-algebras.  Muhly and
Williams~\cite{Muhly-Williams:Equivalence.FellBundles} allow upper
semicontinuous fields throughout, but require Fell bundles to be separable and
saturated.  Since basic results of the theory are established in
different sources under slightly different assumptions, we develop the
theory of Fell bundles over groupoids from scratch here.

Our theory is not yet as general as possible because
we only consider Hausdorff groupoids.  The non-Hausdorff case adds
further complications.  These are of a different nature and will be
treated in a separate article.

A Fell bundle has Banach spaces~\((\A_g)_{g\in G}\) as fibres.  These
form an upper semicontinuous field of Banach spaces and carry
continuous families of involutions \(\A_g \to \A_{g^{-1}}\) and
multiplication maps
\[
  \A_g \times \A_h \to \A_{g h},\qquad (a,b)\mapsto a\cdot b,
\]
for \((g,h)\in G^2 \defeq G\times_{\s,G^0,\rg} G\); here, \(G^0\)
denotes the space of objects of~\(G\).  The multiplication maps
satisfy the standard algebraic properties of a Fell bundle and a
positivity condition for \(a^* \cdot a\) for all \(a\in \A_g\) and
\(g\in G\) (see Definition~\ref{def:Fell_bundle}).  The restriction
of~\(\A\) to~\(G^0\) is an upper
semicontinuous field of \(\Cst\)\nb-algebras on~\(G^0\), which we
denote by~\(\FU\).  Let \(\SUF \defeq \Cont_0(G^0,\FU)\) be the
\(\Cont_0(G^0)\)-\(\Cst\)-algebra of sections of~\(\FU\).  Since a
Hilbert bimodule is a partial Morita--Rieffel equivalence, we
interpret the Fell bundle~\(\A\) as a generalised partial action
of~\(G\) on~\(\SUF\).

A \emph{partial action} of~\(G\) on~\(\FU\) consists of isomorphisms
\(\theta_g\colon \Dom_{g^{-1}} \congto \Dom_g\) for certain ideals
\(\Dom_g \subseteq \FU_{\rg(g)}\).  These ideals must depend
continuously on~\(G\) in the sense that there is an ideal
\(\Dom\idealin \rg^*(\FU)\) with fibres
\(\Dom_g\idealin \FU_{\rg(g)}\); here~\(\rg^*(\FU)\) denotes the
pullback of~\(\FU\) to a \(\Cont_0(G)\)-\(\Cst\)-algebra along
\(\rg\colon G \to G^0\).  In addition, we require that
\(\Dom_{u(x)} = \FU_x\) for all \(x\in G^0\) and that
\(\theta_g\circ\theta_h\subseteq\theta_{g h}\) as partial maps.
This is the groupoid generalisation of a partial group action as
defined in~\cite{Exel:Partial_dynamical}.  An \emph{action} of~\(G\)
is a partial action where \(\Dom_g = \FU_{\rg(g)}\) for all
\(g\in G\).  Any partial action of~\(G\) defines a Fell bundle,
which is saturated if and only if the partial action is an action.

Section~\ref{sec:representations} defines representations of a Fell
bundle on a Hilbert module~\(\F\) over an arbitrary
\(\Cst\)\nb-algebra~\(D\).  As in~\cite{Buss-Holkar-Meyer:Universal},
this is based on \(\Cst\)\nb-correspondences associated to open
continuous surjections \(\rg,\s\colon G\onto G^0\) and
\(d_0,d_1,d_2\colon G^2\onto G\) and families of measures along
their fibres induced by the Haar system on~\(G\).
In the case of Fell bundles, these \(\Cst\)\nb-correspondences acquire
``coefficients'' related to the Fell bundle.
That is, instead of scalar-valued \(\Lt^2\)\nb-functions we consider
\(\Lt^2\)\nb-sections of various bundles.
A representation on a Hilbert module~\(\F\) has two pieces.
One piece is a representation of the \(\Cst\)\nb-algebra~\(\SUF\)
associated to the Fell bundle.
The other piece is a unitary operator between certain Hilbert
\(D\)\nb-modules that are built by tensoring~\(\F\) over~\(\SUF\) with
certain Hilbert \(\SUF\)\nb-modules.
This is subject to a condition that is formulated using the
space~\(G^2\) of composable pairs.
There are two new technical issues compared
to~\cite{Buss-Holkar-Meyer:Universal}.
Namely, adding coefficients to the \(\Cst\)\nb-correspondences and
restricting to appropriate Hilbert submodules to get a unitary operator
even if the action of~\(G\) on~\(\FU\) is only partial.

The continuous, compactly supported sections of a Fell bundle form a
\Star{}algebra \(\Contc(G,\A)\).  Section~\ref{sec:integrate}
integrates a representation of~\(\A\) on a Hilbert module~\(\F\) to
a representation of~\(\Contc(G,\A)\) on the same Hilbert
module~\(\F\).  By construction, this integrated representation is
bounded by a variant of the \(I\)\nb-norm.

Section~\ref{sec:disintegrate} disintegrates a representation
of~\(\Contc(G,\A)\) on~\(\F\) to a representation of~\(\A\)
on~\(\F\).  As in~\cite{Buss-Holkar-Meyer:Universal}, this also
works for representations by densely defined unbounded operators,
assuming only some weak continuity in the inductive limit topology
on~\(\Contc(G,\A)\).  Thus any densely defined representation
of~\(\Contc(G,\A)\) that is continuous in the inductive limit topology
extends to a representation by adjointable operators that is bounded
for the \(I\)\nb-norm on~\(\Contc(G,\A)\).

The integration and disintegration processes are inverse to each
other.  Thus there is a maximal \(\Cst\)\nb-norm on~\(\Contc(G,\A)\)
that is continuous in the inductive limit topology; and
this \(\Cst\)\nb-norm is already bounded by the \(I\)\nb-norm.  The
full section \(\Cst\)\nb-algebra \(\Cst(G,\A)\) of the Fell
bundle is the completion of~\(\Contc(G,\A)\) in this \(\Cst\)\nb-norm.
Our analysis shows that representations of~\(\Cst(G,\A)\) on a
Hilbert module~\(\F\) are equivalent to Fell-bundle
representations of~\(\A\)
on~\(\F\).  This is our universal property
for~\(\Cst(G,\A)\).  Since it covers representations on
arbitrary Hilbert modules, it determines~\(\Cst(G,\A)\)
uniquely up to isomorphism.

We also define the reduced section \(\Cst\)\nb-algebra \(\Cred(G,\A)\)
of a Fell bundle~\(\A\) over~\(G\), which is based on a ``regular
representation'' of the Fell bundle~\(\A\).
We show that the regular representation generates a norm
on~\(\Contc(G,\A)\) --~not just a seminorm.
Thus the canonical maps from~\(\Contc(G,\A)\) to \(\Cred(G,\A)\)
and~\(\Cst(G,\A)\) are injective.
This implies that the canonical \Star{}homomorphisms \(\SUF \to
\Mult(\Cst(G,\A))\) and \(\SUF \to \Mult(\Cred(G,\A))\) are injective.

In Section~\ref{sec:functorial_exact}, we prove several expected
results about section \(\Cst\)\nb-algebras of Fell bundles.
We describe the universal representation of the Fell bundle in
\(\Cst(G,\A)\).
We prove that both full and reduced crossed products are functorial
for \Star{}homomorphisms of Fell bundles --~defined suitably.
And we show that these functors preserve surjections and that the
reduced crossed product preserves injections.
We also prove that the full crossed product functor is exact.
For this, we study ideals in Fell bundles and relate them to invariant
ideals in~\(\SUF\).
We use split extensions to define homomorphisms from a Fell bundle to
the ``multipliers'' of another Fell bundle.
We show that these induce \Star{}homomorphisms \(\Cst(G,\A) \to
\Mult(\Cst(G,\B))\) and \(\Cred(G,\A) \to
\Mult(\Cred(G,\B))\).
In the end of Section~\ref{sec:functorial_exact}, we discuss
approximate units for the section \(\Cst\)\nb-algebras of general Fell
bundles.
Although our main theorems are proven without relying on
approximate units, we include this construction for future
convenience.
To the best of our knowledge, the existence of such approximate units
is currently only documented in the literature for saturated and
separable Fell bundles (see
\cite{Muhly-Williams:Equivalence.FellBundles}*{Proposition~5.1}).

In Section~\ref{sec:dual_groupoid_quasi-orbit}, we build a dual
action on the space~\(\dual{\SUF}\) of equivalence classes of
irreducible representations of~\(\SUF\).  We show that the open subsets
in~\(\dual{\SUF}\) invariant under the dual action correspond
precisely to the Fell ideals in~\(\A\).  We prove that the
``restriction'' of an ideal in \(\Cst(G,\A)\) to~\(\SUF\) is always
invariant in this sense.  Using this result and the machinery
developed in~\cite{Kwasniewski-Meyer:Stone_duality}, we define a
quasi-orbit space \(\dual{\SUF}/{\sim}\) and a quasi-orbit map
\[
  \dual{\Cst(G,\A)} \to \dual{\SUF}/{\sim}.
\]
Quasi-orbit spaces
are a key ingredient in the Effros--Hahn Conjecture (see
\cites{Effros-Hahn:Transformation_groups,
  GootmanRosenberg.StructureOfCrossedProducts}).
They are somewhat implicit in the groupoid version of this conjecture
in~\cite{Renault:Ideal_structure}, and interesting objects in their
own right (compare~\cite{GootmanLazar.DualityCrossedProduct}).

In Section~\ref{sec:Renault}, we study Fell-bundle representations
on a separable Hilbert space.
The \(\Cont_0(G^0)\)-\(\Cst\)-algebra structure on~\(\SUF\) allows us
to analyse a representation of~\(\SUF\) through a measurable field of
representations of the field~\(\FU\).
Based on this, we deduce a version of Renault's
Integration--Disintegration Theorem from our universal property.
We also bring in results of Ramsay in order to improve statements that
initially hold only almost everywhere so that they hold everywhere.

The last, rather short sections establish results that are specific to
certain classes of Fell bundles.  In
Section~\ref{sec:twisted_partial_special}, we consider Fell bundles
coming from twisted partial actions.  We show that measurable
Fell-bundle representations on a Hilbert space are equivalent to
measurable covariant representations, and we also decompose
representations on Hilbert modules as ``covariant representations''
provided the action is global and the twist is scalar-valued.

In Section~\ref{sec:Fell_line}, we consider Fell line bundles.  These
are often described through circle bundles.  We show that
representations of a Fell line bundle are equivalent to
representations of the corresponding circle bundle that are
homogeneous in a suitable sense.  We prove this both for the
measurable representations on a separable Hilbert space and for
Fell-bundle representations on Hilbert modules.

In Section~\ref{sec:Fell_over_groups}, we specialise to Fell bundles
over groups.  In this case, there is a rather obvious concept of a
continuous Fell-bundle representation on a Hilbert module.  We prove
that these continuous Fell-bundle representations are equivalent to
our Fell-bundle representations.  For representations on Hilbert space
and under separability assumptions, our results imply that a
measurable Fell-bundle representation is almost everywhere equal to a
continuous Fell-bundle representation.

Section~\ref{sec:trafo_groupoids} considers Fell bundles over
transformation groupoids \(X\rtimes G\).  Here a Fell bundle~\(\B\)
over~\(X\rtimes G\) induces a Fell bundle~\(\A\) over~\(G\) by
forgetting the fibrewise decomposition over~\(X\).  We show that the
Fell bundles~\(\B\) over \(X\rtimes G\) and~\(\A\) over~\(G\) have the
same Fell-bundle representations and thus have canonically isomorphic
section \(\Cst\)\nb-algebras.

\section{Upper semicontinuous fields of Banach spaces}
\label{sec:usc_fields}

The main new issue in this article compared
to~\cite{Buss-Holkar-Meyer:Universal} is that we must add fields of
\(\Cst\)\nb-algebras or \(\Cst\)\nb-correspondences as coefficients
to various constructions in~\cite{Buss-Holkar-Meyer:Universal}.  In
this section, we recall some definitions and results needed for
this.  We recall upper semicontinuous fields of Banach spaces,
\(\Cst\)\nb-algebras, \(\Cst\)\nb-correspondences, and Hilbert
bimodules.  We build a \(\Cst\)\nb-correspondence between the
\(\Cst\)\nb-algebras of \(\Cont_0\)\nb-sections of two upper
semicontinuous fields of \(\Cst\)\nb-algebras from a topological
correspondence decorated by an upper semicontinuous field of
\(\Cst\)\nb-correspondences, and we prove some lemmas about these
constructions.

Let~\(X\) be a locally compact space.
An \emph{upper semicontinuous field of Banach spaces} over~\(X\) is a
family of Banach spaces~\(\B_x\) for \(x\in X\) with a topology on the
total space \(\B = \bigsqcup_{x\in X} \B_x\) satisfying, among others,
the following properties
(see \cite{Doran-Fell:Representations} or
\cite{Williams:crossed-products}*{Appendix~C} for the full
definition): the addition and scalar multiplication are continuous
maps \(\B\times_X \B \to \B\) and \(\C\times\B \to \B\), respectively,
and the norm function \(\B\to[0,\infty)\) is upper semicontinuous.
Here \(\B\times_X \B\) denotes the fibre product over~\(X\),
consisting of all pairs \((a,b)\in \B\times \B\) with \(p(a)=p(b)\),
where \(p\colon \B\to X\) denotes the bundle projection.
We often describe the topology on~\(\B\) by giving the space
\(\Cont_0(X,\B)\) of continuous sections \(X \to \B\)
vanishing at infinity.
There is a unique topology on~\(\B\) with the specified continuous
sections by
\cite{BussExel:Fell.Bundle.and.Twisted.Groupoids}*{Proposition~2.4}.

An \emph{upper semicontinuous field of \(\Cst\)\nb-algebras}
over~\(X\) is a family of \(\Cst\)\nb-algebras~\(\B_x\) for
\(x\in X\) with a topology on \(\B \defeq \bigsqcup \B_x\) making it
an upper semicontinuous field of Banach spaces and making the
fibrewise multiplication and involution continuous maps
\(\B\times_X \B \to \B\) and \(\B\to \B\), respectively.
There is a bijection between isomorphism classes of upper
semicontinuous fields of \(\Cst\)\nb-algebras over~\(X\) and of
\(\Cont_0(X)\)-\(\Cst\)-algebras (see~\cite{Nilsen:Bundles}), mapping
a field~\(\B\) to its \(\Cst\)\nb-algebra of \(\Cont_0\)\nb-sections
\(\Cont_0(X,\B)\).

Let \(A\) and~\(B\) be \(\Cst\)\nb-algebras.
An \emph{\(A\)\nb-\(B\)-correspondence} is a Hilbert
\(B\)\nb-module~\(\E\) with a nondegenerate representation of~\(A\) by
adjointable operators, \(A\to \Bound(\E)\).
A \emph{Hilbert \(A\)-\(B\)\nb-bimodule} is a vector space~\(\E\) that
is both a right Hilbert \(B\)\nb-module and a left Hilbert
\(A\)\nb-module, such that the structures are compatible in the sense
that the actions of \(A\) and~\(B\) commute and \(\xi\cdot
\braket{\eta}{\zeta} = \BRAKET{\xi}{\eta}\cdot \zeta\) for all
\(\xi,\eta,\zeta\in \E\).
Here \(\braket{\cdot}{\cdot}\) and \(\BRAKET{\cdot}{\cdot}\) denote
the right \(B\)- and left \(A\)\nb-valued inner products,
respectively.
Such a Hilbert \(A\)-\(B\)-bimodule is, in particular, an
\(A\)\nb-\(B\)-correspondence; conversely, an
\(A\)\nb-\(B\)-correspondence comes from a Hilbert
\(A\)-\(B\)-bimodule if and only if the left action \(\varphi\colon A
\to \Bound(\E)\) restricts to an isomorphism \(\varphi|_I\colon I
\congto \Comp(\E)\) between some (twosided) ideal \(I\idealin A\) and
the ideal \(\Comp(\E)\) of compact operators in \(\Bound(\E)\) (see
\cite{Buss-Meyer:Actions_groupoids}*{Lemma~4.2}).
Then the left inner product on~\(\E\) is the composite of the standard
\(\Comp(\E)\)-valued inner product and~\(\varphi|_I^{-1}\).
This ideal~\(I\) and the left inner product on~\(\E\) are unique if
they exist.
Namely, \(I\) must be Katsura's ideal \(\varphi^{-1}(\Comp(\E)) \cap
\ker \varphi\), and \(\BRAKET{\xi}{\eta} =
\varphi|_I^{-1}\bigl(\ket{\xi}\bra{\eta}\bigr)\) for all
\(\xi,\eta\in\E\).
The left action \(A\to \Bound(\E)\) is the composition of the
canonical homomorphism \(A\to \Mult(I)\) with the isomorphism
\(\Mult(I)\congto \Mult(\Comp(\E))\cong \Bound(\E)\) induced
by~\(\varphi|_I\).

A \emph{correspondence isometry} between two
\(A\)\nb-\(B\)-correspondences \(\E\) and~\(\F\) is an
\(A\)\nb-\(B\)-bimodule homomorphism \(f\colon \E\to \F\) that
preserves the right inner product, that is,
\(\braket{f(\xi)}{f(\eta)}=\braket{\xi}{\eta}\) for all
\(\xi,\eta\in \E\).
Then~\(f\) is an isometry, but need not be
adjointable.
A \emph{Hilbert bimodule map} between two Hilbert
\(A\)\nb-\(B\)-bimodules \(\E\) and~\(\F\) is an
\(A\)\nb-\(B\)-bimodule homomorphism \(f\colon \E\to \F\) that
preserves both inner products, that is,
\(\braket{f(\xi)}{f(\eta)}=\braket{\xi}{\eta}\) and
\(\BRAKET{f(\xi)}{f(\eta)}=\BRAKET{\xi}{\eta}\) for all
\(\xi,\eta\in \E\).
Thus it is a correspondence isometry.
All correspondence isometries in our universal property are invertible
because we will restrict the codomains suitably.

Let \(\A=(\A_x)_{x\in X}\) and \(\B=(\B_x)_{x\in X}\) be two upper
semicontinuous fields of \(\Cst\)\nb-algebras over~\(X\).
An \emph{upper semicontinuous field of \(\Cst\)\nb-correspondences}
between them is a family of
\(\A_x\)\nb-\(\B_x\)-correspondences~\(\E_x\) with a topology on
\(\E = (\E_x)_{x\in X}\) that makes it an upper semicontinuous field
of Banach spaces such that the left and right multiplications and
the inner product are continuous maps \(\A\times_X \E \to \E\),
\(\E\times_X \B \to \E\) and \(\E \times_X \E \to \B\),
respectively.
We also ask the norms on the fibres~\(\E_x\) that make~\(\E\) an upper
semicontinuous field to induced by the inner products, that is,
\(\norm{\xi}_{\E_x} = \norm{\braket{\xi}{\xi}}_{\B_x}^{1/2}\) for
all \(x\in X\), \(\xi\in\E_x\).

An \emph{upper semicontinuous field of Hilbert bimodules} between
\(\A\) and~\(\B\) is a family of \(\A_x\)\nb-\(\B_x\)-Hilbert
bimodules with a topology that makes it an upper semicontinuous field
of \(\Cst\)\nb-correspondences and also makes the left inner product
continuous.  The norm is given both by the left and the right inner
product: if \(x\in X\), \(\xi\in\E_x\), then
\[
  \norm{\xi}_{\E_x} = \norm{\braket{\xi}{\xi}}_{\B_x}^{1/2} =
  \norm{\BRAKET{\xi}{\xi}}_{\A_x}^{1/2}.
\]

An important way to build new upper semicontinuous fields of Banach
spaces is the pull-back construction.  Let \(X\) and~\(Y\) be locally
compact spaces, \(f\colon X\to Y\) a continuous map, and
\(\B = (\B_y)_{y\in Y}\) an upper semicontinuous field of Banach
spaces over~\(Y\).  Then there is an upper semicontinuous field of
Banach spaces \(f^*(\B)\) over~\(X\) with fibre~\(\B_{f(x)}\) at
\(x\in X\) and with the fibre product topology on the total space
\(\bigsqcup_{x\in X} \B_{f(x)} = X\times_Y \bigsqcup_{y\in Y} \B_y\).
The space of \(\Cont_0\)\nb-sections of~\(f^*(\B)\) is the closed
linear span of sections of the form \(x\mapsto h(x) b(f(x))\) for
\(h\in \Cont_0(X)\) and \(b\in \Cont_0(Y,\B)\).

The pull-back of an upper semicontinuous field of
\(\Cst\)\nb-algebras is again an upper semicontinuous field of
\(\Cst\)\nb-algebras.
The pull-back of an upper semicontinuous field of
\(\Cst\)\nb-correspondences or Hilbert bimodules between two upper
semicontinuous fields of \(\Cst\)\nb-algebras \(\A\) and~\(\B\) is an upper
semicontinuous field of \(\Cst\)\nb-correspondences or Hilbert
bimodules between \(f^*(\A)\) and~\(f^*(\B)\), respectively.

\emph{In the following, we usually abbreviate ``upper semicontinuous
  field'' to ``field'', that is, all fields of Banach spaces,
  \(\Cst\)\nb-algebras or \(\Cst\)\nb-correspondences over locally
  compact spaces are tacitly assumed to be upper semicontinuous.}

We sometimes have to prove that two fields of Banach spaces are
isomorphic.  When we already have an isometric map between them, then
this is a fibrewise issue:

\begin{lemma}
  \label{lem:subfield_equality}
  Let \(\A\) and~\(\B\) be fields of Banach spaces over~\(X\) and let
  \(\psi\colon \A \to \B\) be a map between them that is fibrewise
  isometric.  In other words, let~\(\A\) be a subfield in the field of
  Banach spaces~\(\B\).  The map~\(\psi\) is an isomorphism of fields
  of Banach spaces if and only if~\(\psi_x\) has dense range for all
  \(x\in X\).
\end{lemma}

Since~\(\psi\) is isometric, \(\psi_x\) is surjective when its range
is dense.
So the density condition in the lemma is weaker than surjectivity; we
use this weaker form in later applications.

\begin{proof}
  It is clear that~\(\psi_x\) has dense range if~\(\psi\) is an
  isomorphism.
  Conversely, assume this.  Since the topology of a field of Banach
  spaces is determined by the continuous sections, it is enough to
  prove that any \(\Cont_0\)\nb-section~\(h\) of~\(\B\) is the
  \(\psi\)\nb-image of a \(\Cont_0\)\nb-section of~\(\A\).
  We are going to find \(\Cont_0\)\nb-sections~\(f_\varepsilon\)
  of~\(\A\) for \(\varepsilon>0\) such that
  \(\norm{\psi(f_\varepsilon) - h}_\infty<\varepsilon\).
  Then the sections~\(f_\varepsilon\) form a Cauchy net converging
  towards a \(\psi\)\nb-preimage of~\(h\) because
  \[
    \norm{f_{\varepsilon_1}-f_{\varepsilon_2}}
    = \norm{\psi(f_{\varepsilon_1}) - h -(\psi(f_{\varepsilon_2})-h)}
    < 2 \max \{\varepsilon_1,\varepsilon_2\}.
  \]

  Fix \(\varepsilon>0\).
  For each \(x\in X\), there is a continuous section~\(f_x\) of~\(\A\)
  with \(\norm{\psi(f_x(x)) - h(x)}<\varepsilon\).
  Since the norm function on~\(\B\) is upper semicontinuous and
  \(\psi(f_x)-h\) is a continuous section, there is a
  neighbourhood~\(U_x\) of~\(x\) such that \(\norm{\psi(f_x(y)) -
    h(y)} < \varepsilon\) for all \(y\in U_x\).
  There is also a compact subset \(K\subseteq X\) such that
  \(\norm{h(y) - 0}<\varepsilon\) for \(y\notin K\).
  Cover~\(K\) by finitely many open subsets of the form~\(U_x\).
  Adding the complement of~\(K\) gives a finite cover of~\(X\).
  Using a partition of unity subordinate to this finite cover, we
  build a \(\Cont_0\)\nb-section~\(f_\varepsilon\) of~\(\A\) with
  \(\norm{\psi(f_\varepsilon(y)) - h(y)} < \varepsilon\) for all
  \(y\in X\).
\end{proof}

\begin{definition}
  \label{def:continuous_family_subspaces}
  Let~\(X\) be a locally compact space and let
  \(\B = (\B_x)_{x\in X}\) be a field of Banach spaces over~\(X\).
  A family of closed subspaces \(\A_x \subseteq \B_x\) is
  \emph{continuous} if for each \(x\in X\), \(a\in \A_x\), there is a
  continuous section~\(f\) of~\(\B\) with \(f(x) = a\) and
  \(f(y)\in \A_y\) for all \(y\in X\).
\end{definition}

\begin{lemma}
  \label{lem:subbundle_from_continuous_family}
  Let \(\A_x\subseteq \B_x\) for \(x\in X\) be a continuous family
  of subspaces.  Then there is a unique field~\(\A\) of Banach
  spaces with fibres~\(\A_x\) such that a continuous section
  of~\(\A\) is the same as a continuous section of~\(\B\) with
  values in~\(\A\).  The implied topology on~\(\A\) is the subspace
  topology from~\(\B\).
\end{lemma}

\begin{proof}
  The proof is left to the reader.
\end{proof}

\begin{lemma}
  \label{lem:subbundle_with_dense_set}
  Let~\(X\) be a locally compact space, let \(\B = (\B_x)_{x\in X}\)
  be a field of Banach spaces over~\(X\), and let
  \(\A_x \subseteq \B_x\) be closed subspaces such that for each
  \(x\in X\), the set of \(f(x) \in \A_x\) for continuous
  sections~\(f\) of~\(\B\) with \(f(y)\in \A_y\) for all \(y\in X\) is
  dense in~\(\A_x\).  Then the field of subspaces~\((\A_x)_{x\in X}\) is
  continuous.
\end{lemma}
\begin{proof}
  Let \(a\in \A_x\).
  We must build a continuous section~\(f\) as in
  Definition~\ref{def:continuous_family_subspaces}.
  By assumption, there are continuous sections~\(f_n\) of~\(\B\) with
  \(f_n(y) \in\A_y\) for all \(y\in X\) and with \(\norm{f_n(x) -
    a}<2^{-n}\) for all \(n\in\N\).
  Then the set of \(y\in X\) with \(\norm{f_n(y) - f_{n+1}(y)} <
  2^{-n-1}\) is an open neighbourhood of~\(x\).
  There are \(h_n\in\Cont_0(X)\) that are supported in this open
  neighbourhood with \(\norm{h_n}_\infty=1\) and \(h_n(x)=1\) for all
  \(n\in\N\).
  Then the series \(f_1(y) + \sum_{n=1}^\infty h_n(y) (f_{n+1}(y)-f_n(y))\) converges
  uniformly towards a continuous section~\(f\) of~\(\B\) with the
  required properties, that is, \(f(y) \in\A_y\) for all \(y\in X\)
  and \(f(x) = f_1(x) + (a - f_1(x)) = a\).
\end{proof}

\begin{remark}
  \label{rem:continuous_family_of_ideals}
  Let~\(X\) be a locally compact space and let
  \(\B = (\B_x)_{x\in X}\) be a field of \(\Cst\)\nb-algebras
  over~\(X\).  A continuous family of ideals in~\(\B\) is equivalent
  to an ideal in \(\Cont_0(X,\B)\).  In one direction, an ideal
  \(J\idealin \Cont_0(X,\B)\) gives a continuous family of ideals by
  \(I_x \defeq \ev_x(J)\) for all \(x\in X\).  Here
  \(\ev_x\colon \Cont_0(X,\B) \to \B_x\) denotes the evaluation map.
  The ideals~\(I_x\) determine~\(J\) uniquely.  In the other
  direction, the space of \(\Cont_0\)\nb-sections of a continuous
  family of ideals is an ideal in \(\Cont_0(X,\B)\).  These two
  constructions are inverse to each other.  (See also
  Corollary~\ref{cor:hereditary_subbundle_subalgebra} for a more
  general statement involving hereditary Fell subbundles.)
\end{remark}

Let \(\E\colon A\to B\) and \(\F\colon B\to D\) be
\(\Cst\)\nb-correspondences.  Their composite is the (balanced) tensor
product \(\E\otimes_B\F\)
(see~\cite{Lance:Hilbert_modules}*{Chapter~4}).  This is the
completion of the algebraic (balanced) tensor product
\(\E\odot_B\F\) with respect to the \(D\)\nb-valued inner product
\[
  \braket{\xi_1\otimes\eta_1}{\xi_2\otimes \eta_2}\defeq
  \braket{\eta_1}{\varphi(\braket{\xi_1}{\xi_2})\eta_2},
\]
where the homomorphism \(\varphi\colon B\to \Bound(\F)\)
gives the left \(B\)\nb-module structure of~\(\F\).  There is a
similar operation for fields of \(\Cst\)\nb-correspondences.  Let
\(\A\), \(\B\) and~\(\Bun[D]\) be fields of \(\Cst\)\nb-algebras
over~\(X\) and let \(\E\) and~\(\F\) be fields of
\(\Cst\)\nb-correspondences between them.  Then there is a unique
field of \(\Cst\)\nb-correspondences over~\(X\) between \(\A\)
and~\(\Bun[D]\) whose space of \(\Cont_0\)\nb-sections is the closed
linear span of sections of the form \(\xi\otimes \eta\) for
\(\xi\in\Cont_0(X,\E)\), \(\eta\in\Cont_0(X,\F)\).  We denote this
field of \(\Cst\)\nb-correspondences by \(\E \otimes_\B \F\).

\begin{lemma}
  \label{lem:sections_tensor_product_of_fields}
  There is a canonical isomorphism
  \begin{equation*}
    \Cont_0(X,\E) \otimes_{\Cont_0(X,\B)} \Cont_0(X,\F)
    \congto \Cont_0(X,\E \otimes_\B \F),\quad
    \xi\otimes \eta\mapsto \bigl[x\mapsto \xi(x) \otimes \eta(x)\bigr],
  \end{equation*}
  of \(\Cont_0(X,\A)\)-\(\Cont_0(X,\Bun[D])\)-correspondences.
\end{lemma}

\begin{proof}
  The map above is a \(\Cont_0(X,\A)\)-\(\Cont_0(X,\Bun[D])\)-bimodule
  map and isometric for the \(\Cont_0(X,\Bun[D])\)-valued inner
  products.
  Its image is dense by definition.
  So it is an isomorphism of
  \(\Cont_0(X,\A)\)-\(\Cont_0(X,\Bun[D])\)-correspondences by
  Lemma~\ref{lem:subfield_equality}.
  
\end{proof}

Let \(X\) and~\(Y\) be locally compact spaces and let
\(f\colon X \to Y\) be a continuous map.  Let~\(\lambda\) be a
continuous family of measures along the fibres of~\(f\).  (The
map~\(f\) must be open if~\(\lambda\) has full support.)  This gives
rise to a \(\Cont_0(X)\)-\(\Cont_0(Y)\)\nb-correspondence
\(\Lt^2(X,f,\lambda)\) as in
\cite{Buss-Holkar-Meyer:Universal}*{Definition~2.1}.  We now
decorate this construction with upper semicontinuous fields.  Let
\(\A\) and~\(\B\) be fields of \(\Cst\)\nb-algebras over \(X\)
and~\(Y\), respectively.  Let \(\Cont_0(X,\A)\) and
\(\Cont_0(Y,\B)\) denote their spaces of continuous sections
vanishing at~\(\infty\).  To build a \(\Cst\)\nb-correspondence
between \(\Cont_0(X,\A)\) and \(\Cont_0(Y,\B)\), we use a field of
\(\Cst\)\nb-correspondences \(\E = \bigsqcup_{x\in X} \E_x\) between
\(\A\) and~\(f^*\B\); equivalently, each~\(\E_x\) is an
\(\A_x\)-\(\B_{f(x)}\)-correspondence, and the multiplications and
inner products are continuous maps \(\A \times_X \E \to \E\),
\(\E \times_Y \B \to \E\), and \(\E \times_X \E \to \B\), where we
turn~\(\E\) into a space over~\(Y\) using~\(f\) to form the fibre
product over~\(Y\).  Let \(\Contc(X,\E)\) be the space of continuous sections
of~\(\E\) with compact support.  Equip \(\Contc(X,\E)\) with the
\(\Cont_0(X,\A)\)-\(\Cont_0(Y,\B)\)-bimodule structure
\[
  (\varphi\cdot \xi\cdot \psi)(x) \defeq \varphi(x)\cdot \xi(x)\cdot
  \psi(f(x))
\]
for \(\varphi\in\Cont_0(X,\A)\), \(\xi\in\Contc(X,\E)\),
\(\psi\in\Cont_0(Y,\B)\) and the \(\Cont_0(Y,\B)\)-valued inner
product
\[
  \braket{\xi_1}{\xi_2}(y) \defeq \int_X {}\braket{\xi_1(x)}{\xi_2(x)}
  \,\diff\lambda_y(x)
\]
for \(\xi_1,\xi_2\in \Contc(X,\mathcal{E})\) and \(y\in Y\).  The
continuity of the inner product follows from the following general
fact:

\begin{lemma}
  \label{lem:continuity_integral}
  Let \(f\colon X\to Y\) be a continuous map with a family of
  measures~\(\lambda\) along its fibres and let~\(\B\) be an upper
  semicontinuous field of Banach spaces over~\(Y\).  Then the
  integration map
  \(\lambda(\varphi)(y) \defeq \int_X \varphi(x) \,\diff\lambda_y(x)\)
  maps \(\Contc(X,f^*(\B))\) into \(\Contc(Y,\B)\).
\end{lemma}

\begin{proof}
  If \(\varphi(x) = h(x) b(f(x))\) for \(h\in\Contc(X)\) and \(b\in
  \Contc(Y,\B)\), then \(\lambda(\varphi)(y) = \bigl(\int_X h(x)
  \,\diff\lambda_y(x)\bigr) \cdot b(y)\) is a continuous section with
  compact support.
  This remains true for linear combinations of sections of this form,
  and also for norm limits of such linear combinations.
  Any continuous section of~\(f^*(\B)\) is of this form because of the
  definition of the topology on~\(f^*(\B)\).
\end{proof}

The bimodule structure and inner product on \(\Contc(X,\E)\) satisfy
the algebraic conditions for a \(\Cst\)\nb-correspondence, and the
inner product is positive.  So it defines a norm on
\(\Contc(X,\E)\).  Its completion is a \(\Cst\)\nb-correspondence
between \(\Cont_0(X,\A)\) and \(\Cont_0(Y,\B)\).  We denote it by
\(\Lt^2(X,\E,f,\lambda)\), and we write
\[
  \Cont_0(X,\A) \xrightarrow[\lambda, \E]{f} \Cont_0(Y,\B)
\]
in commuting diagrams to keep them readable.

Let \(\A\), \(\B\) and~\(\Bun[D]\) be fields of \(\Cst\)\nb-algebras
over locally compact spaces \(X\), \(Y\) and~\(Z\).
Let \(f\colon X \to Y\) and \(g\colon Y \to Z\) be continuous maps
with families of measures \(\lambda\) and~\(\mu\) along their fibres
and fields of \(\Cst\)\nb-correspondences \(\E\) and~\(\F\) over \(X\)
and~\(Y\), respectively.
Then the composite integration map
\[
  \mu\circ\lambda\colon \Contc(X)\to\Contc(Y)\to\Contc(Z),\qquad
  (\mu\circ\lambda)(\varphi)(z) = \int_Y \int_X
  \varphi(x)\,\diff\lambda_y(x)\,\diff\mu_z(y),
\]
describes a continuous family of measures \(\mu\circ\lambda\) along
\(g\circ f\).
Recall that \(\E \otimes_{f^*(\B)} f^*(\F)\) is the unique field of
\(\A\)-\((gf)^*\Bun[D]\)-correspondences over~\(X\) with fibres \(\E_x
\otimes_{\B_{f(x)}} \F_{f(x)}\) and \(\Cont_0\)\nb-sections isomorphic
to \(\Cont_0(X,\E) \otimes_{\Cont_0(X,f^*(\B))} \Cont_0(X,f^*(\F))\).

\begin{lemma}
  \label{lem:compose_corr_from_measure-family}
  With the notation above, the map
  \[
    \gamma\colon \Contc(X,\E)\odot \Contc(Y,\F) \to \Contc(X,\E
    \otimes_{f^*(\B)} f^*(\F)),\quad \gamma(\varphi\otimes\psi)(x)
    \defeq \varphi(x)\otimes \psi(f(x)),
  \]
  extends uniquely to an isomorphism of
  \(\Cont_0(X,\A)\)-\(\Cont_0(Z,\Bun[D])\)-correspondences
  \[
    \Lt^2(X,\E,f,\lambda) \otimes_{\Cont_0(Y,\B)} \Lt^2(Y,\F,g,\mu)
    \cong \Lt^2(X, \E \otimes_{f^*(\B)}f^*(\F),g\circ f,\mu\circ
    \lambda).
  \]
\end{lemma}

\begin{proof}
  The proof is exactly the same as for
  \cite{Buss-Holkar-Meyer:Universal}*{Lemma~2.3}.
  Easy computations show that the map~\(\gamma\) is a
  \(\Cont_0(X,\A)\)-\(\Cont_0(Z,\Bun[D])\)-bimodule homomorphism that
  preserves the \(\Cont_0(Z,\Bun[D])\)-valued inner products.
  The topology on \(f^*(\F)\) is defined so that sections of the form
  \(\psi(g(x))h(x)\) with \(g\in \Contc(Y,\F)\) and \(h\in \Contc(X)\)
  span a dense subspace of \(\Contc(X,f^*(\F))\) with respect to the
  inductive limit topology.
  Since \(\gamma(h\cdot \varphi\otimes\psi)(x)=h(x)\varphi(x)\otimes
  \psi(g(x))=\varphi(x)\otimes h(x)\psi(g(x))\) and sections of the
  form \(h(x) \psi(g(x))\) span a dense subspace of
  \(\Contc(X,f^*\F)\), it follows that~\(\gamma\) has dense image.
  Thus the isometry~\(\gamma\) is unitary.
\end{proof}

In our brief notation, the following diagram of
\(\Cst\)\nb-correspondences commutes up to the canonical
isomorphism~\(\gamma\) of
Lemma~\ref{lem:compose_corr_from_measure-family}:
\begin{equation}
  \label{eq:compose_measure_family}
  \begin{tikzpicture}[baseline=(current bounding box.west)]
    \matrix (m) [cd,column sep=4em, row sep=2.5em]
    {\Cont_0(X,\A)& \Cont_0(Y,\B)\\
      &\Cont_0(Z,\Bun[D]).\\
    };
    \meascor{m-1-1}{f}{\lambda,\E}{m-1-2}
    \meascor{m-1-2}{g}{\mu,\mathcal{F}}{m-2-2}
    \meascor{m-1-1}{g\circ f}{\mu\circ
      \lambda,\mathcal{E}\otimes_{f^*(\B)}f^*(\mathcal{F})}{m-2-2}
  \end{tikzpicture}
\end{equation}

The regular representation was defined
in~\cite{Buss-Holkar-Meyer:Universal} using topological
correspondences.
These generalise the construction above by adding a map backwards that
affects the left action.
For the same purpose, we will use topological correspondences
decorated with fields of \(\Cst\)\nb-correspondences.
As before, let \(X\) and~\(Y\) be locally compact spaces and let
\(\A\) and~\(\B\) be fields of \(\Cst\)\nb-algebras over \(X\)
and~\(Y\), respectively.
Let \(\Cont_0(X,\A)\) and \(\Cont_0(Y,\B)\) denote their spaces of
continuous sections vanishing at~\(\infty\).
Let~\(V\) be a third locally compact space equipped with continuous
maps \(b\colon V\to X\) and \(f\colon V\to Y\).
Let~\(\lambda\) be a continuous family of measures along the fibres
of~\(f\).
Let \(\E= (\E_v)_{v\in V}\) be a field of \(\Cst\)\nb-correspondences
between \(b^*(\A)\) and~\(f^*(\B)\); thus each~\(\E_v\) is an
\(\A_{b(v)}\)-\(\B_{f(v)}\)-correspondence.
Let \(\Contc(V,\E)\) be the space of continuous sections of~\(\E\)
with compact support.
Equip \(\Contc(V,\E)\) with the
\(\Cont_0(X,\A)\)-\(\Cont_0(Y,\B)\)-bimodule structure
\[
  (\varphi\cdot \xi\cdot \psi)(v) \defeq \varphi(b(v))\cdot
  \xi(v)\cdot \psi(f(v))
\]
for \(\varphi\in\Cont_0(X,\A)\), \(\xi\in\Contc(V,\E)\),
\(\psi\in\Cont_0(Y,\B)\) and the \(\Cont_0(Y,\B)\)-valued inner
product
\[
  \braket{\xi_1}{\xi_2}(y)
  \defeq \int_X {}\braket{\xi_1(x)}{\xi_2(x)} \,\diff\lambda_y(x).
\]
This is a pre-Hilbert \(\Cont_0(Y,\B)\)-module.
Let \(b^*\Lt^2(V,\E,f,\lambda)\) be its completion to a
\(\Cont_0(X,\A)\)-\(\Cont_0(Y,\B)\)-correspondence.

\begin{lemma}[compare
  \cite{Buss-Holkar-Meyer:Universal}*{Proposition~2.6}]
  \label{lem:functoriality-of-corr}
  Let \(X\), \(Y\) and~\(Z\) be locally compact spaces with fields
  of \(\Cst\)\nb-algebras \(\A\), \(\B\) and \(\Bun[D]\),
  respectively.
  Let
  \[
    \begin{tikzpicture}
      \matrix (m) [cd] { &(V,\E)
        &&(W,\F)&\\
        (X,\A)&&(Y,\B)&&(Z,\Bun[D])\\
      };
      \meascor{m-1-2}{f_V}{\lambda}{m-2-3}
      \meascor{m-1-4}{f_W}{\mu}{m-2-5}
      \meascor{m-1-2}{}{b_V}{m-2-1}
      \meascor{m-1-4}{}{b_W}{m-2-3}
    \end{tikzpicture}
  \]
  be decorated topological correspondences from \((X,\A)\) to
  \((Y,\B)\) and from \((Y,\B)\) to \((Z,\Bun[D])\), respectively.
  Let \(\pr_1\colon V\times_Y W \to V\) and
  \(\pr_2\colon V\times_Y W \to W\) be the coordinate projections and
  define \(b\colon V\times_Y W \to X\) and
  \(f\colon V\times_Y W \to Z\) by \(b(v,w) \defeq b_V(v)\) and
  \(f(v,w) \defeq f_W(w)\).
  Then
  \[
    \E \otimes_{\B} \F \defeq \pr_1^*(\E)
    \otimes_{(f_V\circ\pr_1)^*\B} \pr_2^*(\F)
  \]
  is a field of \(\Cst\)\nb-correspondences over \(V\times_Y W\)
  with fibres \(\E_v \otimes_{\B_{f_V(v)}} \F_w\) at
  \((v,w) \in V\times_Y W\).
  There is a continuous family of measures \(\lambda\times_Y\mu\)
  along~\(f\) defined by
  \[
    (\lambda\times_Y \mu) \psi(z)
    \defeq \int_W \int_V \psi(v,w) \,\diff\lambda_{b_W(w)}(v)
    \,\diff \mu_z(w).
  \]
  The canonical map
  \begin{align*}
    \gamma\colon\Contc(V,\E)\odot\Contc(W,\F)
    &\to \Contc(V\times_{f_V,Y,b_W}W, \E \otimes_{\B} \F), \\
    \gamma(\xi\otimes \eta)(v,w)
    &\defeq \xi(v)\otimes\eta(w),
  \end{align*}
  extends to an isomorphism of
  \(\Cont_0(X,\A)\)-\(\Cont_0(Z,\Bun[D])\)-correspondences
  \begin{multline*}
    b_V^* \Lt^2(V,\E,f_V,\lambda)\otimes_{\Cont_0(Y,\B)}
    b_W^* \Lt^2(W,\F,f_W,\mu) \\
    \congto b^* \Lt^2(V\times_{f_V,Y,b_W}W, \E \otimes_{\B}
    \F,f,\lambda\times_Y\mu).
  \end{multline*}
\end{lemma}

\begin{proof}
  Simple computations show that~\(\gamma\) is a bimodule map and an
  isometry for the right inner product.  Since it also has dense
  range, it is unitary by Lemma~\ref{lem:subfield_equality}.
\end{proof}

We shall later need that certain Hilbert modules built from
\(\Lt^2\)\nb-sections are full.
This will follow from the following lemma:

\begin{lemma}
  \label{lem:fullness-Hilbert-modules-fields}
  Let \(X\) and~\(Y\) be locally compact Hausdorff spaces,
  \(f\colon X\to Y\) a continuous map, \(\B=(\B_y)_{y\in Y}\) a
  field of \(\Cst\)\nb-algebras over~\(Y\), \(\E=(\E_x)_{x\in X}\) a
  field of Hilbert modules over~\(f^*\B\), and
  \(\lambda=(\lambda_y)_{y\in Y}\) a continuous family of measures
  along~\(f\) with full support.  The Hilbert
  \(\Cont_0(Y,\B)\)-module \(\Lt^2(X,\E,f,\lambda)\) is full if, for
  every \(y\in Y\), the linear span of
  \(\setgiven{\braket{\E_x}{\E_x}}{x\in f^{-1}(y)}\) is dense
  in~\(\B_y\).
\end{lemma}

\begin{proof}
  Recall that \(\Lt^2(X,\E,f,\lambda)\) is the completion of
  \(\Contc(X,\E)\) in the norm associated to the
  \(\Cont_0(Y,\B)\)\nb-valued inner product
  \[
    \braket{\xi}{\eta}(y) \defeq \int_X {}\braket{\xi(x)}{\eta(x)}
    \,\diff\lambda_y(x).
  \]
  The Hilbert \(\Cont_0(Y,\B)\)\nb-module \(\Lt^2(X,\E,f,\lambda)\)
  is isomorphic to the space of \(\Cont_0\)-sections of the field of
  Hilbert \(\B_y\)\nb-modules \(\Lt^2(f^{-1}(y),\E,\lambda_y)\).
  Lemma~\ref{lem:subfield_equality} implies that
  \(\Lt^2(X,\E,f,\lambda)\) is full if and only if
  \(\Lt^2(f^{-1}(y),\E,\lambda_y)\) is full as a Hilbert
  \(\B_y\)\nb-module for every \(y\in Y\).
  So we only need to prove the lemma in the case where~\(Y\) is a
  point.
  In other words, we must check that \(\Lt^2(X,\E,\lambda)\) is full
  as a Hilbert \(B\)\nb-module if~\(B\) is a \(\Cst\)\nb-algebra,
  \(\lambda\) a (Radon) measure on~\(X\) with full support, and
  \(\E=(\E_x)_{x\in X}\) a field of Hilbert \(B\)\nb-modules
  over~\(X\) such that the linear span of
  \(\setgiven{\braket{\E_x}{\E_x}}{x\in X}\) is dense in~\(B\).
  We fix \(x_0\in X\) and \(\xi_0\in \E_{x_0}\) and take a section
  \(\xi\in \Contc(X,\E)\) such that \(\xi(x_0)=\xi_0\).
  If \(\varepsilon>0\), the set \(U\defeq \setgiven{x\in X}
  {\norm{\braket{\xi(x)}{\xi(x)}-\braket{\xi_0}{\xi_0}}<\varepsilon}\)
  is an open neighbourhood of~\(x_0\) in~\(X\) because~\(\E\) is an
  upper semicontinuous field of Hilbert \(B\)\nb-modules.
  There is \(\psi\in \Contc(X)\) with \(\supp(\psi)\subseteq U\) and
  \(\int {}\abs{\psi(x)}^2\,\diff\lambda(x)=1\).
  Then
  \begin{multline*}
    \norm{\braket{\psi\cdot \xi}{\psi\cdot \xi}(x)-\braket{\xi_0}{\xi_0}}\\
    =\norm*{\int
      {}\abs{\psi(x)}^2\braket{\xi(x)}{\xi(x)}\,\diff\lambda(x) - \int
      {}\abs{\psi(x)}^2 \braket{\xi_0}{\xi_0} \,\diff\lambda(x)}\\
    \le \int {}\abs{\psi(x)}^2\norm{\braket{\xi(x)}{\xi(x)} -
      \braket{\xi_0}{\xi_0}} \,\diff\lambda(x) < \varepsilon.
  \end{multline*}
  This implies that \(\braket{\xi_0}{\xi_0}_B\) for
  \(\xi_0\in \E_{x_0}\) belongs to the ideal \(I\subseteq B\)
  generated by the \(B\)\nb-valued inner product on
  \(\Lt^2(X,\E,\lambda)\).  The polarisation identity now implies
  \(\braket{\E_{x_0}}{\E_{x_0}}\subseteq I\).  Then \(I=B\) by
  assumption.
\end{proof}

\section{Fell bundles over groupoids}
\label{sec:Fell_bundles}

Fell bundles over groups have been introduced in
\cites{Fell:induced, Doran-Fell:Representations} and were
generalised to groupoids in \cites{Yamagami:IdealStructure,
  Kumjian:Fell_bundles}.
They are treated in depth
in~\cite{Muhly-Williams:Equivalence.FellBundles}.
Unlike the Fell bundles
in~\cite{Muhly-Williams:Equivalence.FellBundles}, ours need not be
saturated, their fibres need not be separable, and~\(G\) need not be
second countable.
Let~\(G\) be a locally compact, Hausdorff groupoid.
As usual, we also write~\(G\) for its arrow space.
Let~\(G^0\) be its object space, \(\s,\rg\colon G\rightrightarrows
G^0\) the source and range maps, and \(u\colon G^0 \to G\) the unit
map.
Throughout the paper, we identify units \(x\in G^0\) with their
identity arrows \(u(x)\in G\), and we often write \(x\) instead of
\(u(x)\) when no confusion can arise.
Let
\begin{align}
  \label{eq:def_G2}
  G^2 &\defeq G \times_{\s,G^0,\rg} G
        = \setgiven{(g,h)\in G}{\s(g) = \rg(h)},\\
  \label{eq:def_G3}
  G^3 &\defeq G \times_{\s,G^0,\rg} G \times_{\s,G^0,\rg} G
        = \setgiven{(g,h,k)\in G}{\s(g) = \rg(h),\ \s(h) = \rg(k)}.
\end{align}
We shall also need the three maps \(d_0,d_1,d_2\colon G^2\to G\)
defined by
\[
  d_0(g,h) = h,\qquad d_1(g,h) = g\cdot h,\qquad d_2(g,h) = g
\]
for \(g,h\in G\) with \(\s(g)=\rg(h)\), and the composite maps
\begin{equation}
  \label{eq:vertex_maps}
  v_0 = \rg\circ d_1 = \rg\circ d_2,\qquad
  v_1 = \rg\circ d_0 = \s\circ d_2,\qquad
  v_2 = \s\circ d_0 = \s\circ d_1.
\end{equation}
These maps are illustrated in Figure~\ref{fig:triangle_arrows}.
\begin{figure}[htbp]
  \begin{tikzpicture}[xscale=1.8]
    \node (v1) at (90:1) {\(v_1(g,h)\)};
    \node (v0) at (210:1) {\(v_0(g,h)\)};
    \node (v2) at (330:1) {\(v_2(g,h)\)};
    \draw[cdar] (v1) -- node[swap] {\(\scriptstyle d_2(g,h)=g\)} (v0);
    \draw[cdar] (v2) -- node {\(\scriptstyle d_1(g,h)=gh\)} (v0);
    \draw[cdar] (v2) -- node[swap] {\(\scriptstyle d_0(g,h)=h\)} (v1);
  \end{tikzpicture}
  \caption{A pair \((g,h)\in G^2\) of composable arrows generates a
    commutative triangle of arrows in~\(G\).  We number the edges so
    that the one opposite the vertex~\(v_i(g,h)\) is~\(d_i(g,h)\)
    for \(i=0,1,2\).}
  \label{fig:triangle_arrows}
\end{figure}
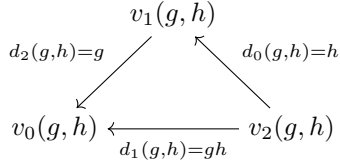

\begin{definition}
  \label{def:Fell_bundle}
  A \emph{Fell bundle over~\(G\)} is an upper semicontinuous field of
  Banach spaces \(\A=(\A_g)_{g\in G}\) with bilinear maps
  \(\A_g \times \A_h \to \A_{g\cdot h}\), \((a,b)\mapsto a\cdot b\),
  for \((g,h)\in G^2\) (\emph{multiplication}) and conjugate-linear maps
  \(\A_g \to \A_{g^{-1}}\), \(a\mapsto a^*\), for \(g\in G\)
  (\emph{involution}) with the following properties:
  \begin{enumerate}
  \item \label{def:Fell_bundle_1}%
    \((a\cdot b)\cdot c = a\cdot (b\cdot c)\) for all
    \((g,h,k)\in G^3\), \(a\in \A_g\), \(b\in \A_h\), \(c\in \A_k\);
  \item \label{def:Fell_bundle_2}%
    \((a\cdot b)^* = b^* \cdot a^*\) for all \((g,h)\in G^2\),
    \(a\in \A_g\), \(b\in \A_h\);
  \item \label{def:Fell_bundle_3}%
    \((a^*)^* =a\) for all \(g\in G\), \(a\in \A_g\);
  \item \label{def:Fell_bundle_4}%
    \(\norm{a\cdot b} \le \norm{a}\cdot \norm{b}\) for all
    \((g,h)\in G^2\), \(a\in \A_g\), \(b\in \A_h\);
  \item \label{def:Fell_bundle_5}%
    \(\norm{a^* a} = \norm{a}^2\) for all \(g\in G\), \(a\in \A_g\);
  \item \label{def:Fell_bundle_6}%
    for each \(g\in G\), \(a\in \A_g\), there is \(b\in \A_{u(\s(g))}\)
    with \(b^* b = a^* a\);
  \item \label{def:Fell_bundle_7}%
    the multiplication and involution are continuous as maps
    \(\A\times_{G^0} \A \to \A\) and \(\A \to \A\), respectively,
    where
    \[
      \A\times_{G^0} \A \defeq \setgiven{(a,b)\in \A_g\times \A_h}
      {g,h\in G,\ \s(g)=\rg(h)} \subseteq \A \times \A.
    \]
  \end{enumerate}
  Let \(\A_g \A_h \subseteq \A_{g h}\) denote the closed linear span
  of products \(a\cdot b\) for \(a\in \A_g\), \(b\in \A_h\).  The Fell
  bundle is called \emph{saturated} if \(\A_g \A_h = \A_{g h}\) for
  all \((g,h)\in G^2\).
\end{definition}

Conditions~\ref{def:Fell_bundle_1}--\ref{def:Fell_bundle_5} for
\(g=h=k=u(x)\) with \(x\in G^0\) say that~\(\A_{u(x)}\) with its
multiplication, involution and norm is a \(\Cst\)\nb-algebra.
Thus the pullback~\(u^*\A\) of~\(\A\) for the inclusion map \(u\colon
G^0\hookrightarrow G\), that is, the restriction of~\(\A\) to~\(G^0\),
becomes an upper semicontinuous field of \(\Cst\)\nb-algebras
over~\(G^0\).
Throughout the article, we denote this field by \(\FU\defeq u^*\A\),
its fibres by \(\FU_x\defeq \A_{u(x)}\) for \(x\in G^0\), and we let
\(\SUF\defeq \Cont_0(G^0,\FU)\).

If \(g\in G\), then we equip~\(\A_g\) with the products
\(\FU_{\rg(g)} \times \A_g \to \A_g\) and
\(\A_g \times \FU_{\s(g)} \to \A_g\) and the left and right inner
products \(\BRAKET{a}{b} \defeq a \cdot b^*\) and
\(\braket{a}{b} \defeq a^* \cdot b\).
This makes it a Hilbert \(\FU_{\rg(g)}\)-\(\FU_{\s(g)}\)-bimodule.
The right inner product is positive by~\ref{def:Fell_bundle_6}, and
the left one because \(\BRAKET{a}{b} = \braket{a^*}{b^*}\).
Condition~\ref{def:Fell_bundle_3} implies that the involution
\(^*\colon \A_g\to\A_{g^{-1}}\) for \(g\in G\) is bijective.
Therefore, we also write~\(\A_g^*\) for~\(\A_{g^{-1}}\).
The definition also implies \(\norm{a^*} = \norm{a}\) for all
\(a\in\A_g\), \(g\in G\).
The closed linear spans of products
\[
  \A_g \A_{g^{-1}} = \A_g \A_g^* = \BRAKET{\A_g}{\A_g},\qquad
  \A_{g^{-1}} \A_g = \A_g^* \A_g = \braket{\A_g}{\A_g}
\]
are closed ideals in \(\FU_{\rg(g)}\) and~\(\FU_{\s(g)}\),
respectively.
By construction, \(\A_g\) is an imprimitivity bimodule between these
two ideals.

The product \(\A_g \times \A_h \to \A_{g h}\) for \((g,h)\in G^2\)
induces a Hilbert bimodule map
\begin{equation}
  \label{eq:multiplication_isometries}
  m_{g,h}\colon \A_g \otimes_{\FU_{\s(g)}} \A_h \hookrightarrow \A_{g h}.
\end{equation}
Its range is equal to \(\A_g \A_h \subseteq \A_{g h}\).  This is a
Hilbert subbimodule of~\(\A_{g h}\).

We define a field of \(\Cst\)\nb-algebras~\(\A\A^*\) over~\(G\) with
fibres~\(\A_g\A_g^*\).  First, there is a field of Hilbert bimodules
\(\pr_1^*(\A)\otimes_{\FU} \pr_2^*(\A)\) over
\(G\times_{\s,G^0,\rg}G\).  Its fibre at \((g,h)\in G^2\) with
\(\s(g) = \rg(h) = y\) is \(\A_g \otimes_{\FU_y} \A_h\), and its
\(\Cont_0\)\nb-sections are generated by
\((g,h) \mapsto f_1(g) \otimes f_2(h)\) for \(\Cont_0\)\nb-sections
\(f_1,f_2\) of~\(\A\).  Let~\(\A\A^*\) denote the pullback of
\(\pr_1^*(\A)\otimes_{\FU} \pr_2^*(\A)\) along the map
\(G\to G\times_{\s,G^0,\rg}G\), \(g\mapsto (g,g^{-1})\).  The fibre
of this pullback at \(g\in G\) is
\(\A_g\A^*_g \cong \A_g \otimes_{\FU_y} \A_g^*\).  The
multiplication in the Fell bundle maps~\(\A_g\A_g^*\) isometrically onto a closed
ideal in~\(\FU_{\rg(g)}\).  This is a continuous family of
ideals~\(\A\A^*\) in~\(\rg^*(\FU)\) as in
Definition~\ref{def:continuous_family_subspaces} by
Lemma~\ref{lem:subbundle_with_dense_set}.  In particular,
\(\A \A^*\) becomes a field of \(\Cst\)\nb-algebras over~\(G\).  By
construction, its restriction to~\(G^0\) is equal to
\(\A|_{G^0}=\FU\).  Similarly, there is a continuous family of
ideals~\(\A^*\A\) in~\(\s^*(\FU)\) with fibres
\(\A_g^* \otimes_{\FU_{\rg(g)}} \A_g \cong \A_g^*\A_g \subseteq
\FU_{\s(g)}\).

\begin{remark}
  \label{rem:multiply_fibres}
  Let \(n\in\N\) and let \((g_1,\dotsc,g_n)\) be a chain of
  composable arrows.
  Then we often tacitly identify the Hilbert bimodule tensor product
  \[
    \A_{g_1} \A_{g_2} \dotsm \A_{g_n}
    \defeq \A_{g_1} \otimes_{\FU_{\s(g_1)}} \A_{g_2}
    \otimes_{\FU_{\s(g_2)}} \dotsb \otimes_{\FU_{\s(g_{n-1})}} \A_{g_n}
  \]
  with its image under the iterated multiplication map
  to~\(\A_{g_1\cdot g_2\dotsm g_n}\).
  The iterated multiplication map is an isometry of Hilbert bimodules.
  It does not depend on how we put parentheses, that is, it is
  associative.
  Let~\(G^n\) be the space of composable \(n\)\nb-chains in~\(G\).
  The Hilbert bimodules \(\A_{g_1} \A_{g_2} \dotsm \A_{g_n}\) and
  \(\A_{g_1 \cdot g_2\dotsm g_n}\) are the fibres of continuous
  fields of Hilbert bimodules over~\(G^n\), and the iterated
  multiplication map identifies the former with a subfield of the
  latter.  We may also write~\(\A_{g_j^{-1}}^*\) instead
  of~\(\A_{g_j}\) in such chains.  Often, there are relations among
  the elements \(g_1,\dotsc,g_n\).  Namely, we pull the bundles
  above back along a map \(G^m \to G^n\) for some \(m\in\N\) that is
  defined using the groupoid structure.  For instance, the next
  lemma will consider the products \(\A_{g^{-1}}\A_{g h}\)
  and~\(\A_g \A_g^* \A_g\) for \((g,h)\in G^2\).
\end{remark}

\begin{lemma}
  \label{lem:equality_of_products}
  Let~\(\A\) be a Fell bundle over~\(G\).
  \begin{enumerate}
  \item \label{lem:equality_of_products_1}%
    If \(g\in G\), then
    \[
      \A_g \A_{g^{-1}} \A_g = \A_g,\qquad \FU_{\rg(g)} \A_g = \A_g =
      \A_g \FU_{\s(g)}.
    \]
  \item \label{lem:equality_of_products_2}%
    If \((g,h)\in G^2\), then
    \[
      \A_{g^{-1}} \A_{g h} = \A_{g^{-1}} \A_g \A_h,\qquad \A_{g h}
      \A_{h^{-1}} = \A_g \A_h \A_{h^{-1}}
    \]
    as subspaces of \(\A_h\) and~\(\A_g\), respectively.

  \item \label{lem:equality_of_products_3}%
    The Fell bundle~\(\A\) is saturated if and only if
    \(\A_g\A_{g^{-1}}=\FU_{\rg(g)}\) for all \(g\in G\), if and only
    if \(\A_{g^{-1}}\A_g=\FU_{\s(g)}\) for all \(g\in G\).
  \end{enumerate}
  These equalities are the restrictions to the fibres of
  isomorphisms between suitable continuous fields of Hilbert
  bimodules.
\end{lemma}

\begin{proof}
  We prove~\ref{lem:equality_of_products_1}.  Any Hilbert
  module~\(\E\) is nondegenerate as a module
  over~\(\braket{\E}{\E}\).  Applied to~\(\A_g\), this says that
  \(\A_g = \A_g \braket{\A_g}{\A_g} = \A_g \A_{g^{-1}} \A_g\).  Since
  \(\FU_{\rg(g)} \supseteq \A_g \A_{g^{-1}}\) and
  \(\FU_{\s(g)} \supseteq \A_{g^{-1}} \A_g\), this implies
  \(\FU_{\rg(g)} \A_g = \A_g\) and \(\A_g \FU_{\s(g)} = \A_g\).

  Next, we prove~\ref{lem:equality_of_products_2}.  We use
  \(\A_g \A_h \subseteq \A_{g h}\) and
  \(\A_{g^{-1}} \A_{g h} \subseteq \A_h\)
  and~\ref{lem:equality_of_products_1} to prove
  \[
    \A_{g^{-1}}\A_g\A_h \subseteq \A_{g^{-1}} \A_{gh} = \A_{g^{-1}}
    \A_g \A_{g^{-1}} \A_{gh} \subseteq \A_{g^{-1}} \A_g \A_h.
  \]
  Then \(\A_{g^{-1}} \A_{g h} = \A_{g^{-1}} \A_g \A_h\).  This
  implies \(\A_{gh} \A_{h^{-1}}=\A_g\A_h\A_{h^{-1}}\) by taking
  adjoints and inverting \(g\) and~\(h\).  The Hilbert bimodules
  \(\A_g\), \(\FU_{\rg(g)} \A_g\), \(\A_g \FU_{\s(g)}\), and
  \(\A_g \A_{g^{-1}} \A_g\) are the fibres of the fields of Hilbert
  \(\rg^*\FU\)-\(\s^*\FU\)-bimodules
  \[
    \A,\quad \rg^*\FU \otimes_{\rg^*\FU} \A,\quad \A \otimes_{\s^*\FU}
    \s^*\FU,\quad \A \otimes_{\rg^*\FU} \inv^*\A\otimes_{\s^*\FU} \A,
  \]
  respectively.  Here~\(\inv\) denotes the inversion map.  The
  multiplication in the Fell bundle gives Hilbert bimodule maps
  between these fields.  We have shown above that these are
  isomorphisms fibrewise.  Then they are isomorphisms of fields by
  Lemma~\ref{lem:subfield_equality}.  In fact, this may also
  be shown directly by applying the proof that
  \(\A_g = \FU_{\rg(g)} \A_g = \A_g \FU_{\s(g)} = \A_g \A_{g^{-1}}
  \A_g\) to the fields instead of their fibres.  We wrote the argument
  down fibrewise to simplify the notation.  Similarly, the map
  \[
    (\inv\circ d_2)^*\A \otimes_{v_0^*\FU} d_2^*\A \otimes_{v_1^*\FU}
    d_0^* \A \to (\inv\circ d_2)^*\A \otimes_{v_0^*\FU} d_1^*\A
  \]
  induced by the product in the Fell bundle is an isomorphism of
  fields of Hilbert bimodules over~\(G^2\) because it is an
  isomorphism between the fibres \(\A_{g^{-1}} \A_g \A_h\) and
  \(\A_{g^{-1}} \A_{g h}\) of these fields.

  Finally, we prove~\ref{lem:equality_of_products_3}.  If~\(\A\) is
  saturated, then \(\A_g\A_{g^{-1}}=\A_{gg^{-1}}=\FU_{\rg(g)}\) for
  all \(g\in G\).  Using the inversion in~\(G\), this is
  equivalent to \(\A_{g^{-1}}\A_g=\FU_{\s(g)}\) for all \(g\in G\).
  Conversely, assume that this is true.  Then
  \ref{lem:equality_of_products_1}
  and~\ref{lem:equality_of_products_2} imply
  \[
    \A_{gh}
    = \A_{gh} \FU_{\s(gh)}
    = \A_{gh} \FU_{\s(h)}
    = \A_{gh}\A_{h^{-1}} \A_h
    = \A_g \A_h \A_{h^{-1}} \A_h
    = \A_g \A_h
  \]
  for all \((g,h)\in G^2\).  Thus~\(\A\) is saturated.
\end{proof}

\subsection{Full section C*-algebras}
\label{sec:full_section_cstar-algebra}

Let~\(G\) be a locally compact, Hausdorff groupoid and let~\(\A\) be
a Fell bundle over~\(G\).  Let \(\alpha = (\alpha^x)_{x\in G^0}\) be
a left invariant Haar system on~\(G\) and let
\(\tilde\alpha = (\tilde\alpha_x)_{x\in G^0}\) be the corresponding
right invariant Haar system, defined by
\(\tilde\alpha_x(S) \defeq \alpha^x(S^{-1})\) for
\(S\subseteq G_x\).  As in~\cite{Buss-Holkar-Meyer:Universal}, we
define continuous families of measures~\(\lambda_i\) along the
maps~\(d_i\) for \(i=0,1,2\) in Figure~\ref{fig:triangle_arrows} by
\begin{align}
  \label{eq:lambda_definition0}
  (\lambda_0 f)(h)
  &= \int_{G_{\rg(h)}} f(g,h) \,\diff\tilde{\alpha}_{\rg(h)}(g),\\
  \label{eq:lambda_definition1}
  (\lambda_1 f)(k)
  &= \int_{G^{\rg(k)}} f(g,g^{-1}k) \,\diff\alpha^{\rg(k)}(g)
    = \int_{G_{\s(k)}} f(kh^{-1},h) \,\diff\tilde{\alpha}_{\s(k)}(h),\\
  \label{eq:lambda_definition2}
  (\lambda_2 f)(g)
  &= \int_{G^{\s(g)}} f(g,h) \,\diff\alpha^{\s(g)}(h),
\end{align}
and continuous families of measures~\(\mu_i\) along the maps~\(v_i\)
by
\begin{align}
  \label{eq:compose_integration_maps0}
  \mu_0
  &\defeq \alpha\circ \lambda_1
    = \alpha\circ \lambda_2,\\
  \label{eq:compose_integration_maps1}
  \mu_1
  &\defeq \alpha\circ \lambda_0
    = \tilde{\alpha}\circ \lambda_2,\\
  \label{eq:compose_integration_maps2}
  \mu_2
  &\defeq \tilde{\alpha}\circ \lambda_0
    = \tilde{\alpha}\circ \lambda_1.
\end{align}
Here we identify a continuous family of measures with the
integration map that it defines.
Equation~\eqref{eq:compose_integration_maps2} is equivalent to the
following equality of integrals for all \(f\in \Contc(G^2)\),
\(x\in G^0\):
\begin{equation}
  \label{eq:compose_integration_nontrivial}
  \int\displaylimits_{G_x} \int\displaylimits_{G_{\rg(h)}}
  f(g,h) \,\diff \tilde\alpha_{\rg(h)}(g) \,\diff\tilde\alpha_x(h)
  = \int\displaylimits_{G_x} \int\displaylimits_{G^{\rg(k)}}
  f(g,g^{-1} k) \,\diff \alpha^{\rg(k)}(g) \,\diff\tilde\alpha_x(k).
\end{equation}

We use the algebraic operations of a Fell bundle~\(\A\) to turn the
space \(\Contc(G,\A)\) of compactly supported continuous sections
\(G\to \A\) into a \Star{}algebra.  The involution is defined by
\(\xi^*(g)\defeq \xi(g^{-1})^*\).  The multiplication is the
convolution defined by
\[
  (\xi*\eta)(g)\defeq \int_{G} \xi(h)\eta(h^{-1}g)
  \,\diff\alpha^{\rg(g)}(h).
\]
To see that this is well defined, we first realise that
\((h,k)\mapsto \xi(h)\eta(k)\) is a continuous, compactly supported
section of the field of Hilbert bimodules~\(d_1^*(\A)\).  The integral
above applies the integration map~\(\lambda_1\) to this section and
thus gives an element of \(\Contc(G,\A)\) by
Lemma~\ref{lem:continuity_integral}.  Standard computations show that
\(\Contc(G,\A)\) with these operations is a \Star{}algebra.  The proof
that the convolution is associative is the same as in the scalar case
of groupoids without Fell bundles, except that we also use the
associativity of the underlying Fell bundles.  For the sake of
completeness we add this computation.  Let
\(\xi,\eta,\zeta\in \Contc(G,\A)\).  Then
\begin{align*}
  (\xi*\eta)*\zeta(g)
  &=\int_G {}\left(\int_G
    \xi(k)\eta(k^{-1}h)\,\diff\alpha^{\rg(h)}(k)\right)
    \zeta(h^{-1}g)\,\diff\alpha^{\rg(g)}(h)\\
  &=\int_G \int_G {}(\xi(k)\eta(k^{-1}h))\zeta(h^{-1}g)
    \,\diff\alpha^{\rg(g)}(k)\,\diff\alpha^{\rg(g)}(h)\\
  &=\int_G \int_G {}(\xi(k)\eta(k^{-1}h))\zeta(h^{-1}g)
    \,\diff\alpha^{\rg(k)}(h)\,\diff\alpha^{\rg(g)}(k)\\
  &=\int_G \int_G \xi(k)(\eta(h)\zeta(h^{-1}k^{-1}g))
    \,\diff\alpha^{\s(k)}(h)\,\diff\alpha^{\rg(g)}(k)\\
  &=\int_G \xi(k)(\eta*\zeta)(k^{-1}g)\,\diff\alpha^{\rg(g)}(k)
    = \xi*(\eta*\zeta)(g).
\end{align*}
In the step from line~3 to line~4, we substituted~\(k h\) for~\(h\),
which is possible because~\(\alpha\) is invariant under left
multiplication.

Let \(\xi\in\Contc(G,\A)\).  Since~\(\A\) is an upper semicontinuous
field, the function
\[
  \abs{\xi}\colon G\to \C,\qquad g\mapsto \norm{\xi(g)},
\]
is upper semicontinuous and hence Borel.  For a compactly supported
Borel function \(f\colon G\to \C\), we define
\(\alpha(f)\colon G^0\to \C\) by
\(\alpha(f)(x)\defeq \int_G f(g)\,\diff\alpha^x(g)\).  Now
\[
  \norm{\xi}_I\defeq \max \left\{\norm{\alpha(\abs{\xi})}_\infty,
    \norm{\tilde\alpha(\abs{\xi})}_\infty\right\}
\]
defines a norm on~\(\Contc(G,\A)\).
Well known computations show that this norm turns \(\Contc(G,\A)\)
into a normed \Star{}algebra, that is, \(\norm{\xi*\eta}_I \le
\norm{\xi}_I\cdot \norm{\eta}_I\) and \(\norm{\xi}_I=\norm{\xi^*}_I\)
(see \cite{Renault:Groupoid_Cstar}*{Proposition~1.4}).
Therefore, the completion of \(\Contc(G,\A)\) with respect to
\(\norm{\cdot}_I\) is a Banach \Star{}algebra, denoted \(L^1_I(\A)\)
or \(L^1_I(G,\A)\).

\begin{definition}
  \label{def:full_section_Cstar-algebra}
  The \emph{full section \(\Cst\)\nb-algebra}
  \(\Cst(G,\A)=\Cst(G,\A,\alpha)\) of~\(\A\) is the enveloping
  \(\Cst\)\nb-algebra of~\(L^1_I(\A)\).
\end{definition}

Equivalently, \(\Cst(\A)\) is the (Hausdorff) completion
of~\(\Contc(G,\A)\) in the maximal \(\Cst\)\nb-seminorm
on~\(\Contc(G,\A)\) with \(\norm{\xi}\le \norm{\xi}_I\) for all
\(\xi\in\Contc(G,\A)\).  This seminorm is, in fact, a norm (see
Proposition~\ref{pro:section_algebra_has_faithful_rep}).  Our main
results imply that this norm is also the maximal \(\Cst\)\nb-seminorm that
is continuous in the inductive limit topology.

\subsection{Twisted partial actions and Fell bundles}
\label{sec:actions_partial_twisted}

Fell bundles over groupoids contain other kinds of
generalised actions of groupoids as special cases.  The most basic
notion is that of ordinary (global) actions by automorphisms.  These
have been introduced by Le Gall~\cite{LeGall:KK_groupoid}.  He
defines a \emph{continuous action} of~\(G\) as a
\(\Cont_0(G^0)\)-\(\Cst\)-algebra~\(\SUF\) with an isomorphism of
\(\Cont_0(G)\)-\(\Cst\)-algebras
\(\act\colon \s^* \SUF \congto \rg^* \SUF\) such that
\(\act_{g h} = \act_g\act_h\) for all \((g,h)\in G^2\); we use unusual
symbols for actions because \(\alpha\) and~\(\beta\) usually denote
Haar systems in this article.
We turn the \(\Cont_0(G^0)\)-\(\Cst\)-algebra~\(\SUF\) into an upper
semicontinuous field of \(\Cst\)\nb-algebras
\(\FU= (\FU_x)_{x\in G^0}\) as usual (see~\cite{Nilsen:Bundles}).
Namely, \(\FU_x \defeq \SUF/(\Cont_0(X\setminus\{x\})\cdot \SUF)\) and
\(\SUF = \Cont_0(G^0,\FU)\).  Let \(\A\defeq \rg^*\FU\) be the
pullback along the range map.  This is a field of \(\Cst\)\nb-algebras
over~\(G\) with fibres~\(\FU_{\rg(g)}\).  To distinguish elements
of~\(\A_g\) and~\(\FU_{\rg(g)}\), we write \(a\delta_g\in \A_g\) for
\(a\in \FU_{\rg(g)}\).  The product and involution on~\(\A\) are
defined by
\[
  (a\delta_g)\cdot (b\delta_h) \defeq a \cdot \act_g(b)
  \,\delta_{gh}, \qquad (a\delta_g)^* \defeq
  \act_{g^{-1}}(a^*)\,\delta_{g^{-1}}
\]
for \((g,h)\in G^2\), \(a\in \A_g = \FU_{\rg(g)}\), and
\(b\in \A_h= \FU_{\rg(h)} = \FU_{\s(g)}\).  It is easy to check that
this is a saturated Fell bundle over~\(G\).

\begin{definition}
  \label{def:continuous_twisted_action}
  A \emph{continuous twisted action} of~\(G\) is an upper
  semicontinuous field of \(\Cst\)\nb-algebras~\(\FU\) over~\(G^0\)
  with an isomorphism of fields of \(\Cst\)\nb-algebras
  \(\act\colon \s^* \FU \congto \rg^* \FU\) and with unitary
  multipliers \(w(g,h)\in \UM(\FU_{\rg(g)})\) for \((g,h)\in G^2\)
  such that \(w(g,u(\s(g)))=1\) and \(w(u(\rg(g)),g)=1\) for all
  \(g\in G\) and
  \begin{alignat}{2}
    \label{eq:twisted_action_1}
    \act_g\act_h &= \Ad(w(g,h))\cdot \act_{g h}
    &\qquad&\text{for all }(g,h)\in G^2,\\
    \label{eq:twisted_action_2}
    \act_g(w(h,k)) w(g,h\cdot k) &= w(g,h)\cdot w(g\cdot h,k)
    &\qquad&\text{for all }(g,h,k)\in G^3;
  \end{alignat}
  here \(G^2\) and~\(G^3\) are as in \eqref{eq:def_G2}
  and~\eqref{eq:def_G3}.
  In addition, we require that pointwise multiplication by \(w(g,h)\)
  preserves continuity for sections of the field \(v_0^*(\FU)\)
  over~\(G^2\).
\end{definition}

Since multiplication by~\(w(g,h)\) maps continuous sections to
continuous sections and is fibrewise bijective, it is bijective on
sections globally by Lemma~\ref{lem:subfield_equality}.
Therefore, pointwise multiplication by~\(w(g,h)^*\) also preserves
continuity.
So we may briefly call~\(w\) \emph{strictly continuous}.

The conditions above imply \(\act_{u(x)} = \Id_{\A_x}\) for all \(x\in
G^0\): first, \eqref{eq:twisted_action_1} and the normalisation
\(w(u(x),u(x))=1\) imply \(\act_{u(x)}\act_{u(x)} = \act_{u(x)}\),
then we use that~\(\act_{u(x)}\) is invertible.

A continuous twisted action as above also yields a Fell bundle.  As
for an ordinary action, we let \(\A\defeq \rg^*\FU\), viewed as a
field of Banach spaces over~\(G\).  We define the product
\(\A_g \times \A_h \to \A_{g h} = \FU_{\s(h)}\) and the involution
\(\A_g \to \A_{g^{-1}}\) by
\begin{equation}
  \label{eq:A2_for_twisted_action}
  (a\delta_g)\cdot (b\delta_h)
  \defeq a\cdot \act_g(b) \cdot w(g,h) \,\delta_{g h},
  \qquad
  (a\delta_g)^*
  \defeq \act_{g}^{-1}(a^*) \cdot w(g^{-1},g)^* \,\delta_{g^{-1}}
\end{equation}
for \((g,h)\in G^2\), \(a\in \A_g\), \(b\in \A_h\).  Notice
that~\(\act_g^{-1}\) may differ from~\(\act_{g^{-1}}\).  They are
related by
\begin{equation}
  \label{eq:act_inverse_with_w}
  \act_{g^{-1}}(a)=w(g^{-1},g)\act_g^{-1}(a)w(g^{-1},g)^*.
\end{equation}
This follows from~\eqref{eq:twisted_action_1}.
The formulas~\eqref{eq:A2_for_twisted_action} are exactly as in the case
of twisted actions of groups (see~\cite{Exel:TwistedPartialActions}).
The same arguments as in~\cite{Exel:TwistedPartialActions} show that
this defines a saturated Fell bundle over~\(G\).
We will prove a more general result in
Proposition~\ref{pro:FellBundle-twisted-partial}.

\begin{remark}
  Let~\(G\) be a locally compact group.  Then it is common to study
  measurable twisted actions (see, for instance,
  \cite{Busby-Smith:Representations_twisted_group}).  Exel and Laca
  show in~\cite{Exel-Laca:Continuous_Fell} that measurable twisted
  actions of~\(G\) also give \emph{continuous} Fell bundles
  over~\(G\).  We are not aware of a generalisation of this to
  groupoids.  It is clear that mere measurability is not enough
  because we need continuity of the bundle on~\(G^0\).
\end{remark}

\begin{definition}
  \label{def:twisted-partial-action-groupoid}
  A \emph{twisted partial action} of a locally compact groupoid~\(G\)
  has the following data:
  \begin{enumerate}
  \item \label{def:twisted-partial-action-groupoid_1}%
    an (upper semicontinuous) field of \(\Cst\)\nb-algebras
    \(\FU=(\FU_x)_{x\in G^0}\) over~\(G^0\);
  \item \label{def:twisted-partial-action-groupoid_2}%
    ideals \(\Bun[D]_g\idealin \FU_{\rg(g)}\) for all \(g\in G\);
  \item \label{def:twisted-partial-action-groupoid_3}%
    isomorphisms
    \(\act_g\colon \Bun[D]_{g^{-1}}\congto \Bun[D]_{g}\) for all
    \(g\in G\);
  \item \label{def:twisted-partial-action-groupoid_4}%
    unitary multipliers \(w(g,h)\) of \(\Bun[D]_g\cap \Bun[D]_{gh}\)
    for all \((g,h)\in G^2\).
  \end{enumerate}
  We require the following conditions:
  \begin{enumerate}[resume]
  \item \label{def:twisted-partial-action-groupoid_5}%
    \(\Bun[D]_{u(x)}= \FU_x\) and \(\act_{u(x)}=\Id_{\FU_x}\) for all
    \(x\in G^0\), and \(w(g,u(\s(g)))=1\) and \(w(u(\rg(g)),g)=1\) for all
    \(g\in G\);
  \item \label{def:twisted-partial-action-groupoid_6}%
    \(\act_g(\Bun[D]_{g^{-1}}\cap \Bun[D]_h)\subseteq \Bun[D]_{gh}\)
    for all \((g,h)\in G^2\);
  \item \label{def:twisted-partial-action-groupoid_7}%
    \(\act_g(\act_h(a))=w(g,h)\act_{gh}(a)w(g,h)^*\) for
    \((g,h)\in G^2\) and
    \(a\in \act_{h^{-1}}(\Bun[D]_{g^{-1}}\cap \Bun[D]_h)\);
  \item \label{def:twisted-partial-action-groupoid_8}%
    \(\act_g(a\cdot w(h,k))\cdot w(g,hk) = \act_g(a)\cdot
    w(g,h)\cdot w(gh,k)\) for all \((g,h,k)\in G^3\) and
    \(a\in \Bun[D]_{g^{-1}}\cap \Bun[D]_h\cap \Bun[D]_{hk}\); for
    this formula to make sense, we restrict multipliers of ideals to
    smaller ideals, so as to turn \(w(h,k) \in \UM(\Bun[D]_h\cap
    \Bun[D]_{h k})\) into a multiplier of \(\Bun[D]_{g^{-1}}\cap
    \Bun[D]_h\cap \Bun[D]_{hk}\) and \(w(g,h k)\), \(w(g,h)\)
    and~\(w(g h,k)\) into multipliers of \(\Bun[D]_g\cap \Bun[D]_{g
      h}\cap \Bun[D]_{g h k}\);
  \item \label{def:twisted-partial-action-groupoid_9}%
    the ideals \(\Bun[D]_g\idealin \FU_{\rg(g)}\) form a continuous
    family of ideals~\(\Bun[D]\) in~\(\rg^*(\FU)\);
  \item \label{def:twisted-partial-action-groupoid_10}%
    the isomorphisms~\(\act_g\) depend continuously on~\(g\), that
    is, they combine to an isomorphism
    \(\act\colon \inv^*\Bun[D]\congto \Bun[D]\);
  \item \label{def:twisted-partial-action-groupoid_11}%
    the unitaries \(w(g,h)\) depend continuously on
    \((g,h)\in G^2\), that is, they form a unitary multiplier of
    \(\Cont_0(G^2,J)\), where~\(J\) is the continuous family of
    ideals \(\Bun[D]_g\cap \Bun[D]_{gh}\) in~\(\rg^*(\FU)\).
  \end{enumerate}
  A \emph{partial action} is defined as a twisted partial action where all
  unitaries~\(w(g,h)\) are the unit multiplier of the appropriate
  ideal.
  A \emph{twisted action} is a twisted partial action with
  \(\Bun[D]_g=\FU_{\rg(g)}\) for all \(g\in G\).
  Twisted partial
  actions with \(w(g,h)=1\) for all \((g,h)\in G^2\) and
  \(\Bun[D]_g=\FU_{\rg(g)}\) for all \(g\in G\) are ordinary actions
  of~\(G\).
\end{definition}

Partial actions of discrete groups were introduced by McClanahan in
\cite{McClanahan:K-theory-partial}.
Exel extended this definition to locally compact groups and added
twists in~\cite{Exel:TwistedPartialActions}.
Some algebraists have considered partial actions of discrete groupoids
on rings (see
\cites{Bagio:Partial,Bagio-Paques:Partial,Bagio-Pinedo:Globalization}).
Anantharaman-Delaroche~\cite{AnantharamanDelaroch:Partial_actions}
defined partial actions of topological groupoids and assigned Fell
bundles to them.
The definitions in \cites{Exel:TwistedPartialActions,
  AnantharamanDelaroch:Partial_actions} are special cases of ours.

\begin{lemma}
  \label{lem:partial_action_domains}
  The conditions above imply~\eqref{eq:act_inverse_with_w} as an
  equality of partial maps, and \(\act_g(\Bun[D]_{g^{-1}}\cap
  \Bun[D]_h) = \Bun[D]_g\cap \Bun[D]_{gh}\) for all \((g,h)\in G^2\).
\end{lemma}

\begin{proof}
  The two partial maps in~\eqref{eq:act_inverse_with_w} have the same
  domain and codomain by~\ref{def:twisted-partial-action-groupoid_3},
  and~\ref{def:twisted-partial-action-groupoid_7} implies that they
  are inverse to each other.
  The inclusion \(\act_g(\Bun[D]_{g^{-1}}\cap \Bun[D]_h) \subseteq
  \Bun[D]_g\cap \Bun[D]_{gh}\) follows from
  \ref{def:twisted-partial-action-groupoid_3}
  and~\ref{def:twisted-partial-action-groupoid_6}.
  This inclusion for~\(g^{-1}\) and~\eqref{eq:act_inverse_with_w}
  imply the reverse inclusion: \(\Bun[D]_g\cap \Bun[D]_{gh} =
  \act_g\bigl(\act_{g^{-1}}(\Bun[D]_g\cap \Bun[D]_{gh})\bigr)
  \subseteq \act_g(\Bun[D]_{g^{-1}}\cap \Bun[D]_h)\).
\end{proof}

\begin{lemma}
  \label{lem:cocycle_inverse_adjoint}
  For a twisted partial action of~\(G\) and any \(g\in G\),
  \[
    \act_g(w(g^{-1},g)) = w(g,g^{-1}).
  \]
\end{lemma}

\begin{proof}
  The cocycle identity~\ref{def:twisted-partial-action-groupoid_8} for
  \((g, g^{-1}, g)\in G^3\) yields
  \[
    \act_g(a\cdot w(g^{-1},g))\cdot w(g,g^{-1}g) = \act_g(a)\cdot w(g,g^{-1})\cdot w(g g^{-1},g)
  \]
  for all \(a\in \Bun[D]_{g^{-1}}\cap
  \Bun[D]_{g^{-1}g} = \Bun[D]_{g^{-1}}\).
  Since \(g^{-1}g = u(\s(g))\) and \(g g^{-1} = u(\rg(g))\), the
  conditions \(w(g,u(\s(g))) = 1\) and \(w(u(\rg(g)),g) = 1\)
  in~\ref{def:twisted-partial-action-groupoid_5} simplify this to
  \[
    \act_g(a) \cdot \act_g(w(g^{-1},g)) = \act_g(a) \cdot w(g,g^{-1}).
  \]
  Since \(a\in \Bun[D]_{g^{-1}}\) is arbitrary, we may
  cancel~\(\act_g(a)\) here.
\end{proof}

\begin{definition}
  \label{def:Fell_twisted_partial_action}
  The \emph{Fell bundle associated to a twisted partial action} is
  defined as follows.  The underlying field of Banach spaces is
  \(\A\defeq \Bun[D]\) with the fibres \(\A_g=\Bun[D]_g\) for
  \(g\in G\).  We write~\(a\delta_g\) for \(a\in \Bun[D]_g\).  The
  algebraic operations on~\(\A\) are as in the case of groups
  in~\cite{Exel:TwistedPartialActions}.  Namely, if \((g,h)\in G^2\)
  and \(a\in \Bun[D]_g\), \(b\in \Bun[D]_h\), then
  \begin{align}
    \label{eq:Fell_twisted_partial_action_multiply}
    (a\delta_g)\cdot (b\delta_h)
    &\defeq
      \act_g(\act_g^{-1}(a)\cdot b)\cdot w(g,h) \,\delta_{gh},\\
    \label{eq:Fell_twisted_partial_action_involution}
    (a\delta_g)^*
    &\defeq
      \act_g^{-1}(a^*)\cdot w(g^{-1},g)^*\,\delta_{g^{-1}}.
  \end{align}
\end{definition}

\begin{proposition}
  \label{pro:FellBundle-twisted-partial}
  The above data defines a Fell bundle over~\(G\).
\end{proposition}

\begin{proof}
  The most difficult part is to show that the multiplication operation
  is associative.  The proof in~\cite{Exel:TwistedPartialActions} uses
  approximate units.  We prove this in a more algebraic way.  Let
  \((g,h,k)\in G^3\).  We must show that the two maps
  \begin{align*}
    \A_g\otimes_{\FU_{\s(g)}}\A_h\otimes_{\FU_{\s(h)}}\A_k
    &\to \A_{ghk}\\
    a\otimes b\otimes c
    &\mapsto ((a\delta_g)\cdot (b\delta_h))\cdot (c\delta_k),\\
    a\otimes b\otimes c
    &\mapsto (a\delta_g)\cdot ((b\delta_h)\cdot (c\delta_k)),
  \end{align*}
  are equal.  Let
  \[
    I\defeq \A_h\A_k\A_k^*\A_h^*,\quad J\defeq
    \A_h^*\A_g^*\A_g\A_h\A_k\A_k^*,\quad
    K\defeq\A_k^*\A_h^*\A_g^*\A_g\A_h\A_k.
  \]
  These are ideals of \(\FU_{\s(g)}=\FU_{\rg(h)}\),
  \(\FU_{\s(h)} = \FU_{\rg(k)}\) and~\(\FU_{\s(k)}\), respectively.
  We claim that it suffices to prove the equation
  \[
    ((a\delta_g)\cdot (b\delta_h))\cdot (c\delta_k) = (a\delta_g)\cdot
    ((b\delta_h)\cdot (c\delta_k))
  \]
  for \(a\in \A_g\cdot I\), \(b\in \A_h\cdot J\) and
  \(c\in \A_k\cdot K\).

  To prove the claim, we first compute \(\A_k\cdot K = J\cdot \A_k\)
  and
  \begin{align*}
    I\cdot \A_h\cdot J
    &=\A_h(\A_k\A_k^*)(\A_h^*\A_h)
      (\A_h^*\A_g^*\A_g\A_h)\A_k\A_k^*\\
    &=(\A_h\A_h^*)(\A_g^*\A_g)\A_h(\A_k\A_k^*)\\
    &=(\A_g^*\A_g)\A_h(\A_k\A_k^*).
  \end{align*}
  It follows that
  \begin{multline*}
    (\A_g\cdot I)\otimes_{\FU_{\s(g)}} (\A_h\cdot
    J)\otimes_{\FU_{\s(h)}}(\A_k\cdot K) = \A_g\otimes_{\A_x} (I\cdot
    \A_h\cdot J\cdot J) \otimes_{\A_y}\A_k \\ =
    \A_g\otimes_{\FU_{\s(g)}}(\A_g^*\A_g) \A_h(\A_k\A_k^*)
    \otimes_{\FU_{\s(h)}}\A_k =
    \A_g\otimes_{\FU_{\s(g)}}\A_h\otimes_{\FU_{\s(h)}}\A_k.
  \end{multline*}
  Therefore, associativity follows in general once it holds for
  \(a\in \A_g\cdot I\), \(b\in \A_h\cdot J\) and \(c\in \A_k\cdot K\).
  If
  \(b\in \A_h\cdot J = \A_h\A_h^*\A_g^*\A_g\A_h\A_k\A_k^* =
  \A_g^*\A_g\A_h\A_k\A_k^*\), then there are
  \(b_1\in \A_g^*\A_g=\Bun[D]_{g^{-1}}\), \(b_2\in \A_h\) and
  \(b_3\in \A_k\A_k^*=\Bun[D]_k\) with \(b=b_1 b_2 b_3\).  Thus~\(b\)
  belongs to the domain~\(\Bun[D]_{g^{-1}}\) of~\(\act_g\).  It
  follows that
  \[
    (a\delta_g)\cdot (b\delta_h) = a\act_g(b)w(g,h)\delta_{gh}
  \]
  for \(a\in \A_g\cdot I\) and \(b\in \A_h\cdot J\).  Let
  \(c\in \A_k\cdot K = \A_h^*\A_g^*\A_g\A_h\A_k\).  Then there are
  \(c_1\in \A_h^*\A_g^*\A_g\A_h\) and \(c_2\in \A_k\) with
  \(c=c_1c_2\).  The ideal \(\A_h^*\A_g^*\A_g\A_h\) in~\(\FU_{\s(h)}\)
  is equal to the domain
  \(\Bun[D]_{h^{-1}} \cap \act_h^{-1}(\Bun[D]_{g^{-1}})\)
  of~\(\act_g \act_h\), and it is contained in the domain
  \(\Bun[D]_{(gh)^{-1}}\) of~\(\act_{g h}\).  We compute
  \begin{align*}
    ((a\delta_g)\cdot (b\delta_h))\cdot (c\delta_k)
    &= a\act_g(b)w(g,h)\act_{gh}(c)w(gh,k)\delta_{ghk}
    \\&= a\act_g(b)(\act_g\circ\act_h)(c)w(g,h)w(gh,k)
    \delta_{ghk}
    \intertext{by Condition~\ref{def:twisted-partial-action-groupoid_7} in
    Definition~\ref{def:twisted-partial-action-groupoid}.  A similar
    computation gives}
    (a\delta_g)\cdot ((b\delta_h)\cdot (c\delta_k))
    &= a\act_g(b\act_h(c)w(h,k))w(g,hk)
    \\&= a\act_g(b)(\act_g\circ\act_h)(c)
    \act_g(w(h,k))w(g,hk)\delta_{ghk}.
  \end{align*}
  Then Condition~\ref{def:twisted-partial-action-groupoid_8} in
  Definition~\ref{def:twisted-partial-action-groupoid} implies the
  associativity for elements of the special form.  And this
  suffices.

  Another axiom requiring care is the anti-multiplicativity of the
  involution:
  \begin{equation}
    \label{eq:involution_anti_mult}
    ((a\delta_g)\cdot (b\delta_h))^* = (b\delta_h)^*\cdot (a\delta_g)^*
  \end{equation}
  for \((g,h)\in G^2\), \(a\in \A_g\), and \(b\in \A_h\).
  Let \(K \defeq \Bun[D]_{g^{-1}}\cap \Bun[D]_h\), which is an ideal
  in~\(\FU_{\s(g)}\).
  As above, \(\A_g\otimes_{\FU_{\s(g)}} \A_h \cong \A_g \A_g^*\A_g
  \otimes_{\FU_{\s(g)}} \A_h \A_h^* \A_h \cong \A_g (\A_g^*\A_g \A_h \A_h^*)
  \otimes_{\FU_{\s(g)}}  (\A_g^*\A_g \A_h \A_h^*) \A_h\).
  The right action on~\(\A_g\) is given by \(a\cdot b =
  a\act_g(b)\), so
  \[
    \A_g  (\A_g^*\A_g \A_h \A_h^*)
    = \Bun[D]_g\cap \act_g(\Bun[D]_{g^{-1}}\cap\Bun[D]_h)
    = \Bun[D]_g \cap \Bun[D]_{g h}.
  \]
  Similarly, \((\A_g^*\A_g \A_h \A_h^*) \A_h = \Bun[D]_{g^{-1}}\cap
  \Bun[D]_h\).
  So we may assume without loss of generality that
  \[
    a\in \Bun[D]_g\cap \Bun[D]_{g h},\qquad
    b\in \Bun[D]_{g^{-1}}\cap \Bun[D]_h.
  \]
  Since \(b\in \Bun[D]_{g^{-1}}\), we may simplify \((a\delta_g)\cdot
  (b\delta_h) = a\act_g(b)w(g,h)\delta_{gh}\).
  Taking adjoints gives
  \[
    ((a\delta_g)\cdot (b\delta_h))^*
    = \act_{gh}^{-1}(w(g,h)^*\act_g(b^*)a^*) w((gh)^{-1},gh)^*\delta_{(gh)^{-1}}.
  \]
  Condition~\ref{def:twisted-partial-action-groupoid_7} gives
  \(\act_{gh}^{-1}(c) = \act_h^{-1}(\act_g^{-1}(w(g,h)cw(g,h)^*))\)
  for~\(c\) in the domain of \(\act_h^{-1} \circ \act_g^{-1}\).
  Applying this above yields
  \[
    ((a\delta_g)\cdot (b\delta_h))^*
    = \act_h^{-1}(b^*\act_g^{-1}(a^*)\act_g^{-1}(w(g,h)^*)) w((gh)^{-1},gh)^*\delta_{(gh)^{-1}}.
  \]
  For the right-hand side of~\eqref{eq:involution_anti_mult}, we first
  take the adjoints:
  \[
    (b\delta_h)^* = \act_h^{-1}(b^*)w(h^{-1},h)^*\delta_{h^{-1}}, \qquad
    (a\delta_g)^* = \act_g^{-1}(a^*)w(g^{-1},g)^*\delta_{g^{-1}}.
  \]
  The product of these elements simplifies to
  \[
    \act_h^{-1}(b^*)w(h^{-1},h)^*
    \act_{h^{-1}}(\act_g^{-1}(a^*)w(g^{-1},g)^*)w(h^{-1},g^{-1})
  \]
  because \(\act_g^{-1}(a^*)\) belongs to~\(\Bun[D]_h\).
  Next we replace~\(\act_{h^{-1}}\)  by~\(\act_h^{-1}\)
  using~\eqref{eq:act_inverse_with_w}.
  This has the side effect of moving the multipliers \(w(h^{-1},h)^*\)
  to a later position, giving
  \[
    (b\delta_h)^*\cdot (a\delta_g)^*
    = \act_h^{-1}(b^*\act_g^{-1}(a^*)w(g^{-1},g)^*)w(h^{-1},h)^*w(h^{-1},g^{-1})\delta_{(gh)^{-1}}.
  \]
  The term \(\act_h^{-1}(b^*\act_g^{-1}(a^*))\) appears in our
  formulas for both sides of~\eqref{eq:involution_anti_mult}, and it
  belongs to~\(\act_h^{-1}(K) = \Bun[D]_{h^{-1}} \cap
  \Bun[D]_{(g h)^{-1}}\).
  Therefore,  \eqref{eq:involution_anti_mult}
  reduces to the following identity in the multiplier algebra
  of~\(\Bun[D]_{h^{-1}} \cap \Bun[D]_{(g h)^{-1}}\):
  \[
    \label{eq:anti_mult_cocycle}
    \act_h^{-1}(\act_g^{-1}(w(g,h)^*))w((gh)^{-1},gh)^*
    = \act_h^{-1}(w(g^{-1},g)^*)w(h^{-1},h)^*w(h^{-1},g^{-1}).
  \]
  Here we use that each factor is a multiplier of an ideal containing
  \(\Bun[D]_{h^{-1}} \cap \Bun[D]_{(g h)^{-1}}\) and that
  multipliers restrict to smaller ideals, so that each factor becomes
  a multiplier of~\(\Bun[D]_{h^{-1}} \cap \Bun[D]_{(g h)^{-1}}\).
  Taking the adjoint of both sides, this is equivalent to
  \begin{equation}
    \label{eq:anti_mult_cocycle_adj}
    w((gh)^{-1},gh)\act_h^{-1}(\act_g^{-1}(w(g,h)))
    = w(h^{-1},g^{-1})^*w(h^{-1},h)\act_h^{-1}(w(g^{-1},g)).
  \end{equation}

  The \(2\)\nb-cocycle
  identity~\ref{def:twisted-partial-action-groupoid_8} for
  \(((gh)^{-1},g,h)\) gives
  \[
    \act_{(gh)^{-1}}(w(g,h))w((gh)^{-1},gh) = w((gh)^{-1},g)w(h^{-1},h)
  \]
  as multipliers of \(\Bun[D]_{(g h)^{-1}} \cap \Bun[D]_{h^{-1}}\).
  We use~\eqref{eq:act_inverse_with_w} to rewrite this as
  \[
    w((gh)^{-1},gh)\act_{gh}^{-1}(w(g,h)) =
    w(h^{-1}g^{-1},g)w(h^{-1},h).
  \]
  Since \(\act_{gh}^{-1} = \act_h^{-1}\circ \act_g^{-1}\), also on
  multipliers, this simplifies the left side
  of~\eqref{eq:anti_mult_cocycle_adj} to
  \(w(h^{-1}g^{-1},g)w(h^{-1},h)\).

  Next, the \(2\)\nb-cocycle identity~\ref{def:twisted-partial-action-groupoid_8} on
  \((h^{-1},g^{-1},g)\) and that~\(w\) is normalised imply
  \[
    \act_{h^{-1}}(w(g^{-1},g)) = w(h^{-1},g^{-1})w(h^{-1}g^{-1},g)
  \]
  as multipliers of \(\Bun[D]_{h^{-1}} \cap \Bun[D]_{(g h)^{-1}}\).
  We use this and~\eqref{eq:act_inverse_with_w} for~\(h\) to rewrite
  the right hand side of~\eqref{eq:anti_mult_cocycle_adj} as
  \begin{multline*}
    w(h^{-1},g^{-1})^*w(h^{-1},h)\act_h^{-1}(w(g^{-1},g))
    = w(h^{-1},g^{-1})^*\act_{h^{-1}}(w(g^{-1},g)) w(h^{-1},h)
    \\= w(h^{-1},g^{-1})^*  w(h^{-1},g^{-1}) w(h^{-1}g^{-1},g)
    w(h^{-1},h)
    = w(h^{-1}g^{-1},g)w(h^{-1},h).
  \end{multline*}
  Having simplified both sides of~\eqref{eq:anti_mult_cocycle_adj} to
  the same expression, we have proven this equation and thus finished
  the proof of~\eqref{eq:involution_anti_mult}.

  Next we check the condition \(((a\delta_g)^*)^* = a\delta_g\).
  By definition,
  \[
    (a\delta_g)^* = \act_g^{-1}(a^*) w(g^{-1},g)^* \delta_{g^{-1}}.
  \]
  Applying the involution again gives
  \begin{align*}
    ((a\delta_g)^*)^*
    &= \act_{g^{-1}}^{-1} \bigl( (\act_g^{-1}(a^*) w(g^{-1},g)^*)^* \bigr) w(g,g^{-1})^* \delta_g \\
    &= \act_{g^{-1}}^{-1} \bigl( w(g^{-1},g) \act_g^{-1}(a) \bigr) w(g,g^{-1})^* \delta_g.
  \end{align*}
  Next, we use \eqref{eq:act_inverse_with_w} and
  Lemma~\ref{lem:cocycle_inverse_adjoint} to simplify this:
  \begin{multline*}
    ((a\delta_g)^*)^*
    = \act_{g^{-1}}^{-1} \bigl( w(g^{-1},g) \act_g^{-1}(a)
    \act_g(w(g^{-1},g)) w(g^{-1},g)^* \bigr) w(g,g^{-1})^*
    \delta_g
    \\= a \act_g(w(g^{-1},g)) w(g,g^{-1})^* \delta_g
    =a \delta_g.
  \end{multline*}

  The norm is submultiplicative, that is, \(\norm{(a\delta_g)\cdot
    (b\delta_h)} \le \norm{a} \norm{b}\), for all \(a\in \A_g\),
  \(b\in\A_h\) because~\(\act_g\) is an isometric isomorphism between
  ideals and \(w(g,h)\) is unitary.

  Next, we check the \(\Cst\)\nb-identity and positivity.
  If \(a\delta_g \in \A_g\), then \(a \in \Bun[D]_g\).
  Since the domain of~\(\act_{g^{-1}}\) is exactly \(\Bun[D]_{(g^{-1})^{-1}} =
  \Bun[D]_g\), we may always use the simpler formula for the product
  as for global actions.
  So
  \begin{align*}
    (a\delta_g)^* \cdot (a\delta_g)
    &= \bigl( \act_g^{-1}(a^*) w(g^{-1},g)^* \delta_{g^{-1}} \bigr) \cdot (a\delta_g) \\
    &= \act_g^{-1}(a^*) w(g^{-1},g)^* \act_{g^{-1}}(a) w(g^{-1},g) \delta_{u(\s(g))}.
  \end{align*}
  We use~\eqref{eq:act_inverse_with_w} to simplify this to
  \[
    (a\delta_g)^* \cdot (a\delta_g)
    = \act_g^{-1}(a^*) \act_g^{-1}(a) \delta_{u(\s(g))}
    = \act_g^{-1}(a^* a) \delta_{u(\s(g))}.
  \]
  This is clearly positive in~\(\FU_{\s(g)}\) as required.
  In addition, the \(\Cst\)\nb-identity in
  Condition~\ref{def:Fell_bundle_5} follows because~\(\act_g^{-1}\) is
  isometric.

  The continuity of the multiplication and involution follow directly
  from the definition of the topology on the field~\(\A\) because the
  isomorphisms~\(\act_g\) and unitary multipliers~\(w(g,h)\) are continuous.
\end{proof}

The following proposition shows that the section \Star{}algebra
\(\Contc(G,\A)\) of the Fell bundle defined above is the same
\(\ast\)\nb-algebra that is used in \cites{Exel:TwistedPartialActions,
  BussExel:Regular.Fell.Bundle} to define the crossed product
\(\Cst\)\nb-algebra of a twisted partial action.

\begin{proposition}
  \label{pro:twisted-partial-sections}
  Let \((\FU,\Dom,\act,w)\) be a twisted partial action of~\(G\) as in
  Definition~\textup{\ref{def:twisted-partial-action-groupoid}}.
  Let~\(\A\) be the Fell bundle in
  Definition~\textup{\ref{def:Fell_twisted_partial_action}}.
  Then
  \[
    \Contc(G,\A)
    = \setgiven{f\in \Contc(G,\rg^*(\FU))}{f(g)\in \Dom_g\text{ for all }g\in G}.
  \]
  Under this identification, the convolution and involution in
  \(\Contc(G,\A)\) become
  \begin{align*}
    (f_1*f_2)(g)
    &= \int_{G^{\rg(g)}} f_1(h)\cdot f_2(h^{-1}g)\,\diff\alpha^{\rg(g)}(h) \\
    &= \int_{G^{\rg(g)}}
      \act_h\bigl(\act_{h^{-1}}(f_1(h))f_2(h^{-1}g)\bigr) w(h,h^{-1}g)
      \,\diff\alpha^{\rg(g)}(h),
    \\
    f^*(g)
    &= f(g^{-1})^*
      = w(g,g^{-1})^* \act_g\bigl(f(g^{-1})^*\bigr)
  \end{align*}
  for all \(f,f_1,f_2\in \Contc(G,\Dom)\) and \(g\in G\).
\end{proposition}

\begin{proof}
  The proof is immediate from the definitions.
\end{proof}

\begin{example}[Restriction of actions]
  Let \(\FU[B]=(\FU[B]_x)_{x\in G^0}\) be an upper semicontinuous field of
  \(\Cst\)\nb-algebras over~\(G^0\) and let
  \(\actB\colon \s^*(\FU[B])\congto \rg^*(\FU[B])\) be an action of~\(G\)
  on~\(\FU[B]\).
  Let \(\FU=(\FU_x)_{x\in G^0}\) be a continuous family of ideals
  in~\(\FU[B]\).
  For \(g\in G\), let
  \[
    \Bun[D]_g \defeq
    \setgiven{b\in \FU_{\rg(g)}} {\actB_g^{-1}(b)\in \FU_{\s(g)}}.
  \]
  This is a continuous family of ideals in~\(\rg^*(\FU[B])\).
  The automorphism~\(\actB_g\) maps~\(\Bun[D]_{g^{-1}}\)
  isomorphically onto~\(\Bun[D]_g\).
  Let~\(\act_g\) be the restriction of~\(\actB_g\)
  to~\(\Bun[D]_{g^{-1}}\).
  These domains and automorphisms define a partial action of~\(G\)
  on~\(\FU\) called the \emph{restriction} of~\(\actB\) to~\(\FU\).
  We may restrict a twisted or a twisted partial action in the same
  way to a twisted partial action.
\end{example}

\begin{example}[Green--Renault twists]
  Green's twisted actions for groups
  (see~\cite{Green:Local_twisted}) have been generalised to
  groupoids by Kumjian~\cite{Kumjian:Diagonals} and
  Renault~\cite{Renault:Representations}.  These differ from the
  twisted actions above, which generalise Busby--Smith twisted
  actions (see~\cite{Busby-Smith:Representations_twisted_group}).
  Renault uses an extension of locally compact groupoids
  \(S \into \Sigma \onto G\), where~\(S\) is a bundle of Abelian
  groups acted upon by~\(G\).  A twisted action of~\(G\) is an
  action of~\(\Sigma\) that is implemented by specified unitaries
  on~\(S\), which satisfy the conditions for a crossed module
  representation as in
  \cites{Buss-Meyer-Zhu:Non-Hausdorff_symmetries,
    Buss-Meyer-Zhu:Higher_twisted, Buss-Meyer:Crossed_products}.  A
  Green--Renault twisted action as above gives rise to a Fell bundle
  over~\(G\) in a canonical way.  This is explained in detail in
  \cite{Muhly-Williams:Equivalence.FellBundles}*{Example~2.5 and
    discussion thereafter}, and it is also related to the
  construction of Fell bundles in
  \cites{Buss-Meyer-Zhu:Higher_twisted,
    Buss-Meyer:Groupoid_fibrations, Buss-Meyer:Crossed_products}.
\end{example}

In summary, Fell bundles encompass twisted partial actions, twisted
actions, and ordinary actions as special cases.
Saturated Fell bundles may be interpreted as actions by Morita
equivalences.
For groups, this is explained in~\cite{Buss-Meyer-Zhu:Higher_twisted};
see also~\cite{Buss-Meyer:Actions_groupoids} for the context of
inverse semigroup actions.
So we interpret unsaturated Fell bundles as partial actions by partial
Morita equivalences.

The following theorem characterises which Fell bundles come from
twisted partial actions and twisted actions.  It generalises analogous
criteria for groups in~\cite{Exel:TwistedPartialActions} and for
inverse semigroups in~\cite{BussExel:Regular.Fell.Bundle}.

\begin{theorem}
  \label{the:Fell_bundle_from_twisted_partial_action}
  Let~\(G\) be a Hausdorff, locally compact groupoid.  A Fell
  bundle~\(\A\) over~\(G\) is isomorphic to the Fell bundle defined
  by a twisted partial action if and only if~\(\A\) is isomorphic to
  a continuous field of ideals in~\(\rg^*\FU\) as a left Hilbert
  \(\rg^*\FU\)\nb-module; equivalently, \(\A\) is isomorphic to a
  continuous field of ideals in~\(\s^*\FU\) as a right Hilbert
  \(\s^*\FU\)\nb-module.  A Fell bundle is isomorphic to the Fell
  bundle defined by a twisted action if and only if~\(\A\) is
  isomorphic to~\(\rg^*\FU\) as a field of left Hilbert
  \(\rg^*\FU\)\nb-modules, if and only if~\(\A\) is isomorphic
  to~\(\s^*\FU\) as a field of right Hilbert \(\s^*\FU\)\nb-modules.
\end{theorem}

\begin{proof}
  We may use the inversion in~\(G\) to switch between the left and
  right Hilbert module structures.
  Therefore, it suffices to prove the left-handed versions of the two
  claims.
  By construction, the Fell bundle associated to a twisted partial
  action is isomorphic to a continuous field of ideals in~\(\rg^*\FU\)
  as a left Hilbert \(\rg^*\FU\)\nb-module.
  And the Fell bundle associated to a twisted action is isomorphic
  to~\(\rg^*\FU\) as a left Hilbert \(\rg^*\FU\)\nb-module.
  Conversely, let~\(\A\) be a Fell bundle that is isomorphic to a
  continuous field of ideals~\((\Bun[D]_g)_{g\in G}\) in~\(\rg^*\FU\)
  as a field of left Hilbert \(\rg^*\FU\)\nb-modules.
  Without loss of generality, we assume \(\A_g = \Bun[D]_g\).
  Then \(\Bun[D]_g = \A_g \A_g^* = \A_g \A_{g^{-1}}\) as ideals
  in~\(\A_g\).
  Since~\(\A_g\) is a Hilbert bimodule, the right multiplication
  action of~\(\FU_{\s(g)}\) on~\(\A_g\) restricts to an isomorphism
  from the ideal \(\braket{\A_g}{\A_g} = \A_{g^{-1}} \A_g =
  \Bun[D]_{g^{-1}}\) onto the \(\Cst\)\nb-algebra of compact operators
  on the left Hilbert \(\FU_{\rg(g)}\)-module~\(\Bun[D]_g\).
  The latter is canonically isomorphic to~\(\Bun[D]_g\).
  This gives an isomorphism \(\act_g\colon \Bun[D]_{g^{-1}} \congto
  \Bun[D]_g\).
  These isomorphisms depend continuously on~\(g\).
  By Lemma~\ref{lem:subfield_equality}, they form an isomorphism of
  continuous families of ideals \(\act\colon \inv^*(\Bun[D]) \congto
  \Bun[D]\).

  Let \((g,h)\in G^2\).  There are canonical isomorphisms
  \begin{multline*}
    \A_g \otimes_{\FU_{\s(g)}} \A_h = \A_g \A_g^* \A_g
    \otimes_{\FU_{\s(g)}} \A_h \\\cong \A_g \otimes_{\FU_{\s(g)}}
    \A_g^* \A_g \A_h \cong \Bun[D]_g \cdot
    \act_g(\Bun[D]_{g^{-1}}\cdot \Bun[D]_h)
    = \act_g(\Bun[D]_{g^{-1}} \cap \Bun[D]_h).
  \end{multline*}
  They map \(a\otimes b\) with \(a\in \Bun[D]_g\), \(b\in
  \Bun[D]_{g^{-1}}\cdot \Bun[D]_h\) to \(a\cdot \act_g(b)\).
  The product~\(m_{g,h}\) in the Fell bundle gives another Hilbert
  bimodule map from \(\A_g \otimes_{\FU_{\s(g)}} \A_h\) to \(\A_{g h}
  = \Bun[D]_{g h}\).
  Two ideals that are isomorphic as left Hilbert modules are equal
  because the ideal is also the range ideal of the left scalar
  product.
  So the two maps above have the same image.
  Hence \(\act_g(\Bun[D]_{g^{-1}} \cap \Bun[D]_h) \subseteq \Bun[D]_{g
    h}\).
  As in Lemma~\ref{lem:partial_action_domains}, it follows that the
  common image is \(\Bun[D]_g\cap \Bun[D]_{g h}\).
  Since the left Hilbert module structure is the obvious one, there
  must be a unitary multiplier~\(w(g,h)\) of \(\Bun[D]_g\cap
  \Bun[D]_{g h}\) with \(m_{g,h}(a\otimes b) = a\cdot \act_g(b) \cdot
  w(g,h)\) for all \(a\in \Bun[D]_g\), \(b\in \Bun[D]_{g^{-1}}\cap
  \Bun[D]_h\).
  This gives the data of a twisted partial action.
  The proof above shows that the ideals~\(\Bun[D]_g\) form a
  continuous family of ideals, that~\((\act_g)\) is an isomorphism
  \(\inv^* \Bun[D] \congto \Bun[D]\), and that right multiplication
  by~\(w(g,h)\) for \((g,h)\in G^2\) preserves continuity of sections.

  We now verify the algebraic conditions for a twisted partial action
  in Definition~\ref{def:twisted-partial-action-groupoid}.
  They follow from the associativity of the multiplication in the Fell
  bundle.
  The computations are similar to those in
  \cites{Exel:TwistedPartialActions, BussExel:Regular.Fell.Bundle} for
  Fell bundles over groups and inverse semigroups.

  For \(x\in G^0\), \(\A_{u(x)} = \FU_x\) implies \(\Bun[D]_{u(x)} =
  \A_{u(x)} \A_{u(x)}^* = \FU_x\).
  Since~\(\A_{u(x)}\) is the trivial Hilbert bimodule~\(\FU_x\), its
  right multiplication coincides with the usual one, giving
  \(\act_{u(x)} = \Id_{\FU_x}\).
  Condition~\ref{def:twisted-partial-action-groupoid_6} says that
  \(\act_g(\Bun[D]_{g^{-1}} \cap \Bun[D]_h) \subseteq \Bun[D]_{g h}\),
  which was already shown above.

  We denote the Fell-bundle multiplication as~\(\circ\) and express it
  in terms of \(\act\) and~\(w\).
  The definition of \(w(g,h)\) through the isomorphism~\(m_{g,h}\)
  implies
  \begin{equation}
    \label{eq:Fell_prod_w}
    x \circ y = x \cdot \act_g(y) \cdot w(g,h)
  \end{equation}
  for all \(x\in \A_g\) and \(y\in \Bun[D]_{g^{-1}} \cap \Bun[D]_h
  \subseteq \A_h\); here the product on the right is in \(\A_{gh}
  \cong \Bun[D]_{gh}\).
  Similarly, the identification of the right action of~\(\FU_{\s(g)}\)
  on~\(\A_g\) with the action by compact operators gives \(x \circ a = x
  \cdot \act_g(a)\) for \(x\in \A_g\) and \(a\in \Bun[D]_{g^{-1}}\).
  This gives the normalisation condition
  in~\ref{def:twisted-partial-action-groupoid_5} in
  Definition~\ref{def:twisted-partial-action-groupoid} and finishes
  the proof of~\ref{def:twisted-partial-action-groupoid_5}.

  To check condition~\ref{def:twisted-partial-action-groupoid_7}, let
  \((g,h)\in G^2\).
  Let \(x \in \A_g = \Bun[D]_g\), \(y \in \Bun[D]_{g^{-1}}\cap
  \Bun[D]_h \subseteq \A_h\), and let \(a \in \Bun[D]_{h^{-1}} \cap
  \Bun[D]_{(g h)^{-1}}  =
  \act_{h^{-1}}(\Bun[D]_{g^{-1}}\cap \Bun[D]_h)\), viewed as an element
  of~\(\FU_{\s(h)}\).
  Since the Fell bundle is associative, we may expand \((x\circ y)\circ a
  = x\circ (y\circ a)\) in two ways.
  First, since \(a\in \Bun[D]_{h^{-1}} \cap \Bun[D]_{(g h)^{-1}}\) and \(y\in
  \Bun[D]_h \cap \Bun[D]_{g^{-1}}\), \eqref{eq:Fell_prod_w} implies
  \[
    (x\circ y)\circ a
    = (x\circ y) \cdot \act_{gh}(a)
    = x \cdot \act_g(y) \cdot w(g,h) \cdot \act_{gh}(a).
  \]
  Similarly, we compute
  \[
    x\circ (y\circ a)
    = x\circ (y \cdot \act_h(a))
    = x \cdot \act_g(y \cdot \act_h(a)) \cdot w(g,h)
    = x \cdot \act_g(y) \cdot \act_g(\act_h(a)) \cdot w(g,h).
  \]
  Equating both sides gives
  \[
    x \cdot \act_g(y) \cdot w(g,h) \cdot \act_{gh}(a)
    = x \cdot \act_g(y) \cdot \act_g(\act_h(a)) \cdot w(g,h).
  \]
  As \(x\) and~\(y\) vary, the elements \(x\cdot \act_g(y)\) span the
  ideal \(\act_g(\Bun[D]_{g^{-1}}\cap \Bun[D]_h) =
  \Bun[D]_g \cap \Bun[D]_{g h}\).
  Thus, \(w(g,h) \cdot \act_{gh}(a) = \act_g(\act_h(a)) \cdot w(g,h)\)
  as multipliers of \(\Bun[D]_g \cap \Bun[D]_{gh}\).
  Multiplying by \(w(g,h)^*\) on the right yields
  condition~\ref{def:twisted-partial-action-groupoid_7} in
  Definition~\ref{def:twisted-partial-action-groupoid}.

  To prove condition~\ref{def:twisted-partial-action-groupoid_8}, let
  \((g,h,k)\in G^3\).
  Let \(x\in \A_g\), \(y \in \Bun[D]_{g^{-1}} \cap \Bun[D]_h \subseteq
  \A_h\), and take \(z \in \Bun[D]_{h^{-1}} \cap \Bun[D]_{(gh)^{-1}}
  \cap \Bun[D]_k \subseteq \A_k\).
  By Lemma~\ref{lem:partial_action_domains}, \(\act_h\) restricts to a
  bijection from \(\Bun[D]_{h^{-1}} \cap \Bun[D]_{(gh)^{-1}} \cap
  \Bun[D]_k\) to \(\Bun[D]_{g^{-1}} \cap \Bun[D]_h \cap
  \Bun[D]_{hk}\).

  Thus, as \(y\) and~\(z\) vary, the elements \(a \defeq y \cdot
  \act_h(z)\) span the ideal \(\Bun[D]_{g^{-1}} \cap \Bun[D]_h \cap
  \Bun[D]_{hk}\).
  We compute \((x\circ y)\circ z\) and \(x\circ (y\circ z)\).
  Our assumptions on \(x,y,z\) imply that all relevant products are
  computed through~\eqref{eq:Fell_prod_w}.
  On the one hand,
  \[
    (x\circ y)\circ z
    = x \cdot \act_g(y) \cdot w(g,h) \cdot \act_{gh}(z) \cdot w(gh,k)
    = x \cdot \act_g(y) \cdot \act_g(\act_h(z)) \cdot w(g,h) \cdot w(gh,k),
  \]
  where the second equation uses
  Condition~\ref{def:twisted-partial-action-groupoid_7} for \(z \in
  \Bun[D]_{h^{-1}} \cap \Bun[D]_{(gh)^{-1}}\).
  On the other hand,
  \[
    x\circ (y\circ z)
    = x \cdot \act_g(y \cdot \act_h(z) \cdot w(h,k)) \cdot w(g,hk)
    = x \cdot \act_g(y) \cdot \act_g(\act_h(z)) \cdot \act_g(w(h,k)) \cdot w(g,hk).
  \]
  We may cancel~\(x\) because these elements span~\(\Bun[D]_g\), and
  find
  \[
    \act_g(a) \cdot w(g,h) \cdot w(gh,k)
    = \act_g(a) \cdot \act_g(w(h,k)) \cdot w(g,hk)
    = \act_g(a \cdot w(h,k)) \cdot w(g,hk).
  \]
  Since~\(a\) spans \(\Bun[D]_{g^{-1}} \cap \Bun[D]_h \cap
  \Bun[D]_{hk}\), this proves
  condition~\ref{def:twisted-partial-action-groupoid_8}.
\end{proof}

\section{Representations of Fell bundles on Hilbert modules}
\label{sec:representations}

Let~\(G\) be a locally compact groupoid with Haar system~\(\alpha\),
and let~\(\A\) be a Fell bundle over~\(G\).
Let~\(D\) be a \(\Cst\)\nb-algebra and let~\(\F\) be a Hilbert
\(D\)\nb-module.
We are going to define representations of~\(\A\) on~\(\F\),
generalising the definition in~\cite{Buss-Holkar-Meyer:Universal}
without Fell bundles.
Adding the coefficient bundles requires some careful choices in
various definitions, especially when the Fell bundle fails to be
saturated.
It also requires extra steps in most proofs, which check new aspects
related to the bundles that have been added.

A representation of the groupoid~\(G\) on~\(\F\) is defined
in~\cite{Buss-Holkar-Meyer:Universal} as a pair \((\varphi,U)\),
where~\(\varphi\) is a nondegenerate \Star{}homomorphism
\(\Cont_0(G^0) \to \Bound(\F)\) and~\(U\) is a unitary operator
\(\Lt^2(G,\s,\tilde\alpha) \otimes_\varphi \F \to \Lt^2(G,\rg,\alpha)
\otimes_\varphi \F\) that intertwines the left actions of
\(\Cont_0(G)\); in other words, \(\varphi\) makes~\(\F\) into a
\(\Cont_0(G)\)\nb-\(D\)-correspondence and~\(U\) is an isomorphism of
\(\Cont_0(G)\)\nb-\(D\)-correspondences.  It is required, in addition,
that two isomorphisms of \(\Cont_0(G^2)\)\nb-\(D\)-correspondences
\(\Lt^2(G^2,v_2,\mu_0) \otimes_\varphi \F \rightrightarrows
\Lt^2(G^2,v_0,\mu_0) \otimes_\varphi \F\) that are built from~\(U\)
are equal.  We now add the Fell bundle~\(\A\) to this definition.

The natural replacement for \(\Cont_0(G^0)\) is
\(\Cont_0(G^0, \FU)\), the \(\Cst\)\nb-algebra of
\(\Cont_0\)\nb-sections of~\(\FU\).
So now~\(\varphi\) is a nondegenerate \Star{}homomorphism
\[
  \varphi\colon \Cont_0(G^0,\FU) \to \Bound(\F)
\]
that makes~\(\F\) a \(\Cont_0(G^0,\FU)\)\nb-\(D\)-correspondence.

We replace \(\Lt^2(G,\s,\tilde\alpha)\) and \(\Lt^2(G,\rg,\alpha)\)
by \(\Lt^2(G,\A,\s,\tilde\alpha)\) and
\(\Lt^2(G,\A\A^*,\rg,\alpha)\), respectively.  Here~\(\A\) is viewed as a field of
\(\A\A^*\)-\(\s^*\FU\)-correspondences and~\(\A\A^*\) as a field of
\(\A\A^*\)-\(\rg^*\FU\)-correspondences; in particular, each
fibre~\(\A_g\) of~\(\A\) is an
\(\A_g\A_g^*\)-\(\FU_{\s(g)}\)-correspondence and each
fibre~\(\A_g \A_g^*\) of~\(\A\A^*\) is an
\(\A_g\A^*_{g}\)-\(\FU_{\rg(g)}\)-correspondence.  The
correspondences \(\Lt^2(G,\A,\s,\tilde\alpha)\) and
\(\Lt^2(G,\A\A^*,\rg,\alpha)\) are defined in
Section~\ref{sec:usc_fields}.  Recall that
\(\Lt^2(G,\A,\s,\tilde\alpha)\) is the completion of
\(\Contc(G,\A)\) in the norm associated to the
\(\Cont_0(G^0,\FU)\)-valued inner product
\[
  \braket{\xi}{\eta}(x)\defeq \int_G
  \xi(g)^*\eta(g)\,\diff\tilde\alpha_x(g).
\]
Thus, \(U\) in the definition of the representation, becomes a
unitary
\[
  U\colon \Lt^2(G,\A,\s,\tilde\alpha) \otimes_\varphi \F \to
  \Lt^2(G,\A\A^*,\rg,\alpha) \otimes_\varphi \F
\]
that intertwines the left actions of \(\Cont_0(G,\A\A^*)\).
In the correspondence bicategory, this says that~\(U\) is an
invertible \(2\)\nb-arrow between arrows \(\Cont_0(G,\A\A^*)
\leftarrow D\) that makes the following diagram commute:
\begin{equation}
  \label{eq:U_diagram}
  \begin{tikzpicture}[baseline=(current bounding box.west)]
    \matrix (m) [cd,column sep=4em, row sep=2.5em] {
      \Cont_0(G,\A\A^*)&\Cont_0(G^0,\FU)\\
      \Cont_0(G^0,\FU)&D\\
    };
    \meascor{m-1-1}{\s}{\tilde{\alpha},\A}{m-1-2}
    \meascor{m-1-1}{\rg}{\alpha,\A\A^*}{m-2-1}
    \meascor{m-2-1}{\F}{}{m-2-2}
    \meascor{m-1-2}{\F}{}{m-2-2}
    \draw[dar,mid] (m-1-2) -- node[narrowfill] {\(\scriptstyle U\)}
    (m-2-1);
  \end{tikzpicture}
\end{equation}

The Hilbert modules \(\Lt^2(G,\A,\s,\tilde\alpha)\) and
\(\Lt^2(G,\A\A^*,\rg,\alpha)\) also carry natural left multiplication
actions of the larger \(\Cst\)\nb-algebra \(\Cont_0(G,\rg^*\FU)\).
Asking~\(U\) to intertwine the left actions of \(\Cont_0(G,\rg^*\FU)\)
or \(\Cont_0(G,\A\A^*)\) makes no difference:

\begin{lemma}
  \label{lem:nondegenerate_over_AA-star}
  Both \(\Lt^2(G,\A,\s,\tilde\alpha)\) and
  \(\Lt^2(G,\A\A^*,\rg,\alpha)\) are nondegenerate as left modules over
  \(\Cont_0(G,\A\A^*)\).  A unitary operator~\(U\) as above that
  intertwines the left actions of \(\Cont_0(G,\A\A^*)\) also
  intertwines the left actions of \(\Cont_0(G,\rg^*\FU)\).
\end{lemma}

\begin{proof}
  The multiplication map \(\Cont_0(G,\A\A^*)
  \otimes_{\Cont_0(G,\rg^*\FU)} \Cont_0(G,\A) \to \Cont_0(G,\A)\) is a
  fibrewise isomorphism by Lemma~\ref{lem:equality_of_products}, and
  then it is an isomorphism of fields of Banach spaces by
  Lemma~\ref{lem:subfield_equality}.
  So the multiplication map \(\Cont_0(G,\A\A^*)
  \otimes_{\Cont_0(G,\rg^*\FU)} \Lt^2(G,\A,\s,\tilde\alpha) \to
  \Lt^2(G,\A,\s,\tilde\alpha)\) has dense range.
  It is also isometric, hence unitary.
  Thus \(\Lt^2(G,\A,\s,\tilde\alpha)\) is a nondegenerate left module
  over \(\Cont_0(G,\A\A^*)\).
  The same argument works for \(\Lt^2(G,\A\A^*,\rg,\alpha)\).
  A nondegenerate representation~\(\varrho\) of an ideal~\(I\) in a
  \(\Cst\)\nb-algebra~\(B\) extends \emph{uniquely} to a
  representation~\(\varrho'\) of~\(B\).
  If a unitary~\(U\) is an intertwiner for~\(\varrho\), then
  \(\varrho'\) and \(\Ad_U \circ \varrho'\) are two extensions
  of~\(\varrho\) to~\(B\).
  Hence they are equal.
  This means that~\(U\) is also an intertwiner for~\(\varrho'\).
  This implies the second claim.
\end{proof}

Next we must carry over the condition \(d_1^*(U) = d_2^*(U) d_0^*(U)\)
in \cite{Buss-Holkar-Meyer:Universal}*{Equation~(3.19)} that singles
out when a pair \((\varphi,U)\) is a representation.  Here
\(d_j^*(U)\) for \(j=0,1,2\) are isomorphisms of
\(\Cont_0(G^2),D\)-correspondences that are built from~\(U\) and
various canonical isomorphisms of correspondences by ``pasting'' them
in the correspondence bicategory.  These pasting diagrams are shown in
\cite{Buss-Holkar-Meyer:Universal}*{Figure~3}.  We will only need two
pasting diagrams, which we are going to explain in great detail.  Thus
the general meaning of pasting diagrams in bicategories as explained
in~\cite{Johnson-Yau:2-Dim} will not be needed here.

To extend the pasting diagrams in~\cite{Buss-Holkar-Meyer:Universal}
to possibly nonsaturated Fell bundles, we must decorate the vertices
and edges in the diagram in
\cite{Buss-Holkar-Meyer:Universal}*{Figure~3} with appropriate fields
of \(\Cst\)\nb-algebras and \(\Cst\)\nb-correspondences.  These fields
are dictated once we fix the field of \(\Cst\)\nb-algebras
over~\(G^2\) to be the field~\(\Bun[I]\) with fibre
\(\A_g \A_h \A_h^* \A_g^*\) at \((g,h)\in G^2\); this is the range
ideal
\(\BRAKET{\A_g \otimes_{\FU_{\s(g)}} \A_h}{\A_g \otimes_{\FU_{\s(g)}}
  \A_h}\) of the composite Hilbert bimodule
\(\A_g \otimes_{\FU_{\s(g)}} \A_h\).
It contains the range ideals of both \(\A_g\) and~\(\A_{g h}\) by
Lemma~\ref{lem:equality_of_products}.
To be more precise,
\[
  \Bun[I] \defeq d_2^*\A \otimes_{v_1^*\FU} d_0^*\A \otimes_{v_2^*\FU}
  (d_0^*\A)^* \otimes_{v_1^*\FU} (d_2^*\A)^*.
\]

The multiplication in the Fell bundle defines an injective map
\(\Bun[I] \hookrightarrow v_0^*\FU\).
If~\(\A\) is saturated, then this is also surjective and hence an
isomorphism, otherwise this is a proper subbundle.
The subfield~\(\Bun[I]\) of
\(v_0^*\FU\) is already determined by its fibres
\(\A_g \A_h \A_h^* \A_g^*\) because of
Lemma~\ref{lem:subfield_equality}.
The space of \(\Cont_0\)\nb-sections of~\(\Bun[I]\) is the closed
linear span of products \(\varphi_1 \cdot \varphi_2 \cdot
\varphi_3^\dagger \cdot \varphi_4^\dagger\); here~\(\varphi_j\) for
\(j=1,2,3,4\) are \(\Cont_0\)\nb-sections of \(d_2^*\A\), \(d_0^*\A\),
\(d_0^*\A\), and \(d_2^*\A\), respectively; and~\(\dagger\) denotes
the pointwise involution, \(\varphi^\dagger(g,h) \defeq
\varphi(g,h)^*\) for all \((g,h)\in G^2\).

The decorated form of \cite{Buss-Holkar-Meyer:Universal}*{Figure~3} in
our context is shown in Figure~\ref{fig:U_coherence}.
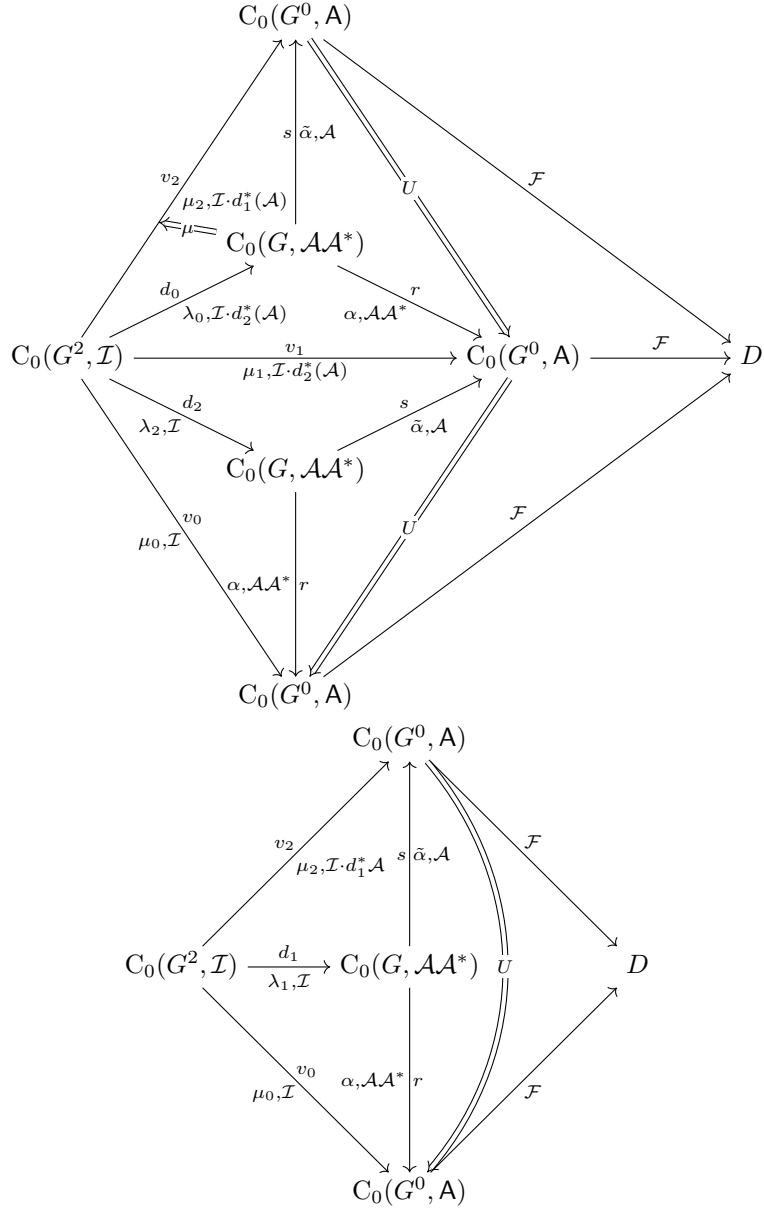
\begin{figure}[htbp]
  \begin{tikzpicture}[scale=3]
    \node (G2) at (0,1.5) {\(\Cont_0(G^2,\Bun[I])\)}; \node (G0a) at
    (1,0) {\(\Cont_0(G^0,\FU)\)}; \node (G1b) at (1,1)
    {\(\Cont_0(G,\A\A^*)\)}; \node (G1c) at (1,2)
    {\(\Cont_0(G,\A\A^*)\)}; \node (G0d) at (1,3)
    {\(\Cont_0(G^0,\FU)\)}; \node (G0e) at (2,1.5)
    {\(\Cont_0(G^0,\FU)\)}; \node (D) at (3,1.5) {\(D\)};
    \meascor[0]{G2}{v_0}{\mu_0,\Bun[I]}{G0a}
    \meascor[0]{G2}{d_2}{\lambda_2,\Bun[I]}{G1b}
    \meascor[0]{G2}{d_0}{\lambda_0,\Bun[I]\cdot d_2^*(\A)}{G1c}
    \meascor[0]{G2}{v_2}{\mu_2,\Bun[I]\cdot d_1^*(\A)}{G0d}
    \meascor[0]{G2}{v_1}{\mu_1,\Bun[I]\cdot d_2^*(\A)}{G0e}
    \meascor{G1b}{\rg}{\alpha,\A\A^*}{G0a}
    \meascor[0]{G1b}{\s}{\tilde{\alpha},\A}{G0e}
    \meascor{G1c}{\s}{\tilde{\alpha},\A}{G0d}
    \meascor[0]{G1c}{\rg}{\alpha,\A\A^*}{G0e}
    \meascor[0]{G0a}{\F}{}{D} \meascor[0]{G0d}{\F}{}{D}
    \meascor{G0e}{\F}{}{D} \draw[dar,mid] (G0e) -- node[narrowfill]
    {\(\scriptstyle U\)} (G0a); \draw[dar,mid] (G0d) --
    node[narrowfill] {\(\scriptstyle U\)} (G0e); \draw[dar,mid] (G1c)
    to node[narrowfill] {\(\scriptstyle \mu\)} (0.4,2.1);
  \end{tikzpicture}\\
  \begin{tikzpicture}[scale=3]
    \node (G2) at (0,1) {\(\Cont_0(G^2,\Bun[I])\)}; \node (G0a) at
    (1,0) {\(\Cont_0(G^0,\FU)\)}; \node (G1b) at (1,1)
    {\(\Cont_0(G,\A\A^*)\)}; \node (G0c) at (1,2)
    {\(\Cont_0(G^0,\FU)\)}; \node (D) at (2,1) {\(D\)};
    \meascor[0]{G2}{v_0}{\mu_0,\Bun[I]}{G0a}
    \meascor{G2}{d_1}{\lambda_1,\Bun[I]}{G1b}
    \meascor[0]{G2}{v_2}{\mu_2,\Bun[I]\cdot d_1^*\A}{G0c}
    \meascor{G1b}{\rg}{\alpha,\A\A^*}{G0a}
    \meascor{G1b}{\s}{\tilde{\alpha},\A}{G0c}
    \meascor[0]{G0a}{}{\F}{D}
    \meascor[0]{G0c}{\F}{}{D}
    \draw[dar,mid]
    (G0c) to[bend left=40] node[narrowfill] {\(\scriptstyle U\)}
    (G0a);
  \end{tikzpicture}
  \caption{Two parallel isomorphisms of correspondences built
    from~\(U\).  Each triangle or quadrilateral means an isomorphism of
    \(\Cst\)\nb-correspondences.  The three quadrilaterals
    containing~\(U\) are copies of~\eqref{eq:U_diagram}.  The
    triangles without label mean the canonical isomorphisms of
    \(\Cst\)\nb-correspondences in~\eqref{eq:compose_measure_family}.
    The triangle marked~\(\mu\) also involves the multiplication map
    \(\mu\colon d_2^*\A \otimes_{v_1^*\FU} d_0^*\A \hookrightarrow
    d_1^*\A\) that acts on fibres as
    in~\eqref{eq:multiplication_isometries}.}
  \label{fig:U_coherence}
\end{figure}
This figure encodes complex information in a single diagram.
Each vertex is a \(\Cst\)\nb-algebra.
In addition to~\(D\), we use the \(\Cst\)\nb-algebras of
\(\Cont_0\)\nb-sections of the fields of \(\Cst\)\nb-algebras \(\FU\),
\(\A\A^*\) and~\(\Bun[I]\) over \(G^0\), \(G\) and~\(G^2\); their
fibres at \(x \in G^0\), \(g\in G\), and \((g,h)\in G^2\) are \(\FU_x
\defeq \A_{u(x)}\), \(\A_g \A_g^*\), and \(\A_g \A_h \A_h^* \A_g^*\),
respectively.
Each arrow describes a \(\Cst\)\nb-correspondence.
The arrows pointing to~\(D\) are~\(\F\), viewed as a
\(\Cst\)\nb-correspondence \(\Cont_0(G^0,\FU) \to D\)
through~\(\varphi\).
The other \(\Cst\)\nb-correspondences are associated to continuous
maps decorated by continuous families of measures along their fibres
and by fields of \(\Cst\)\nb-correspondences.
At this point, the reader should check that the decorating fields are
fields of \(\Cst\)\nb-correspondences over the appropriate fields of
\(\Cst\)\nb-algebras.
For instance, \(\Bun[I]\cdot d_1^*(\A)\) is the field whose fibre at
\((g,h)\in G^2\) is \(\A_g \A_h \A_h^* \A_g^* \A_{g h}\), which is a full left Hilbert
module over~\(\Bun[I]_{g,h}\) and a right Hilbert module over
\(\FU_{\s(h)} = (v_2^* \FU)_{(g,h)}\).
And \(\Bun[I]\cdot d_2^*(\A)\) is the field whose fibre at \((g,h)\in
G^2\) is \(\A_g \A_h \A_h^* \A_g^* \A_g\), which is a full left
Hilbert module over~\(\Bun[I]_{g,h}\) and a right Hilbert module over
\(\FU_{\s(g)} = (v_1^* \FU)_{(g,h)}\).
Thus each arrow in the diagram exists.
To treat the maps~\(v_j\) for \(j=0,1,2\) simultaneously, it is useful
to number the fields of Hilbert bimodules by which the maps~\(v_j\) in
Figure~\ref{fig:U_coherence} are decorated:
\begin{equation}
  \label{eq:E_012}
  \E_0 \defeq \Bun[I],\qquad
  \E_1 \defeq \Bun[I] \cdot d_2^*(\A),\qquad
  \E_2 \defeq \Bun[I] \cdot d_1^*(\A).
\end{equation}
The fields of Hilbert bimodules in some of the arrows simplify by
Lemma~\ref{lem:equality_of_products}.
Namely, \(\A_g \A_h \A_h^* \A_g^* \A_{g h} = \A_g \A_h\) and \(\A_g
\A_h \A_h^* \A_g^* \A_g = \A_g \A_h \A_h^*\) for \((g,h)\in G^2\).
Then Lemma~\ref{lem:subfield_equality} implies
\[
  \E_2 = \Bun[I] \cdot d_1^*(\A) \cong d_2^*(\A) \cdot d_0^*(\A),\qquad
  \E_1 = \Bun[I] \cdot d_2^*(\A) \cong d_2^*(\A) \cdot d_0^*(\A) \cdot
  d_0^*(\A)^*.
\]
This helps to prove that the squares and triangles in the diagram
commute up to isomorphisms of correspondences.
Each square in the diagrams is decorated by~\(U\) and illustrates
that~\(U\) is an isomorphism between the two composite
\(\Cst\)\nb-correspondences that make up the boundary of the square.
Each unmarked triangle in the diagram illustrates one of the canonical isomorphisms of
\(\Cst\)\nb-correspondences from
Lemma~\ref{lem:compose_corr_from_measure-family}, in the shorthand
notation in~\eqref{eq:compose_measure_family}.
The reader should check that the maps, measure families and fields of
\(\Cst\)\nb-correspondences compose as indicated.
This is obvious for the maps (see Figure~\ref{fig:triangle_arrows}) and
for the fields of \(\Cst\)\nb-correspondences; for the measure
families, it is checked in
\cite{Buss-Holkar-Meyer:Universal}*{Equations (3.5)--(3.10)}.
The triangle marked~\(\mu\) is of the same type as the other
triangles, but also uses a nontrivial isomorphism of fields of
\(\Cst\)\nb-correspondences, which acts on the fibres by the
restriction of the multiplication maps \(\A_g \otimes_{\FU_{\s(g)}}
\A_h \hookrightarrow \A_{g h}\).
The multiplication on the left with the ideal \(\Bun[I]_{(g,h)} = \A_g
\A_h \A_h^* \A_g^*\) makes this Hilbert bimodule map an isomorphism.
Finally, we may combine the isomorphisms of correspondences that
correspond to the triangles and squares in the diagrams by horizontal
and vertical products in the correspondence bicategory.
In this way, the second diagram in Figure~\ref{fig:U_coherence}
describes an isomorphism of correspondences between the composite
correspondences in the top and bottom of the diagram,
\begin{equation}
  \label{eq:domain_codomain_U_coherence}
  \Lt^2(G^2,\E_2,v_2,\mu_2) \otimes_\varphi \F
  \congto
  \Lt^2(G^2,\E_0,v_0,\mu_0) \otimes_\varphi \F.
\end{equation}
We denote this isomorphism by~\(d_1^*(U)\) because it generalises the
correspondence isomorphism~\(d_1^*(U)\)
in~\cite{Buss-Holkar-Meyer:Universal}.
More concretely, it is given as the composition of unitary
isomorphisms in Figure~\ref{fig:def_d1U}.
\begin{figure}[htbp]
  \begin{tikzpicture}[baseline=(current bounding box.west)]
    \node (m-1-1) at (0,3.5)
    {\(\Lt^2(G^2,\E_2,v_2,\mu_2)\otimes_{\varphi}\F\)}; \node (m-1-2) at
    (5,2.5)
    {\(\Lt^2(G^2,\Bun[I],d_1,\lambda_1)\otimes_{\Cont_0(G,\A\A^*)}
      \Lt^2(G,\A,\s,\tilde\alpha)\otimes_{\varphi}\F\)}; \node (m-2-1)
    at (0,0) {\(\Lt^2(G^2,\E_0,v_0,\mu_0) \otimes_\varphi \F\)};
    \node (m-2-2) at (5,1)
    {\(\Lt^2(G^2,\Bun[I],d_1,\lambda_1)\otimes_{\Cont_0(G,\A\A^*)}
      \Lt^2(G,\A\A^*,\rg,\alpha)\otimes_{\varphi}\F\)};
    \meascor{m-1-1}{\gamma_1^{-1}\otimes\Id}{}{m-1-2}
    \meascor{m-1-1}{}{d_1^*(U)}{m-2-1}
    \meascor{m-2-2}{\gamma_2\otimes\Id}{}{m-2-1}
    \meascor{m-1-2}{\Id\otimes U}{}{m-2-2}
  \end{tikzpicture}
  \caption{The unitary \(d_1^*(U)\) as the composite of
    \(\Id \otimes U\) with two canonical isomorphisms \(\gamma_1\)
    and~\(\gamma_2\) from
    Lemma~\ref{lem:compose_corr_from_measure-family}.}
  \label{fig:def_d1U}
\end{figure}
Another isomorphism of correspondences as
in~\eqref{eq:domain_codomain_U_coherence} is described by the first
diagram in Figure~\ref{fig:U_coherence}.  It is useful to split this
as the composition of two isomorphisms of correspondences
\begin{align*}
  d_0^*(U)\colon \Lt^2(G^2,\E_2,v_2,\mu_2)
  \otimes_\varphi \F
  &\congto
    \Lt^2(G^2,\E_1,v_1,\mu_1) \otimes_\varphi
    \F,\\
  d_2^*(U)\colon  \Lt^2(G^2,\E_1,v_1,\mu_1)
  \otimes_\varphi \F
  &\congto
    \Lt^2(G^2,\E_0,v_0,\mu_0) \otimes_\varphi \F.
\end{align*}
These unitaries are composites of \(\Id \otimes U\) with canonical
maps as in Figure~\ref{fig:U_coherence}.
They generalise the unitaries \(d_0^*(U)\) and \(d_2^*(U)\)
in~\cite{Buss-Holkar-Meyer:Universal} to Fell bundles.

\begin{definition}
  \label{def:representation_Fell}
  Let~\(D\) be a \(\Cst\)\nb-algebra and let~\(\F\) be a Hilbert
  \(D\)\nb-module.  Let~\(G\) be a locally compact Hausdorff groupoid
  with Haar system~\(\alpha\), and let~\(\A\) be a Fell bundle
  over~\(G\).  A \emph{representation of~\((G,\A,\alpha)\) on~\(\F\)}
  is a pair \((\varphi,U)\) consisting of a nondegenerate
  \Star{}homomorphism
  \(\varphi\colon \Cont_0(G^0,\FU) \to \Bound(\F)\), which
  turns~\(\F\) into a \(\Cont_0(G^0,\FU)\)-\(D\)-correspondence, and
  an isomorphism of \(\Cont_0(G,\A\A^*)\)-\(D\)-correspondences
  \[
    U\colon \Lt^2(G,\A,\s,\tilde\alpha) \otimes_\varphi \F \congto
    \Lt^2(G,\A\A^*,\rg,\alpha) \otimes_\varphi \F,
  \]
  such that the isomorphisms~\(d_j^*(U)\) built in
  Figure~\ref{fig:U_coherence} satisfy
  \(d_2^*(U) d_0^*(U) = d_1^*(U)\).  That is, the following diagram of
  isomorphisms of \(\Cont_0(G^2,\Bun[I])\)-\(D\)-correspondences
  commutes:
  \[
    \begin{tikzpicture}[baseline=(current bounding box.west)]
      \matrix (m) [cd,column sep=4em, row sep=2.5em] {
        \Lt^2(G^2,\E_2,v_2,\mu_2) \otimes_\varphi
        \F &
        \Lt^2(G^2,\E_1,v_1,\mu_1) \otimes_\varphi \F\\
        &\Lt^2(G^2,\E_0,v_0,\mu_0) \otimes_\varphi \F.\\
      }; \draw[cdar] (m-1-1) -- node {\(\scriptstyle d_0^*(U)\)}
      (m-1-2); \draw[cdar] (m-1-1) -- node[swap]
      {\(\scriptstyle d_1^*(U)\)} (m-2-2); \draw[cdar] (m-1-2) -- node
      {\(\scriptstyle d_2^*(U)\)} (m-2-2);
    \end{tikzpicture}
  \]
\end{definition}

The point of this definition is the following universal property,
which is analogous to
\cite{Buss-Holkar-Meyer:Universal}*{Theorem~3.23}:

\begin{theorem}
  \label{the:universal_groupoid_Haus}
  Let~\(G\) be a locally compact Hausdorff groupoid with Haar
  system~\(\alpha\), and let~\(\A\) be a Fell bundle over~\(G\).
  Let~\(D\) be a \(\Cst\)\nb-algebra and~\(\F\) a Hilbert \(D\)\nb-module.
  There is a bijection between nondegenerate representations of the
  \(\Cst\)\nb-algebra \(\Cst(G,\A)\) on~\(\F\) and representations of
  \((G,\A,\alpha)\) on~\(\F\).
  This bijection is natural in the following two ways:
  \begin{enumerate}
  \item \label{en:universal_groupoid_Haus1}%
    Let \(\F_1\) and~\(\F_2\) be two Hilbert \(D\)\nb-modules and let
    \(V\colon \F_1\injto\F_2\) be an isometry.  Then it intertwines two
    representations of \(\Cst(G,\A)\) on \(\F_1\) and~\(\F_2\)
    if and only if it intertwines the corresponding representations of
    \((G,\A,\alpha)\).
  \item \label{en:universal_groupoid_Haus2}%
    Let~\(\E\) be a \(\Cst\)\nb-correspondence from~\(D\) to a
    \(\Cst\)\nb-algebra~\(D'\).  A representation of \(\Cst(G,\A)\)
    or of \((G,\A,\alpha)\) on~\(\F\) induces a representation of the
    same type on \(\F\otimes_D\E\).  The bijections between the
    two types of representations on \(\F\) and \(\F\otimes_D\E\)
    are compatible with these induction processes.
  \end{enumerate}
  This characterises \(\Cst(G,\A)\) uniquely up to canonical
  isomorphism.
\end{theorem}

We will prove this in Sections \ref{sec:integrate}
and~\ref{sec:disintegrate}.

\section{Integration of representations}
\label{sec:integrate}

Let \((G,\alpha)\) be a locally compact, Hausdorff groupoid with Haar
system and let~\(\A\) be a Fell bundle over~\(G\).  Let~\(D\) be a
\(\Cst\)\nb-algebra, \(\F\) a Hilbert \(D\)\nb-module,
and~\((\varphi,U)\) a representation of the Fell bundle~\(\A\)
on~\(\F\).  We are going to ``integrate'' this representation to a
nondegenerate \Star{}homomorphism
\(L\colon \Contc(G,\A)\to \Bound(\F)\) bounded in the \(I\)\nb-norm.
We follow the construction in
\cite{Buss-Holkar-Meyer:Universal}*{Section~4}.

Let \(f\in \Contc(G,\A)\).  Below, we will decompose
\(f(g)=f_1(g)^*\cdot f_2(g)\) with \(f_1\in \Contc(G,\A\A^*)\) and
\(f_2\in \Contc(G,\A)\).  We briefly write this as
\(f = f_1^\dagger \cdot f_2\), where
\(f_1^\dagger(g) \defeq f_1(g)^*\) is the involution in the
\(\Cst\)\nb-algebra \(\Cont_0(G,\A\A^*)\); our notation avoids confusion
with the involution \(f^*(g) \defeq f(g^{-1})^*\) in the convolution
algebra \(\Contc(G,\A)\).  We may view \(f_1\) and~\(f_2\) as elements
of \(\Lt^2(G,\A\A^*,\rg,\alpha)\) and \(\Lt^2(G,\A,\s,\tilde\alpha)\),
respectively.  These elements give creation operators
\begin{alignat*}{2}
  T^\rg_{f_1}\colon \F &\to \Lt^2(G,\A\A^*,\rg,\alpha) \otimes_\varphi
  \F,&\qquad
  \xi &\mapsto f_1 \otimes \xi,\\
  T^\s_{f_2}\colon \F &\to \Lt^2(G,\A,\s,\tilde\alpha) \otimes_\varphi
  \F,&\qquad \xi &\mapsto f_2 \otimes \xi.
\end{alignat*}
We shall use many similar creation operators in the following.  The
general pattern is as follows.  Let~\(\E\) be a Hilbert module over
a \(\Cst\)\nb-algebra~\(B\), let~\(\F\) be a
\(B\)\nb-\(D\)-correspondence, and let \(x\in\E\).  Then there is
an adjointable operator
\[
  T_x\colon \F\to \E\otimes_B \F,\qquad y\mapsto x\otimes y.
\]
Its adjoint~\(T_x^*\) maps \(z\otimes y \mapsto \braket{x}{z}\cdot y\)
for \(z\in\E\), \(y\in\F\).  When we write~\(T^h_\eta\), then~\(h\)
is one of the continuous maps \(\s,\rg,d_j,v_j\) in
Figure~\ref{fig:U_coherence}, \(\eta\) is a \(\Contc\)\nb-section of
the continuous field of \(\Cst\)\nb-correspondences by which~\(h\) in
that diagram is decorated, and~\(T^h_\eta\) means the creation
operator for the resulting \(\Cst\)\nb-correspondence of
\(\Lt^2\)\nb-sections.

We are going to define
\begin{equation}
  \label{eq:def-integration-FB}
  L(f)\defeq (T^\rg_{f_1})^*U T^\s_{f_2}
\end{equation}
as in \cite{Buss-Holkar-Meyer:Universal}*{Equation~(4.4)}.
The following two lemmas show that the factorisation \(f = f_1^\dagger
\cdot f_2\) exists, that the operator \(L(f)\in\Bound(\F)\)
in~\eqref{eq:def-integration-FB} does not depend on it, and that
\(\norm{L(f)} \le \norm{f}_I\) (compare
\cite{Buss-Holkar-Meyer:Universal}*{Lemma~4.3}).

\begin{lemma}
  \label{lem:factor_f}
  Let \(f\in \Contc(G,\A)\).  There are \(f_1\in \Contc(G,\A\A^*)\)
  and \(f_2\in \Contc(G,\A)\) with \(f=f_1^\dagger\cdot f_2\) and
  \(\norm{f(g)}=\norm{f_1(g)}^2=\norm{f_2(g)}^2\) for all \(g\in G\).
\end{lemma}

\begin{proof}
  Any element~\(x\) of a \(\Cst\)\nb-algebra~\(B\) may be written as
  \(x=bc\) with \(b=(xx^*)^{1/4}\) and
  \(c\defeq \lim_{n\to\infty} {}(1/n+xx^{*})^{-1/4}x\), and then
  \(\norm{c}^2=\norm{x}=\norm{b}^2\) (see
  \cite{Davidson:Cstar_example}*{Lemma~I.5.2}).  Working in the
  linking \(\Cst\)\nb-algebra of a left Hilbert
  \(B\)\nb-module~\(\E\), it follows that any \(\xi\in\E\) may be
  written as \(\xi=b\cdot \eta\) with
  \(b=\braket{\xi}{\xi}^{1/4}\in B\) and
  \(\eta \defeq \lim_n {}(\frac{1}{n}+\braket{\xi}{\xi})^{-1/4}\cdot
  \xi\in \E\), and then \(\norm{\eta}^2=\norm{\xi}=\norm{b}^2\).  We
  apply this to the left \(\Cont_0(G,\A\A^*)\)\nb-Hilbert module
  \(\Cont_0(G,\A)\) to write \(f=f_1^\dagger\cdot f_2\) with
  \(f_1(g)\defeq (f(g)\cdot f(g)^*)^{1/4} = f_1^\dagger(g)\) and
  \[
    f_2(g)\defeq \lim_{n\to \infty} {}\Bigl(\frac{1}{n}+f(g)\cdot
    f(g)^*\Bigr)^{-1/4}\cdot f(g).
  \]
  These factors have the desired norm.
\end{proof}

\begin{definition}
  \label{def:Cont_KpartialK}
  Let \(K\subseteq G\) be compact and let~\(\A\) be a field of
  Banach spaces over~\(G\).  We let
  \(\Cont_0(K\setminus \partial K,\A) \subseteq \Contc(G,\A)\) be
  the subspace of those continuous sections of~\(\A\) over~\(G\) that
  are supported on~\(K\).  Equivalently,
  \(\Cont_0(K\setminus \partial K,\A)\) consists of continuous
  sections of~\(\A\) over~\(K\) that vanish on the boundary of~\(K\)
  in~\(G\).
\end{definition}

\begin{lemma}
  \label{lem:Lf_well-defined}
  There is an approximate unit \((e_n)_{n\in N}\) of the
  \(\Cst\)\nb-algebra \(\Cont_0(G,\A\A^*)\) that is contained in
  \(\Contc(G,\A\A^*)\).
  If \(f_1\in \Contc(G,\A\A^*)\) and \(f_2\in \Contc(G,\A)\), then
  \[
    T_{f_1}^* U T_{f_2} = \lim T_{e_n}^* U T_{f_1^\dagger \cdot f_2}.
  \]
  The operator \(L(f)\in \Bound(\F)\) in~\eqref{eq:def-integration-FB}
  is well defined and depends linearly on~\(f\).
  It is bounded and \(\norm{L(f)}\le \norm{f}_I\).
\end{lemma}

\begin{proof}
  Since \(\Cont_0(K\setminus \partial K,\A\A^*)\) is a
  \(\Cst\)\nb-algebra, it has an approximate unit.
  Letting~\(K\) run through the directed set of compact subsets
  of~\(G\) gives a net~\((e_n)\) in \(\Contc(G,\A\A^*)\) that is an
  approximate unit in \(\Cont_0(G,\A\A^*)\).
  Let~\(M_\psi\) for a function \(\psi\in \Contc(G,\A\A^*)\) denote
  the operator of pointwise multiplication by~\(\psi\) on
  \(\Lt^2(G,\A\A^*,\rg,\alpha)\otimes_\varphi\F\) or
  \(\Lt^2(G,\A,\s,\tilde\alpha)\otimes_\varphi\F\).
  Since~\(U\) is an isomorphism of correspondences, \(M_\psi
  U=UM_\psi\).
  The relations
  \[
    M_\psi T^\rg_{f_1} = T^\rg_{\psi\cdot f_1},\qquad M_\psi
    T^\s_{f_2} = T^\s_{\psi\cdot f_2}
  \]
  are obvious.  Let \(K\subseteq G\) be the compact support
  of~\(f_1\).  Then \(\lim f_1 e_n = f_1\) in the norm topology on
  \(\Cont_0(K\setminus \partial K,\A\A^*)\).  The map
  \(f_1 \mapsto T^\rg_{f_1}\) is bounded on the latter Banach space.
  Thus \(\lim {}\norm{T^\rg_{f_1 e_n} - T^\rg_{f_1}}=0\).  Now we
  compute
  \begin{multline*}
    (T^\rg_{f_1})^* U T^\s_{f_2} = \lim {} (T^\rg_{f_1 e_n})^* U
    T^\s_{f_2} = \lim {} (M_{f_1} T^\rg_{e_n})^* U T^\s_{f_2} \\= \lim
    {} (T^\rg_{e_n})^* M_{f_1^\dagger} U T^\s_{f_2} = \lim
    {}(T^\rg_{e_n})^* U M_{f_1^\dagger} T^\s_{f_2} = \lim
    {}(T^\rg_{e_n})^* UT^\s_{f_1^\dagger\cdot f_2}.
  \end{multline*}
  This implies that~\(L(f)\) is well defined and depends linearly
  on~\(f\).  Now choose \(f_1,f_2\) with
  \(\norm{f_1(g)}^2=\norm{f_2(g)}^2=\norm{f(g)}\) for all \(g\in G\)
  as in Lemma~\ref{lem:factor_f}.  Since~\(U\) is unitary, the
  definition of~\(L(f)\) using this factorisation of~\(f\) implies the
  estimate
  \[
    \norm{L(f)}\le \norm{f_1}_{\Lt^2(G,\A\A^*,\rg,\alpha)} \cdot
    \norm{f_2}_{\Lt^2(G,\A,\s,\tilde\alpha)} \le \norm{f}_I.\qedhere
  \]
\end{proof}

\begin{remark}
  The proof above shows the following sharper norm estimate
  for~\(L(f)\):
  \[
    \norm{L(f)}^2 \le \sup_{x\in G^0} \norm*{ \int_{G} {}
      (f(g)f(g)^*)^{1/2} \,\diff\alpha^x(g)} \cdot \sup_{x\in G^0}
    \norm*{ \int_{G} {}(f(g)^*f(g))^{1/2} \,\diff\tilde\alpha_x(g)}.
  \]
  Here we use that \(f_1(g)^*f_1(g)=(f(g)f(g)^*)^{1/2}\) and
  \[
    f_2(g)^*f_2(g) =\lim_{n\to\infty} f(g)^*(1/n+f(g)f(g)^*)^{-1/2}f(g)
    =(f(g)^*f(g))^{1/2}.
  \]
  If~\(h\) is a continuous function on the spectrum of~\(aa^*\),
  then \(a^* h(a a^*) a = (h\cdot x)(a^*\cdot a)\) because this
  holds for all polynomials.
\end{remark}

\begin{lemma}
  \label{lem:Fell_bundle_section_full}
  The Hilbert \(\Cont_0(G^0,\FU)\)-modules
  \(\Lt^2(G,\A,\s,\tilde\alpha)\) and \(\Lt^2(G,\A\A^*,\rg,\alpha)\)
  are full.
\end{lemma}

\begin{proof}
  This follows from Lemma~\ref{lem:fullness-Hilbert-modules-fields}
  because \(\Lt^2(G,\A,\s,\tilde\alpha)\) and
  \(\Lt^2(G,\A\A^*,\rg,\alpha)\) are completions
  of spaces of sections of upper semicontinuous fields of Hilbert
  modules over~\(G\) whose restrictions to~\(G^0\) are isomorphic
  to~\(\FU\).
\end{proof}

\begin{lemma}[compare \cite{Buss-Holkar-Meyer:Universal}*{Lemma~4.5}]
  \label{lem:Lf_nondegenerate}
  The closed linear span of~\(L(f)\xi\) for \(f\in\Contc(G,\A\A^*)\)
  and \(\xi\in\F\) is equal to~\(\F\).
\end{lemma}

\begin{proof}
  Because of~\eqref{eq:def-integration-FB}, we must show that the
  linear span of \((T^\rg_{f_1})^* U T^\s_{f_2}\eta\) for
  \(f_1\in\Contc(G,\A\A^*)\), \(f_2\in\Contc(G,\A)\), \(\eta\in \F\)
  is dense in~\(\F\).  By construction, \(\Contc(G,\A)\) is dense in
  \(\Lt^2(G,\A,\s,\tilde{\alpha})\) and \(\Contc(G,\A\A^*)\) is
  dense in \(\Lt^2(G,\A\A^*,\rg,\alpha)\).  So the linear span of
  \(T^\s_{f_2}\eta\) for \(f_2\in\Contc(G,\A)\), \(\eta\in \F\) is
  dense in \(\Lt^2(G,\A,\s,\tilde{\alpha}) \otimes_\varphi \F\).
  Since~\(U\) is unitary, the linear span of \(U T^\s_{f_2}\eta\)
  for such \(f_2\) and~\(\eta\) is dense in
  \(\Lt^2(G,\A\A^*,\rg,\alpha)\otimes_\varphi \F\).  Then the linear
  span of \((T^\rg_{f_1})^* U T^\s_{f_2}\eta\) for
  \(f_1\in\Contc(G,\A\A^*)\), \(f_2\in\Contc(G,\A)\), \(\eta\in \F\)
  is dense in
  \(\varphi(\braket{\Lt^2(G,\A\A^*,\rg,\alpha)}{\Lt^2(G,\A\A^*,\rg,\alpha)})
  \F\).  This is equal to~\(\F\) because
  \(\Lt^2(G,\A\A^*,\rg,\alpha)\) is full by
  Lemma~\ref{lem:Fell_bundle_section_full} and~\(\varphi\) is
  nondegenerate.
\end{proof}

\begin{proposition}[compare
  \cite{Buss-Holkar-Meyer:Universal}*{Proposition~4.6}]
  \label{pro:Lf_star-representation}
  The map~\(L\) is a nondegenerate \Star{}homomorphism from the
  convolution algebra \(\Contc(G,\A)\) to \(\Bound(\F)\).
\end{proposition}

The proposition follows as in the proof of
\cite{Buss-Holkar-Meyer:Universal}*{Proposition~4.6} once we show
\begin{equation}
  \label{eq:Lf_star-rep}
  L(f_1)^*L(f_2) = L(f_1^* * f_2)
\end{equation}
for all \(f_1,f_2\in \Contc(G,\A)\).  This will take a while and
will be accomplished after
Lemma~\ref{lem:approximate_proof_representation_condition}.  We
introduce some more creation operators.  Recall the fields of
Hilbert bimodules~\(\E_j\) for \(j=0,1,2\) over~\(G^2\) defined
in~\eqref{eq:E_012}.  Let
\[
  \eta_0\in \Contc(G^2,\E_0),\qquad
  \eta_1\in \Contc(G^2,\E_1),\qquad
  \eta_2\in \Contc(G^2,\E_2).
\]
These define adjointable creation operators
\begin{alignat*}{2}
  T^{v_0}_{\eta_0}\colon \F &\to \Lt^2(G^2, \E_0,v_0,\mu_0)
  \otimes_\varphi \F,&\qquad
  \xi &\mapsto \eta_0 \otimes \xi,\\
  T^{v_1}_{\eta_1}\colon \F &\to \Lt^2(G^2, \E_1,v_1,\mu_1)
  \otimes_\varphi \F,&\qquad
  \xi &\mapsto \eta_1 \otimes \xi,\\
  T^{v_2}_{\eta_2}\colon \F &\to \Lt^2(G^2, \E_2,v_2,\mu_1)
  \otimes_\varphi \F,&\qquad \xi &\mapsto \eta_2 \otimes \xi.
\end{alignat*}
Then \((T^{v_1}_{\eta_1})^* d_0^*(U) T^{v_2}_{\eta_2}\),
\((T^{v_0}_{\eta_0})^* d_1^*(U) T^{v_2}_{\eta_2}\) and
\((T^{v_0}_{\eta_0})^* d_2^*(U) T^{v_1}_{\eta_1}\) are adjointable
operators on~\(\F\).  The following lemma relates these operators to
the
map~\(L\):

\begin{lemma}
  \label{lem:whiskered_d0U}
  In the situation above,
  \[
    \eta_1^\dagger \cdot \eta_2 \in \Contc(G^2,d_0^*(\A)),\qquad
    \eta_0^\dagger \cdot \eta_2 \in \Contc(G^2,d_1^*(\A)),\quad
    \eta_0^\dagger \cdot \eta_1 \in \Contc(G^2,d_2^*(\A)).
  \]
  Then
  \(\lambda_0(\eta_1^\dagger \cdot \eta_2)\),
  \(\lambda_1(\eta_0^\dagger \cdot \eta_2)\), and
  \(\lambda_2(\eta_0^\dagger \cdot \eta_1)\) are well-defined elements
  of \(\Contc(G,\A)\) by Lemma~\textup{\ref{lem:continuity_integral}},
  and
  \begin{align*}
    (T^{v_1}_{\eta_1})^* d_0^*(U) T^{v_2}_{\eta_2}
    &= L\bigl(\lambda_0(\eta_1^\dagger \cdot \eta_2)\bigr),\\
    (T^{v_0}_{\eta_0})^* d_1^*(U) T^{v_2}_{\eta_2}
    &= L\bigl(\lambda_1(\eta_0^\dagger \cdot \eta_2)\bigr),\\
    (T^{v_0}_{\eta_0})^* d_2^*(U) T^{v_1}_{\eta_1}
    &= L\bigl(\lambda_2(\eta_0^\dagger \cdot \eta_1)\bigr).
  \end{align*}
\end{lemma}

\begin{proof}
  By definition, \(\eta^\dagger_1 \cdot \eta_2\) is a section of the
  field \(d_2^*(\A)^* \cdot \Bun[I]\cdot d_1^*(\A)\).  Since
  \(\Bun[I]\cdot d_1^*(\A) = \Bun[I]\cdot d_2^*(\A) d_0^*(\A)\), this
  field is contained in
  \(d_2^*(\A)^* \cdot d_2^*(\A) \cdot d_0^*(\A) \subseteq d_0^*(\A)\).
  Then \(\lambda_0(\eta_1^\dagger \cdot \eta_2)\) is a well-defined
  element of \(\Contc(G,\A)\) by Lemma~\ref{lem:continuity_integral}.
  Explicitly,
  \[
    \lambda_0(\eta_1^\dagger \cdot \eta_2)(h) \defeq \int_{G^2} {}
    (\eta_1^\dagger \cdot \eta_2)(g,h) \,\diff \lambda_0^h(g,h) =
    \int_{G} \eta_1(g,h)^* \cdot \eta_2(g,h) \,\diff
    \tilde\alpha_{\rg(h)}(g).
  \]
  It is trivial that
  \(\eta_0^\dagger \cdot \eta_2 \in \Contc(G^2,d_1^*(\A))\) and
  \(\eta_0^\dagger \cdot \eta_1 \in \Contc(G^2,d_2^*(\A))\).  Thus
  \(\lambda_1(\eta_0^\dagger \cdot \eta_2)\) and
  \(\lambda_2(\eta_0^\dagger \cdot \eta_1)\) are well-defined
  elements of \(\Contc(G,\A)\).  The domain and codomain
  \(\Lt^2(G^2,\E_2,v_2,\mu_2) \otimes_\varphi \F\) and
  \(\Lt^2(G^2,\E_1,v_1,\mu_1) \otimes_\varphi \F\) of \(d_0^*(U)\)
  are isomorphic to
  \begin{gather*}
    \Lt^2(G^2,\Bun[I] \cdot d_2^*(\A),d_0,\lambda_0)
    \otimes_{\Cont_0(G,\A\A^*)} \Lt^2(G,\A,\s,\tilde\alpha)
    \otimes_\varphi \F,\\
    \Lt^2(G^2,\Bun[I] \cdot d_2^*(\A),d_0,\lambda_0)
    \otimes_{\Cont_0(G,\A\A^*)} \Lt^2(G,\A,\rg,\alpha) \otimes_\varphi
    \F,
  \end{gather*}
  respectively, as in Figure~\ref{fig:U_coherence}.  We have defined
  \(d_0^*(U)\) as the composite of these isomorphisms with the
  operator
  \(\Id_{\Lt^2(G^2,\Bun[I] \cdot d_2^*(\A),d_0,\lambda_0)}\otimes
  U\) between these two tensor products
  (compare Figure~\ref{fig:def_d1U}).  Since
  \[
    \bigl(\Bun[I]\cdot d_2^*(\A)\bigr)_{g,h} = \A_g \A_h \A_h^* \A_g^*
    \A_g = \A_g \A_h \A_h^*,
  \]
  it follows that \(\Bun[I]\cdot d_2^*(\A)\) is a field of Hilbert
  modules over \(d_0^*(\A\A^*)\).  Therefore, linear combinations of
  products \(\eta_3 \cdot d_0^*(\eta_4)\) for
  \(\eta_3\in \Contc(G^2,\Bun[I]\cdot d_2^*(\A))\),
  \(\eta_4\in \Contc(G,\A\A^*)\) are dense in
  \(\Contc(G^2,\Bun[I]\cdot d_2^*(\A))\) in the inductive limit
  topology.  So it is no loss of generality to assume that~\(\eta_1\)
  has this form.  Then
  \(T^{v_1}_{\eta_1} = T^{d_0}_{\eta_3} T^\rg_{\eta_4}\).
  Here~\(T^\rg_{\eta_4}\) maps~\(\F\) to
  \(\Lt^2(G,\A\A^*,\rg,\alpha) \otimes_\varphi \F\)
  and~\(T^{d_0}_{\eta_3}\) maps this on to
  \begin{multline*}
    \Lt^2(G^2,\Bun[I] \cdot d_2^*(\A),d_0,\lambda_0)
    \otimes_{\Cont_0(G,\A\A^*)} \Lt^2(G,\A\A^*,\rg,\alpha)
    \otimes_\varphi \F \\\cong \Lt^2(G^2,\E_1,v_1,\mu_1)
    \otimes_\varphi \F.
  \end{multline*}
  If~\(\eta_1\) has this form, then
  \[
    (T^{v_1}_{\eta_1})^* d_0^*(U) = (T^\rg_{\eta_4})^*
    (T^{d_0}_{\eta_3})^* d_0^*(U) = (T^\rg_{\eta_4})^* U
    (T^{d_0}_{\eta_3})^*.
  \]
  As an operator from \(\Lt^2(G^2,\E_2,v_2,\mu_2)\)
  to \(\Lt^2(G,\A,\s,\tilde\alpha)\), the annihilation
  operator~\((T^{d_0}_{\eta_3})^*\) acts by
  \[
    (T^{d_0}_{\eta_3})^*(\psi)(h) = \int_{G^2} \eta_3(g,h)^* \psi(g,h)
    \,\diff\lambda_0^h(g,h) = \lambda_0\bigl( \eta_3^\dagger \cdot
    \psi\bigr)(h).
  \]
  Thus
  \((T^{d_0}_{\eta_3})^* T^{v_2}_{\eta_2} =
  T^\s_{\lambda_0(\eta_3^\dagger \eta_2)}\) as an operator from~\(\F\)
  to \(\Lt^2(G,\A,\s,\tilde\alpha) \otimes_\varphi \F\).  Then
  \begin{multline*}
    (T^{v_1}_{\eta_1})^* d_0^*(U) (T^{v_2}_{\eta_2}) =
    (T^\rg_{\eta_4})^* U (T^{d_0}_{\eta_3})^* T^{v_2}_{\eta_2} =
    (T^\rg_{\eta_4})^* U T^\s_{\lambda_0(\eta_3^\dagger \eta_2)} \\=
    L\bigl(\eta_4^\dagger \cdot \lambda_0(\eta_3^\dagger \eta_2)\bigr)
    = L\bigl(\lambda_0(\eta_1^\dagger \cdot \eta_2)\bigr).
  \end{multline*}
  The other composites
  \((T^{v_0}_{\eta_0})^* d_1^*(U) T^{v_2}_{\eta_2}\) and
  \((T^{v_0}_{\eta_0})^* d_2^*(U) T^{v_1}_{\eta_1}\) are similar.
\end{proof}

Since \((\varphi,U)\) is a representation, it satisfies
\(d_2^*(U)^* d_1^*(U) = d_0^*(U)\).  This implies
\begin{equation}
  \label{eq:integration_representation_to_check}
  (T^{v_1}_{\eta_1})^* d_2^*(U)^* d_1^*(U) T^{v_2}_{\eta_2}
  = (T^{v_1}_{\eta_1})^* d_0^*(U) T^{v_2}_{\eta_2}
\end{equation}
for all \(\eta_1,\eta_2\) as above.  Now we fix
\(f_1,f_2\in\Contc(G,\A)\).  Assume first that there are
\(\eta_1,\eta_2\) as above with
\(\eta_1^\dagger \eta_2 = d_2^*(f_1^\dagger) d_1^*(f_2)\), that is,
\((\eta_1^\dagger \eta_2)(g,h) = f_1(g)^* f_2(g\cdot h)\).  Then
\[
  (T^{v_1}_{\eta_1})^* d_0^*(U) T^{v_2}_{\eta_2} =
  L(\lambda_0(\eta_1^\dagger \eta_2)) =
  L\bigl(\lambda_0(d_2^*(f_1^\dagger) \cdot d_1^*(f_2))\bigr)
\]
with
\begin{align*}
  \lambda_0(d_2^*(f_1^\dagger) d_1^*(f_2))(h)
  &= \int_{G^2} f_1(g)^* f_2(g\cdot h) \,\diff \lambda_0^h(g,h)
  \\&= \int_{G} f_1(g^{-1})^* f_2(g^{-1}\cdot h)
  \,\diff \alpha^{\rg(h)}(g)
  = (f_1^* * f_2)(h)
\end{align*}
for all \(h\in G\); here we have used the definition of~\(\lambda_0\)
in~\eqref{eq:lambda_definition0} and
substituted~\(g^{-1}\) for~\(g\).  Unfortunately, we cannot just take
\(\eta_1 =d_2^*( f_1)\) and \(\eta_2 = d_1^*(f_2)\) because \(\eta_1\)
and~\(\eta_2\) must be compactly supported sections and take values in
\(\Bun[I]\cdot d_2^*(\A)\) and \(\Bun[I]\cdot d_1^*(\A)\),
respectively.  The following lemma avoids this issue:

\begin{lemma}
  \label{lem:approximate_proof_representation_condition}
  Let \((e_n)_{n\in N}\) be an approximate unit in
  \(\Contc(G^2,\Bun[I])\) that is quasi-central for
  \(\Contc(G^2,v_0^*\FU)\).  Let \(f_1,f_2\in\Contc(G,\A)\),
  \(\eta_{1,n} \defeq e_n d_2^*(f_1)\), and
  \(\eta_{2,n} \defeq e_n d_1^*(f_2)\).  Then
  \begin{align}
    \label{eq:approximate_proof_representation_condition_1}
    \lim {}(T^{v_1}_{\eta_{1,n}})^* d_0^*(U) T^{v_2}_{\eta_{2,n}}
    &= L(f_1^* * f_2),\\
    \label{eq:approximate_proof_representation_condition_2}
    \lim {}(T^{v_1}_{\eta_{1,n}})^* d_2^*(U)^* d_1^*(U)
    T^{v_2}_{\eta_{2,n}} &= L(f_1)^* L(f_2)
  \end{align}
  with norm convergence.
\end{lemma}

\begin{proof}
  The asserted equations are meaningful and
  Lemma~\ref{lem:whiskered_d0U} applies because
  \(\eta_{1,n} \in \Contc(G^2,\Bun[I]\cdot d_2^*(\A))\) and
  \(\eta_{2,n} \in \Contc(G^2,\Bun[I]\cdot d_1^*(\A))\) for all
  \(n\in N\).  Thus
  \((T^{v_1}_{\eta_{1,n}})^* d_0^*(U) T^{v_2}_{\eta_{2,n}} =
  L(\lambda_0(\eta_{1,n}^\dagger \cdot \eta_{2,n}))\).  A
  computation as above shows that
  \[
    \lambda_0(\eta_{1,n}^\dagger \cdot \eta_{2,n})(h) = \int_{G}
    f_1(g)^* \cdot (e_n^\dagger\cdot e_n)(g,h) \cdot f_2(g h)
    \,\diff\tilde\alpha_{\rg(h)}(g).
  \]
  Here \((e_n^\dagger \cdot e_n)_{n\in N}\) is again a quasi-central
  approximate unit in \(\Contc(G^2,\Bun[I])\).  We claim that the net
  of sections
  \(f_1(g)^* \cdot (e_n^\dagger\cdot e_n)(g,h) \cdot f_2(g h)\)
  converges in the supremum norm towards \(f_1(g)^* \cdot f_2(g h)\).
  Since \(f_1,f_2\) are compactly supported, there is a compact subset
  \(K\subseteq G^2\) so that \(f_1(g)\) or \(f_2(g h)\)
  vanishes for \((g,h)\notin K\).
  Therefore, the convergence in supremum norm is a convergence in the
  inductive limit topology.
  And then \(\lim {}(T^{v_1}_{\eta_{1,n}})^* d_0^*(U)
  T^{v_2}_{\eta_{2,n}} = L(f_1^* * f_2)\) follows.
  Lemma~\ref{lem:factor_f} allows to write \(f_1 = f_3 \cdot f_4\) and
  \(f_2 = f_5\cdot f_6\) with \(f_4,f_6\in \Contc(G,\A)\) and \(f_3,
  f_5 \in \Contc(G,\A\A^*)\).
  It suffices to prove that \(f_3(g)^* \cdot (e_n^\dagger\cdot
  e_n)(g,h) \cdot f_5(g h)\) converges uniformly towards \(f_3(g)^*
  \cdot f_5(g h)\).
  Now \(f_3(g)^*\) and \(f_5(g)^*\) are compactly supported continuous
  sections of the fields of ideals \(d_2^*(\A\A^*)\) and
  \(d_1^*(\A\A^*)\) in the field of \(\Cst\)\nb-algebras \(\rg^*\FU\),
  and \(\Bun[I] = d_2^*(\A\A^*)\cdot d_1^*(\A\A^*)\) is the product of
  these two ideals.
  The desired convergence follows from the following elementary
  general fact: let~\(A\) be a \(\Cst\)\nb-algebra, let \(I,J\subseteq
  A\) be two ideals, and let \((e_n)_{n\in N}\) be a quasi-central
  approximate unit for \(I\cdot J = I\cap J\).
  Then \(\lim x\cdot e_n \cdot y=\lim x\cdot y\cdot e_n = x\cdot y\)
  for all \(x\in I\), \(y\in J\).

  This finishes the proof
  of~\eqref{eq:approximate_proof_representation_condition_1}.  We turn
  to~\eqref{eq:approximate_proof_representation_condition_2}.  We keep
  using the factorisations \(f_1 = f_3 \cdot f_4\) and
  \(f_2 = f_5\cdot f_6\) and compute
  \begin{multline*}
    (T^{v_1}_{e_n d_2^*(f_1)})^* d_2^*(U)^* d_1^*(U) T^{v_2}_{e_n
      d_1^*(f_2)}
    = (T^\s_{f_1})^* (T^{d_2}_{e_n})^* d_2^*(U)^* d_1^*(U)
    T^{d_1}_{e_n} T^\s_{f_2}
    \\= (T^\s_{f_3 \cdot f_4})^* U^*   (T^{d_2}_{e_n})^*
    T^{d_1}_{e_n} U T^\s_{f_5\cdot f_6}
    = (T^\s_{f_4})^* U^* M_{f_3}^* (T^{d_2}_{e_n})^* T^{d_1}_{e_n}
    M_{f_5} U T^\s_{f_6}.
  \end{multline*}
  The operator \(M_{f_3}^* (T^{d_2}_{e_n})^* T^{d_1}_{e_n} M_{f_5}\)
  in the middle is the exterior tensor product with \(\Id_{\F}\) of an
  operator on \(\Lt^2(G,\A\A^*,\rg,\alpha)\), which we also denote by
  the same name.  If
  \(\psi\in \Contc(G,\A) \subseteq \Lt^2(G,\A,\rg,\alpha)\), then
  \begin{multline*}
    M_{f_3}^* (T^{d_2}_{e_n})^* T^{d_1}_{e_n} M_{f_5} \psi(g)
    = f_3(g)^* \cdot \lambda_2( e_n^\dagger \cdot e_n\cdot
    d_1^*(f_5\cdot \psi))(g)
    \\= \int_{G^2} f_3(g)^* (e_n^\dagger\cdot e_n)(g,h)
    f_5(g\cdot h) \psi(g\cdot h) \,\diff\lambda_2^g(g,h).
  \end{multline*}
  An argument as above shows that the sections
  \((g,h)\mapsto f_3(g)^* (e_n^\dagger\cdot e_n)(g,h) f_5(g\cdot h)\)
  have a uniform compact support and converge in supremum norm towards
  the section \((g,h)\mapsto f_3(g)^* f_5(g\cdot h)\).  The resulting
  operator on \(\Lt^2(G,\A,\rg,\alpha)\) maps~\(\psi\) to the function
  \begin{multline*}
    g\mapsto \int_{G^2} f_3(g)^* f_5(g\cdot h) \psi(g\cdot h)
    \,\diff\lambda_2^g(g,h)
    = \int_{G} f_3(g)^* f_5(g\cdot h) \psi(g\cdot h) \,\diff\alpha^{\s(g)}(h)
    \\= f_3(g)^* \cdot \int_{G} f_5(k) \psi(k) \,\diff\alpha^{\rg(g)}(k).
  \end{multline*}
  As a result,
  \[
    \lim M_{f_3}^* (T^{d_2}_{e_n})^* T^{d_1}_{e_n} M_{f_5} =
    T^\rg_{f_3^\dagger} (T^\rg_{f_5^\dagger})^*.
  \]
  This implies
  \begin{multline*}
    \lim {}(T^{v_1}_{e_n d_2^*(f_1)})^* d_2^*(U)^* d_1^*(U)
    T^{v_2}_{e_n d_1^*(f_2)} = (T^\s_{f_4})^* U^* T^\rg_{f_3^\dagger}
    (T^\rg_{f_5^\dagger})^* U T^\s_{f_6} \\= ((T^\rg_{f_3^\dagger})^*
    U T^\s_{f_4})^* (T^\rg_{f_5^\dagger})^* U T^\s_{f_6} = L(f_3\cdot
    f_4)^* L(f_5 \cdot f_6).
  \end{multline*}
  This computation finishes the proof of the lemma.
\end{proof}

Lemma~\ref{lem:approximate_proof_representation_condition}
and~\eqref{eq:integration_representation_to_check}
imply~\eqref{eq:Lf_star-rep}.  And this finishes the proof of
Proposition~\ref{pro:Lf_star-representation}.  As a result, any
representation of a Fell bundle~\(\A\) integrates to an \(I\)\nb-norm
bounded, nondegenerate representation of the \Star{}algebra
\(\Contc(G,\A)\).  This extends uniquely to a nondegenerate
representation of \(\Cst(G,\A)\) by the universal property.

\subsection{The regular representation}
\label{sec:regular_rep}

We are going to define a ``regular'' representation of the Fell
bundle~\(\A\) and show that it integrates to the usual regular
representation, namely, the action of \(\Contc(G,\A)\) on
\(\Lt^2(G,\A,\s,\tilde\alpha)\) by convolution.  As a consequence, the
latter representation of \(\Contc(G,\A)\) is bounded in the
\(I\)\nb-norm.  The constructions below are the same as
in~\cite{Buss-Holkar-Meyer:Universal}.  We merely add appropriate
continuous fields of \(\Cst\)\nb-algebras and
\(\Cst\)\nb-correspondences everywhere.  We construct the regular
representation on dense subspaces of continuous sections of compact
support.  In this form, the construction will reappear in the proof
of the disintegration theorem in Section~\ref{sec:disintegrate}.

The regular representation \((\varphi,U)\) of~\(\A\) is defined on
the Hilbert \(\Cont_0(G^0,\FU)\)-module
\(\F\defeq \Lt^2(G,\A,\s,\tilde\alpha)\).  Defining~\(\varphi\)
on~\(\F\) amounts to making it a
\(\Cont_0(G^0,\FU)\)-\(\Cont_0(G^0,\FU)\)-correspondence.  We simply
take the \(\Cst\)\nb-correspondence
\(\rg^* \Lt^2(G,\A,\s,\tilde\alpha)\) that is defined by the
decorated topological correspondence
\[
  \begin{tikzpicture}
    \matrix (m) [cd] {
      &(G, \A)\\
      (G^0,\FU)&&(G^0,\FU).\\
    };
    \meascor{m-1-2}{\s}{\tilde\alpha}{m-2-3}
    \meascor{m-1-2}{}{\rg}{m-2-1}
  \end{tikzpicture}
\]
In other words, if \(f\in\Cont_0(G^0,\FU)\), \(\xi\in \Contc(G,\A)\),
then
\begin{equation}\label{eq:left-reg-rep-multi}
(\varphi(f)\xi)(g) \defeq f(\rg(g))\cdot \xi(g)
\end{equation}
for all
\(g\in G\); this extends to a nondegenerate representation~\(\varphi\)
of \(\Cont_0(G^0,\FU)\) by adjointable operators on
\(\Lt^2(G,\A,\s,\tilde\alpha)\).  The operator~\(U\) should be an
isomorphism of
\(\Cont_0(G,\A\A^*)\)-\(\Cont_0(G^0,\FU)\)-correspondences
\begin{multline*}
  \Lt^2(G,\A,\s,\tilde\alpha) \otimes_{\Cont_0(G^0,\FU)}
  \rg^* \Lt^2(G,\A,\s,\tilde\alpha)
  \\\congto \Lt^2(G,\A\A^*,\rg,\alpha) \otimes_{\Cont_0(G^0,\FU)}
  \rg^* \Lt^2(G,\A,\s,\tilde\alpha).
\end{multline*}
And then the operators~\(d_j^*(U)\) for \(j=0,1,2\) are isomorphisms of
\(\Cont_0(G^2,\Bun[I])\)-\(\Cont_0(G^0,\FU)\)-correspondences
\begin{multline}
  \label{eq:dj_U_domain_codomain}
  \Lt^2(G^2,\E_k,v_k,\mu_k) \otimes_{\Cont_0(G^0,\FU)}
  \rg^*\Lt^2(G,\A,\s,\tilde\alpha) \\\congto
  \Lt^2(G^2,\E_l,v_l,\mu_l) \otimes_{\Cont_0(G^0,\FU)}
  \rg^*\Lt^2(G,\A,\s,\tilde\alpha);
\end{multline}
here \(k,l\) are such that \(k>l\) and \(\{j,k,l\} = \{0,1,2\}\), and
\(\E_0 = \Bun[I]\), \(\E_1 = \Bun[I] \cdot d_2^*(\A)\) and
\(\E_2 = \Bun[I] \cdot d_1^*(\A)\) as in Figure~\ref{fig:U_coherence}.
All five tensor products above are again associated to topological
correspondences decorated by fields of \(\Cst\)\nb-correspondences by
Lemma~\ref{lem:functoriality-of-corr}, which also describes the
underlying topological correspondences.  We will get the
isomorphism~\(U\) from an isomorphism between the underlying
topological correspondences.  Then the induced isomorphisms
\(d_j^*(U)\) are also induced by isomorphisms of their underlying
topological correspondences, and the equation
\(d_2^*(U) d_0^*(U) = d_1^*(U)\) holds already on the level of
isomorphisms of topological correspondences.  We use
Lemma~\ref{lem:functoriality-of-corr} to describe
\(\Lt^2(G,\A,\s,\tilde\alpha) \otimes_\varphi \rg^*
\Lt^2(G,\A,\s,\tilde\alpha)\) and
\(\Lt^2(G,\A\A^*,\rg,\alpha) \otimes_\varphi \rg^*
\Lt^2(G,\A,\s,\tilde\alpha)\) through topological correspondences.
After some simplification, we get the two topological correspondences
in the top and bottom row of the following diagram:
\begin{equation}
  \label{eq:regular_rep_as_iso_of_top_corr}
  \begin{tikzpicture}[baseline=(current bounding box.west)]
    \matrix (m) [cd] {
      &(G\times_{\s,G^0,\rg} G, \A \otimes_{\FU} \A)\\
      (G,\A\A^*)&&(G^0,\FU)\\
      &(G\times_{\rg,G^0,\rg} G, \A\A^* \otimes_{\FU} \A)\\
    };
    \draw[->,dashed] (m-1-2) -- node {\(\scriptstyle U\)} (m-3-2);
    \meascor{m-1-2}{{\s\circ
        \pr_2}}{{\tilde\alpha\times_{G^0}\tilde\alpha}}{m-2-3}
    \meascor{m-1-2}{}{\pr_1}{m-2-1}
    \meascor{m-3-2}{\s\circ \pr_2}
    {\tilde\alpha\times_{G^0}\alpha}{m-2-3}
    \meascor{m-3-2}{}{\pr_1}{m-2-1}
  \end{tikzpicture}
\end{equation}
Here \(\A \otimes_{\FU} \A\) and \(\A\A^* \otimes_{\FU} \A\) are the
fields of \(\Cst\)\nb-correspondences with fibres
\(\A_g\otimes_{\FU_{\s(g)}} \A_h\) at
\((g,h) \in G\times_{\s,G^0,\rg} G\) and
\(\A_g\A_g^*\otimes_{\FU_{\rg(g)}} \A_k\) at
\((g,k) \in G\times_{\rg,G^0,\rg} G\), respectively.  The fibre
product \(G\times_{\s,G^0,\rg} G\) is equal to~\(G^2\), and
\(G\times_{\rg,G^0,\rg} G\) is homeomorphic to~\(G^2\) through the
homeomorphism
\begin{equation}
  \label{eq:Upsilon}
  \Upsilon\colon G^2 = G\times_{\s,G^0,\rg} G \congto
  G\times_{\rg,G^0,\rg} G,\qquad
  (g,h)\mapsto (g,g\cdot h),
\end{equation}
in \cite{Buss-Holkar-Meyer:Universal}*{Equation~(3.22)}.  This
homeomorphism intertwines the maps to \(G\) and~\(G^0\) in the two
topological correspondences
in~\eqref{eq:regular_rep_as_iso_of_top_corr}, which are
\(\pr_1=d_2\) and
\(\s\circ \pr_2=v_2\) on~\(G^2\).  And it is also exactly compatible
with the measure families on the map to~\(G^0\), which are the
equivalent ways of defining the measure family~\(\mu_2\) along the
map~\(v_2\); these computations are done already
in~\cite{Buss-Holkar-Meyer:Universal} to build the regular
representation of~\(G\) without Fell bundle.  In addition, we now need
an isomorphism between the fields of
\(\Bun[I]\)-\(v_2^*\FU\)-correspondences \(\A \otimes_{\FU} \A\) and
\(\Upsilon^*(\A\A^* \otimes_{\FU} \A)\) over~\(G^2\).  Fibrewise, this
isomorphism~\(\upsilon\) is given by
Lemma~\ref{lem:equality_of_products}, by the multiplication on fibres,
\[
  \upsilon_{g,h}\colon \A_g\otimes_{\FU_{\s(g)}} \A_h \congto \A_g
  \A_g^*\otimes_{\FU_{\rg(g)}} \A_{g h};
\]
since the multiplication in the Fell bundle is continuous, these maps
on the fibres form a continuous map between the relevant fields of
\(\Cst\)\nb-correspondences.  Now
Lemma~\ref{lem:subfield_equality} implies that~\(\upsilon\) is
an isomorphism of fields of \(\Cst\)\nb-correspondences.  Let~\(U\) be
the isomorphism of
\(\Cont_0(G,\A\A^*)\)-\(\Cont_0(G^0,\FU)\)-\(\Cst\)\nb-correspondences
induced by~\(\upsilon\).  This finishes the construction
of~\(U\) as an isomorphism of decorated topological correspondences.
We remark for later use that, by construction, \(U\) is the unique
unitary extension of a
\(\Cont_0(G,\A\A^*)\)-\(\Cont_0(G^0,\FU)\)-bimodule isomorphism
\begin{equation}
  \label{eq:regular_U_on_dense_subspaces}
  U_0\colon \Contc(G\times_{\s,G^0,\rg} G, \A \otimes_{\FU} \A)
  \congto \Contc(G\times_{\rg,G^0,\rg} G, \A\A^* \otimes_{\FU} \A).
\end{equation}

Next we check the identity \(d_2^*(U) d_0^*(U) = d_1^*(U)\).  Let
\(j,k,l\) be as in~\eqref{eq:dj_U_domain_codomain}.  The domain
of~\(d_j^*(U)\) is associated to the topological correspondence
\[
  (G^2\times_{v_k,G^0,\rg} G, \E_k \otimes_{\FU} \A,
  \pr_1,\s\circ\pr_2,\tilde\alpha\circ\mu_k)
\]
from \((G^2,\Bun[I])\) to \((G^0,\FU)\), and the codomain has the same
form with \(l\) instead of~\(k\); here we use the triangles for the
composition \(v_k = \s\circ d_j\) and \(v_l = \rg\circ d_j\) in
Figure~\ref{fig:U_coherence} to simplify the measure families and
fields of \(\Cst\)\nb-correspondences to \(\E_k\) and~\(\E_l\)
as in~\eqref{eq:dj_U_domain_codomain}.  The isomorphisms \(d_j^*(U)\)
come from isomorphisms of decorated topological correspondences.
Their underlying homeomorphisms are
\begin{alignat*}{2}
  d_1^*(\Upsilon)&\colon G^2\times_{v_2,G^0,\rg} G\congto
  G^2\times_{v_0,G^0,\rg} G,&\qquad
  (g,h,l)&\mapsto (g,h,ghl),\\
  d_0^*(\Upsilon)&\colon G^2\times_{v_2,G^0,\rg} G\congto
  G^2\times_{v_1,G^0,\rg} G,&\qquad
  (g,h,l)&\mapsto (g,h,hl),\\
  d_2^*(\Upsilon)&\colon G^2\times_{v_1,G^0,\rg} G\congto
  G^2\times_{v_0,G^0,\rg} G,&\qquad (g,h,l)&\mapsto (g,h,gl),
\end{alignat*}
as in~\cite{Buss-Holkar-Meyer:Universal} after Equation~(3.22).  These
homeomorphisms intertwine the maps to \(G^2\) and~\(G^0\) and exactly
preserve the measure families along the fibres of the map to~\(G^0\)
because~\(U\) preserves the measure families exactly.  Each
\(d_j^*(U)\) also contains an isomorphism between the fields of
\(\Cst\)\nb-correspondences by which its domain and codomain are
decorated.  The restriction of this isomorphism for~\(d_j^*(U)\) to
the fibre at \((g,h,l) \in G^2 \times_{v_k,G^0,\rg} G\) is an
isomorphism
\begin{align*}
  d_1^*(\upsilon)_{g,h,l}\colon \Bun[I]_{g,h}\cdot \A_{g h}
  \otimes_{\FU_{\s(h)}} \A_l
  &\congto \Bun[I]_{g,h} \otimes_{\FU_{\rg(g)}} \A_{g h l},\\
  d_0^*(\upsilon)_{g,h,l}\colon \Bun[I]_{g,h}\cdot \A_{g h}
  \otimes_{\FU_{\s(h)}} \A_l
  &\congto \Bun[I]_{g,h} \cdot \A_g \otimes_{\FU_{\s(g)}} \A_{h l},\\
  d_2^*(\upsilon)_{g,h,l}\colon \Bun[I]_{g,h}\cdot \A_{g\phantom{h}}
  \otimes_{\FU_{\s(g)}} \A_l &\congto \Bun[I]_{g,h}
  \otimes_{\FU_{\rg(g)}} \A_{g l}.
\end{align*}
Each of the isomorphisms \(d_j^*(\upsilon)_{g,h,l}\) is a composite of
a multiplication map in the Fell bundle with the inverse of a
multiplication map in the Fell bundle; the factor~\(\Bun[I]_{g,h}\)
restricts the multiplication maps to suitable Hilbert subbimodules on
which multiplication is an isomorphism even if the Fell
bundle is not saturated.

The associativity of the multiplication in~\(G\) implies
\(d_2^*(\Upsilon) d_0^*(\Upsilon) = d_1^*(\Upsilon)\).
The associativity of the multiplication in~\(\A\) implies
\(d_2^*(\upsilon)_{g,h,h l} d_0^*(\upsilon)_{g,h,l} =
d_1^*(\upsilon)_{g,h,l}\).
These computations together give \(d_2^*(U) d_0^*(U) = d_1^*(U)\).
Thus the pair \((\varphi,U)\) is a representation of~\(\A\) on \(\F =
\Lt^2(G,\A,\s,\tilde\alpha)\).
We remark for later use that the proof also shows the identity
\(d_2^*(U_0) d_0^*(U_0) = d_1^*(U_0)\) for the operators
\[
  d_j^*(U_0) \colon \Contc(G^2\times_{v_k,G^0,\rg} G,\E_k
  \otimes_{\FU} \A) \congto \Contc(G^2\times_{v_l,G^0,\rg} G,\E_l
  \otimes_{\FU} \A)
\]
that we get by restricting \(d_j^*(U)\) to \(\Contc\)\nb-sections.

\begin{proposition}
  \label{lem:integrated_regular}
  The nondegenerate \Star{}homomorphism
  \[
    \Lambda\colon \Contc(G,\A)\to \Bound(\Lt^2(G,\A,\s,\tilde\alpha))
  \]
  that integrates the representation of~\(\A\) on
  \(\Lt^2(G,\A,\s,\tilde\alpha)\) constructed above is the ``usual''
  regular representation given by \(\Lambda(f)\xi = f*\xi\) for all
  \(f,\xi\in \Contc(G,\A)\).
\end{proposition}

\begin{proof}
  Let \(f,\xi\in \Contc(G,\A)\).  Factor~\(f\) as a pointwise product
  \(f_1\cdot f_2\) with \(f_1\in \Contc(G,\A\A^*)\),
  \(f_2\in \Contc(G,\A)\) as in Lemma~\ref{lem:factor_f}.  Then
  \[
    \Lambda(f)\xi = (T^\rg_{f_1^\dagger})^* U T^\s_{f_2}\xi =
    (T^\rg_{f_1^\dagger})^* U (f_2 \otimes \xi)
  \]
  Here \(f_2\otimes \xi\) is the section
  \(G^2 \to d_2^*(\A) \otimes_{\A\A^*} d_0^*(\A)\),
  \((g,h)\mapsto f_2(g) \otimes \xi(h)\).  The unitary~\(U\) maps this
  to the section \(G\times_{\rg,\rg} G \to \A\A^* \cdot \pr_2^*(\A)\),
  \((g, k) \mapsto f_2(g)\cdot \xi(g^{-1} k)\).  Applying
  \((T^\rg_{f_1^\dagger})^*\) to this gives the section
  \[
    k\mapsto \int_{G} f_1^\dagger(g)^* \cdot f_2(g)\cdot \xi(g^{-1} k)
    \,\diff\alpha^{\rg(k)}(g) = \int_{G} f(g)\cdot \xi(g^{-1} k)
    \,\diff\alpha^{\rg(k)}(g) = (f*\xi)(k)
  \]
  of~\(\A\) over~\(G\).  Thus \(\Lambda(f)(\xi) = f*\xi\).
\end{proof}

\begin{proposition}
  \label{pro:section_algebra_has_faithful_rep}
  The regular representation~\(\Lambda\) of \(\Contc(G,\A)\) is
  \(I\)\nb-norm bounded and injective.
  Thus the maximal \(I\)\nb-bounded \(\Cst\)\nb-seminorm on
  \(\Contc(G,\A)\) is a \(\Cst\)\nb-norm.
\end{proposition}

\begin{proof}
  Lemma~\ref{lem:Lf_well-defined} implies the estimate
  \(\Lambda(f) \le \norm{f}_I\) for all \(f\in \Contc(G,\A)\).  If
  \(f\in \Contc(G,\A)\) and \(\Lambda(f)=0\), then \(f*\xi=0\) for all
  \(\xi\in \Contc(G,\A)\).  Then
  \[
    0=(f*f^*)(x)=\int_{G} f(g)f(g)^* \,\diff\alpha^x(g)
  \]
  for all \(x\in G^0\).  Since the measures~\(\alpha^x\) in a Haar
  system have full support, this implies \(f=0\).
\end{proof}

The regular representation defines the reduced \(\Cst\)\nb-algebra
of~\(\A\):

\begin{definition}
  \label{def:regular_section_Cstar-algebra}
  The \emph{reduced section \(\Cst\)\nb-algebra}
  \(\Cred(G,\A)\) of~\(\A\) is the image of~\(\Cst(\A)\) in
  the regular representation.  Equivalently, it is the completion of
  \(\Contc(G,\A)\) for the norm
  \(\norm{f}_\red\defeq \norm{\Lambda(f)}\).
\end{definition}

\begin{remark}
  \label{rem:induce_reps}
  Let \(\varrho\colon \Cont_0(G^0,\FU) \to \Bound(\E)\) be a
  representation on a Hilbert module over another
  \(\Cst\)\nb-algebra~\(D\) or a Hilbert space.
  Then the regular representation on \(\Lt^2(G,\A,\s,\tilde\alpha)\)
  induces a representation
  \[
    \Ind \varrho\colon \Cred(G,\A) \to
    \Bound\bigl(\Lt^2(G,\A,\s,\tilde\alpha) \otimes_{\Cont_0(G^0,\FU)}
    \E\bigr),
    \qquad
    f \mapsto \Lambda(f) \otimes \Id_{\E}.
  \]
  If~\(\varrho\) is faithful, then \(\Ind \varrho\) is faithful
  because the homomorphism
  \[
    \Bound(\Lt^2(G,\A,\s,\tilde\alpha)) \to
    \Bound(\Lt^2(G,\A,\s,\tilde\alpha) \otimes_{\Cont_0(G^0,\FU)} \E),
    \qquad
    T\mapsto T\otimes \Id_{\E},
  \]
  is faithful.
  More generally, let~\(I\) be a set whose elements are
  representations \(\varrho_i\colon \Cont_0(G^0,\FU) \to
  \Bound(\E_i)\) on certain Hilbert modules~\(\E_i\).
  Assume that this set of representations is faithful in the sense
  that \(\bigcap_{i\in I} \ker \varrho_i = \{0\}\).
  Then the family of induced representations \(\Ind \varrho_i\) is
  faithful in the same sense.
  This is because \(\setgiven{\varrho_i}{i\in I}\) is faithful if and
  only if \(\bigoplus_{i\in I} \varrho_i\) is faithful.
  In particular, the family of restriction maps \(\Cont_0(G,\FU) \to
  \FU_x\) for \(x\in G^0\) is faithful.
  It follows that the reduced norm defining \(\Cred(G,\A)\) is
  equal to a supremum over~\(G^0\), where the norm at~\(x\) is the
  norm in the regular representation of \(\Cst(G,\A)\) in the Hilbert
  \(\FU_x\)-module \(\Lt^2(G_x,\A|_{G_x},\tilde\alpha_x)\); this is
  the fibre of \(\Lt^2(G,\A,\s,\tilde\alpha)\) at~\(x\).
  We could also use the family of all irreducible representations
  of~\(\SUF\) or, equivalently, of the fibres~\(\A_x\) for \(x\in X\).
  If the field of \(\Cst\)\nb-algebras~\(\A\) is continuous and
  \(X_0\subseteq X\) is a dense subset, then the family of restriction
  maps for \(x\in X_0\) is faithful as well.
  Thus the reduced norm is also equal to a supremum over
  \(X_0\subseteq X\).
  These and similar alternative descriptions of the reduced
  \(\Cst\)\nb-norm are used in the literature for saturated Fell bundles
  (see \cites{Exel:noncomm.cartan,
    Sims-Williams:Equivalence_reduced_groupoid, Moutuou-Tu:Equivalence}).
\end{remark}

\section{Disintegration of representations}
\label{sec:disintegrate}

We are going to ``disintegrate'' a representation
\(L\colon \Contc(G,\A)\to\Bound(\F)\) to a representation
\((\varphi,U)\) of \((G,\A,\alpha)\).
This disintegration still works for some densely defined ``unbounded
representations'', showing that all such representations are in fact
bounded.
This works exactly as in~\cite{Buss-Holkar-Meyer:Universal}, and the
required notion of ``unbounded representation'' is a literal
translation of \cite{Buss-Holkar-Meyer:Universal}*{Definition~5.1}.
Recall the subspaces \(\Cont_0(K\setminus \partial K,\A)\) of
\(\Contc(G,\A)\) defined in Definition~\ref{def:Cont_KpartialK}.

\begin{definition}
  \label{def:pre-representation}
  Let~\(D\) be a \(\Cst\)\nb-algebra and let~\(\F\) be a Hilbert
  \(D\)\nb-module.  A \emph{pre-\alb{}representation}
  of~\(\Contc(G,\A)\) on~\(\F\) consists of a vector space~\(\F_0\),
  a linear map \(\iota\colon \F_0\to\F\) with dense image, and a
  bilinear map \(L\colon \Contc(G,\A)\times \F_0 \to \F\),
  \((f,\xi)\mapsto L(f)\xi\), such that
  \begin{enumerate}
  \item if \(\xi, \eta\in \F_0\) and \(K\subseteq G\) is compact, then
    the map \(\Cont_0(K\setminus \partial K,\A)\to D\),
    \(f\mapsto \braket{\iota(\xi)}{L(f)\eta}_D\), is bounded in norm;
    roughly speaking, \(L\) is continuous for the inductive limit
    topology on \(\Contc(G,\A)\) and a suitable weak topology on the
    space of linear maps \(\F_0\to\F\);
  \item
    \(\braket{L(f_1)\xi}{L(f_2)\eta}_D =
    \braket{\iota(\xi)}{L(f_1^**f_2)\eta}_D\) for
    \(\xi,\eta\in \F_0\), \(f_1,f_2\in \Contc(G,\A)\);
  \item the linear span of \(L(f)\xi\) for \(f\in\Contc(G,\A)\),
    \(\xi\in \F_0\) is dense in~\(\F\).
  \end{enumerate}
\end{definition}

Our first goal is to define a nondegenerate representation
\(\varphi\colon \Contc(G^0,\FU) \to \Bound(\F)\) from a
pre-representation of \(\Contc(G,\A)\).  Let~\(B\) be the
\(\Cst\)\nb-algebra of all bounded sections of the field of multiplier
algebras \((\Mult(\FU_x))_{x\in G^0}\) with the property that pointwise
multiplication with them maps \(\Cont_0(G^0,\FU)\) again to
\(\Cont_0(G^0,\FU)\).  In fact, \(B\) is naturally isomorphic to the
multiplier algebra of \(\Cont_0(G^0,\FU)\); but we shall not use this
fact.  The \(\Cst\)\nb-algebra~\(B\) is unital and contains
\(\Cont_0(G^0,\FU)\).  It acts on \(\Contc(G,\A)\) by multipliers
through
\[
  (b \dagger f)(g) \defeq b(\rg(g))\cdot f(g),\qquad (f\ddagger b)(g)
  \defeq f(g) \cdot b(\s(g))
\]
for all \(b\in B\), \(f\in \Contc(G,\A)\), \(g\in G\).  Easy
computations show that these multiplication maps are bilinear and
satisfy \((b\dagger f)^* = f^*\ddagger b^*\) and
\[
  (f_1\ddagger b)*f_2 = f_1 * (b\dagger f_2),\qquad (b_1\cdot
  b_2)\dagger f = b_1 \dagger (b_2 \dagger f).
\]
for all \(f,f_1,f_2\in \Contc(G,\A)\), \(b_1,b_2,b\in B\).
The properties \(b\dagger(f_1*f_2) = (b\dagger f_1)*f_2\), \(f\ddagger
(b_1\cdot b_2) = (f \ddagger b_1) \ddagger b_2\) and \((f_1*f_2)
\dagger b = f_1*(f_2\ddagger b)\) follow from this, and they are also
easy to check directly.

\begin{lemma}[compare \cite{Buss-Holkar-Meyer:Universal}*{Lemma~5.2}]
  \label{lem:extension-multiplier}
  Let \((L,\iota,\F_0)\) be a pre-representation of \(\Contc(G,\A)\)
  on a Hilbert \(D\)\nb-module~\(\F\).
  There is a unique unital \Star{}homomorphism \(\varphi\colon B \to
  \Bound(\F)\) such that \(\varphi(b) L(a)\xi = L(b\dagger a)\xi\) for
  all \(b\in B\), \(a\in \Contc(G,\A)\), \(\xi\in \F_0\).
  Its restriction to \(\Cont_0(G^0,\FU)\) is nondegenerate.
\end{lemma}

\begin{proof}
  Fix \(b\in B\).  Let \(\F_{00}\subseteq \F\) be the linear span
  of~\(L(a)\xi\) for \(a\in \Contc(G,\A)\), \(\xi\in\F_0\).  So a
  typical element of~\(\F_{00}\) is of the form
  \(\zeta \defeq \sum_{j=1}^n L(a_j)\xi_j\) for some \(n\in\N\),
  \(a_1,\dotsc,a_n \in \Contc(G,\A)\) and
  \(\xi_1,\dotsc,\xi_n\in \F_0\).  The subspace~\(\F_{00}\) is dense
  in~\(\F\) by assumption.  We show that there is a linear map
  \(\varphi(b)\colon \F_{00} \to \F_{00}\) with
  \(\norm{\varphi(b)} \le \norm{b}\) and
  \(\varphi(b)L(a)\xi = L(b\dagger a)\xi\) for all
  \(a\in \Contc(G,\A)\), \(\xi\in\F_0\).  Let
  \(b_2 \defeq (\norm{b}^2\cdot 1 - b^* b)^{1/2} \in B\).  Thus
  \(b^*b + b_2^* b_2 = \norm{b}^2\cdot 1\).  For
  \(\zeta = \sum_{j=1}^n L(a_j) \xi_j\) as above, put
  \(\zeta_2 \defeq \sum_{j=1}^n L(b\dagger a_j) \xi_j\) and
  \(\zeta_3 \defeq \sum_{j=1}^n L(b_2\dagger a_j) \xi_j\).  Then
  \begin{multline*}
    \braket{\zeta_2}{\zeta_2}
    \le \braket{\zeta_2}{\zeta_2} + \braket{\zeta_3}{\zeta_3}
    \\= \sum_{j,k=1}^n {}
    \braket[\big]{\iota(\xi_j)} {L\bigl((b\dagger a_j)^* * (b\dagger
      a_k)\bigr) \xi_k}
    + \braket[\big]{\iota(\xi_j)}
    {L\bigl((b_2\dagger a_j)^** (b_2\dagger a_k)\bigr) \xi_k}
    \\= \sum_{j,k=1}^n{} \norm{b}^2 \cdot
    \braket[\big]{\iota(\xi_j)} {L(a_j^* a_k) \xi_k}
    = \norm{b}^2\braket{\zeta}{\zeta}.
  \end{multline*}
  So \(\zeta=0\) implies \(\zeta_2=0\).  Thus, there is a well-defined
  linear map~\(\varphi(b)\) with \(\varphi(b)\zeta = \zeta_2\), that
  is,
  \(\varphi(b)(\sum_{j=1}^n L(a_j) \xi_j)= \sum_{j=1}^n L(b\dagger
  a_j) \xi_j\).  Clearly, this is the only linear map on~\(\F_{00}\)
  with \(\varphi(b)L(a)\xi = L(b\dagger a)\xi\) for all
  \(a\in \Contc(G,\A)\), \(\xi\in\F_0\).  In addition to this, the
  computation above implies that \(\norm{\varphi(b)} \le \norm{b}\).
  Since~\(\F_{00}\) is dense in~\(\F\), \(\varphi(b)\) extends
  uniquely to a bounded linear map on~\(\F\).  A direct computation
  shows that the map \(\varphi\colon B \to \Endo(\F_{00})\) is a unital
  algebra homomorphism and that
  \(\braket{\varphi(b)\zeta_1}{\zeta_2} =
  \braket{\zeta_1}{\varphi(b^*)\zeta_2}\) for all
  \(\zeta_1,\zeta_2\in \F_{00}\).  So the unique extension
  of~\(\varphi(b)\) to~\(\F\) defines a unital \Star{}homomorphism
  \(B\to\Bound(\F)\).

  The \(\Cont_0(G^0,\FU)\)-module structure~\(\dagger\) on
  \(\Contc(G,\A)\) is nondegenerate because each fibre~\(\A_g\)
  of~\(\A\) is a nondegenerate \(\FU_{\rg(g)}\)\nb-module (compare
  Lemma~\ref{lem:subfield_equality}).  Since
  \(\varphi(b) L(a)\xi = L(b\dagger a)\xi\), it follows that
  \(\varphi(\Cont_0(G,\FU)) L(\Contc(G,\A))\F_0\) and
  \(L(\Contc(G,\A))\F_0\) have the same closed linear span in~\(\F\).
  Since \(L(\Contc(G,\A))\F_0\) is dense in~\(\F\), it follows that
  the restriction of~\(\varphi\) to \(\Cont_0(G^0,\FU)\) is still
  nondegenerate.
\end{proof}

We used the extension of~\(\varphi\) to~\(B\) to prove that it is
bounded.  In the following, we restrict~\(\varphi\) to
\(\Cont_0(G^0,\FU)\).
This is the first ingredient in the Fell-bundle representation
associated to~\(L\).
The unitary~\(U\) needs some more preparation.

Let~\(X\) be a locally compact Hausdorff space with a continuous
map \(p\colon X\to G^0\) with a continuous family of
measures~\(\lambda\) along its fibres.
Let \(\E=(\E_x)_{x\in X}\) be a field of right Hilbert
\(p^*\FU\)-modules.
So \(\Lt^2(X,\E,p,\lambda)\) is a Hilbert module over
\(\Cont_0(G^0,\FU)\).
The representation~\(\varphi\) defined above provides a Hilbert
\(D\)\nb-module \(\Lt^2(X,\E,p,\lambda)\otimes_\varphi\F\).
Let \(\E\otimes_{p^*\FU}\A\) be the ``composite'' field of Hilbert
modules over \(X\times_{p,G^0,\rg} G\) with fibre \(\E_x
\otimes_{\FU_{p(x)}} \A_g\) at \((x,g)\in X\times_{p,G^0,\rg} G\) as
in Lemma~\textup{\ref{lem:compose_corr_from_measure-family}}.
Inner products from this field take values in the field \((\s\circ
\pr_2)^*\FU\) with fibre \(\A_{\s(g)}\) at \((x,g)\in
X\times_{p,G^0,\rg} G\), where \(\pr_2\colon X\times_{p,G^0,\rg} G
\to G\) denotes the coordinate projection.

\begin{lemma}[compare \cite{Buss-Holkar-Meyer:Universal}*{Lemma~5.3}]
  \label{lem:technical_disintegration_lemma}
  In the above situation, the map
  \[
    \tau\colon \Contc(X,\E)\odot \Contc(G,\A)\odot \F_0\to
    \Contc(X,\E)\odot \F,\qquad f_0\otimes f_1\otimes \xi\mapsto
    f_0\otimes L(f_1)\xi,
  \]
  extends uniquely to a linear map with dense range
  \[
    \bar\tau\colon
    \Contc(X\times_{p,G^0,\rg}G,\E\otimes_{p^*\FU}\A)\odot \F_0 \to
    \Lt^2(X,\E,p,\lambda)\otimes_\varphi\F
  \]
  that is continuous in the inductive limit topology on
  \(\Contc(X\times_{p,G^0,\rg}G,\E\otimes_{p^*\FU}\A)\) for all fixed
  \(\xi\in \F_0\).  If
  \(F_1,F_2\in \Contc(X\times_{p,G^0,\rg}G,\E\otimes_{p^*\FU}\A)\) and
  \(\xi_1,\xi_2\in \F_0\), then
  \begin{equation}
    \label{eq:innprod_dense_in_HilmF0}
    \braket{\bar\tau(F_1\otimes\xi_1)}{\bar\tau(F_2\otimes\xi_2)}
    = \braket{\iota(\xi_1)}{L(\braket{F_1}{F_2})\xi_2}
  \end{equation}
  with \(\braket{F_1}{F_2}\in \Contc(G,\A)\) defined by
  \begin{align}
    \label{eq:innprod_dense_in_HilmF}
    \braket{F_1}{F_2}(g)
    &\defeq \iint \braket{F_1(x,h^{-1})}{F_2(x,h^{-1}g)}
      \,\diff\lambda_{\s(h)}(x) \,\diff\alpha^{\rg(g)}(h)\\ \notag
    &= \iint \braket{F_1(x,h)}{F_2(x,hg)}
      \,\diff\lambda_{\rg(h)}(x) \,\diff\tilde\alpha_{\rg(g)}(h).
  \end{align}
\end{lemma}

\begin{proof}
  The proof is a routine translation of the proof of
  \cite{Buss-Holkar-Meyer:Universal}*{Lemma~5.3}.  The linear span of
  \(L(f)\xi\) for \(f\in \Contc(G,\A)\), \(\xi\in\F_0\) is dense
  in~\(\F\) by assumption.  This implies that the image of~\(\tau\) is
  dense in \(\Lt^2(X,\E,p,\lambda)\otimes_\varphi\F\).  There is a
  field \(\E\otimes_{p^*\FU}\A\) of Hilbert
  \((\s\circ\pr_2)^*\FU\)-modules with fibres
  \(\E_x \otimes_{\FU_{\rg(g)}} \A_g\) at \(x\in X\), \(g\in G\) with
  \(p(x) = \rg(g)\) as a Hilbert \(\FU_{\s(g)}\)-module, and such that
  the space of continuous sections of \(\E\otimes_{p^*\FU}\A\) is the
  closed linear span of the sections of the form
  \((x,g)\mapsto \eta_1(x)\otimes \eta_2(g)\) for
  \(\eta_1 \in \Cont_0(X,\E)\), \(\eta_2\in \Cont_0(G,\A)\).
  Restricting to fixed compact subsets of \(X\) and~\(G\), we see that
  \(\Contc(X,\E)\odot \Contc(G,\A)\) is dense in
  \(\Contc(X\times_{p,G^0,\rg} G,\E\otimes_{p^*\FU}\A)\) in the
  inductive limit topology.  This already explains the uniqueness of
  the extension, if it exists.

  Let \(f_1,f_3\in \Contc(X,\E)\), \(f_2,f_4\in \Contc(G,\A)\).
  Define sections
  \[
    F_1\colon (x,k)\mapsto f_1(x)\otimes f_2(k) \quad\text{and}\quad
    F_2\colon (x,k)\mapsto f_3(x)\otimes f_4(k)
  \]
  in \(\Contc(X\times_{p,G^0,\rg} G,\E\otimes_{p^*\FU}\A)\).  Then
  \begin{multline*}
    \braket{\tau(f_1\otimes f_2\otimes\xi_1)} {\tau(f_3\otimes
      f_4\otimes \xi_2)}
    = \braket{f_1\otimes L(f_2)\xi_1}{f_3\otimes L(f_4)\xi_2}
    \\= \braket{L(f_2)\xi_1}{\varphi(\braket{f_1}{f_3})L(f_4)\xi_2}
    = \braket{\xi_1}{L(f_2^** (\braket{f_1}{f_3} \dagger f_4))\xi_2},
  \end{multline*}
  where~\(\varphi\) is the action of \(\Cont_0(G, \rg^*\FU)\) on
  \(\mathcal{L}^2(G,\A,s,\tilde{\alpha})\).  And
  \begin{multline*}
    (f_2^** (\braket{f_1}{f_3} \dagger f_4))(g)
    = \int_G f_2^*(h) (\braket{f_1}{f_3} \dagger f_4)(h^{-1}g)
    \,\diff\alpha^{\rg(g)}(h)\\
    = \int_G\int_X f_2(h^{-1})^* \braket{f_1(x)}{f_3(x)} f_4(h^{-1}g)
    \,\diff\lambda_{\rg(h^{-1}g)}(x) \,\diff\alpha^{\rg(g)}(h)\\
    = \int_G\int_X {} \braket{F_1(x,h^{-1})}{F_2(x,h^{-1}g)}
    \,\diff\lambda_{\s(h)}(x) \,\diff\alpha^{\rg(g)}(h).
  \end{multline*}
  This proves~\eqref{eq:innprod_dense_in_HilmF0} for~\(\tau\)
  instead of~\(\bar\tau\) and for \(F_1,F_2\) in the image of
  \(\Contc(X,\E)\odot \Contc(G,\A)\) in
  \(\Contc(X\times_{p,G^0,\rg} G,\E\otimes_{p^*\FU}\A)\).  Now let
  \(K\subseteq X\) and \(L\subseteq G\) be compact.  Let \((F_n)\)
  be a sequence in
  \(\Cont_0(K\setminus \partial K,\E)\odot \Cont_0(L\setminus
  \partial L,\A)\) that is Cauchy in the supremum norm on
  \(\Cont(K\times_{p,G^0,\rg} L,\E\otimes_{p^*\FU}\A)\).  The same
  estimates as in the proof of
  \cite{Buss-Holkar-Meyer:Universal}*{Lemma~5.3} show that~\(\tau\)
  maps \((F_n \otimes \xi)_{n\in\N}\) for \(\xi\in\F_0\) to a Cauchy
  sequence in \(\Lt^2(X,\E,p,\lambda)\otimes_\varphi\F\); here the
  continuity assumption in Definition~\ref{def:pre-representation}
  is used.  Since the image of \(\Contc(X,\E)\odot \Contc(G,\A)\) in
  \(\Contc(X\times_{p,G^0,\rg} G,\E\otimes_{p^*\FU}\A)\) is dense
  and since both sides in~\eqref{eq:innprod_dense_in_HilmF0} are
  continuous in the inductive limit topology on
  \(\Contc(X\times_{p,G^0,\rg} G,\E\otimes_{p^*\FU}\A)\), it follows
  that~\(\tau\) extends uniquely to~\(\bar\tau\) and that the
  extension still verifies~\eqref{eq:innprod_dense_in_HilmF0}.
\end{proof}

Along with the unitary in~\eqref{eq:regular_U_on_dense_subspaces}
built in the regular representation in
Section~\ref{sec:regular_rep}, the lemma above will help us in
constructing the required unitary~\(U\) in the disintegrated
representation.  Now as in~\cite{Buss-Holkar-Meyer:Universal}, we
apply Lemma~\ref{lem:technical_disintegration_lemma} to two cases,
namely, when \((X,\E,p,\lambda)\) is \((G,\A,\s,\tilde\alpha)\) or
\((G,\A\A^*,\rg,\alpha)\).  This gives maps
\begin{align*}
  \bar\tau_\s\colon
  \Contc(G\times_{\s,G^0,\rg}G,\A\otimes_{p^*\FU} \A) \odot \F_0
  &\to \Lt^2(G,\A,\s,\tilde\alpha) \otimes_\varphi\F,\\
  \bar\tau_\rg\colon
  \Contc(G\times_{\rg,G^0,\rg}G,\A\A^* \otimes_{p^*\FU} \A) \odot \F_0
  &\to \Lt^2(G,\A\A^*,\rg,\alpha) \otimes_\varphi\F.
\end{align*}
The isomorphism~\(U_0\) in~\eqref{eq:regular_U_on_dense_subspaces}
induces a map
\begin{multline*}
  U_0\otimes \Id_{\F_0}\colon
  \Contc(G\times_{\s,G^0,\rg}G,\A\otimes_{p^*\FU} \A) \odot \F_0
  \\\congto \Contc(G\times_{\rg,G^0,\rg}G,\A\A^* \otimes_{p^*\FU} \A)
  \odot \F_0.
\end{multline*}
We are going to extend it to the unitary~\(U\) needed for a
Fell-bundle representation.

\begin{lemma}
  \label{lem:isometric-D-inner-products}
  The operator \(U_0\otimes \Id_{\F_0}\) is isometric for the
  \(D\)\nb-valued inner products on
  \(\Lt^2(G,\A,\s,\tilde\alpha)\otimes_\varphi\F\) and
  \(\Lt^2(G,\A\A^*,\rg,\alpha) \otimes_\varphi\F\).
\end{lemma}

\begin{proof}
  The isomorphism~\(U_0\) in~\eqref{eq:regular_U_on_dense_subspaces} is
  given by the formula \(U_0(F)(g,h) \defeq F(g,g^{-1}h)\) for
  \(F\in \Contc(G\times_{\s,G^0,\rg}G,\A\otimes_{p^*\FU}\A)\) and
  \((g,h)\in G\times_{\rg,G^0,\rg}G\).  To see that
  \(U_0\otimes\Id_{\F_0}\) extends to an isometry as in the statement,
  it is enough to check
  \begin{equation}
    \label{eq:isometric-relation}
    \braket{\bar\tau_\s(F_1\otimes\xi_1)}{\bar\tau_\s(F_2\otimes\xi_2)}
    = \braket{\bar\tau_\rg(U_0(F_1)\otimes\xi_1)}{\bar\tau_\rg(U_0(F_2)\otimes\xi_2)}
  \end{equation}
  for all
  \(F_1,F_2\in \Contc(G\times_{\s,G^0,\rg}G,\A\otimes_{p^*\FU}\A)\) and
  \(\xi_1,\xi_2\in \F_0\) because then the map
  \(\bar\tau_\s(F\otimes\xi)\mapsto \bar\tau_\rg(F\otimes\xi)\) is well
  defined and isometric and therefore gives an isometry as in
  the statement.  Equation~\eqref{eq:innprod_dense_in_HilmF0} implies
  \begin{align*}
    \braket{\bar\tau_\s(F_1\otimes\xi_1)}{\bar\tau_\s(F_2\otimes\xi_2)}
    &= \braket{\iota(\xi_1)}{L(\braket{F_1}{F_2}_\s)\xi_2},\\
    \braket{\bar\tau_\rg(U_0(F_1)\otimes\xi_1)}{\bar\tau_\rg(U_0(F_2)\otimes\xi_2)}
    &=\braket{\iota(\xi_1)}{L(\braket{F_1}{F_2}_\rg)\xi_2},
  \end{align*}
  where
  \begin{align}
    \label{eq:inner-prod-s}
    \braket{F_1}{F_2}_\s(g)
    &= \int_G\int_G F_1(x,h)^*F_2(x,h g)
      \,\diff\tilde\alpha_{\rg(h)}(x)
      \,\diff\tilde\alpha_{\rg(g)}(h),\\
    \label{eq:inner-prod-r}
    \braket{F_1}{F_2}_\rg(g)
    &= \int_G\int_G F_1(x,x^{-1}k)^*F_2(x,x^{-1}k g)
      \,\diff\alpha^{\rg(k)}(x)
      \,\diff\tilde\alpha_{\rg(g)}(k).
  \end{align}
  These two integrals are equal
  by~\eqref{eq:compose_integration_nontrivial} (compare
  \cite{Buss-Holkar-Meyer:Universal}*{Equation~(3.11) on page~351}).
  This yields~\eqref{eq:isometric-relation} and finishes the proof.
\end{proof}

Since the maps \(\bar\tau_\rg\) and~\(\bar\tau_\s\) have dense range,
it follows that \(U_0\otimes \Id_{\F_0}\) extends uniquely to a
unitary operator
\[
  U\colon \Lt^2(G,\A,\s,\tilde\alpha)\otimes_\varphi\F\congto
  \Lt^2(G,\A\A^*,\rg,\alpha) \otimes_\varphi\F.
\]
This intertwines the left \(\Cont_0(G,\A\A^*)\)-actions
because~\(U_0\) is a bimodule map.

\begin{lemma}
  \label{lem:disintegration_rep}
  The pair \((\varphi,U)\) associated to a pre-representation
  of~\(\Contc(G,\A)\) is a representation of \((G,\A,\alpha)\).
\end{lemma}

\begin{proof}
  Apply Lemma~\ref{lem:technical_disintegration_lemma} to the cases
  \((G^2,\E_k,v_k,\mu_k)\) with~\(\E_k\) as
  in~\eqref{eq:dj_U_domain_codomain}.  This provides well-defined
  linear maps with dense range
  \[
    \Contc(G^2 \times_{v_k,G^0,\rg} G,\E_k \otimes_{\FU} \A) \odot
    \F_0 \to \Lt^2(G^2,\E_k,v_k,\mu_k) \otimes_\varphi \F.
  \]
  We claim that the operator~\(d_j^*(U)\) is the unique unitary
  extension of the operator \(d_j^*(U_0) \otimes \Id_{\F_0}\)
  with~\(U_0\) as in~\eqref{eq:dj_U_domain_codomain}.  On the subspace
  \(\Contc(G^2,\E_k) \odot \Contc(G,\A) \odot \F_0\), this is true
  because of the definition of the map~\(\tau\) in
  Lemma~\ref{lem:technical_disintegration_lemma}.  Since this subspace
  is dense, the equation holds everywhere.  So the identity
  \(d_2^*(U_0) d_0^*(U_0) = d_1^*(U_0)\) shown during the construction
  of the regular representation implies that
  \[
    (d_2^*(U_0) \otimes \Id_{\F_0}) (d_0^*(U_0) \otimes \Id_{\F_0}) =
    d_1^*(U_0) \otimes \Id_{\F_0}.
  \]
  This implies \(d_2^*(U) d_0^*(U) = d_1^*(U)\).
\end{proof}

The following is the Fell-bundle analogue of
\cite{Buss-Holkar-Meyer:Universal}*{Proposition~6.1}.

\begin{proposition}
  \label{pro:back_and_forth_from_L}
  Let \((L,\F_0,\iota)\) be a pre-representation of
  \(\Contc(G,\A)\) on a Hilbert module~\(\F\).
  Define the Fell-bundle representation~\((\varphi,U)\) as above, and
  let~\(L'\) be the representation of \(\Contc(G,\A)\) associated
  to~\((\varphi,U)\) as in Section~\textup{\ref{sec:integrate}}.
  Let \(f,f_2\in\Contc(G,\A)\), \(\xi\in\F_0\).
  Then \(L'(f)(\iota(\xi)) = L(f)(\xi)\) and \(L'(f)(L(f_2)\xi) =
  L(f*f_2)(\xi)\).
\end{proposition}

\begin{proof}
  We first check \(L'(f)(L(f_2)\xi) = L(f*f_2)(\xi)\).  This is
  equivalent to
  \[
    \braket{L(f_4)\xi_2}{L'(f)(L(f_2)\xi)} =
    \braket{L(f_4)\xi_2}{L(f*f_2)(\xi)}
  \]
  for all \(f_4\in \Contc(G,\A)\), \(\xi_2\in\F_0\) because the
  linear span of~\(L(f_4)\xi_2\) is dense in~\(\F\).  Factor
  \(f = f_3^\dagger \cdot f_1\) with \(f_3\in\Contc(G,\A\A^*)\) and
  \(f_1\in\Contc(G,\A)\) as in Lemma~\ref{lem:factor_f}.  The
  integrated form~\(L'\) of~\((\varphi,U)\) is defined by
  \(L'\bigl(f_3^\dagger\cdot f_1\bigr) \defeq (T^\rg_{f_3})^* U
  T^\s_{f_1}\).  So
  \begin{multline*}
    \braket{L(f_4)\xi_2}{L'(f)(L(f_2)\xi)}
    = \braket{T^\rg_{f_3}  L(f_4)\xi_2}{U T^\s_{f_1} L(f_2)\xi}
    \\= \braket{\tau_\rg(f_3 \otimes f_4\otimes \xi_2)}
    {U \tau_\s(f_1 \otimes f_2 \otimes\xi)}
  \end{multline*}
  with the maps \(\tau_\s\) and~\(\tau_\rg\) from
  Lemma~\ref{lem:technical_disintegration_lemma}.
  By construction, \(U\) maps \(\tau_\s(f_1 \otimes f_2 \otimes\xi)\)
  to \(\bar\tau_\rg(F\otimes \xi)\) with \(F(g,k) \defeq f_1(g) \cdot
  f_2(g^{-1} k)\) for \((g,k)\in G \times_{\rg,\rg} G\); here we have
  reparametrised \(f_1 \otimes f_2\) with the inverse of the
  homeomorphism~\(\Upsilon\) in~\eqref{eq:Upsilon}.
  So Lemma~\ref{lem:technical_disintegration_lemma} implies
  \[
    \braket[\big]{\tau_\rg(f_3 \otimes f_4\otimes \xi_2)}
    {U \tau_\s(f_1 \otimes f_2 \otimes\xi)} =
    \braket[\big]{\iota\xi_2}{L(\braket{f_3 \otimes f_4}{F})\xi}
  \]
  with
  \begin{align*}
    \braket{f_3 \otimes f_4}{F}(g)
    &\defeq
      \int_G\int_G (f_3 \otimes f_4)^\dagger(x,h^{-1}) F(x,h^{-1} g)
      \,\diff\alpha^{\rg(h)}(x) \,\diff\alpha^{\rg(g)}(h)
    \\&= \int_G\int_G f_4(h^{-1})^* f_3(x)^* f_1(x) f_2(x^{-1} h^{-1} g)
    \,\diff\alpha^{\rg(h)}(x)\,\diff\alpha^{\rg(g)}(h).
  \end{align*}
  Now \(f_4(h^{-1})^* = f_4^*(h)\) and \(f_3(x)^* f_1(x) = f(x)\).
  The integral over~\(x\) gives the convolution \((f*f_2)(h^{-1} g)\).
  Then the integral over~\(h\) gives \((f_4^**f*f_2)(g)\).
  So
  \[
    \braket{L(f_4)\xi_2}{L'(f)(L(f_2)\xi)} =
    \braket{\iota\xi_2}{L(f_4^* *(f*f_2))(\xi)} =
    \braket{L(f_4)\xi_2}{L(f*f_2)(\xi)}.
  \]
  This finishes the proof that \(L'(f)(L(f_2)\xi) = L(f*f_2)(\xi)\).

  Since~\(L'\) is a \Star{}homomorphism to \(\Bound(\F)\), we
  further compute
  \begin{multline*}
    \braket{L'(f)\iota(\xi_1)}{L(f_2)\xi_2} =
    \braket{\iota(\xi_1)}{L'(f^*)L(f_2)\xi_2} \\=
    \braket{\iota(\xi_1)}{L(f^* * f_2)\xi_2} =
    \braket{L(f)\xi_1}{L(f_2)\xi_2}
  \end{multline*}
  for all \(f,f_2\in\Contc(G,\A)\), \(\xi_1,\xi_2\in\F_0\).
  Since vectors of the form \(L(f_2)\xi_2\) span a dense subspace
  in~\(\F\), this implies \(L'(f)\iota(\xi_1) = L(f)\xi_1\).
\end{proof}

\begin{corollary}
  \label{cor:prerep}
  For any pre-representation~\((L,\F_0,\iota)\) of
  \(\Contc(G,\A)\) there is a representation
  \(L'\colon \Contc(G,\A)\to\Bound(\F)\) that is bounded with
  respect to the \(I\)\nb-norm and satisfies
  \(L'(f)(\iota(\xi)) = L(f)(\xi)\) and
  \(L'(f)(L(f_2)\xi) = L(f*f_2)(\xi)\) for all
  \(f,f_2\in\Contc(G,\A)\), \(\xi\in\F_0\).  In particular, a
  representation of \(\Contc(G,\A)\) is \(I\)\nb-norm bounded if and
  only if it is continuous in the inductive limit topology.
\end{corollary}

\begin{lemma}
  \label{lem:U_from_L}
  Let \((\varphi,U)\) be a representation of~\((G,\A,\alpha)\) on a
  Hilbert module~\(\F\).
  Its integral \(L\colon \Contc(G,\A) \to \Bound(\E)\) uniquely
  determines \((\varphi,U)\).
\end{lemma}

\begin{proof}
  Vectors of the form \(L(f)\xi\) for \(f\in\Contc(G,\A)\) and
  \(\xi\in\F\) are linearly dense in~\(\F\) by
  Lemma~\ref{lem:Lf_nondegenerate}.
  The definition of~\(\varphi\) implies that \(\varphi(\eta) (L(f)\xi)
  = L(\eta \dagger f)\xi\) for all \(\eta\in \Cont_0(G^0,\FU)\).
  Since the linear span of \(L(f)\xi\) is dense in~\(\F\), the
  map~\(L\) uniquely determines~\(\varphi\).

  If \(f_1\in\Contc(G,\A\A^*)\), \(f_2\in\Contc(G,\A)\) and
  \(\xi_1,\xi_2\in\F\), then
  \[
    \braket{f_1 \otimes \xi_1}{U(f_2\otimes \xi_2)} =
    \braket{\xi_1}{(T^\rg_{f_1})^* U T^\s_{f_2}(\xi_2)} =
    \braket{\xi_1}{L(f_1^\dagger\cdot f_2)\xi_2}.
  \]
  Since elements \(f_1 \otimes \xi_1\) and \(f_2 \otimes \xi_2\) as
  above are linearly dense in
  \(\Lt^2(G,\A,\rg,\alpha) \otimes_\varphi \F\) and
  \(\Lt^2(G,\A,\s,\tilde\alpha) \otimes_\varphi \F\), respectively,
  the inner products above determine~\(U\) uniquely.
\end{proof}

Now the proof of Theorem~\ref{the:universal_groupoid_Haus} is finished
exactly as in~\cite{Buss-Holkar-Meyer:Universal}.  We have integrated
a representation of \((G,\A,\alpha)\) to a representation of
\(\Contc(G,\A)\) that is bounded in the \(I\)\nb-norm; hence it
extends uniquely to a representation of the \(\Cst\)\nb-algebra
\(\Cst(G,\A)\).  Conversely, we have disintegrated a
representation of \(\Cst(G,\A)\) to a representation of
\((G,\A,\alpha)\).  Proposition~\ref{pro:back_and_forth_from_L} shows
that integration after disintegration is the identity map on
representations of \(\Cst(G,\A)\).  And
Lemma~\ref{lem:U_from_L} shows that integration is injective.  So
disintegration and integration are inverse to each other.
The naturality properties in Theorem~\ref{the:universal_groupoid_Haus}
are rather evident; it is left to the reader to transfer this part of
the proof of Theorem~3.23 in~\cite{Buss-Holkar-Meyer:Universal}.

We are going to derive several consequences of our results in the
following sections.  Here is an immediate one:

\begin{corollary}
  \label{cor:continuity-rep}
  Let~\(\A\) be a Fell bundle over a locally compact groupoid with
  Haar system \((G,\alpha)\) and let~\(D\) be a \(\Cst\)\nb-algebra.
  A \Star{}homomorphism \(\varphi\colon \Contc(G,\A)\to D\) is
  continuous in the inductive limit topology if and only if it is
  bounded in the \(I\)\nb-norm.  In that case it extends uniquely to a
  \Star{}homomorphism \(\Cst(G,\A) \to D\).
\end{corollary}

\begin{proof}
  By definition, the homomorphism \(\varphi\) extends to
  \(\Cst(G,\A)\) if and only if it is bounded in the \(I\)\nb-norm;
  such an extension, if it exists, is necessarily unique.
  If \(\varphi\) is bounded in the \(I\)\nb-norm, then it is
  continuous in the inductive limit topology.

  Conversely, suppose that \(\varphi\) is continuous in the inductive
  limit topology.
  Then the formula \(\norm{f} \defeq \norm{\varphi(f)}_D\) defines a
  \(\Cst\)\nb-seminorm on~\(\Contc(G,\A)\).
  Let~\(D'\) denote the Hausdorff completion of \(\Contc(G,\A)\) with
  respect to this seminorm.
  By construction, \(\varphi\) descends to an isometric
  \Star{}homomorphism \(D' \to D\), identifying~\(D'\) with the
  \(\Cst\)\nb-subalgebra \(\overline{\varphi(\Contc(G,\A))}\subseteq D\)
  generated by the image of~\(\varphi\).

  The canonical map \(\Contc(G,\A) \to D'\) --- namely, the
  co-restriction of~\(\varphi\) --- is nondegenerate (this is
  precisely why we replaced~\(D\) with~\(D'\)), and hence defines a
  pre-representation of \(\Contc(G,\A)\) on~\(D'\).
  As such, it is bounded in the \(I\)\nb-norm by
  Corollary~\ref{cor:prerep}, and this bound is inherited
  by~\(\varphi\).
\end{proof}

\section{Functoriality of full and reduced section
  \texorpdfstring{$\Cst$-}{C*-}algebras}
\label{sec:functorial_exact}

In this section, we prove some rather basic results about section
\(\Cst\)\nb-algebras of Fell bundles over groupoids.  Throughout
this section, let \((G,\alpha)\) be a locally compact groupoid with
Haar system and let~\(\A\) be a Fell bundle over~\(G\).  Let
\(\FU\defeq \A|_{G^0}\) be the restriction of~\(\A\) to units.  This
is a field of \(\Cst\)\nb-algebras over~\(G^0\).  Let
\(\SUF\defeq\Cont_0(G^0,\FU)\) be the corresponding
\(\Cont_0(G^0)\)-\(\Cst\)-algebra.

We are going to describe a universal Fell-bundle representation on
\(\Cst(G,\A)\).  It contains a canonical \Star{}homomorphism
\(\SUF \to \Mult(\Cst(G,\A))\).  We show that the latter is
injective.  Then we study homomorphisms of Fell bundles.  These
induce homomorphisms of full and reduced section
\(\Cst\)\nb-algebras.  It is easy to see that a surjection of Fell
bundles induces a surjection on full and reduced section
\(\Cst\)\nb-algebras and that an injection of Fell bundles induces
an injection of \emph{reduced} section \(\Cst\)\nb-algebras.  An
injection of Fell bundles need not induce an injective
\Star{}homomorphism on the full section \(\Cst\)\nb-algebras: this
already fails for crossed products by group actions.  We define
hereditary and ideal Fell subbundles and show that \(\Cst(G,\I)\) is naturally
isomorphic to a \(\Cst\)-ideal  in \(\Cst(G,\A)\)
if~\(\I\) is an ideal Fell subbundle in~\(\A\) and the quotient corresponds to the
quotient Fell bundle~\(\A/\I\).  In other words, the canonical maps
\[
  \Cst(G,\I) \to \Cst(G,\A) \to \Cst(G,\A/\I)
\]
form an extension of \(\Cst\)\nb-algebras.

\subsection{The universal representation}
\label{sec:universal_rep}

The second part of the universal property,
Theorem~\ref{the:universal_groupoid_Haus}.\ref{en:universal_groupoid_Haus2},
implies that the bijection between representations of the Fell bundle
and its section \(\Cst\)\nb-algebra is induced by a universal
representation, which we are now going to describe.
Briefly, it is the Fell-bundle representation on \(B\defeq
\Cst(G,\A,\alpha)\) that corresponds to the \emph{identity
  representation} of~\(B\) on itself
under the bijection in Theorem~\ref{the:universal_groupoid_Haus}.
In the following, we usually omit~\(\alpha\) from our notation and
write \(B=\Cst(G,\A)\).
The identity map of~\(B\) is a nondegenerate representation of~\(B\)
on itself, where we view~\(B\) as a Hilbert module over itself.
We call the disintegrated form of the identity representation the
\emph{universal representation} of the Fell bundle~\((\A, G)\).
It is a pair
\[
  \iota\colon \SUF\to\Mult(B),\qquad
  U_\univ\colon \Lt^2(G,\A,\s,\tilde\alpha) \otimes_\iota B
  \congto \Lt^2(G,\A\A^*,\rg,\alpha) \otimes_\iota B,
\]
where~\(\iota\) is a nondegenerate \Star{}homomorphism
and~\(U_\univ\) is a unitary operator.
The action of the \(\Cst\)\nb-algebra~\(B\) on the Hilbert
\(B\)\nb-module~\(B\) in the identity representation \(B\to B\) is the
left multiplication in the \(\Cst\)\nb-algebra~\(B\).
Therefore, in the disintegration of the identity representation
of~\(B\), Lemma~\ref{lem:extension-multiplier} implies
\begin{equation}
  \label{eq:uni-rep-multi}
  \iota(f)(g)(x) = f(\rg(x)) g(x)
\end{equation}
for \(f\in \Cont_0(G^0, \FU[B])\) and \(g\in \Contc(G,\A)\).

The proof of disintegration shows that the unitary~\(U_\univ\) is the
unique extension of the map~\(U_0\)
in~\eqref{eq:regular_U_on_dense_subspaces}.

By the naturality of the universal property, the universal
representation~\((\iota, U_\univ)\) helps in computing the
disintegrations of representations of~\(B\).
Suppose that~\(D\) is a
\(\Cst\)\nb-algebra, and a nondegenerate representation
\(L\colon B\to \Bound(\F)\) of~\(B\) on a Hilbert
\(D\)\nb-module~\(\F\) is given.
Then~\(L\) extends uniquely to \(\Mult(B)\), and allows us to define a
composite nondegenerate \Star{}homomorphism
\(L\circ \iota\colon \SUF \to \Bound(\F)\).
Next we define an isomorphism of correspondences~\(U\) as the
composite in the following diagram:
\[
  \begin{tikzcd}[column sep=huge]
    \Lt^2(G,\A,\s,\tilde\alpha) \otimes_{L\circ \iota} \F
    \ar[r, dashed, "U"]
    \ar[d, "\cong", "\mathrm{canonical}"']
    & \Lt^2(G,\A\A^*,\rg,\alpha) \otimes_{L\circ \iota} \F
    \ar[d, "\cong"', "\mathrm{canonical}"]\\
    (\Lt^2(G,\A,\s,\tilde\alpha) \otimes_\iota B) \otimes_L \F
    \ar[r, "\cong"', "U_\univ \otimes \Id_{\F}"]
    &
    (\Lt^2(G,\A\A^*,\rg,\alpha) \otimes_\iota B) \otimes_L \F
  \end{tikzcd}
\]
The vertical isomorphisms use that the composition of
\(\Cst\)\nb-correspondences is associative and unital, so that 
\(B \otimes_L \F \cong \F\).
The naturality properties in Theorem~\ref{the:universal_groupoid_Haus}
imply that the pair \((L\circ \iota,U)\) is the disintegration
of~\(L\) to a Fell-bundle representation.
In this way, the universal Fell-bundle representation
\((\iota,U_\univ)\) on~\(B\) determines the bijection between
representations of \(B= \Cst(G, \A, \alpha)\) and of~\((\A, G,
\alpha)\).

\begin{proposition}
  \label{pro:coefficients_embed}
  Let \((G,\alpha)\) be a locally compact groupoid with Haar system
  and let~\(\A\) be a Fell bundle over~\(G\).
  Let \(B\defeq
  \Cst(G,\A)\).
  The canonical \Star{}homomorphism \(\iota\colon \SUF \to \Mult(B)\)
  and the following composite map are injective:
  \[
    \SUF\to\Mult(\Cst(G,\A)) \to \Mult(\Cred(G,\A)).
  \]
\end{proposition}

\begin{proof}
  Let \(f\in \SUF = \Cont_0(G^0,\FU)\) be such that \(\iota(f)(g)(x) =
  f(\rg(x)) g(x) = 0\) for all \(g\in \Contc(G,\A)\).
  We claim that \(f=0\), so that~\(\iota\) is injective.
  The naturality of the universal property  in
  Theorem~\ref{the:universal_groupoid_Haus} shows that the composite
  of~\(\iota\) with the quotient homomorphism to \(\Cred(G,\A)\)
  is the representation~\(\varphi\) of~\(\SUF\) in the disintegration
  \((\varphi,U)\) of the regular representation.
  The construction of the regular representation in
  Section~\ref{sec:regular_rep} defines~\(\varphi\) as the
  representation of \(\SUF\) on the Hilbert \(\SUF\)\nb-module
  \(\Lt^2(G,\A,s,\tilde{\alpha})\) defined by \(\varphi(f)\xi(g) =
  f(\rg(g)) \cdot \xi(g)\) for \(f\in \SUF\) and \(\xi \in
  \Lt^2(G,\A,s,\tilde{\alpha})\), see
  Equation~\eqref{eq:left-reg-rep-multi}.
  Now \(\varphi(f)\xi (x) = 0\) for
  \(\xi \in \Contc(G,\A) \subseteq \Lt^2(G,\A,s,\tilde{\alpha})\)
  implies \(f=0\).
  This proves the second injectivity claimed in the proposition.
\end{proof}

\subsection{Functoriality}
\label{sec:functoriality}

We now fix a locally compact Hausdorff groupoid~\(G\) with a Haar system and study
morphisms between Fell bundles over~\(G\).
Such morphisms are given by fibrewise linear maps that are compatible
with the multiplication, involution, and the underlying groupoid
structure.

Throughout this section, we fix a Fell bundle~\(\A\) over~\(G\).  To discuss
functoriality, we consider a second Fell bundle~\(\B\) over~\(G\).  We write
\(\FU[B] \defeq \B|_{G^0}\) for the restriction of~\(\B\) to the unit space and
\(\SUF[B] \defeq \Cont_0(G^0,\FU[B])\).

\begin{definition}
  \label{def:morphisms-Fell-bundles}
  A \emph{homomorphism of Fell bundles} \(\A\to\B\) is a continuous
  family of linear maps \(\mor_g\colon \A_g\to \B_g\) for all \(g\in
  G\) such that
  \[
    \mor_{g^{-1}}(a^*) = \mor_g(a)^*
    \quad\text{and}\quad
    \mor_g(a)\cdot \mor_h(b) = \mor_{g h}(a\cdot b)
  \]
  for all \((g,h) \in G^2\), \(a\in \A_g\), \(b\in\A_h\).
  Here ``continuity'' means that
  \(g\mapsto \mor_g(\xi(g))\) is a continuous section of~\(\B\)
  whenever~\(\xi\) is a continuous section of~\(\A\).
  The homomorphism is called \emph{injective} or \emph{surjective} if
  all maps~\(\mor_g\) are injective or surjective, respectively.
\end{definition}

\begin{remark}
  \label{rem:tau_injective}
  Let \(\mor=(\mor_g)_{g\in G}\) be a homomorphism of Fell bundles.
  If \(x\in G^0\), then \(\mor_{u(x)}\colon \FU_x \to \FU[B]_x\) is a
  \Star{}homomorphism, hence norm-contractive.
  Then~\(\mor_g\) is norm-contractive for all \(g\in G\) because
  \(\norm{a} = \norm{a^* a}^{1/2}\) for \(a\in \A_g\).
  The same idea shows that~\((\mor_g)_{g\in G}\) is injective if and
  only if the \Star{}homomorphisms~\(\mor_{u(x)}\) are injective for
  all \(x\in G^0\).
  In this case, each~\(\mor_g\) is isometric.
\end{remark}

\begin{proposition}
  \label{pro:functorial}
  Let \(\mor=(\mor_g)_{g\in G}\colon \A \to \B\) be a Fell-bundle homomorphism.
  Then there is a natural commuting diagram of \Star{}homomorphisms
  \[
    \begin{tikzcd}[column sep=large]
      \Cst(G,\A) \arrow{r}{\Cst(G,\mor)}
      \arrow[d,->>, "\mathrm{canonical}"'] &
      \Cst(G,\B)
      \arrow[->>]{d}{\mathrm{canonical}}\\
      \Cred(G,\A) \arrow{r}{\Cred(G,\mor)} &
      \Cred(G,\B).
    \end{tikzcd}
  \]
  The vertical maps are the canonical quotient maps.
  If~\(\mor\) is surjective, then so are \(\Cst(G,\mor)\) and
  \(\Cred(G,\mor)\).
  If~\(\mor\) is injective, then so is \(\Cred(G,\mor)\).
\end{proposition}

\begin{proof}
  The induced map
  \[
    \Contc(G,\mor) \colon \Contc(G,\A) \to \Contc(G,\B),
    \qquad
    \Contc(G,\mor)(h)(g) \defeq \Contc(G,\mor)(h(g)),
  \]
  is a \Star{}homomorphism and is continuous for the inductive limit topology.
  By Corollary~\ref{cor:continuity-rep}, \(\Contc(G,\mor)\) extends uniquely
  to a \Star{}homomorphism \(\Cst(G,\A) \to \Cst(G,\B)\), which we
  denote by \(\Cst(G,\mor)\).

  If~\(\mor\) is surjective, then so is \(\Contc(G,\mor)\colon \Contc(G,\A)
  \to \Contc(G,\B)\); this may be seen as a consequence of
  Lemma~\ref{lem:subfield_equality}.
  Since the range of \(\Cst(G,\mor)\) is always closed, this implies that
  \(\Cst(G,\mor)\) is surjective.
  Consequently, \(\Cred(G,\mor)\) is also surjective.
  Let \(\varrho\colon \Cont_0(G^0,\FU[B]) \to \Bound(\Hils)\) be a representation
  on a Hilbert space~\(\Hils\).
  We form an induced representation~\(\tilde\varrho\) of
  \(\Cred(G,\B)\) on \(\Lt^2(G,\B,\s,\tilde\alpha)\otimes_{\SUF[B]}
  \Hils\) by tensoring \(\Lt^2(G,\B,\s,\tilde\alpha)\)
  over~\(\Cont_0(G^0,\FU[B])\) with~\(\varrho\).
  This representation is faithful if~\(\varrho\) is faithful by
  Remark~\ref{rem:induce_reps}.

  The homomorphism~\(\mor\) also induces a \Star{}homomorphism
  \(\mor_*\colon \SUF \to \SUF[B]\).
  The map
  \[
    \Contc(G,\mor)\odot \Id_{\Hils}\colon
    \Contc(G,\A)\odot \Hils \to \Contc(G,\B)\odot \Hils
  \]
  extends to an isometry
  \[
    \Lt^2(G,\A,\s,\tilde\alpha)
    \otimes_{\SUF,\varrho\circ\mor_*} \Hils
    \hookrightarrow
    \Lt^2(G,\B,\s,\tilde\alpha)
    \otimes_{\SUF[B],\varrho} \Hils,
  \]
  because it preserves inner products and \(\Contc(G,\A)\) is dense in
  \(\Lt^2(G,\A,\s,\tilde\alpha)\).

  The representation \(\varrho\circ\mor_*\) of~\(\SUF\) induces a
  representation
  \[
    \omega\colon \Cred(G,\A)
    \to
    \Bound\bigl(
      \Lt^2(G,\A,\s,\tilde\alpha)
      \otimes_{\SUF,\varrho\circ\mor_*} \Hils
    \bigr).
  \]
  Extend~\(\omega\) by~\(0\) on the orthogonal complement to get a
  representation~\(\tilde\omega\) on the Hilbert space
  \(\Lt^2(G,\B,\s,\tilde\alpha)\otimes_{\SUF[B],\varrho} \Hils\).
  A direct inspection shows that \(\tilde\omega=\tilde\varrho \circ
  \Contc(G,\mor)\) on \(\Contc(G,\A)\).
  Since~\(\tilde\varrho\) is faithful and~\(\tilde\omega\) extends to
  \(\Cred(G,\A)\), the map \(\Cst(G,\mor)\) descends to a
  well-defined \Star{}homomorphism \(\Cred(G,\A) \to
  \Cred(G,\B)\).

  If~\(\mor\) is injective, then \(\varrho\circ\mor_*\) is injective
  on~\(\SUF\).
  Hence the induced representation of \(\Cred(G,\A)\) is faithful
  by Remark~\ref{rem:induce_reps}.
  Therefore, \(\Cred(G,\mor)\) is injective.
\end{proof}

\begin{remark}
  The assignments \(\A \mapsto \Cst(G,\A)\) and \(\A \mapsto
  \Cred(G,\A)\) define functors from the category of Fell bundles
  over~\(G\) with homomorphisms as in
  Definition~\ref{def:morphisms-Fell-bundles} to the category of
  \(\Cst\)\nb-algebras.

  Later we shall show that if~\(\mor_{u(x)}\) is nondegenerate for all
  \(x\in G^0\), then so are \(\Cst(G,\mor)\) and \(\Cred(G,\mor)\);
  this will be proved in the more general context of multiplier
  morphisms between Fell bundles, see
  Section~\ref{sec:multiplier_split_exact}.
\end{remark}

\subsection{Hereditary Fell subbundles and exactness for Fell ideals}
\label{sec:hereditary}

\begin{definition}
  \label{def:Fell_subbundle}
  Let~\(\B\) be a Fell bundle over~\(G\).
  A \emph{Fell subbundle} is a continuous family of Banach subspaces
  \(\A_g \subseteq \B_g\) as in
  Definition~\ref{def:continuous_family_subspaces} that satisfies
  \(\A_g^* \subseteq \A_{g^{-1}}\) and \(\A_g\cdot \A_h \subseteq
  \A_{g h}\) for all \((g,h)\in G^2\).
  A Fell subbundle \(\A\subseteq \B\) is \emph{hereditary} if \(\A_g
  \B_h \A_k \subseteq \A_{g h k}\) for all \((g,h,k) \in G^3\);
  here~\(G^3\) is defined in~\eqref{eq:def_G3}.
  It is a \emph{Fell ideal} if \(\A_g \B_h \subseteq \A_{g h}\) and
  \(\B_g \A_h \subseteq \A_{g h}\) for all \(g,h\in G\).
\end{definition}

The two conditions \(\A_g \B_h \subseteq \A_{g h}\) and \(\B_g \A_h
\subseteq \A_{g h}\) for all \(g,h\in G\) are equivalent for a Fell
subbundle because \(\A_g^* = \A_{g^{-1}}\) for all \(g\in G\).

\begin{remark}
  An injective Fell-bundle homomorphism \(i\colon \A \to \B\) is the
  same as an isomorphism of Fell bundles onto a Fell subbundle
  \(i(\A) \subseteq \B\).
  Fell ideals are hereditary Fell subbundles.
\end{remark}

By Lemma~\ref{lem:subbundle_from_continuous_family}, a Fell
subbundle is a Fell bundle in its own right, such that the
continuous sections of~\(\A\) are exactly those continuous sections
\(f\colon G\to \B\) with \(f(g)\in\A_g\) for all \(g\in G\).  It is
well known that any closed two-sided ideal in a \(\Cst\)\nb-algebra
is a \Star{}ideal.  This is proven using approximate units.  The
same argument shows that the condition \(\A_g^* = \A_{g^{-1}}\) is
redundant for Fell ideals:

\begin{lemma}
  \label{lem:Fell_ideal_star}
  Let \(\I_g\subseteq \A_g\) be a continuous family of Banach
  subspaces.  Assume
  \[
    \I_g \cdot \A_h \subseteq \I_{g\cdot h},\qquad
    \A_g \cdot \I_h \subseteq \I_{g\cdot h}
  \]
  for all \((g,h)\in G^2\).  Then \(\I_g^*=\I_{g^{-1}}\) for all
  \(g\in G\).  This makes~\(\I\) a Fell ideal.
\end{lemma}

\begin{proof}
  Let \(g\in G\).  By assumption,
  \(\I_{u\s(g)}\subseteq \FU_{\s(g)}\) is a closed ideal.  Then it
  is a \(\Cst\)\nb-subalgebra.  So~\(\I_{u\s(g)}\) has an
  approximate unit~\((e_i)\).  Since \(\I_g\subseteq \A_g\) is a
  right \(\FU_{\s(g)}\)-submodule and
  \(\I_g^*\I_g\subseteq \A_{g^{-1}}\I_g\subseteq \I_{u\s(g)}\), it
  follows that~\(\I_g\) is a Hilbert \(\I_{u\s(g)}\)-module.  This
  makes it nondegenerate as a right \(\I_{u\s(g)}\)-module.  That
  is, \(\lim x\cdot e_i= x\) for all \(x\in \I_g\).  Then
  \(x^* = \lim e_i x^*\) as well, and this belongs to
  \(\I_{u\s(g)} \A_{g^{-1}} \subseteq \I_{g^{-1}}\).  This proves
  that~\(\I_g^*\) is contained in~\(\I_{g^{-1}}\).  Replacing~\(g\)
  by~\(g^{-1}\) and taking adjoints gives the reverse inclusion.  So
  \(\I_g^*=\I_{g^{-1}}\).
\end{proof}

Recall that \(\FU[B]\defeq \B|_{G^0}\) and \(\SUF[B]\defeq
\Cont_0(G^0,\FU[B])\).

\begin{proposition}
  \label{pro:hereditary_Fell_subbundle}
  Let~\(\B\) be a Fell bundle over~\(G\).  Let
  \(\FU_x \subseteq \FU[B]_x\) for \(x\in G^0\) be a continuous family
  of hereditary subalgebras.  Let \(\A_g\subseteq \B_g\) be
  \(\FU_{\rg(g)} \B_g \FU_{\s(g)}\), the closed linear span of
  products \(a_1 b_2 a_3\) with \(a_1\in \FU_{\rg(g)}\),
  \(b_2\in \B_g\), \(a_3\in \FU_{\s(g)}\).  Then \((\A_g)_{g\in G}\)
  is a hereditary Fell subbundle of~\(\B\).  Conversely, any
  hereditary Fell subbundle is of this form for a unique continuous
  family of hereditary subalgebras \(\FU_x \subseteq \FU[B]_x\) for
  \(x\in G^0\).  So there is a natural bijection between hereditary
  Fell subbundles in~\(\B\) and continuous families of hereditary
  subalgebras in~\(\FU[B]\).
\end{proposition}

\begin{proof}
  It is easy to see that the family of Banach subspaces
  \((\FU_{\rg(g)} \B_g \FU_{\s(g)})_{g\in G}\) in~\(\B\) is
  continuous.  We check
  \begin{align*}
    (\FU_{\rg(g)} \B_g \FU_{\s(g)})^*
    &= \FU_{\rg(g^{-1})} \B_{g^{-1}} \FU_{\s(g^{-1})},\\
    (\FU_{\rg(g)} \B_g \FU_{\s(g)})\cdot
    (\FU_{\rg(h)} \B_h \FU_{\s(h)})
    &\subseteq \FU_{\rg(g)} \B_{g h} \FU_{\s(h)},\\
    (\FU_{\rg(g)} \B_g \FU_{\s(g)})\cdot \B_h \cdot
    (\FU_{\rg(k)} \B_k \FU_{\s(k)})
    &\subseteq \FU_{\rg(g)} \B_{g h k} \FU_{\s(k)}
  \end{align*}
  for all \((g,h,k)\in G^3\).  This says that
  \((\FU_{\rg(g)} \B_g \FU_{\s(g)})_{g\in G}\) is a hereditary Fell
  subbundle of~\(\B\).  Since each~\(\FU_x\) is hereditary
  in~\(\FU[B]_x\), it follows that \(\FU_x \FU[B]_x \FU_x\) is
  equal to~\(\FU_x\).  Thus
  \((\FU_{\rg(g)} \B_g \FU_{\s(g)})_{g\in G}\) determines
  \((\FU_x)_{x\in G^0}\).

  Now let \((\A_g)_{g\in G}\) be any hereditary Fell subbundle
  of~\(\B\).
  If \(x\in G^0\), then \(\FU_x \defeq \A_{u(x)}\) is a hereditary
  subalgebra in~\(\FU[B]_x\).
  Since~\((\A_g)_{g\in G}\) is a Fell bundle in its own right,
  \(\A_g\) is nondegenerate as a bimodule over \(\FU_{\rg(g)}\)
  and~\(\FU_{\s(g)}\).  That is,
  \(\A_g = \FU_{\rg(g)} \A_g \FU_{\s(g)} \subseteq \FU_{\rg(g)} \B_g
  \FU_{\s(g)}\).  The converse inclusion
  \(\FU_{\rg(g)} \B_g \FU_{\s(g)} \subseteq \A_g\) follows from the
  definition of a hereditary Fell subbundle.  Thus
  \(\A_g =\FU_{\rg(g)} \B_g \FU_{\s(g)}\) for all \(g\in G\).
\end{proof}

\begin{corollary}
  \label{cor:hereditary_subbundle_subalgebra}
  Hereditary Fell subbundles of~\(\B\) are in bijection with
  hereditary \(\Cst\)\nb-subalgebras of~\(\SUF[B]\).
\end{corollary}

\begin{proof}
  Proposition~\ref{pro:hereditary_Fell_subbundle} describes a
  bijection between hereditary Fell subbundles of~\(\B\) and
  continuous families of hereditary \(\Cst\)\nb-subalgebras
  \(\FU_x \subseteq \FU[B]_x\) for \(x\in G^0\).  Then the
  subalgebra of~\(\SUF\) consisting of those sections with values
  in~\((\FU_x)\) is a hereditary subalgebra of~\(\SUF[B]\).
  Conversely, let \(D\subseteq \SUF[B]\) be hereditary.  Then
  \(\Cont_0(G^0)\cdot D = \Cont_0(G^0)\cdot D\cdot D = D\cdot
  \Cont_0(G^0)\cdot D \subseteq D\) because~\(D\) is hereditary and
  \(\Cont_0(G^0)\) is central.  Thus~\(D\) is also a
  \(\Cont_0(G^0)\)-\(\Cst\)-algebra.  It follows that there is a
  continuous family of \(\Cst\)\nb-subalgebras
  \(\FU_x \subseteq \FU[B]_x\) for \(x\in G^0\) such that
  \(D = \Cont_0(G^0, (\FU_x)_{x\in G^0})\).  Since~\(D\) is
  hereditary, it follows easily that each fibre~\(\FU_x\) of~\(D\)
  is hereditary in~\(\FU[B]_x\).
\end{proof}

\begin{remark}
  \label{rem:hered-subalg}
  If \(\A\subseteq \B\) is a hereditary subbundle, then
  \(\tilde\A\defeq \Cont_0(G,\A\A^*)\) is a hereditary \(\Cst\)\nb-subalgebra in
  \(\tilde\B\defeq \Cont_0(G,\B\B^*)\).
\end{remark}

\begin{proposition}
  \label{pro:Fell_ideals}
  Let \(\FU_x \subseteq \FU[B]_x\) for \(x\in G^0\) be a continuous
  family of hereditary subalgebras.  The corresponding hereditary
  Fell subbundle is a Fell ideal if and only if
  \(\FU_{\rg(g)} \B_g = \B_g \FU_{\s(g)}\) for all \(g\in G\).  Then
  each \(\FU_x\subseteq \FU[B]_x\) is an ideal.
\end{proposition}

\begin{proof}
  Assume first that \(\FU_{\rg(g)} \B_g = \B_g \FU_{\s(g)}\) for all
  \(g\in G\).  Then
  \[
    \FU_{\rg(g)} \B_g \FU_{\s(g)}
    = \FU_{\rg(g)} \B_g = \B_g \FU_{\s(g)}.
  \]
  This implies that \((\FU_{\rg(g)} \B_g \FU_{\s(g)})_{g\in G}\) is both a left
  and a right ideal in~\(\B\), and this makes it a Fell ideal.
  Conversely, if the subspaces
  \(\A_g \defeq \FU_{\rg(g)} \B_g \FU_{\s(g)}\) form a Fell ideal,
  then
  \[
    \FU_{\rg(g)} \B_g = \A_{u(\rg(g))} \B_g
    \subseteq \A_g
    = \FU_{\rg(g)} \B_g \FU_{\s(g)}
    \subseteq \FU_{\rg(g)} \B_g \FU[B]_{\s(g)}
    = \FU_{\rg(g)} \B_g.
  \]
  So \(\A_g = \FU_{\rg(g)} \B_g\).
  A symmetric argument shows \(\A_g = \B_g \FU_{\s(g)}\).
\end{proof}

\begin{definition}
  \label{def:invariant_ideal}
  Let~\(\B\) be a Fell bundle over~\(G\) and let
  \(\FU = (\FU_x)_{x\in G^0}\) be a continuous family of ideals
  in \(\FU[B]=\B|_{G^0}\).
  We call~\(\FU\) \emph{\(\B\)\nb-invariant} if \(\FU_{\rg(g)} \B_g =
  \B_g \FU_{\s(g)}\) for all \(g\in G\).
 \end{definition}

 \begin{lemma}
   \label{lem:invariant_ideal}
   A continuous family of ideals \(\FU =  (\FU_x)_{x\in G^0}\) is
   \(\B\)\nb-invariant if and only if
   \(\B_g\FU_{\s(g)}\B_{g^{-1}}\subseteq \FU_{\rg(g)}\) for all \(g\in
   G\).
 \end{lemma}

 \begin{proof}
   If \(\FU_{\rg(g)} \B_g = \B_g \FU_{\s(g)}\), then \(\B_g
   \FU_{\s(g)} \B_{g^{-1}} = \FU_{\rg(g)} \B_g \B_{g^{-1}} \subseteq
   \FU_{\rg(g)}\).
   Conversely, assume \(\B_g \FU_{\s(g)} \B_{g^{-1}} \subseteq
   \FU_{\rg(g)}\).
   Then \(\B_g \FU_{\s(g)} \subseteq \B_g \FU_{\s(g)} \B_{g^{-1}} \B_g
   \subseteq \FU_{\rg(g)} \B_g\).
   Taking the adjoint of this relation for~\(g^{-1}\) gives the
   reverse inclusion as well.
 \end{proof}

 \begin{corollary}
   \label{cor:Fell_ideal_family_ideals}
   A continuous family of ideals~\(\FU\) is \(\B\)\nb-invariant if and only
   if the hereditary subbundle in~\(\B\) that corresponds to~\(\FU\)
   is a Fell ideal in~\(\B\).
   As a consequence, Fell ideals in~\(\B\) are in bijection with
   \(\B\)\nb-invariant ideals in \(B= \Cont_0(G^0,\FU[B])\), which are in
   turn in bijection with continuous families of \(\B\)\nb-invariant ideals
   of~\(\FU[B]\).
\end{corollary}

\begin{example}
  Assume~\(\B\) is the Fell bundle associated to an
  ordinary action of~\(G\) by a continuous family of isomorphisms
  \(\beta_g\colon \FU[B]_{\s(g)}\congto \FU[B]_{\rg(g)}\) for \(g\in
  G\).
  Then a continuous family over the units \(\FU\subseteq \FU[B]\) is
  \(\B\)\nb-invariant if and only if it is invariant in the usual sense of
  actions, that is, \(\beta_g(\FU_{\s(g)})\subseteq \FU_{\rg(g)}\) for all
  \(g\in G\); this follows quickly from the definition of the product
  in~\(\B\).
  As a consequence, Fell ideals of~\(\B\) correspond to invariant
  ideals in the usual sense.
  If \(\FU\subseteq \FU[B]\) is any continuous family of ideals, then
  the associated hereditary Fell subbundle is the one associated to
  the partial action, given by restricting the action~\(\beta\)
  to~\(\FU\).
\end{example}

\begin{example}
  \label{exa:ideal_from_invariant_open}
  Let \(U\subseteq G^0\) be an open \(G\)\nb-invariant subset, that is,
  if \(g\in G\) and \(\s(g)\in U\), then \(\rg(g)\in U\).
  Then the reduction
  \[
    G_U \defeq \s^{-1}(U) = \rg^{-1}(U)
  \]
  is an open subgroupoid of~\(G\).
  Consider the trivial one-dimensional Fell bundle
  \(\B^{\mathrm{triv}}\) over~\(G\), and the trivial one-dimensional
  Fell bundle \(\B^{\mathrm{triv}}_U\) over~\(G_U\).
  Extending sections by zero outside~\(G_U\) identifies
  \(\B^{\mathrm{triv}}_U\) with a Fell subbundle \(\A\subseteq
  \B^{\mathrm{triv}}\), where
  \[
    \A_g =
    \begin{cases}
      \B^{\mathrm{triv}}_g \cong \C & \text{if } g\in G_U,\\
      \{0\} & \text{if } g\notin G_U.
    \end{cases}
  \]
  Then \(\A\) is a Fell ideal in~\(\B^{\mathrm{triv}}\).
  Indeed, the corresponding continuous family \(\FU = (\FU_x)_{x\in
    G^0}\) of ideals in the trivial bundle \(\FU[B]^{\mathrm{triv}}\cong
  G^0\times\C\) is given by \(\FU_x \cong \C\) for \(x\in U\) and
  \(\FU_x = \{0\}\) for \(x\notin U\), and \(G\)-invariance of \(U\) is
  precisely the condition
  \(\FU_{\rg(g)}\B^{\mathrm{triv}}_g=\B^{\mathrm{triv}}_g\FU_{\s(g)}\)
  from Definition~\ref{def:invariant_ideal}.

  Conversely, every Fell ideal \(\A\subseteq \B^{\mathrm{triv}}\)
  is of this form: since each fibre \(\B^{\mathrm{triv}}_g\cong\C\),
  one has \(\A_g\in\{0,\C\}\), so \(H=\setgiven{g\in G}{\A_g=\C}\) is open and
  \(U=H\cap G^0\) is open.
  The Fell-bundle axioms imply that~\(H\) is a subgroupoid, and the
  ideal conditions force \(\s(g),\rg(g)\in U\) for \(g\in H\), hence
  \(H\subseteq G_U\).
  Conversely, if \(g\in G_U\), then \(\A_{\rg(g)}=\C\) and the ideal
  property gives \(\A_g=\C\), so \(G_U\subseteq H\).
  Thus \(H=G_U\), and \(\A\) is as above.

  Therefore, Fell ideals of the trivial bundle~\(\B^{\mathrm{triv}}\)
  are in bijection with open invariant subsets of~\(G^0\).
\end{example}

\begin{proposition}
  \label{pro:Fell_ideal_inclusion_sections}
  Let~\(\B\) be a Fell bundle and let~\(\A\) be an ideal Fell
  subbundle of~\(\B\).
  Then the
  \Star{}homomorphisms \(\Cst(G,\A) \to \Cst(G,\B)\) and
  \(\Cred(G,\A) \to \Cred(G,\B)\) induced by the inclusion
  \(\A\subseteq \B\) are isomorphisms onto \(\Cst\)\nb-ideals
  of\/ \(\Cst(G,\B)\) and \(\Cred(G,\B)\), respectively.
  More precisely, the range is the \(\Cst\)\nb-ideal generated by
  \(\Cont_0(G^0,\A|_{G^0})\), in either \(\Cst(G,\B)\) or
  \(\Cred(G,\B)\).
\end{proposition}

\begin{proof}
  Let \(\SUF\defeq \Cont_0(G^0,\A|_{G^0})\subseteq \SUF[B]\).
  Since \(\A\) is an ideal Fell subbundle of~\(\B\), we have
  \(\A_g=\FU_{\rg(g)}\B_g=\B_g\FU_{\s(g)}\) for all \(g\in G\).
  Hence Lemma~\ref{lem:subfield_equality} implies
  \[
    \SUF\Contc(G,\B)=\Contc(G,\B)\SUF=\Contc(G,\A).
  \]
  Therefore, the range of the induced \Star{}homomorphism
  \(\Cst(G,\A)\to \Cst(G,\B)\) is the closed two-sided \(\Cst\)\nb-ideal of
  \(\Cst(G,\B)\) generated by \(\SUF\subseteq \Mult(\Cst(G,\B))\).
  The same argument applies to the reduced section algebras, so the range of
  \(\Cred(G,\A)\to \Cred(G,\B)\) is the \(\Cst\)\nb-ideal generated by
  \(\SUF\) in~\(\Cred(G,\B)\).

  Proposition~\ref{pro:functorial} already implies that the induced map
  \(\Cred(G,\A)\to \Cred(G,\B)\) is injective.  It remains to prove that
  \(\Cst(G,\A)\to \Cst(G,\B)\) is injective.

  Let \(\iota\colon \Cst(G,\A)\to \Cst(G,\B)\) be the
  \Star{}homomorphism induced  by the inclusion \(\A\hookrightarrow
  \B\).
  To prove that~\(\iota\) is isometric, let \((\varphi,U)\) be a
  representation of the Fell bundle~\(\A\) on a Hilbert module~\(\F\)
  over some \(\Cst\)\nb-algebra~\(D\), and let \(L_{\varphi,U}\) be its
  integrated form.

  Since \(\SUF\) is an ideal in~\(\SUF[B]\) and \(\varphi\colon
  \SUF\to \Bound(\F)\) is nondegenerate, \(\varphi\) extends uniquely
  to a nondegenerate representation \(\bar\varphi\colon \SUF[B]\to
  \Bound(\F)\).
  Explicitly, if \(b\in \SUF[B]\), then we define \(\bar\varphi(b)\) by
  \[
    \bar\varphi(b)\bigl(\varphi(a)\xi\bigr)\defeq \varphi(ba)\xi
    \qquad (a\in \SUF,\ \xi\in \F).
  \]
  Notice that \(\varphi(\SUF)\F=\F\) because~\(\varphi\) is
  nondegenerate.

  The inclusion \(\Contc(G,\A)\hookrightarrow \Contc(G,\B)\) extends to an
  isometric inclusion of Hilbert \(\SUF\)-modules
  \(\Lt^2(G,\A,\s,\tilde\alpha)\hookrightarrow
  \Lt^2(G,\B,\s,\tilde\alpha)\).
  Moreover, there is a canonical isomorphism of Hilbert \(D\)-modules
  \[
    \Lt^2(G,\A,\s,\tilde\alpha)\otimes_\varphi \F
    \cong
    \Lt^2(G,\B,\s,\tilde\alpha)\otimes_{\bar\varphi} \F,
  \]
  characterised on elementary tensors by
  \[
    f\otimes \varphi(a)\xi \longmapsto (f\cdot a)\otimes \xi
    \qquad (f\in \Contc(G,\A),\ a\in \SUF,\ \xi\in \F),
  \]
  where \(f\cdot a\in \Contc(G,\A)\subseteq \Contc(G,\B)\) denotes the
  right action of~\(\SUF\) on \(\Contc(G,\B)\).

  Under this identification, the unitary
  \[
    U\colon \Lt^2(G,\A,\s,\tilde\alpha)\otimes_\varphi \F
    \xrightarrow{\ \sim\ }
    \Lt^2(G,\A\A^*,\rg,\alpha)\otimes_\varphi \F
  \]
  may be viewed as a unitary
  \[
    \bar U\colon \Lt^2(G,\B,\s,\tilde\alpha)\otimes_{\bar\varphi} \F
    \xrightarrow{\sim}
    \Lt^2(G,\B\B^*,\rg,\alpha)\otimes_{\bar\varphi} \F.
  \]
  Similarly, the representation identity for \((\varphi,U)\) implies the corresponding
  identity for \((\bar\varphi,\bar U)\).
  Hence \((\bar\varphi,\bar U)\) is a representation of~\(\B\)
  on~\(\F\) extending \((\varphi,U)\).

  Let \(L_{\bar\varphi,\bar U}\) be the integrated form.  By construction,
  \(L_{\bar\varphi,\bar U}(\xi)=L_{\varphi,U}(\xi)\) for all
  \(\xi\in \Contc(G,\A)\subseteq \Contc(G,\B)\).  Therefore, for such \(\xi\),
  \[
    \norm{L_{\varphi,U}(\xi)}
    \le \norm{L_{\bar\varphi,\bar U}(\xi)}
    \le \norm{\iota(\xi)}_{\Cst(G,\B)}.
  \]
  Taking the supremum over all representations \((\varphi,U)\) of~\(\A\) yields
  \(\norm{\xi}_{\Cst(G,\A)}\le \norm{\iota(\xi)}_{\Cst(G,\B)}\).
  Together with the opposite inequality above, this shows that
  \(\iota\) is isometric on \(\Contc(G,\A)\), and hence injective.
\end{proof}

\begin{remark}
  We strongly believe that Proposition~\ref{pro:Fell_ideal_inclusion_sections}
  remains true for arbitrary hereditary Fell subbundles.
  In the special case of Fell bundles over discrete groups, an
  analogous result is already known; see
  \cite{Exel:Partial_dynamical}*{Theorem 21.13}.
  However, the proof technique used above for ideal Fell subbundles
  does not extend easily to arbitrary hereditary Fell subbundles over
  groupoids.
  Although related methods can be applied to certain `invariant'
  hereditary subbundles, developing the necessary technical framework
  would substantially complicate the exposition.
  For this reason, we restrict ourselves to ideal Fell subbundles and
  leave the general hereditary case for future work.
\end{remark}

\begin{deflem}
  \label{def:quotient_Fell_bundle}
  Let \(\I\subseteq \A\) be a Fell ideal.
  There is a unique Fell-bundle structure on the quotients
  \(\B_g\defeq\A_g/\I_g\) such that the quotient maps \(\A_g \to
  \B_g\) form a surjective Fell-bundle homomorphism.
  We call this the \emph{quotient Fell bundle} of~\(\A\) by~\(\I\).
\end{deflem}

\begin{proof}
  By assumption, \(\I = (\I_g)_{g\in G}\) is a continuous
  family of closed subspaces of the Banach spaces~\(\A_g\).
  For each \(g\in G\), form the quotient \(\B_g\defeq\A_g/\I_g\).
  Since the quotient maps \(\A_g\to\B_g\) and elements of~\(\B_g\)
  lift to elements with almost the same norm, the family
  \((\B_g)_{g\in G}\) becomes an upper semicontinuous Banach bundle
  over~\(G\) as well, such that the quotient maps assemble to a
  continuous open surjection \(\A\to\B\).

  The Fell-bundle structure on~\(\B\) is defined fibrewise by
  \[
    [a_g]\cdot[b_h]\defeq[a_gb_h], \qquad
    [a_g]^*\defeq [a_g^*],
  \]
  for \(a_g\in\A_g\), \(b_h\in\A_h\).
  This is well defined because~\(\I\) is a Fell ideal.
  Thie multiplication and involution on~\(\B\) are continuous because
  they are so on~\(\A\) and the map \(\A\to\B\) is a quotient map.
  The Fell-bundle axioms for~\(\B\) are inherited directly from those
  of~\(\A\).
\end{proof}

This construction is standard in the literature.
For saturated and separable Fell bundles, a closely related quotient
construction is carried out in
\cites{Ionescu-Williams:Remarks_ideal_structure,
  KQW:Rieffel-correspondence}.
We do not need to assume the bundles to be saturated or separable.

\begin{proposition}
  \label{pro:full_Cstar_exact}
  Let~\(\A\) be a Fell bundle over~\(G\), let \(\I\idealin \A\) be a
  Fell ideal, and \(\B \defeq \A/\I\).  The canonical maps
  \(\I\into \A\onto \B\) induce a \(\Cst\)\nb-algebra extension
  \[
    \Cst(\I)\into \Cst(\A)\onto \Cst(\B).
  \]
\end{proposition}

\begin{proof}
  Let \(\FU[I] \defeq \I|_{G^0}\) and
  \(\SUF[I] \defeq \Cont_0(G^0,\FU[I])\) be the continuous family of
  ideals and the ideal in~\(\SUF\) that correspond to the Fell
  ideal~\(\I\) by Corollary~\ref{cor:Fell_ideal_family_ideals}.
  Proposition~\ref{pro:Fell_ideal_inclusion_sections} shows that
  the inclusion map \(\I\into \A\) induces an injective map
  \(\iota\colon \Cst(G,\I)\to\Cst(G,\A)\).
  Its image is an ideal
  because \(\Contc(G,\I)\) is an ideal in \(\Contc(G,\A)\).
  Proposition~\ref{pro:functorial} shows that the quotient map
  \(\A \onto \B\) induces a surjective \Star{}homomorphism
  \(\pi\colon \Cst(G,\A) \onto \Cst(G,\B)\).
  The kernel of~\(\pi\)
  contains \(\iota(\Cst(G,\I))\) because \(\Contc(G,\I)\) is dense
  in \(\Cst(G,\I)\) and the composite map
  \(\Contc(G,\I) \to \Contc(G,\A) \to \Contc(G,\B)\) vanishes.
  It remains to prove, conversely, that \(\ker \pi\) is contained in
  \(\iota(\Cst(G,\I))\).
  For this, we are going to build a Fell-bundle representation
  of~\(\B\) on
  \[
    D\defeq \Cst(G,\A)/\iota(\Cst(G,\I)).
  \]
  Let \(Q\colon \Cst(G,\A)\onto D\) be the quotient homomorphism.
  It corresponds to a Fell-bundle representation \((\psi,V)\)
  of~\(\A\) on~\(D\) by our universal property.  Here~\(\psi\) is a
  (nondegenerate) \Star{}homomorphism
  \(\psi\colon \SUF\to \Mult(D)\) and~\(V\) is a unitary operator
  \[
    V\colon \Lt^2(G,\A,\s,\tilde\alpha)\otimes_\psi D
    \congto \Lt^2(G,\A\A^*,\rg,\alpha)\otimes_\psi D,
  \]
  where~\(D\) is viewed as a Hilbert \(D\)\nb-module.  By
  naturality, \(\psi\) is the composite of the canonical
  \Star{}homomorphism \(\SUF\to \Mult(\Cst(G,\A))\) in
  Proposition~\ref{pro:coefficients_embed} with~\(Q\).  This clearly
  vanishes on the ideal \(\SUF[I]\idealin \SUF\).  So it factors
  through a nondegenerate \Star{}homomorphism
  \[
    \tilde\psi\colon \SUF/\SUF[I]\cong \SUF[B] \to \Mult(D).
  \]

  In general, let \(A\) and~\(D\) be \(\Cst\)\nb-algebras,
  \(I\idealin A\), let~\(\E\) be a Hilbert \(A\)\nb-module, and
  let~\(\F\) be an \(A/I,D\)-correspondence.
  Then we may consider~\(\F\) as an \(A,D\)-correspondence~\(\F'\) and
  define \(\E \otimes_A \F'\).
  Or we may take the quotient \(\E/\E I\), which is a Hilbert
  \(A/I\)-module, and form \(\E/\E I \otimes_{A/I} \F\).
  These two tensor products are canonically isomorphic because \(\E/\E
  I \cong \E \otimes_A A/I\) and \(\F' \cong A/I \otimes_{A/I} \F\),
  where we treat~\(A/I\) as an \(A,A/I\)-correspondence.
  In our case, this gives us canonical isomorphisms
  \begin{align*}
    \Lt^2(G,\A,\s,\tilde\alpha)\otimes_{\SUF} D
    &\cong \Lt^2(G,\B,\s,\tilde\alpha)\otimes_{\SUF[B]} D,\\
    \Lt^2(G,\A\A^*,\rg,\alpha)\otimes_{\SUF} D
    &\cong \Lt^2(G,\B\B^*,\rg,\alpha)\otimes_{\SUF[B]} D,\\
    \Lt^2(G^2,\E_k(\A),v_k,\mu_k)\otimes_{\SUF} D
    &\cong \Lt^2(G^2,\E_k(\B),v_k,\mu_k)\otimes_{\SUF[B]} D.
  \end{align*}
  Here \(\E_k(\A)\) and \(\E_k(\B)\) are the fields of
  correspondences that appear in the definition of a representation
  for the Fell bundles \(\A\) and~\(\B\), respectively.  We used
  that~\(\I\) is a Fell ideal to identify \(\A/\A I \cong \B\),
  \(\A\A^*/ \A\A^* I \cong \B\B^*\), and
  \(\E_k(\A)/\E_k(\A) I\cong \E_k(\B)\) for \(k=0,1,2\).  After
  these identifications, the same unitary~\(V\)
  completes~\(\tilde\psi\) to a Fell-bundle representation
  \((\tilde\psi,\tilde V)\) of~\(\B\) on~\(D\).  The conditions for
  being a Fell-bundle representation of~\(\B\) are equivalent to the
  corresponding conditions as a representation of~\(\A\).

  Let \(\Psi\colon \Cst(G,\B) \to \Mult(D)\) be the representation
  that integrates \((\tilde\psi,\tilde V)\).  The explicit formulas
  for integrating a Fell-bundle representation show that
  \(\Psi\circ\pi\) is the representation of \(\Cst(G,\A)\) that
  integrates \((\psi,V)\).  In other words, \(\Psi\circ\pi\) is the
  quotient map \(\Cst(G,\A) \to D\).  Since this quotient map
  factors through~\(\pi\), it follows that
  \(\ker \pi\subseteq \ker Q =\iota(\Cst(G,\I))\) as desired.
\end{proof}

The above extension generalises the classical case for saturated and
separable Fell bundles, proved by Ionescu--Williams
\cite{Ionescu-Williams:Remarks_ideal_structure}; see also
\cite{KQW:Rieffel-correspondence}, where the same
quotient construction is used to study Fell-bundle equivalence.
Our proof works without assuming the bundles to be saturated or
separable.

In contrast, the sequence
\[
  \Cred(G,\I)
  \into \Cred(G,\A)
  \onto \Cred(G,\A/\I)
\]
may fail to be exact.
This problem already occurs for ordinary actions of discrete groups,
which leads to the definition of an \emph{exact} group.
We will prove later that this sequence is always exact if the
underlying extension \(\I \into \A \onto \A/\I\) splits (see
Proposition~\ref{pro:split_exact}).

\subsection{Multiplier morphisms of Fell bundles}
\label{sec:multiplier_split_exact}

We are going to define the multipler bundle of a Fell bundle.
We prove that a \Star{}homomorphism from a Fell bundle to the
multiplier bundle of another Fell bundle induces homomorphisms to
multipliers for the full and reduced section \(\Cst\)\nb-algebras.

\begin{definition}
  \label{def:multiplier_g}
  Let~\(\A\) be a Fell bundle over~\(G\) and \(g\in G\).  A
  \emph{multiplier of degree~\(g\)} of~\(\A\) is a family of maps
  \(m_h\colon \A_h \to \A_{g h}\) for all \(h\in G^{\s(g)}\) for
  which there is a family of maps
  \(m_k^*\colon \A_k \to \A_{g^{-1} k}\) for all \(k\in G^{\rg(g)}\)
  such that \(a_1^* \cdot m_h(a_2) = m_k^*(a_1)^* \cdot a_2\) for
  all \(a_1 \in \A_k\), \(a_2\in \A_h\), \(k\in G^{\rg(g)}\),
  \(h\in G^{\s(g)}\).
\end{definition}

\begin{lemma}
  \label{lem:multiplier_g}
  Let~\((m_h)_{h\in G^{\s(g)}}\) be a multiplier of degree~\(g\).
  The maps~\(m_h\) for \(h\in G^{\s(g)}\) are linear and uniformly
  bounded with \(\norm{m_h} \le \norm{m_{u\s(h)}}\) for all
  \(h\in G^{\s(g)}\).
  They satisfy \(m_h(a_1)\cdot a_2 = m_{h\cdot k}(a_1\cdot a_2)\)
  whenever \((h,k)\in G^2\), \(a_1\in \A_h\), \(a_2\in \A_k\).
  The family of maps \((m_k^*)_{k\in G^{\rg(g)}}\) is unique, and it
  is a multiplier of degree~\(g^{-1}\) with \((m_k^*)^* = (m_h)\).
\end{lemma}

\begin{proof}
  This is analogous to a result about adjointable operators on
  Hilbert modules.  Since the inner product on~\(\A_{g h}\) is
  nondegenerate, \(m_h(a_2)\) for \(a_2\in \A_h\) is determined by
  \(a_1^* m_h(a_2) = m_{g h}^*(a_1)^* a_2\) for all
  \(a_1\in \A_{g h}\).  And this depends linearly on~\(a_2\).
  Thus~\(m_h\) is linear.
  It also follows that the graph of~\(m_h\) is closed, and this
  implies that each~\(m_h\) is bounded.
  If \((g,h,k)\in G^3\), \(a_1\in \A_h\),
  \(a_2\in \A_k\),   \(a_3\in \A_k\), and \(b\in \A_{g h k}\), then
  \(b^* m_h(a_2) a_3 = m_{g h k}^*(b)^* a_2 a_3 = b^* m_{h k}(a_2
  a_3)\).  Then \(m_h(a_1)\cdot a_2 = m_{h\cdot k}(a_1\cdot a_2)\).
  Since
  \(\FU_{\s(g)}\cdot \A_h = \A_h\) for all \(h\in G^{\s(g)}\), the
  Cohen--Hewitt Factorisation Theorem allows us to write any
  \(a\in \A_h\) as \(a=a_1\cdot a_2\) with \(a_1\in \A_{u\s(g)}\),
  \(a_2\in \A_h\), and
  \(\norm{a_1}\cdot \norm{a_2} \le \norm{a}(1+\varepsilon)\) for any given
  \(\varepsilon>0\).  Then
  \(\norm{m_h(a)} \le \norm{m_{u\s(g)}(a_1)}\cdot \norm{a_2}\).
  This implies \(\norm{m_h} \le (1+\varepsilon) \norm{m_{u\s(h)}}\).
  Since \(\varepsilon>0\) is arbitrary, this implies
  \(\norm{m_h} \le \norm{m_{u\s(h)}}\) as desired.  If
  \(k\in G^{\rg(g)}\), then \(m_k^*(a_1)\) for \(a_1\in \A_k\) is
  determined by \(m_k^*(a_1)^* a_2 = a_1^* m_{g^{-1} k}(a_2)\) for
  all \(a_1\in \A_{g^{-1} k}\).  Thus~\(m_k^*\) is determined
  uniquely.  Taking the adjoint of the relation that links~\((m_h)\)
  and~\((m_k^*)\) shows that~\((m_k^*)\) is also a multiplier, with
  the maps~\(m_h\) giving the adjoints.
\end{proof}

\begin{remark}
  \label{rem:multiplier_saturated}
  We briefly compare our definition of a multiplier  of a Fell bundle
  with another candidate, which only works in the saturated case.
  Let \((m_h)_{h\in G_{\s(g)}}\) be a multiplier of the Fell
  bundle~\(\A\) of degree \(g\in G\).
  Lemma~\ref{lem:multiplier_g} implies that \(m_{u\s(g)}\colon
  \FU_{\s(g)} \to \A_g\) is an adjointable operator of right Hilbert
  \(\FU_{\s(g)}\)-modules.
  Conversely, let \(T\colon \FU_{\s(g)} \to \A_g\) be such an
  adjointable operator.
  If \(h\in G_{\s(g)}\), then~\(T\) induces an adjointable operator
  \(T\otimes_{\FU_{\s(g)}} \Id_{\A_h}\colon \A_h \cong \FU_{\s(g)}
  \otimes_{\FU_{\s(g)}} \A_h \to \A_g \otimes_{\FU_{\s(g)}} \A_h\).
  The multiplication in~\(\A\) gives an isometry \(\A_g
  \otimes_{\FU_{\s(g)}} \A_h \injto \A_{g h}\).
  If~\(\A\) is saturated, then this is unitary, hence adjointable.
  So~\(T\) is part of a family of adjointable operators
  \(\A_h \to \A_{g h}\) for all \(h\in G_{\s(g)}\).
  This is a multiplier of the Fell bundle.
  So we may identify multipliers of degree~\(g\) with
  \(\Bound(\FU_{\s(g)}, \A_g)\) if~\(\A\) is saturated.

  If, however, the Fell bundle is not saturated, then the
  multiplication maps may fail to be adjointable.
  Then we do not know whether the linear maps \(\A_h \to \A_{g h}\)
  induced by~\(T\) are adjointable.
  If, however, \(g = u(x)\) for \(x\in G^0\) is a unit, then the
  multiplication map is unitary even if the Fell bundle is not
  saturated.
  So in this case, multipliers are just elements of
  \(\Bound(\FU_x,\FU_x) \cong \Mult(\FU_x)\).
  That is, a multiplier of degree \(u(x)\) is equivalent to a
  multiplier of the \(\Cst\)\nb-algebra~\(\FU_x\).
\end{remark}

\begin{definition}
  \label{def:multipliers_structure}
  Let~\(\Mult(\A)_g\) be the set of multipliers of degree~\(g\).
  Let \(m = (m_h)_{h\in G_{\s(g)}}\) and
  \(l = (l_h)_{h\in G_{\s(g)}}\) be multipliers of degree~\(g\) and
  \(\lambda\in\C\).  Then
  \(m + \lambda l = (m_h + \lambda l_h)_{h\in G_{\s(g)}}\) is a
  multiplier of degree~\(g\) as well.  We equip~\(\Mult(\A)_g\) with
  the norm
  \[
    \norm{m} \defeq \sup_{h\in G_{\s(g)}} \norm{m_h}
    = \norm{m_{u\s(g)}}.
  \]
  Let \((g,h)\in G^2\) and let \(x = (x_k)_{k\in G_{\s(h)}}\) be a
  multiplier of degree~\(h\).  Then we let
  \[
    m\cdot x \defeq (m_{h k}\circ x_k\colon
    \A_k \to \A_{h k} \to \A_{g h k})_{k \in G_{\s(h)}}.
  \]
\end{definition}

\begin{lemma}
  \label{lem:multiplier_bundle}
  The above structure makes each \(\Mult(\A)_g\) a Banach space.
  The map \((m_h)\mapsto (m_k^*)\) defines a conjugate-linear
  isometry \(\Mult(\A_g) \to \Mult(\A_{g^{-1}})\).  The
  multiplication above defines a bilinear map
  \(\Mult(\A)_g \times \Mult(\A)_h \to \Mult(\A)_{g h}\).  It
  satisfies
  \begin{itemize}
  \item multiplication is associative;
  \item \((a\cdot b)^* = b^* \cdot a^*\) for all \((g,h)\in G^2\),
    \(a\in \Mult(\A)_g\), \(b\in \Mult(\A)_h\);
  \item \((a^*)^* =a\) for all \(g\in G\), \(a\in \Mult(\A)_g\);
  \item \(\norm{a\cdot b} \le \norm{a}\cdot \norm{b}\) for all
    \((g,h)\in G^2\), \(a\in \Mult(\A)_g\), \(b\in \Mult(\A)_h\);
  \item \(\norm{a^* a} = \norm{a}^2\) for all \(g\in G\),
    \(a\in \Mult(\A)_g\);
  \item for each \(g\in G\), \(a\in \Mult(\A)_g\), there is
    \(b\in \Mult(\A)_{u\s(g)}\) with \(b^* b = a^* a\);
  \item for each \(g\in G\), \(a\in\A_g\), the family of maps
    \(m_h(b) \defeq a\cdot b\) for \(h\in G_{\s(g)}\), \(b\in \A_h\)
    is a multiplier.  This defines isometric linear maps
    \(\A_g \to \Mult(\A)_g\) that are multiplicative and preserve
    the involutions.
  \end{itemize}
\end{lemma}

\begin{proof}
  The algebraic properties of the norm, multiplication, and
  involution are mostly obvious.  The positivity of~\(a^* a\) for
  \(a\in \Mult(\A)_g\) follows because
  \(\Mult(\A)_{\s(g)} = \Mult(\FU_{\s(g)})\) and the operator
  \(a^*_{g^{-1}} a_g\colon \FU_{\s(g)} \to \A_g \to \FU_{\s(g)}\) is
  indeed positive.  The proof that \(\Mult(\A)_g\) is a Banach space
  and that the norms on these spaces satisfy the
  \(\Cst\)\nb-condition is like the proof of the analogous
  statements for the \(\Cst\)\nb-algebra of adjointable operators on
  a Hilbert module.
\end{proof}

Thus \(\Mult(\A) \defeq (\Mult(\A)_g)_{g\in G}\) for \(g\in G\) with
the norms, involutions, and multiplication maps above carries most
of the structure of a Fell bundle.  What is missing is a topology on
this bundle.  The most reasonable topology should be an analogue of
the strict topology on a multiplier algebra.  We will not study this
topology here in itself.  Our main purpose is to prove that a
``multiplier morphism''
\(\A \to \Mult(\B)\) for two Fell bundles \(\A\)
and~\(\B\) induces \Star{}homomorphisms
\[
   \Cst(G,\A) \to \Mult(\Cst(G,\B)),\qquad
   \Cred(G,\A) \to \Mult(\Cred(G,\B)).
\]
Here continuity is understood as follows:

\begin{definition}
  \label{def:multiplier_homomorphism}
  Let \(\A\) and~\(\B\) be Fell bundles.  A \emph{multiplier
    homomorphism} \(\mor\colon \A \to \Mult(\B)\) is a family of
  linear maps \(\mor_g\colon \A_g \to \Mult(\B)_g\) with the
  following properties:
  \begin{itemize}
  \item \(\mor_g(a)^* = \mor_{g^{-1}}(a^*)\) for all \(g\in
    G\), \(a\in \A_g\);
  \item \(\mor_g(a_1)\cdot \mor_h(a_2)=\mor_{gh}(a_1\cdot a_2)\) for all \((g,h)\in G^2\),
    \(a_1 \in \A_g\), \(a_2 \in \A_h\);
  \item the family of bilinear maps
    \(\mu_{g, h}\colon \A_g \times \B_h \to \B_{g h}\),
    \((a,b) \to \mor_g(a)_h(b)\), for \((g,h)\in G^2\) is
    continuous, that is, it maps continuous sections to continuous
    sections; here \(\mor_g(a)_h\colon \B_h \to \B_{g h}\) is the
    component at~\(h\) of
    \(\mor_g(a)\in \Mult(\B)_g\).
  \end{itemize}
\end{definition}

More explicitly, a multiplier homomorphism is a continuous family of
bilinear maps
\[
  \mu_{g,h}\colon \A_g \times \B_h \to \B_{g h}, \qquad
  (a,b) \mapsto a\cdot b \defeq \mor_g(a)_h(b),
\]
that satisfies \(b_1^* (a^*\cdot b_2) = (a\cdot b_1)^* b_2\) for all
\(a\in\A_g\), \(b_2\in \B_h\), \(b_1\in \B_{g h}\),
\((g,h)\in G^2\), and
\(a_1 \cdot (a_2\cdot b) = (a_1\cdot a_2)\cdot b\) for all
\(a_1\in \A_g\), \(a_2\in \A_h\), \(b\in\B_k\), \((g,h,k)\in G^3\).
The first property says both that each \(\mor_g(a)\) is a
multiplier and that \(\mor_g(a)^* = \mor_{g^{-1}}(a^*)\).

\begin{lemma}
  \label{lem:multiplier-fiber-products}
  Write a multiplier homomorphism \(\A\to\Mult(\B)\) via multiplication
  maps \(\A_{g}\times\B_h\to \B_{gh}\), \((a,b)\mapsto a\cdot b\), for
  \((g,h)\in G^2\).
  Then
  \[
    \A_g\cdot\B_h=\A_g\A_g^*\cdot\B_{gh}
    \qquad\text{for all }(g,h)\in G^2.
  \]
  Here products of subspaces mean their closed linear spans.
  If~\(\A\) is saturated and \(\mor_x\colon \FU_x\to \Mult(\FU[B]_x)\) is
  nondegenerate for all \(x\in G^0\), then
  \[
    \A_g\cdot\B_h=\B_{gh}
    \qquad\text{for all }(g,h)\in G^2.
  \]
\end{lemma}

\begin{proof}
  Since \(\A_g\A_g^*\A_g=\A_g\), we compute
  \[
    \A_g\cdot\B_h
    = \A_g\A_g^*\A_g\cdot \B_h
    \subseteq \A_g\A_g^*\cdot\B_{gh}
    \subseteq \A_g\cdot \B_h.
  \]
  So all these spaces are equal, giving the first assertion
  of the lemma.
  To prove the second assertion, assume that~\(\A\) is saturated
  and~\(\mor_x\) is nondegenerate for all \(x\in G^0\).
  Then \(\A_g\A_g^*=\FU_{\rg(g)}\) and therefore
  \[
    \A_g\cdot\B_h
    = \FU_{\rg(g)}\cdot \B_{gh}
    = \FU_{\rg(g)} \FU[B]_{\rg(g)}\B_{gh}
    = \FU[B]_{\rg(g)}\B_{gh}
    =\B_{gh}.\qedhere
  \]
\end{proof}

\begin{proposition}[Integrated form of a nondegenerate multiplier morphism]
  \label{pro:integrated-multiplier-morphism}
  Let \(\mor\colon \A\to\Mult(\B)\) be a nondegenerate multiplier
  homomorphism between Fell bundles over~\(G\), written multiplicatively as
  \((a,b)\mapsto a\cdot b\) for \(a\in \A_g\), \(b\in \B_h\) with
  \(\s(g)=\rg(h)\).  Then there is a unique nondegenerate
  \Star{}homomorphism
  \[
    \Cst(G,\mor)\colon \Cst(G,\A)\to \Mult(\Cst(G,\B))
  \]
  such that for all \(f\in\Contc(G,\A)\), \(\xi\in\Contc(G,\B)\) and \(g\in G\),
  \begin{equation}\label{eq:convolution-map}
    (\Cst(G,\mor)(f)\xi)(g)
    = (f*\xi)(g)
    \defeq \int_G f(h)\cdot\xi(h^{-1}g)\,\diff\alpha^{\rg(g)}(h).
  \end{equation}
\end{proposition}
\begin{proof}
  We construct~\(\Cst(G,\mor)\) as the integrated form of a
  Fell-bundle representation \((\varphi,U)\).
  Let \(\FU\defeq \A|_{G^0}\) and \(\FU[B]\defeq \B|_{G^0}\) be the
  restrictions of \(\A\) and \(\B\) to the unit space, and set
  \[
    \SUF\defeq \Cont_0(G^0,\FU),
    \qquad
    \SUF[B]\defeq \Cont_0(G^0,\FU[B]).
  \]
  Restrict \(\mor\) to the unit fibres and compose with the canonical
  nondegenerate homomorphism \(B\to\Mult(\Cst(G,\B))\).
  This gives a \Star{}homomorphism
  \[
    \mor^0\colon \SUF\to \Mult(\Cst(G,\B)).
  \]
  Since \(\mor\) is nondegenerate on each unit fibre, \(\mor^0\) is
  nondegenerate.
  We view \(\Cst(G,\B)\) as a Hilbert \(\Cst(G,\B)\)\nb-module over itself
  and regard \(\mor^0\) as a nondegenerate representation
  \[
    \varphi\colon \SUF\to \Bound(\Cst(G,\B))=\Mult(\Cst(G,\B)).
  \]
  On the dense subspace \(\Contc(G,\B)\subseteq \Cst(G,\B)\), this
  representation is given by the usual formula
  \begin{equation}
    \label{eq:formula-for-varphi}
    (\varphi(f)\eta)(g) = f(\rg(g))\cdot\eta(g),
    \qquad f\in \SUF,\ \eta\in\Contc(G,\B),\ g\in G,
  \end{equation}
  coming from the fibrewise action on~\(\FU[B]\).

  Next we construct the unitary~\(U\).
  Consider the field \(\A\otimes_{\FU}\B\) over \(G^{(2)}\defeq
  G\times_{\s,\rg}G\) with fibre \(\A_g\otimes_{\FU_{\s(g)}}\B_h\) at
  \((g,h)\in G^{(2)}\), and the field \(\A\A^*\otimes_{\FU}\B\) over
  \(G\times_{\rg,\rg}G\) with fibre
  \(\A_g\A_g^*\otimes_{\FU_{\rg(g)}}\B_k\) at $(g,k)\in
  G\times_{\rg,\rg}G$.

  By Lemma~\ref{lem:multiplier-fiber-products}, for each
  \((g,h)\in G^{(2)}\) we have fibrewise isomorphisms of correspondences
  \[
    \A_g\otimes_{\FU_{\s(g)}}\B_h
    \cong \A_g\cdot\B_h
    = \A_g\A_g^*\cdot\B_{gh}
    \cong \A_g\A_g^*\otimes_{\FU_{\rg(g)}}\B_{gh}.
  \]
  These canonical isomorphisms are the fibres of an isomorphism of
  upper semicontinuous fields of correspondences.
  The latter lifts the canonical homeomorphism
  \[
    \Upsilon\colon G^{(2)}\congto G\times_{\rg,\rg}G,
    \qquad
    \Upsilon(g,h)\defeq (g,gh),
  \]
  to an isomorphism of fields of correspondences
  \begin{equation}
    \label{eq:iso-fields-multipliers}
    \A\otimes_{\FU}\B \cong \A\A^*\otimes_{\FU}\B.
  \end{equation}
  On compactly supported sections this yields a
  \(\Contc(G,\A\A^*)\)-\(\Cont_0(G^0,\FU)\)-\alb{}bimodule isomorphism
  \begin{equation}
    \label{eq:unitary-from-morphism}
    U_0\colon
    \Contc(G^{(2)},\A\otimes_{\FU}\B)
    \congto
    \Contc(G\times_{\rg,\rg}G,\A\A^*\otimes_{\FU}\B),
  \end{equation}
  which, under the above identifications, is explicitly given by
  \[
    (U_0\zeta)(g,k) = \zeta(g,g^{-1}k),
    \qquad
    \zeta\in\Contc(G^{(2)},\A\otimes_{\FU}\B),\ (g,k)\in G\times_{\rg,\rg}G.
  \]

  As in Section~\ref{sec:universal_rep}, we may view
  \(\Contc(G^{(2)},\A\otimes_{\FU}\B)\) as a dense subspace of
  \[
    \Lt^2(G,\A,\s,\tilde\alpha)\otimes_\varphi \Cst(G,\B)
  \]
  and \(\Contc(G\times_{\rg,\rg}G,\A\A^*\otimes_{\FU}\B)\) as a dense
  subspace of
  \[
    \Lt^2(G,\A\A^*,\rg,\alpha)\otimes_\varphi \Cst(G,\B),
  \]
  with the \(\Cst(G,\B)\)\nb-valued inner products computed as in
  Section~\ref{sec:disintegrate}, see
  \eqref{eq:inner-prod-s} and~\eqref{eq:inner-prod-r}.
  The same inner-product computation as in
  Lemma~\ref{lem:isometric-D-inner-products} shows that~\(U_0\) is
  isometric for these inner products, hence extends uniquely by
  continuity to a unitary
  \[
    U\colon
    \Lt^2(G,\A,\s,\tilde\alpha)\otimes_\varphi\Cst(G,\B)
    \congto
    \Lt^2(G,\A\A^*,\rg,\alpha)\otimes_\varphi\Cst(G,\B).
  \]

  The pair \((\varphi,U)\) is a Fell-bundle representation of~\(\A\) on the
  Hilbert module \(\F\defeq \Cst(G,\B)\) by the same verification as in
  Lemma~\ref{lem:disintegration_rep}.
  Let
  \[
    L\colon \Contc(G,\A)\to \Bound(\F)=\Mult(\Cst(G,\B))
  \]
  denote its integrated form as in Section~\ref{sec:integrate}.
  If \(f\in\Contc(G,\A)\) is written as \(f=f_1^\dagger\cdot f_2\) with
  \(f_1\in\Contc(G,\A\A^*)\), \(f_2\in\Contc(G,\A)\) as in
  Lemma~\ref{lem:factor_f}, then by definition
  \[
    L(f) = (T^\rg_{f_1})^* U T^\s_{f_2},
  \]
  where \(T^\s_{f_2}\) and~\(T^\rg_{f_1}\) are the creation operators
  associated to \(f_2\) and~\(f_1\), viewed as sections of the
  corresponding fields.

  We now compute \(L(f)\xi\) for \(\xi\in\Contc(G,\B)\subseteq\F\).
  For \((g,h)\in G^{(2)}\) with \(\s(g)=\rg(h)\), the definition of the creation
  operator~\(T^\s_{f_2}\) gives
  \[
    (T^\s_{f_2}\xi)(g,h)
    = [\,f_2(g)\otimes\xi(h)\,]
    \in \A_g\otimes_{\FU_{\s(g)}}\B_h.
  \]
  Under the canonical identifications above, this may be viewed as an element of
  \(\A_g\A_g^*\otimes_{\FU_{\rg(g)}}\B_{gh}\) and hence of the fibre
  \((\A\A^*\otimes_{\FU}\B)_{(g,gh)}\).
  Therefore, the explicit formula for~\(U_0\) shows that
  \[
    (U\,T^\s_{f_2}\xi)(g,k)
    = f_2(g)\cdot\xi(g^{-1}k)
    \in \A_g\A_g^*\cdot\B_k
  \]
  for all \((g,k)\in G\times_{\rg,\rg}G\).
  
  Next, \(T^\rg_{f_1}\) is the creation operator attached to the field
  \(\A\A^*\otimes_{\FU}\B\) over \(G\times_{\rg,\rg}G\), viewed as a
  correspondence via~\(\rg\).  Its adjoint is given by the usual formula for
  creation operators: if
  \(\zeta\in\Contc(G\times_{\rg,\rg}G,\A\A^*\otimes_{\FU}\B)\), then
  \[
    (T^\rg_{f_1})^*\zeta(g)
    = \int_G f_1(h)^*\cdot\zeta(h,g)\,\diff\alpha^{\rg(g)}(h),
    \qquad g\in G.
  \]
  Applying this to \(\zeta = U\,T^\s_{f_2}\xi\) and remembering that 
  \((U T^\s_{f_2}\xi)(h,g) = f_2(h)\cdot\xi(h^{-1}g)\), we obtain
  \begin{align*}
    (L(f)\xi)(g)
    &= (T^\rg_{f_1})^*U T^\s_{f_2}\xi(g) \\
    &= \int_G f_1(h)^*\cdot
      \bigl(U T^\s_{f_2}\xi\bigr)(h,g)\,\diff\alpha^{\rg(g)}(h) \\
    &= \int_G f_1(h)^*\cdot
      \bigl(f_2(h)\cdot\xi(h^{-1}g)\bigr)\,\diff\alpha^{\rg(g)}(h) \\
    &= \int_G \bigl(f_1(h)^* f_2(h)\bigr)\cdot\xi(h^{-1}g)\,\diff\alpha^{\rg(g)}(h).
  \end{align*}
  The factorisation \(f=f_1^\dagger\cdot f_2\) implies
  \(f(h)=f_1(h)^*f_2(h)\) for all \(h\in G\).
  Hence
  \[
    (L(f)\xi)(g)
    = \int_G f(h)\cdot\xi(h^{-1}g)\,\diff\alpha^{\rg(g)}(h)
    = (f*\xi)(g),
  \]
  which is exactly the convolution formula~\eqref{eq:convolution-map} on the
  dense subspace \(\Contc(G,\B)\subseteq\F\).

  By the integration procedure (Section~\ref{sec:integrate}), the
  representation \((\varphi,U)\) integrates to a unique nondegenerate
  \Star{}homomorphism
  \[
    \Cst(G,\mor)\colon \Cst(G,\A)\to\Bound(\F)=\Mult(\Cst(G,\B)),
  \]
  whose restriction to \(\Contc(G,\A)\) is \(L\).
  Since \(\Contc(G,\B)\) is dense
  in~\(\F\), the above computation shows that \(\Cst(G,\mor)\) is characterised on
  \(\Contc(G,\B)\) by the convolution
  formula~\eqref{eq:convolution-map}.
  This finishes the proof.
\end{proof}

\begin{remark}[Degenerate multiplier morphisms]
  \label{rem:degenerate-multiplier}
  Let \(\mor\colon \A\to\Mult(\B)\) be a possibly degenerate
  multiplier morphism.
  Then the restriction of~\(\mor\) to unit fibres still induces a
  \Star{}homomorphism
  \[
    \mor^0 \colon \SUF=\Cont_0(G^0,\FU)\to \Mult(\Cst(G,\B)),
  \]
  but~\(\mor^0\) need not be nondegenerate.
  In this case, we form the standard \emph{nondegenerate part}
  \[
    \F \defeq \mor^0(\SUF)\,\Cst(G,\B) = \SUF\cdot \Cst(G,\B),
  \]
  a right ideal of \(\Cst(G,\B)\), and restrict~\(\mor^0\) to obtain a
  nondegenerate homomorphism
  \[
    \varphi\colon \SUF\to \Bound(\F).
  \]

  Exactly the same construction of the unitary~\(U\) as in
  Proposition~\ref{pro:integrated-multiplier-morphism} applies to~\(\varphi\):
  the fields \(\A\otimes_{\FU}\B\) and \(\A\A^*\otimes_{\FU}\B\), the
  identification~\eqref{eq:iso-fields-multipliers}, and the fibrewise
  formula \((U_0\zeta)(g,k)=\zeta(g,g^{-1}k)\) make sense independently of
  nondegeneracy.
  Thus we again obtain a Fell-bundle representation \((\varphi,U)\)
  on~\(\F\), whose integrated form
  \[
    L \colon \Contc(G,\A)\to \Bound(\F)
  \]
  is given by the same convolution formula
  \[
    (L(f)\xi)(g)=\int_G f(h)\cdot \xi(h^{-1}g)\,\diff\alpha^{\rg(g)}(h).
  \]

  Using the canonical isomorphism
  \[
    \Lt^2(G,\A,\s,\tilde\alpha)\otimes_\varphi \Cst(G,\B)
    \cong
    \Lt^2(G,\A,\s,\tilde\alpha)\otimes_\varphi \F,
  \]
  the same calculation as in
  Proposition~\ref{pro:integrated-multiplier-morphism}
  shows that
  \[
    \Cst(G,\mor)\colon \Cst(G,\A)\to \Mult(\Cst(G,\B))
  \]
  exists for every (possibly degenerate) multiplier morphism
  \(\mor\colon \A\to \Mult(\B)\) and is still given by the convolution
  formula~\eqref{eq:convolution-map}.

  In Section~\ref{sec:Multiplier-as-split-extensions} we shall give an
  alternative existence proof using split extensions, which applies to
  all multiplier morphisms without appealing to the disintegration
  theory and arrives at the same integrated map \(\Cst(G,\mor)\).
\end{remark}

\subsection{Multiplier homomorphisms as split extensions}
\label{sec:Multiplier-as-split-extensions}

We are going to reinterpret multiplier homomorphisms as split
extensions of Fell bundles.
This will allow us to generalise some
results from ordinary Fell-bundle homomorphisms to multiplier
homomorphisms.
We first recall the analogous results for homomorphisms
\(A\to\Mult(B)\) between two \(\Cst\)\nb-algebras.

\begin{definition}
 Let \(A\) and~\(B\) be \(\Cst\)\nb-algebras.
 A \emph{split extension} of \(A\) by~\(B\) is an extension
  \(B \into E \overset{p}\onto A\) of \(\Cst\)\nb-algebras with a
  \Star{}homomorphism \(\sigma\colon A\to E\) such that \(p\circ
  \sigma = \Id_A\).
  A \emph{split extension} of Fell bundles is defined similarly,
  replacing \(\Cst\)\nb-algebras and \Star{}homomorphisms by Fell
  bundles and Fell-bundle homomorphisms.
\end{definition}

The following result is well known (see~\cite{Pedersen:Extensions}),
but we add the proof here for convenience of the reader and also
because we are going to use a similar idea for Fell bundles.

\begin{lemma}
  \label{lem:multiplier_morphism_as_split_extension}
  Let \(B\) and~\(A\) be \(\Cst\)\nb-algebras.
  There is a natural bijection between equivalence classes of split
  extensions of~\(A\) by~\(B\) and \Star{}homomorphisms \(A\to
  \Mult(B)\).
\end{lemma}

\begin{proof}
  Let a split extension as above be given with splitting section
  \(\sigma\colon A\to E\).
  Then \(\tau\colon A\to\Mult(B)\) defined by
  \(\tau\defeq R\circ \sigma\) is a \Star{}homomorphism, where
  \(R\colon E\to \Mult(B)\), \(R(x)b\defeq x\cdot b\), is the
  canonical (restriction) homomorphism obtained from~\(B\) viewed as
  an ideal of~\(E\).
  Conversely, let \(\tau\colon A\to\Mult(B)\) be a
  \Star{}homomorphism.
  Let
  \[
    E \defeq \setgiven{(a,m) \in A \oplus \Mult(B)}{\tau(a)-m \in B}.
  \]
  This is a \(\Cst\)\nb-subalgebra of \(A \oplus \Mult(B)\).
  The map \(B\to E\), \(b\mapsto (0,b)\), identifies~\(B\) with an
  ideal in~\(E\), and the map \(E\to A\), \((a,b)\mapsto a\),
  identifies the quotient of~\(E\) by that ideal with~\(A\).
  The quotient map splits by the map \(A \to E\), \(a\mapsto
  (a,\tau(a))\).
  Thus we get a split extension.
  The homomorphism \(A\to\Mult(B)\) induced by it is the given
  homomorphism~\(\tau\).
  Conversely, let us start with an extension \(B \into E \onto A\)
  with a section \(\sigma\colon A\to E\).
  Define~\(\tau\) as above, and then turn~\(\tau\) back into a split
  extension \(B\into E_\tau \onto A\).
  Then \(E_\tau \to E\), \((a,b) \mapsto \sigma(a) + b\), is an
  isomorphism \(E_\tau \cong E\) that induces identity maps on \(A\)
  and~\(B\) and is also compatible with the natural sections \(A\to
  E\) and \(A\to E_\tau\).
  Thus our two split extensions are equivalent.
\end{proof}

\begin{remark}
  \label{rem:essential-split-ext}
  Notice that the ideal~\(B\) is essential in the split extension \(B
  \into E \onto A\) if and only if the associated \Star{}homomorphism
  \(\tau\colon A \to \Mult(B)\) satisfies \(\tau^{-1}(B) = 0\), that is,
  if the composition \(\eta\defeq q\circ\tau\colon A\to \Q(B)\defeq
  \Mult(B)/B\) is injective, where \(q\colon \Mult(B)\onto \Q(B)\)
  denotes the quotient map.
\end{remark}

\begin{proposition}
  \label{pro:multiplier_hom_to_split_ext}
  Let \(\B\) and~\(\A\) be two Fell bundles over~\(G\).  There is a
  natural bijection between equivalence classes of split Fell-bundle
  extensions
  \begin{equation}
    \label{eq:split_Fell_bundle_ext}
    \begin{tikzcd}
      \B \ar[r, >->, "\iota"] &
      \E \ar[r, ->>, "\pi"] &
      \A \ar[l, bend left, "\sigma"]
    \end{tikzcd}
  \end{equation}
  and multiplier homomorphisms \(\A\to \Mult(\B)\).
\end{proposition}

\begin{proof}
  First, assume that a split Fell-bundle extension is given.
  Then the continuous family of multiplication maps
  \(\A_g \times \B_h \to \A_{g h}\), \((a,b)\mapsto \sigma(a)\cdot
  b\), with the multiplication in~\(\E\) defines a multiplier
  homomorphism \(\A\to\Mult(\B)\).
  Conversely, let \(\varphi\colon \A\to\Mult(\B)\) be a multiplier
  homomorphism.
  Then we define
  \[
    \E_g \defeq \setgiven{(a,m)\in \A_g \oplus \Mult(\B)_g}
    {\varphi_g(a)-m \in \B_g}.
  \]
  These subspaces are closed in norm and closed under the
  multiplication and involution in \(\A \oplus \Mult(\B)\).
  We give~\(\E\) the topology defined by the Banach space isomorphisms
   \(\E_g \to \A_g \oplus \B_g\), \((a,b)\mapsto (a,\varphi_g(a)-b)\).
   Routine computations show that this makes \(\E \defeq (\E_g)_{g\in
     G}\) a Fell bundle over~\(G\).
   The subspaces \(0 \oplus \B_g\) form a Fell ideal that is naturally
   isomorphic to~\(\B\), and the resulting quotient Fell bundle is
   isomorphic to~\(\A\) by the projection \((a,b)\mapsto a\).
   The resulting Fell-bundle extension \(\B \into \E \onto \A\) splits
   by the Fell-bundle homomorphism \(\A\to\E\), \(a\mapsto
   (a,\varphi_g(a))\).
   Thus a multiplier homomorphism determines a split Fell-bundle
   extension.
   The two constructions above are easily seen to be inverse to each
   other up to a canonical isomorphism of split Fell-bundle extensions.
\end{proof}

Proposition~\ref{pro:multiplier_hom_to_split_ext} implies that a
multiplier homomorphism also produces right multiplication maps that
satisfy all associativity conditions for products with three factors
from \(\A\) or~\(\B\), namely,
\(a\cdot (b_1\cdot b_2) = (a\cdot b_1)\cdot b_2\),
\(b_1\cdot (a\cdot b_2) = (b_1\cdot a)\cdot b_2\),
\(b_1\cdot (b_2\cdot a) = (b_1\cdot b_2)\cdot a\),
\(a_1\cdot (a_2\cdot b) = (a_1\cdot a_2)\cdot b\),
\(a_1\cdot (b\cdot a_2) = (a_1\cdot b)\cdot a_2\),
\(b\cdot (a_1\cdot a_2) = (b\cdot a_1)\cdot a_2\), if
\(a,a_j\in \A\), \(b,b_j\in\B\) are composable.
All this follows by computing in the split extension~\(\E\) defined by
the multiplier homomorphism.
The right multiplication maps \(\B_g \times \A_h \to \B_{g h}\) may be
defined directly by
\[
  b\cdot a\defeq (a^*\cdot b^*)^*
\]
for \(b\in \B_g\), \(a\in \A_h\).

\begin{proposition}
  \label{pro:split_exact}
  Let~\eqref{eq:split_Fell_bundle_ext} be a split Fell-bundle
  extension.  Then there is a commuting diagram of
  \Star{}homomorphisms whose rows are split extensions
  \begin{equation}
    \label{eq:split_extension_functorial}
    \begin{tikzcd}
      \Cst(G,\B) \ar[r, >->, "\iota_*"] \ar[d, ->>] &
      \Cst(G,\E) \ar[r, ->>, "\pi_*"]  \ar[d, ->>]&
      \Cst(G,\A) \ar[l, start anchor=west, end anchor=east,
      shift left, bend left, "\sigma_*"]  \ar[d, ->>]\\
      \Cred(G,\B) \ar[r, >->, "\iota_*"] &
      \Cred(G,\E) \ar[r, ->>, "\pi_*"] &
      \Cred(G,\A). \ar[l, start anchor=west, end anchor=east,
      shift left, bend left, "\sigma_*"]
    \end{tikzcd}
  \end{equation}
\end{proposition}

\begin{proof}
  The diagram~\eqref{eq:split_extension_functorial} exists and
  commutes --~except for \(\sigma_* \circ \pi_* \neq
  \Id_{\Cst(G,\E)}\)~-- because the full and reduced section
  \(\Cst\)\nb-algebras and the canonical quotient map between them are
  functorial.
  By Proposition~\ref{pro:full_Cstar_exact}, the top row is exact and,
  in the bottom row, \(\iota_*\) is injective, \(\pi_*\) is
  surjective, and \(\pi_*\circ \iota_*=0\).
  To prove that the bottom row is exact, it only remains to prove that
  \(\ker \pi_* \subseteq \IM \iota_*\).
  Let \(x\in \ker(\pi)\).
  There is a sequence~\((x_n)_{n\in\N}\) in \(\Contc(G,\E)\) with
  \(\lim x_n = x\).
  Let \(\tilde x_n\defeq x_n - (\sigma\circ\pi)_*(x_n)\).
  Since \(\sigma\pi-\Id_{\A}\) projects onto~\(\I\), this sequence
  belongs to \(\Contc(G,\I) \subseteq \Cred(G,\I)\).
  And \(\lim \tilde x_n = x- (\sigma\circ\pi)_*(x) = x\).
  Thus \(x\in \IM(\iota_*)\) as desired.
\end{proof}

\begin{proposition}
  \label{pro:multiplier_homomorphism}
  Let \(\mor=(\mor_g)\colon \A\to\Mult(\B)\) be a multiplier
  homomorphism between Fell bundles over~\(G\).
  Then there are natural \Star{}homomorphisms \(\Cst(G,\mor)\) and
  \(\Cred(G,\mor)\) that make the following diagram commute:
  \[
    \begin{tikzcd}[column sep=large]
      \Cst(G,\A) \arrow{r}{\Cst(G,\mor)}
      \arrow[d,->>, "\mathrm{canonical}"'] &
      \Mult(\Cst(G,\B))
      \arrow[->]{d}{\mathrm{canonical}}\\
      \Cred(G,\A) \arrow{r}{\Cred(G,\mor)} &
      \Mult(\Cred(G,\B)).
    \end{tikzcd}
  \]
  If \(f\in\Contc(G,\A)\), \(\xi\in\Contc(G,\B)\) and \(g\in G\), then
  \begin{equation}
    \label{eq:multiplier-convolution}
    \bigl(\Cst(G,\mor)(f)\xi\bigr)(g)
    = \bigl(\Cred(G,\mor)(f)\xi\bigr)(g)
    = \int_G f(h)\cdot\xi(h^{-1}g)\,\diff\alpha^{\rg(g)}(h).
  \end{equation}
\end{proposition}

\begin{proof}
  By Proposition~\ref{pro:multiplier_hom_to_split_ext}, \(\mor\)
  determines an extension of Fell bundles \(\B\hookrightarrow
  \E\twoheadrightarrow\A\) with a section \(\sigma\colon \A\to\E\).
  Since the full and reduced section algebras are functorial for
  morphisms of Fell bundles by Proposition~\ref{pro:split_exact}, there
  are \Star{}homomorphisms
  \[
    \Cst(G,\sigma)\colon \Cst(G,\A)\to\Cst(G,\E),
    \qquad
    \Cred(G,\sigma)\colon \Cred(G,\A)\to\Cred(G,\E).
  \]
  Since~\(\B\) embeds as a Fell ideal in~\(\E\), we identify
  \(\Cst(G,\B)\) and \(\Cred(G,\B)\) with ideals of the corresponding
  section algebras of~\(\E\).
  The canonical multiplier actions
  \[
    R\colon \Cst(G,\E)\to\Mult(\Cst(G,\B)),
    \qquad
    R_\red\colon \Cred(G,\E)\to\Mult(\Cred(G,\B)),
  \]
  are defined on compactly supported sections by left convolution:
  \[
    (R(\eta)\,\xi)(g)
    =\int_G \eta(h)\cdot\xi(h^{-1}g)\,\diff\alpha^{\rg(g)}(h),
    \qquad
    \eta\in C_c(G,\E),\ \xi\in C_c(G,\B).
  \]

  We define
  \[
    \Cst(G,\mor)\defeq R\circ\Cst(G,\sigma),
    \qquad
    \Cred(G,\mor)\defeq R_\red\circ\Cred(G,\sigma).
  \]
  These maps make the above square commute because both rows come from
  the same split extension.

  Let \(f\in\Contc(G,\A)\) and \(\xi\in\Contc(G,\B)\).
  Using \(\Cst(G,\sigma)(f)=\sigma\circ f\in C_c(G,\E)\) we compute
  \begin{align*}
    \bigl(\Cst(G,\mor)(f)\xi\bigr)(g)
    &= (R(\sigma\circ f)\xi)(g)\\
    &= \int_G (\sigma(f(h)))\cdot\xi(h^{-1}g)\,\diff\alpha^{\rg(g)}(h)\\
    &= \int_G \mor_h(f(h))\cdot\xi(h^{-1}g)\,\diff\alpha^{\rg(g)}(h)\\
    &= \int_G f(h)\cdot\xi(h^{-1}g)\,\diff\alpha^{\rg(g)}(h),
  \end{align*}
  since multiplication in~\(\E\) restricts to \((a,\mor(a))\cdot
  b=\mor(a)\cdot b\) on the ideal~\(\B\).
  This gives \eqref{eq:multiplier-convolution}.
  The reduced case is identical.
\end{proof}

Remark~\ref{rem:essential-split-ext} makes us wonder how essential
extensions behave under full and reduced crossed products.

\begin{proposition}
  \label{pro:Cred_essential_ideal}
  Let \(\FU[I]\idealin\FU\) be an invariant ideal that is essential.
  Then the induced ideal \(\Cred(G,\I)\) in \(\Cred(G,\A)\)
  is essential as well.
\end{proposition}

\begin{proof}
  Let \(\varrho\colon\SUF[I]\hookrightarrow\Bound(\Hils)\) be a faithful
  nondegenerate representation.
  Since~\(\SUF[I]\) is essential in~\(\SUF\),
  \(\varrho\) extends uniquely to a faithful nondegenerate representation
  \(\bar\varrho\colon \SUF\to\Bound(\Hils)\).

  Inducing \(\varrho\) yields a faithful representation of \(\Cred(G,\I)\) on
  \(\Lt^2(G,\I,\s,\tilde\alpha)\otimes_{\SUF[I]}\Hils\) by
  Remark~\ref{rem:induce_reps}.
  Similarly, inducing~\(\bar\varrho\) yields a faithful representation of
  \(\Cred(G,\A)\) on
  \(\Lt^2(G,\A,\s,\tilde\alpha)\otimes_{\SUF}\Hils\).

  Because \(\Hils=\SUF[I]\cdot\Hils\) and
  $\Lt^2(G,\A,\s,\tilde\alpha)\cdot\SUF[I]
  =\Lt^2(G,\I,\s,\tilde\alpha)$,
  we may identify the two induced Hilbert spaces:
  \[
    \Lt^2(G,\A,\s,\tilde\alpha)\otimes_{\SUF}\Hils
    \cong
    \Lt^2(G,\I,\s,\tilde\alpha)\otimes_{\SUF[I]}\Hils.
  \]
  Thus the faithful representation of \(\Cred(G,\A)\) restricts to a
  faithful nondegenerate representation of \(\Cred(G,\I)\)
  on the same Hilbert space, proving that
  \(\Cred(G,\I)\) is an essential ideal in \(\Cred(G,\A)\).
\end{proof}

\begin{remark}
  The analogue of Proposition~\ref{pro:Cred_essential_ideal} fails for full
  section algebras.
  This already happens for actions of groups on \(\Cst\)\nb-algebras.
  Let~\(G\) be a nonamenable locally compact group and let
  \(A=\Cont(G^+)\), the algebra of continuous functions on the one-point
  compactification \(G^+=G\sqcup\{\infty\}\), with the natural action of~\(G\)
  by translation on~\(G\) and fixing~\(\infty\).
  Let \(I=\Cont_0(G)\idealin A\), which is a \(G\)\nb-invariant essential ideal.
  Then \(I\rtimes_{\max}G = I\rtimes_{\red}G\),
  whereas \(A\rtimes_{\max}G\neq A\rtimes_{\red}G\)
  because~\(G\) is nonamenable.
  If \(I\rtimes_{\max}G\) were essential in \(A\rtimes_{\max}G\), the canonical
  map
  \[
    A\rtimes_{\max}G\to \Mult(I\rtimes_{\max}G)
  \]
  would be injective; but since
  \(I\rtimes_{\max}G=I\rtimes_{\red}G\), this would force
  \(A\rtimes_{\max}G=A\rtimes_{\red}G\), a contradiction.
  Thus full crossed products do not preserve essential ideals.
  This cannot get better for full section algebras of Fell bundles.
\end{remark}

Let us now study when multiplier morphisms induce faithful
\Star{}homomorphisms between Fell section \(\Cst\)\nb-algebras.

\begin{definition}
  A multiplier homomorphism \(\mor\colon \A \to\Mult(\B)\) is
  called \emph{faithful} if \(\mor_g\colon \A_g \to \Mult(\B)_g\)
  is isometric for all \(g\in G\).
\end{definition}

\begin{lemma}
  \label{lem:faithful_multiplier_morphism}
  A multiplier homomorphism \(\mor\colon \A \to\Mult(\B)\) is
  faithful if and only if the \Star{}homomorphisms
  \(\mor_{u(x)}\colon \FU_x \to \Mult(\FU[B]_x)\) are injective
  for all \(x\in G^0\).
\end{lemma}

\begin{proof}
  The isomorphism \(\Mult(\B)_{u(x)} \cong \Mult(\FU[B]_x)\) is
  observed in Remark~\ref{rem:multiplier_saturated}.
  It is trivial that~\(\mor_{u(x)}\) for \(x\in G^0\) must be injective
  if~\(\mor\) is faithful.
  Conversely, assume that~\(\mor_{u(x)}\) is injective for \(x\in
  G^0\).
  Since~\(\mor_{u(x)}\) are \Star{}homomorphisms, this makes them
  isometric.
  Then the \(\Cst\)\nb-condition \(\norm{a^* a} = \norm{a}^2\) for
  \(a\in\Mult(\B)_g\) implies that~\(\mor_g\) is isometric as well.
\end{proof}

\begin{proposition}
  \label{pro:faithful-multiplier-morphism}
  If a multiplier homomorphism from \(\A\) to~\(\B\) is injective,
  then so is the induced \Star{}homomorphism \(\Cred(G,\A) \to
  \Mult(\Cred(G,\B))\).
\end{proposition}

\begin{proof}
  We adapt the proof from the case of ordinary homomorphisms as in
  Proposition~\ref{pro:functorial}.
  The injective multiplier homomorphism from~\(\A\) to~\(\B\)
  restricts to an injective multiplier homomorphism \(\psi\colon
  \SUF\to \Mult(\SUF[B])\).
  Let \(\rho\colon \SUF[B]\to \Bound(\Hils)\) be a faithful
  nondegenerate representation.
  This extends to a faithful representation~\(\bar\rho\) of the
  multiplier algebra of~\(\SUF[B]\) and hence also gives a faithful
  representation \(\pi\defeq \bar\rho\circ \psi\colon \SUF\to
  \Bound(\Hils)\).
  These induce faithful representations \(\Ind(\pi)\colon
  \Cred(G,\A)\to
  \Bound(\Lt^2(G,\A,\s,\tilde\alpha)\otimes_\pi\Hils)\) and
  \(\Ind(\rho)\colon \Cred(G,\B)\to
  \Bound(\Lt^2(G,\B,\s,\tilde\alpha)\otimes_\rho\Hils)\) (see
  Remark~\ref{rem:induce_reps}).
  As in the proof of Proposition~\ref{pro:functorial}, one checks that
  there is a canonical embedding of Hilbert spaces
  \[
    \iota\colon \Lt^2(G,\A,\s,\tilde\alpha)\otimes_\pi\Hils\injto
    \Lt^2(G,\B,\s,\tilde\alpha)\otimes_\rho\Hils,
  \]
  which is determined by
  \(\iota(\xi\otimes\rho(b)v)\defeq \Psi(\xi)\cdot b \otimes v\) for
  all \(\xi\in \Contc(G,\A)\), \(b\in \SUF[B]\) and \(v\in \Hils\).
  Here \(\Psi(\xi)\cdot b\) denotes the section in \(\Contc(G,\B)\)
  given by
  \[
    (\Psi(\xi)\cdot b)(g)\defeq \psi(\xi(g))\cdot b(\s(g)).
  \]
  We note in passing that~\(\Psi\) may also be viewed as a multiplier
  morphism of Hilbert modules:
  \(\Psi\colon \Lt^2(G,\A,\s,\tilde\alpha)\to
  \Mult(\Lt^2(G,\B,\s,\tilde\alpha))\defeq
  \Bound(B,\Lt^2(G,\B,\s,\tilde\alpha))\).
  This is compatible with \(\psi\colon \SUF\to \Mult(\SUF[B])\); in
  particular, \(\Psi\) is faithful, even isometric, because~\(\psi\) is.
  Now we get a canonical \(\Cst\)\nb-embedding
  \[
    \Bound(\Lt^2(G,\A,\s,\tilde\alpha)\otimes_\pi\Hils)\to
    \Bound(\Lt^2(G,\B,\s,\tilde\alpha)\otimes_\rho\Hils),\quad
    T\mapsto \tilde T\defeq \iota\circ T\circ\iota^*.
  \]
  This embedding sends \(\Ind(\pi)\) to \(\Ind(\rho)\)
  composed with the multiplier homomorphism
  \(\Cred(G,\A)\to \Mult(\Cred(G,\B))\).
  It follows that this is faithful as desired.
\end{proof}

The result above does not extend to induced \Star{}homomorphisms on
full section \(\Cst\)\nb-algebras; it already fails for ordinary morphisms
\(\A \to \B\), even when \(\A\) and~\(\B\) are Fell bundles arising
from group actions by automorphisms.
The issue is that a group~\(G\) is amenable if and only
if \(\Cst(G) = \Cred(G)\), while it is possible for a non-amenable
group to act amenably on a (say, unital) \(\Cst\)\nb-algebra~\(A\) (see
\cite{Anantharaman-Delaroche:Systemes}).
In this case, the obvious homomorphism \(\C \to A\), \(\lambda \mapsto
\lambda \cdot 1\), which induces a morphism of Fell bundles from the
trivial action of~\(G\) on~\(\C\) to the action of~\(G\) on~\(A\).
On the level of reduced crossed products, the induced map
\(\Cred(G) \to A \rtimes_\red G\) is injective.
However, since~\(G\) is not amenable, the canonical map
\(\Cst(G) \to A \rtimes G = A \rtimes_\red G\) fails to be
injective.
This shows that the full crossed product functor \(\A \mapsto \Cst(G,
\A)\) does not, in general, preserve injectivity of morphisms of Fell
bundles.

\subsection{Approximate units}
\label{sec:approx-units}

Approximate units for convolution \Star{}algebras associated to
groupoids play a fundamental rôle in questions of nondegeneracy,
representation theory, and functoriality.
We are going to carry the approximate units constructed in
Muhly--Renault--Williams~\cite{Muhly-Renault-Williams:Equivalence}
over to the general Fell bundle case.
For saturated and separable Fell bundles, the existence of such
approximate units was established by Muhly and Williams in
\cite{Muhly-Williams:Equivalence.FellBundles}*{Proposition~5.1}.
Extending this construction to arbitrary Fell bundles is useful, for
instance, to prove that multiplier morphisms are already nondegenerate
at the \(\Contc\)\nb-level or to apply the Cohen--Hewitt Factorisation
Theorem to representations of \(L^1_I(\A)\).
The existence of a two-sided approximate unit in \(\Contc(G)\) is asserted
in \cite{Muhly-Renault-Williams:Equivalence}*{Corollary~2.11}, but
only under the standing assumption that the unit space~\(G^0\) be
paracompact.
A closer inspection of the proof shows, however, that this assumption
is redundant:

\begin{lemma}
  \label{lem:MRW-approx-unit-CcG}
  Let~\(G\) be a Hausdorff, locally compact groupoid with Haar system.
  Then there is an \(I\)\nb-norm bounded, positive, two-sided approximate
  unit~\((e_i)_i\) in~\(\Contc(G)\), with respect to the inductive limit topology,
  such that each~\(e_i\) is of the form
  \[
    e_i = \sum_{k=1}^{n_i} \varphi_{k,i} * \varphi_{k,i}^*
  \]
  for finitely many \(\varphi_{k,i} \in \Contc(G)\).
\end{lemma}

\begin{proof}
  The construction follows Proposition~2.10 and Corollary~2.11
  in~\cite{Muhly-Renault-Williams:Equivalence}, where such approximate
  units are obtained under the additional assumption that~\(G^0\) is
  paracompact.
  We explain how to remove this assumption.

  Recall that the inductive limit topology on \(\Contc(G)\) is the
  locally convex topology obtained as the inductive limit of the spaces
  \(\Contc(K)\) over compact subsets \(K \subseteq G\).
  Thus it suffices to show that for each compact \(K \subseteq G\),
  the convergence \(e_i * f \to f\) is uniform for all \(f \in
  \Contc(G)\) with \(\supp(f) \subseteq K\).

  Fix a compact subset \(K \subseteq G\).
  Let \(U_K \subseteq G\) be a compact neighbourhood of~\(\rg(K)\),
  and let \(U \subseteq G\) be an open neighbourhood of \(G^0\).
  For each \(x \in \rg(K)\) choose an open, relatively compact
  neighbourhood \(V_x \subseteq G\) such that \(V_x V_x^{-1} \subseteq
  U_K \cap U\).
  By compactness, there are finitely many such sets \(V_1,\dotsc,V_n\) with
  \(\rg(K) \subseteq \bigcup_{i=1}^n \s(V_i)\) and \(V_i V_i^{-1}
  \subseteq U_K \cap U\).

  Since~\(\rg(K)\) is compact Hausdorff, there is a partition of unity
  \((b_i)_{i=1}^n\) on~\(\rg(K)\) with \(b_i \in \Contc(G^0)\), \(\supp(b_i) \subseteq \s(V_i)\),
  and \(\sum_i b_i(x)=1\) for all \(x \in \rg(K)\).
  As in~\cite{Muhly-Renault-Williams:Equivalence}, there are functions
  \(\psi_i \in \Contc^+(V_i)\) with
  \[
    \int_{G^x} \psi_i(h) \,\diff\alpha^x(h) = b_i(x)
  \]
  for all \(x \in G^0\), and functions \(\varphi_{i,\delta} \in
  \Contc^+(V_i)\) for \(\delta>0\), such that
  \[
    \abs*{  \psi_i(h)
      - \varphi_{i,\delta}(h)
      \int \varphi_{i,\delta}(g^{-1}h) \,\diff\alpha^{\rg(h)}(g)}
    < \delta
  \]
  for all \(h \in V_i\).

  Set \(e_{\delta,K,U} \defeq \sum_{i=1}^n \varphi_{i,\delta} *
  \varphi_{i,\delta}^*\).
  Then \(e_{\delta,K,U} \in \Contc(G)\) is positive, supported in
  \(\bigcup_i V_i V_i^{-1} \subseteq U \cap U_K\), and satisfies
  \[
    \lim_{\delta \to 0}
    \int e_{\delta,K,U}(g) \,\diff\alpha^x(g) = 1
  \]
  for all \(x \in \rg(K)\).
  Using the estimates in Proposition~2.10 and Lemma~2.12
  of~\cite{Muhly-Renault-Williams:Equivalence}, it follows that
  \(\norm{e_{\delta,K,U} * f - f}_\infty \to 0\) uniformly for all
  \(f \in \Contc(G)\) with \(\supp(f) \subseteq K\).

  We now organise these elements into a net.
  Let~\(\mathcal{I}\) be the directed set of triples
  \((\varepsilon,K,U)\), where \(\varepsilon>0\), \(K \subseteq G\) is
  compact, and \(U \subseteq G\) is an open neighbourhood of \(G^0\),
  ordered by
  \[
    (\varepsilon_1,K_1,U_1) \le (\varepsilon_2,K_2,U_2)
    \iff \varepsilon_2 \le \varepsilon_1,\ 
    K_1 \subseteq K_2,\ 
    U_2 \subseteq U_1.
  \]
  For each \((\varepsilon,K,U)\) choose \(\delta>0\) sufficiently
  small so that \(\norm{e_{\delta,K,U} * f - f}_\infty < \varepsilon\)
  for all \(f \in \Contc(G)\) with \(\supp(f) \subseteq K\).
  Then \((e_{\varepsilon,K,U})\) is a net in \(\Contc(G)\) such that
  \(e_{\varepsilon,K,U} * f \to f\) for all \(f \in \Contc(G)\) in the
  inductive limit topology.
  Since \(e_{\varepsilon,K,U} = e_{\varepsilon,K,U}^*\), we also have
  \(f * e_{\varepsilon,K,U} \to f\), and the \(I\)\nb-norms remain
  uniformly bounded (in fact, converge to~\(1\)) by the estimates
  in~\cite{Muhly-Renault-Williams:Equivalence}.
\end{proof}

\begin{proposition}
  \label{pro:special-au-fell-bundle-groupoid}
  Let~\(G\) be a locally compact Hausdorff groupoid with Haar system
  \(\alpha=(\alpha^x)_{x\in G^0}\) and let~\(\A\) be a Fell bundle
  over~\(G\).
  There is an \(I\)\nb-norm bounded, two-sided approximate unit
  \((e_n)_{n}\) in \(\Contc(G,\A)\) consisting of elements of the form
  \(\sum_i \varphi_{i,n} * \varphi_{i,n}^*\) with finitely many elements
  \(\varphi_{i,n} \in \Contc(G,\A)\); in particular, \(e_n \ge0\) for
  all~\(n\).
\end{proposition}

\begin{proof}
  We build an approximate unit~\(e_{\varepsilon,K,U}\) in
  \(\Contc(G)\) as in the proof of
  Lemma~\ref{lem:MRW-approx-unit-CcG}, using certain subsets \(U \cap
  U_K\), \(V_{i,K,U}\), and auxiliary functions~\(b_k\),
  \(\psi_{\delta,K,U,i}\), and \(\varphi_{\delta,K,U,i}\).
  In addition, let~\((\xi'_j)\) be an approximate unit in the
  \(\Cst\)\nb-algebra \(A=\Cont_0(G^0,\FU)\); in particular, \(0\le \xi_j'\le 1\).
  Since~\(\A\) is an upper semicontinuous field of Banach spaces
  over~\(G\), the restriction map \(\Cont_0(G,\A) \to \Cont_0(G^0,\FU)\)
  is surjective.
  We may lift each~\(\xi'_j\) to some \(\xi_j \in \Cont_0(G,\A)\) with
  \(\norm{\xi_j(g)}\le 1\) for all \(g\in G\).
  Define
  \[
    \varphi_{j,\delta,K,U,i}(g) \defeq
    \xi_j(g) \cdot\varphi_{\delta,K,U,i}(g)
  \]
  for all \(g\in G\); this is now a section of~\(\A\) supported
  in~\(V_i\).
  We define
  \begin{align*}
    e_{j,\delta,K,U}
    &\defeq \sum_i \varphi_{j,\delta,K,U,i} *
    \varphi_{j,\delta,K,U,i}^*,\\
    e_{j,\delta,K,U}(g)
    &= \sum_i \int_{G^{\rg(g)}}   \xi_j(h) \xi_j(g^{-1} h)^*\cdot
    \varphi_{\delta,K,U,i}(h) \varphi_{\delta,K,U,i}(g^{-1} h)
    \,\diff\alpha^{\rg(g)}(h).
  \end{align*}
  This is positive in \(\Contc(G,\A)\) and supported in~\(U\) by
  construction.
  Now let \(F\subseteq \Contc(G,\A)\) be a finite subset.
  Let~\(K_0\) be a compact subset containing the supports of all
  \(f\in F\).
  Since \(A=\Cont_0(G^0,\FU)\) acts nondegenerately on \(\Cont_0(G,\A)\),
  there is~\(j_0\) so that \(\norm{\xi_j'(\rg(g))^2 f(g) -
    f(g)}<\varepsilon/2\) for all \(j\ge j_0\), \(f\in F\), and \(g\in
  G\).
  We fix such a~\(j\).
  Since~\(\xi_j\) is a continuous extension of~\(\xi_j'\) and~\(\xi_j'\) is
  selfadjoint, the section \((g,h,k)\mapsto \xi_j(h) \xi_j(g^{-1} h)^*
  f(g^{-1} k)\) is continuous on the space of \((g,h,k)\in G^3\) with
  \(\rg(g) = \rg(h) = \rg(k)\) and restricts to \(\xi_j'(\rg(k))^2 f(k)\)
  when \(g=h=u(\rg(k))\).
  By combining this with the uniform continuity of the elements \(f \in F\), 
  we deduce that for each fixed~\(j\) there is a neighbourhood~\(U_0\)
  of~\(G^0\) such that \(\norm{\xi_j(h)\xi_j(g^{-1} h)^* f(g^{-1} k) -
    f(g^{-1} k)}<\varepsilon\) whenever \(f\in F\), \(\rg(h)= \rg(g) =
  \rg(k)\), and \(h,g\in U_0\).
  Therefore, if \(\delta>0\), \(U_0 \subseteq U\), \(K_0 \subseteq K\),
  then we may bound \(\norm{(e_{j,\delta,K,U}*f - f)(k)}\) for
  \(f\in F\) by splitting the difference into two terms.
  The first term measures the error introduced by replacing the
  vector-valued term \(\xi_j(h)\xi_j(g^{-1} h)^* f(g^{-1} k)\) by the
  scalar-valued term~\(f(g^{-1} k)\) inside the convolution
  \[
    (e_{j,\delta,K,U}*f)(k)
    \defeq \int_{G^{\rg(k)}}
    e_{j,\delta,K,U}(g) f(g^{-1} k) \,\diff\alpha^{\rg(k)}(g).
  \]
  This replacement yields precisely the convolution of~\(f\) with the scalar 
  approximate unit \(e_{\delta,K,U} \in \Contc(G)\).
  The difference between this scalar convolution and~\(f(k)\) forms
  the second term, which goes  to~\(0\) uniformly on~\(K\) if
  \(\delta\searrow 0\) and \(U\searrow G^0\) by
  Lemma~\ref{lem:MRW-approx-unit-CcG}.
  The first term is bounded by the integral of the error:
  \[
    \sum_i \int_{G^{\rg(k)}} \int_{G^{\rg(k)}}   \varepsilon \,
    \varphi_{\delta,K,U,i}(h) \varphi_{\delta,K,U,i}(g^{-1} h)
    \,\diff\alpha^{\rg(k)}(h)
    \,\diff\alpha^{\rg(k)}(g).
  \]
  Since the double integral of the scalar functions is uniformly bounded 
  (it computes the bounded \(I\)\nb-norm of the scalar approximate unit), 
  this entire error can be made less than a constant multiple of~\(\varepsilon\) 
  if \(\delta>0\) is chosen small enough, independently of \(U\) and~\(K\).
  Thus there is a function~\(e_F\) of the form~\(e_{j,\delta,K,U}\)
  with \(\norm{e_F * f -  f}<\varepsilon\) for all \(f\in F\).
  This net~\((e_F)\) is a left approximate unit by design.
  It also inherits the property that each element is positive and has
  \(I\)\nb-norm at most~\(1\).
\end{proof}

Using the approximate identity built above, we may slightly refine the
statement that the integrated map
\[
  \Cst(G,\mor)\colon \Cst(G,\A)\to \Mult(\Cst(G,\B))
\]
of a nondegenerate multiplier morphism of Fell bundles is
nondegenerate:

\begin{proposition}
  \label{pro:nondegeneracy-Cc-level}
  Let~\(G\) be a locally compact Hausdorff groupoid with Haar
  system~\(\alpha\) and let \(\A,\B\) be Fell bundles over~\(G\).
  Let \(\mor\colon \A\to\Mult(\B)\) be a nondegenerate multiplier
  morphism, and let
  \[
    \Contc(G,\mor)\colon \Contc(G,\A)\to \Mult(\Contc(G,\B))
  \]
  be the \Star{}homomorphism given by the convolution formula:
  \[
    \bigl(\Contc(G,\mor)(f)\xi\bigr)(g)
    =
    \int_{G^{\rg(g)}} f(h)\cdot\xi(h^{-1}g)\,\diff\alpha^{\rg(g)}(h),
    \quad f\in\Contc(G,\A),\ \xi\in\Contc(G,\B).
  \]
  Then there is a positive, two-sided, \(I\)\nb-norm bounded
  approximate unit~\((e_n)\) in \(\Contc(G,\A)\) such that \(\lim
  \Contc(G,\mor)(e_n) \xi = \xi\) and \(\lim \xi \Contc(G,\mor)(e_n) =
  \xi\) in the inductive limit topology on \(\Contc(G,\B)\) for all
  \(\xi\in\Contc(G,\B)\).
\end{proposition}

\begin{proof}
  The proof is the same as for
  Proposition~\ref{pro:special-au-fell-bundle-groupoid}, but now using
  a finite subset \(F \subseteq\Contc(G,\B)\).
\end{proof}

\section{Dual groupoid and quasi-orbit space}
\label{sec:dual_groupoid_quasi-orbit}

One important question about the section \(\Cst\)\nb-algebra
\(\Cst(G,\A)\) for a Fell bundle~\(\A\) over~\(G\) concerns its
ideal structure.
In this section, we prove some basic results about ideals in
\(\Cst(G,\A)\), using the theory developed
in~\cite{Kwasniewski-Meyer:Stone_duality}.
The results here are far from a complete answer.
We define the dual groupoid of a Fell bundle over~\(G\) and the
quasi-orbit space of this dual groupoid.
We prove that there is a continuous map from the primitive ideal
space of \(\Cst(G,\A)\) to the quasi-orbit space as in
\cite{Kwasniewski-Meyer:Stone_duality}*{Theorem~7.17}; the
difference to the results in~\cite{Kwasniewski-Meyer:Stone_duality}
is that we now allow possibly nonsaturated Fell bundles over
locally compact groupoids.
Fell bundles over locally compact groups are already considered
in~\cite{Kwasniewski-Meyer:Stone_duality}*{Section~6.2}.
There is only one missing ingredient in the proof of
\cite{Kwasniewski-Meyer:Stone_duality}*{Theorem~7.17}.
Namely, the proof of \cite{Kwasniewski-Meyer:Stone_duality}*{Proposition~7.16}
uses Renault's Disintegration Theorem.
We explain how our universal property may be used instead to complete
this proof.
Our discussion here is not selfcontained.
We only give basic definitions and explain how the information that is
missing in~\cite{Kwasniewski-Meyer:Stone_duality} follows from our
universal property.

As before, let \((G,\alpha)\) be a locally compact groupoid with
Haar system and let~\(\A\) be a Fell bundle over~\(G\).  Let~\(\FU\)
be the restriction of~\(\A\) to units, which is a field of
\(\Cst\)\nb-algebras over~\(G^0\).  Let
\(\SUF\defeq\Cont_0(G^0,\FU)\) be its \(\Cst\)\nb-algebra of
\(\Cont_0\)\nb-sections.

\subsection{The dual groupoid and the quasi-orbit space}
\label{sec:dual_groupoid}

Let~\(\dual{\SUF}\) denote the set of unitary equivalence classes of
irreducible representations of~\(\SUF\).  We equip~\(\dual{\SUF}\)
with the usual topology, whose open subsets are those of the
form~\(\dual{I}\) for ideals \(I\idealin \SUF\).  We are going to
define a partial action of~\(G\) on~\(\dual{\SUF}\).  This has a
transformation groupoid, which we call the \emph{dual groupoid} of
the Fell bundle.  The quasi-orbit space of this groupoid is called
the \emph{quasi-orbit space} of the Fell bundle.

For many purposes, we may replace~\(\dual{\SUF}\) by the primitive
ideal space of~\(\SUF\).  This gives the same quasi-orbit space.  It
makes a big difference, however, for the isotropy of the dual
groupoid.  And the version with~\(\dual{\SUF}\) gives finer
information about the ideal structure of the crossed product (see
\cites{Kwasniewski-Meyer:Aperiodicity,Kwasniewski-Meyer:Essential}).
This is why we work with~\(\dual{\SUF}\) here.

Since~\(\SUF\) is a \(\Cont_0(G^0)\)\nb-\(\Cst\)-algebra, there is a
canonical continuous map
\[
  p\colon \dual{\SUF} \to G^0,
\]
such that the preimage of \(x\in G^0\) is \(\dual{\FU_x}\).  This is
the anchor map of the partial \(G\)\nb-action on~\(\dual{\SUF}\).

Let \(\A\A^* \idealin \rg^*(\FU)\) and \(\A^*\A \idealin \s^*(\FU)\)
be the fields of ideals with fibres \(\A_g\A_g^*\)
and~\(\A_g^*\A_g\), respectively.  By construction, \(\A\) is a
field of equivalence bimodules between them.  Thus \(\Cont_0(G,\A)\)
is an equivalence bimodule between the \(\Cst\)\nb-algebras
\(B_\rg\defeq \Cont_0(G,\A\A^*)\) and
\(B_\s\defeq \Cont_0(G,\A^*\A)\).
By the Rieffel correspondence, \(\Cont_0(G,\A)\) induces a
homeomorphism between the dual spaces \(\dual{B_\rg}\)
and~\(\dual{B_\s}\).
We interpret this homeomorphism as a partial action of~\(G\)
on~\(\dual{\SUF}\).
Here \(g\in G\) acts by a partial homeomorphism~\(\dual{\A}_g\)
from~\(\dual{\FU_{\s(g)}}\) to~\(\dual{\FU_{\rg(g)}}\).
Namely, if \(\pi\colon \FU_{\s(g)} \to \Bound(\Hils)\) is an
irreducible representation, then left multiplication on \(\A_g
\otimes_\pi \Hils\) is either zero or an irreducible representation
of~\(\FU_{\rg(g)}\).
This representation is nonzero if and only if \([\pi] \in
\dual{\A_g^*\A_g}\), and then it is nondegenerate on \(\A_g\A_g^*\),
that is, it belongs to \(\dual{\A_g\A_g^*}\).
So \([\pi] \mapsto [\A_g \otimes_\pi \Hils]\) is a homeomorphism
\[
  \dual{\A_g}\colon \dual{\A_g^*\A_g}\to\dual{\A_g\A_g^*}.
\]
By assumption, \(u(x)\) for \(x\in G^0\) acts by the identity
Hilbert bimodule.
This induces the identity homeomorphism
on~\(\dual{\FU_x}\).
The existence of the involution in the Fell
bundle implies \(\dual{\A_{g^{-1}}} = \dual{\A_g}{}^{-1}\) for all
\(g\in G\).
The existence of the multiplication maps in the Fell
bundle implies that~\(\dual{\A_{g h}}\)
extends~\(\dual{\A_g} \dual{\A_h}\) for all \((g,h)\in G^2\).
So~\(\dual{\A}\) is a partial action of~\(G\) on~\(\dual{\SUF}\).

This partial action has a transformation groupoid
\(\dual{\SUF}\rtimes G\), which we call the \emph{dual groupoid} of
the Fell bundle~\(\A\) over~\(G\).  By definition, its object space
is
\[
  (\dual{\SUF}\rtimes G)^0 \defeq \dual{\SUF}.
\]
We identify its arrow space with
\[
  \dual{\SUF}\rtimes G \defeq \dual{B_\rg},
\]
the domain of the partial action.
Since~\(B_\rg\) is an ideal in~\(\rg^*\FU\), the dual
space~\(\dual{B_\rg}\) is an open subset of
\[
  \dual{\Cont_0(G,\rg^*\FU)} \cong \dual{\SUF} \times_{p,G^0,\rg} G.
\]
Elements of~\(\dual{B_\rg}\) are pairs \((\pi,g)\) with
\(g\in G\) and \(\pi\in \dual{\A_g\A^*_g} \subseteq \dual{\SUF}\)
(using standard results on \(\Cont_0(G)\)-\(\Cst\)-agebras,
see~\cite{Nilsen:Bundles}).
This links~\(\dual{B_\rg}\) to the transformation groupoid.
Since the range map \(\rg\colon G\to G^0\) is open, it follows that
the coordinate projection \(\dual{\SUF} \times_{\rg,G^0,p} G \to
\dual{\SUF}\) is open.
Hence it restricts to an open map on~\(\dual{B_\rg}\).
By definition, this restriction is the range map of
\(\dual{\SUF}\rtimes G\).
The source map is the composite map
\[
  \dual{B_\rg} \cong \dual{B_\s}
  \subseteq \dual{\Cont_0(G,\s^*\FU)}
  \cong \dual{\SUF}\times_{p,G^0,\s}G
  \xrightarrow{\pr_1}\dual{\SUF}.
\]
The factorisation above shows that it is open as well.  The
multiplication is defined simply by
\[
  (\pi,g)\cdot (\varrho,h) \defeq (\pi,g\cdot h)
  \qquad \text{if }\varrho = \s(\pi,g).
\]
This is clearly continuous and associative.  The inverse of
\((\pi,g)\) is \((g^{-1}\cdot\pi, g^{-1})\).  This depends continuously
on~\(g\) because of the homeomorphism
\(\dual{B_\rg} \cong \dual{B_\s}\).  Thus \(\dual{\SUF}\rtimes G\)
becomes a topological groupoid with open range and source maps.  Of
course, its object space is usually not Hausdorff, being the
spectrum of a \(\Cst\)\nb-algebra.

Let \(\pi_1,\pi_2\in \dual{\SUF}\).  We write \(\pi_1\sim\pi_2\) if
the orbits of them for the groupoid \(\dual{\SUF}\rtimes G\) have the
same closure.  Equivalently, there are nets \((g_n)\) and~\((h_n)\)
in~\(G\) such that \(\s(g_n) = p(\pi_1)\), \(\s(h_n) = p(\pi_2)\),
and \(\lim g_n\cdot \pi_1 = \pi_2\) and
\(\lim h_n\cdot \pi_2 = \pi_1\).  The relation~\(\sim\) is an
equivalence relation on~\(\dual{\SUF}\).  Let \(\dual{\SUF}/{\sim}\)
be its set of equivalence classes, equipped with the quotient
topology.  This quotient is called the \emph{quasi-orbit space} of
the Fell bundle~\(\A\) over~\(G\).

The following lemma relates the dual groupoid to invariant ideals
and Fell ideals:

\begin{lemma}
  \label{lem:invariant_ideal_in_dual}
  An ideal \(I\idealin \SUF\) is invariant if and only if the
  corresponding open subset~\(\dual{I}\) of~\(\dual{\SUF}\) is
  invariant for the dual groupoid \(\dual{\SUF}\rtimes G\).
\end{lemma}

\begin{proof}
  By Remark~\ref{rem:continuous_family_of_ideals}, an ideal
  \(I\idealin \SUF\) corresponds to a continuous family of ideals
  \(\FU[I]_x \subseteq \FU_x\).
  By definition, \(I\) is invariant if and only if
  \(\FU[I]_{\rg(g)}\cdot \A_g = \A_g\cdot \FU[I]_{\s(g)}\) for all
  \(g\in G\).
  If this is the case, then~\(I\) also corresponds to a Fell
  ideal~\(\I\) in~\(\A\) by Definition~\ref{def:invariant_ideal}.
  Let \(g\in G\).
  Assume first that~\(I\) is invariant.
  Let~\(\pi\) be an irreducible representation of \(\FU[I]_{\s(g)}
  \cap \A_g^*\A_g\).
  Since \(\FU[I]_{\rg(g)} \A_g \FU[I]_{\s(g)} =  \A_g
  \FU[I]_{\s(g)}\), the left multiplication action on \(\A_g\cdot
  \FU[I]_{\s(g)} \otimes_{\A_{\s(g)}} \pi\) is nondegenerate
  on~\(\FU[I]_{\rg(g)}\).
  Thus the dual action of~\(g\) maps~\(\pi\) to a point in
  \(\dual{\FU[I]_{\rg(g)}}\).
  If this holds for all \(g\in G\), then \(\dual{I}\subseteq
  \dual{\SUF}\) is invariant for the dual action of~\(G\) or,
  equivalently, the dual groupoid.
  Conversely, assume that \(\FU[I]_{\rg(g)}\cdot \A_g \neq \A_g\cdot
  \FU[I]_{\s(g)}\) for some \(g\in G\).
  Replacing~\(g\) by~\(g^{-1}\) if necessary, we may arrange that
  \(\A_g\cdot \FU[I]_{\s(g)}\) is not contained in \(\I_g \defeq
  \FU[I]_{\rg(g)}\cdot \A_g\).
  Then~\(\FU[I]_{\s(g)}\) is not contained in \(\braket{\I_g}{\I_g}\).
  Therefore, there is an irreducible representation~\(\pi\) of the
  ideal~\(\FU[I]_{\s(g)}\) in~\(\FU_{\s(g)}\) that vanishes on
  \(\braket{\I_g}{\I_g}\).
  It follows that the left multiplication representation on \(\A_g
  \otimes_{\FU_{\s(g)}} \pi\) vanishes on~\(\FU[I]_{\rg(g)}\).
  In other words, the dual action of~\(g\) maps \([\pi] \in
  \dual{\FU[I]_{\s(g)}}\) to a point outside
  \(\dual{\FU[I]_{\rg(g)}}\).
  Equivalently, \(\dual{I}\) is not invariant for the dual groupoid.
\end{proof}

The map \(I\mapsto \dual{I}\) identifies the lattice of ideals
\(\Ideals(\SUF)\) with the lattice of open subsets \(\Open(\dual{\SUF})\).
By Lemma~\ref{lem:invariant_ideal_in_dual}, this lattice isomorphism
restricts to a bijection between the lattices of \(\A\)\nb-invariant
ideals in~\(\SUF\) and \(G\)\nb-invariant, open subsets
of~\(\dual{\SUF}\).  The latter is, in turn, isomorphic to the
lattice of open subsets of the quasi-orbit
space~\(\dual{\SUF}/{\sim}\).

The quasi-orbit space does not change
when we replace \(\dual{\SUF}\) by \(\operatorname{Prim}(\SUF)\)
everywhere above.  This is because the lattices of open subsets of
\(\dual{\SUF}\) and \(\operatorname{Prim}(\SUF)\) are both isomorphic
to the lattice of ideals in~\(\SUF\).

\subsection{Induction and restriction of ideals}
\label{sec:ind_res_ideals}

Let \(A\) and~\(B\) be \(\Cst\)\nb-algebras.
We are mainly interested in the case \(\SUF = \Cont_0(G^0,\FU)\)
and \(\SUF[B] = \Cst(G,\A)\).
It is instructive, however, to work in greater generality for a while.

A homomorphism \(\iota\colon A\to \Mult(B)\) allows to induce ideals
from~\(A\) to~\(B\) and to restrict ideals from~\(B\) to~\(A\)
(see~\cite{Kwasniewski-Meyer:Stone_duality}).  Namely, the induced
ideal \(i(I)\) for an ideal \(I\idealin A\) is the closed linear
span \(B\cdot I\cdot B\) of \(b_1\cdot x\cdot b_2\) for
\(b_1,b_2\in B\), \(x\in I\).  And the restricted ideal \(r (J)\)
for an ideal \(J\idealin B\) is the kernel of the composite
homomorphism \(A \to \Mult(B) \to \Mult(B/J)\).  Equivalently,
\(a\in A\) belongs to~\(r(J)\) if and only if
\(a\cdot B\subseteq J\), if and only if \(B\cdot a\subseteq J\).
If \(\iota(A) \subseteq B\), then
\(r(J) = J\cap A\).  The two maps \(i\) and~\(r\) between the
lattices of ideals in \(A\) and~\(B\) form a Galois connection, that
is, \(I\subseteq r(J)\) if and only if \(i(I)\subseteq J\) (see
\cite{Kwasniewski-Meyer:Stone_duality}*{Lemma~2.8}).  Many
consequences of this property are listed in
\cite{Kwasniewski-Meyer:Stone_duality}*{Proposition~2.9}.

Let \(\Ideals(A)\) and \(\Ideals(B)\) denote the lattices of
ideals in \(A\) and~\(B\), respectively.  Let
\begin{align*}
  \Ideals^A(B)
  &\defeq \setgiven{i(I)}{I\in\Ideals(A)} \subseteq \Ideals(B),\\
  \Ideals^B(A)
  &\defeq \setgiven{r(J)}{J\in\Ideals(B)} \subseteq \Ideals(A)
\end{align*}
be the sublattices of induced ideals in~\(B\) and of restricted
ideals in~\(A\), respectively.  Then \(i\circ r\) is a retraction
from~\(\Ideals(B)\) onto~\(\Ideals^A(B)\) and \(r\circ i\) is a
retraction from~\(\Ideals(A)\) onto~\(\Ideals^{B}(A)\).  And
\(i\) and~\(r\) restrict to lattice isomorphisms between
\(\Ideals^A(B)\) and \(\Ideals^B(A)\) that are inverse to each
other.  Our results about the ideal structure of~\(B\) are really
results about these isomorphic lattices
\(\Ideals^A(B)\cong\Ideals^B(A)\).  That is, we only describe
the induced ideals in~\(B\), not all ideals.

\begin{proposition}
  \label{pro:Ker-rep-invariant}
  Let \((\varphi,U)\) be a representation of the Fell bundle
  \(\A=(\A_g)_{g\in G}\).  Then \(I=\ker(\varphi)\) is an invariant
  ideal of \(\SUF\).
\end{proposition}

\begin{proof}
  The ideal~\(I\) corresponds to a continuous family of ideals
  \(\FU[I]_x \idealin \FU_x\) by
  Remark~\ref{rem:continuous_family_of_ideals}.
  We must prove that \(\FU[I]_{\rg(g)}\cdot \A_g=\A_g\cdot
  \FU[I]_{\s(g)}\) for all \(g\in G\).
  By definition, \(U\) is a unitary operator
  \[
    U\colon \Lt^2(G,\A,\s,\tilde\alpha)\otimes_\varphi\F
    \congto \Lt^2(G,\A\A^*,\rg,\alpha)\otimes_\varphi\F.
  \]
  The mere existence of an intertwiner~\(U\) of representations of
  \(\Cont_0(G,\A\A^*)\) implies that the kernels of the left actions of
  \(\Cont_0(G,\A\A^*)\) on the domain and codomain of~\(U\) are the
  same.
  Let~\(J\) be this common kernel.
  This ideal in \(\Cont_0(G,\A\A^*)\) corresponds to a continuous
  family of ideals~\(J_g\) in \(\A_g\A_g^*\idealin \FU_{\rg(g)}\) by
  Remark~\ref{rem:continuous_family_of_ideals}.

  We claim that~\(J\) is at the same time the largest ideal in
  \(\Cont_0(G,\A\A^*)\) with
  \(J\cdot \Lt^2(G,\A,\s,\tilde\alpha)\subseteq
  \Lt^2(G,\A,\s,\tilde\alpha)\cdot I\) and the largest ideal in
  \(\Cont_0(G,\A\A^*)\) with
  \(J\cdot \Lt^2(G,\A\A^*,\rg,\alpha)\subseteq
  \Lt^2(G,\A\A^*,\rg,\alpha)\cdot I\).  Both claims are proven in the
  same way, letting \(\E\defeq\Lt^2(G,\A,\s,\tilde\alpha)\) or
  \(\E\defeq\Lt^2(G,\A\A^*,\rg,\alpha)\) for the two cases.  An
  element~\(a\) of \(\Cont_0(G,\A\A^*)\) belongs to~\(J\) if and only
  if it acts by zero on the correspondence \(\E\otimes_\varphi \F\).
  Equivalently, it annihilates all elementary tensors
  \(\xi\otimes \eta\) with \(\xi\in\E\), \(\eta\in \F\).  This is
  equivalent to
  \[
    \braket{a\cdot\xi_1\otimes\eta_1}{\xi_2\otimes \eta_2}
    =
    \braket{\eta_1}{\varphi(\braket{a\cdot\xi_1}{\xi_2})\eta_2}
    = 0
  \]
  for all \(\xi_1,\xi_2\in \E\) and \(\eta_1,\eta_2\in \F\).  This is
  equivalent to \(\braket{a\cdot \E}{\E}\subseteq \ker(\varphi) = I\).
  And this is equivalent to \(a\cdot \E\subseteq \E\cdot I\).

  We claim that
  \(a\cdot \Lt^2(G,\A,\s,\tilde\alpha) \subseteq
  \Lt^2(G,\A,\s,\tilde\alpha)\cdot I\) holds if and only if
  \(a(g)\cdot \A_g \subseteq \A_g\FU[I]_{\s(g)}\) for all \(g\in G\)
  and that
  \(a\cdot \Lt^2(G,\A\A^*,\rg,\alpha) \subseteq
  \Lt^2(G,\A\A^*,\rg,\alpha)\cdot I\) holds if and only if
  \(a(g)\cdot \A_g\A_g^* \subseteq \A_g\A_g^* \FU[I]_{\rg(g)}\) for all
  \(g\in G\).  The proofs of both claims are so similar that one will
  suffice.  The canonical inclusion
  \(\Contc(G,\A)\cdot I \to \Contc(G,\A\cdot \s^*I)\) induces an isometry
  \(\Lt^2(G,\A,\s,\tilde\alpha)\cdot I \hookrightarrow \Lt^2(G,\A\cdot
  \s^* I,\s,\tilde\alpha)\).  It has dense range because the right hand
  side is a Hilbert module over~\(I\), making it nondegenerate as a
  right \(I\)\nb-module.  So
  \[
    \Lt^2(G,\A,\s,\tilde\alpha)\cdot I
    \cong \Lt^2(G,\A\cdot \s^* I,\s,\tilde\alpha).
  \]
  Then \(a(g)\cdot \A_g \subseteq \A_g\FU[I]_{\s(g)}\) for all
  \(g\in G\) implies
  \(a\cdot \Lt^2(G,\A,\s,\tilde\alpha) \subseteq
  \Lt^2(G,\A,\s,\tilde\alpha)\cdot I\).  It remains to prove the
  converse.  We assume that there is \(g\in G\) with
  \(a(g)\cdot \A_g \not\subseteq \A_g\FU[I]_{\s(g)}\).  So there is
  \(\eta(g)\in \A_g\) with
  \(a(g) \eta(g) \notin \A_g \FU[I]_{\s(g)}\).  This implies
  \(\eta(g)^* a(g)^* a(g)\eta(g) \notin \FU[I]_{\s(g)}\).  As our
  notation suggests, we embed \(\eta(g)\) in a
  \(\Contc\)\nb-section~\(\eta\) of~\(\A\).  So
  \(\eta \in \Lt^2(G,\A,\s,\tilde\alpha)\).  We are going to prove by
  contradiction that
  \(a\cdot \eta \notin\Lt^2(G,\A,\s,\tilde\alpha)\cdot I\), thereby
  finishing the proof of the claim.  Since
  \(\Contc(G,\A\cdot \s^* I)\) is dense in
  \(\Lt^2(G,\A\cdot \s^* I,\s,\tilde\alpha)\), we may assume that
  there is a sequence \((\xi_n)_{n\in\N}\) in
  \(\Contc(G,\A\cdot \s^* I)\) that converges to \(a\cdot \eta\) in
  \(\Lt^2(G,\A,\s,\tilde\alpha)\).  As a consequence, the following
  sequence converges to~\(0\) in norm:
  \begin{align*}
    &\phantom{{}={}}
      \int_{G_{\s(g)}} (a(h)\eta(h) - \xi_n(h))^*(a(h)\eta(h) - \xi_n(h))
    \,\diff \tilde\alpha_{\s(g)}(h)
    \\ &= \int_{G_{\s(g)}} \eta(h)^*a(h)^*a(h) \eta(h)
    \\ &\qquad\qquad- \xi_n(h)^*a(h)\eta(h) - \eta(h)^* a(h)^* \xi_n(h)
    + \xi_n(h)^* \xi_n(h) \,\diff \tilde\alpha_{\s(g)}(h).
  \end{align*}
  The first term on the right hand side has values in
  \(\A_h^* \A_h \subseteq \FU_{\s(g)}\), whereas the last three
  terms have values in~\(\FU[I]_{\s(g)}\) for all \(h\in G_{\s(g)}\).
  These summands are continuous sections of constant fields along the
  fibre~\(G_{\s(g)}\).
  By assumption, \(\varepsilon \defeq \norm{\eta(g)^*a(g)^*
    a(g)\eta(g)}_{\FU_{\s(g)}/\FU[I]_{\s(g)}}\) is positive.
  Since the integrand is a continuous section of a constant field,
  there is a neighbourhood of~\(g\) in~\(G_{\s(g)}\) where
  \(\eta(h)^*a(h)^* a(h)\eta(h)\) is \(\varepsilon/3\)-close to
  \(\eta(g)^*a(g)^* a(g)\eta(g)\).
  Since the measure family~\(\tilde\alpha_{\s(g)}\) has full support
  by the assumptions of a Haar system, this implies a nonzero lower
  bound on the integral above, even in the quotient
  \(\FU_{\s(g)}/\FU[I]_{\s(g)}\).
  This is a contradiction.
  This finishes the proof of the claim.

  Lemma~\ref{lem:equality_of_products} implies
  \(\A_g \A_g^* \A_g = \A_g\).  Therefore, the implications
  \begin{align*}
    a(g)\cdot \A_g \subseteq \A_g\FU[I]_{\s(g)}
    &\Rightarrow
    a(g)\cdot \A_g \A_g^* \subseteq \A_g\FU[I]_{\s(g)} \A_g^*
    \\&\Rightarrow
    a(g)\cdot \A_g \A_g^* \A_g
    \subseteq \A_g\FU[I]_{\s(g)} \A_g^* \A_g
    \subseteq \A_g \A_g^* \A_g \FU[I]_{\s(g)}
  \end{align*}
  show that \(a(g)\cdot \A_g \subseteq \A_g\FU[I]_{\s(g)}\) is
  equivalent to \(a(g)\cdot \A_g \A_g^* \subseteq \A_g\FU[I]_{\s(g)}
  \A_g^*\); here we also used that the two ideals \(\A_g^* \A_g\)
  and~\(\FU[I]_{\s(g)}\) commute with each other.
  Now our claims above show that \(a(g)\cdot \A_g \A_g^* \subseteq
  \A_g\FU[I]_{\s(g)} \A_g^*\) for all \(g\in G\) holds if and only if
  \(a(g)\cdot \A_g\A_g^* \subseteq \A_g\A_g^*\FU[I]_{\rg(g)}\) for all
  \(g\in G\); here \(a\in \Cont_0(G,\A \A^*)\) is fixed.
  The two right hand sides here are continuous families of ideals
  contained in~\(\A \A^*\).
  Therefore, any element in their fibres at~\(g\) is the value of a
  continuous section~\(a\).
  Therefore, the equivalence above shows that
  \(\A_g\A_g^*\FU[I]_{\rg(g)} = \A_g\FU[I]_{\s(g)}\A_g^*\) for all
  \(g\in G\).
  We may rewrite \(\A_g\A_g^*\FU[I]_{\rg(g)} =
  \FU[I]_{\rg(g)}\A_g\A_g^*\).
  An argument as above shows that \(\FU[I]_{\rg(g)}\A_g\A_g^* =
  \A_g\FU[I]_{\s(g)}\A_g^*\) is equivalent to
  \(\FU[I]_{\rg(g)}\A_g\A_g^*\A_g = \A_g\FU[I]_{\s(g)}\A_g^*\A_g\),  
  which simplifies to \(\FU[I]_{\rg(g)}\A_g = \A_g\FU[I]_{\s(g)}\).
  So~\(I\) is invariant.
\end{proof}

\begin{remark}
  The proof of the lemma above did not use all conditions on the
  representation \((\varphi,U)\).  We only needed that there is an
  isomorphism of correspondences
  \[
    \Lt^2(G,\A,\s,\tilde\alpha)\otimes_\varphi\F \congto
    \Lt^2(G,\A\A^*,\rg,\alpha)\otimes_\varphi\F.
  \]
\end{remark}

\begin{theorem}
  \label{the:induced_ideals_vs_quasi-orbit_space}
  Let~\(\A\) be a Fell bundle over a locally compact Hausdorff
  groupoid~\(G\) with Haar system~\(\alpha\).  Let
  \(\SUF \defeq \Cont_0(G^0,\FU)\) and \(B\defeq \Cst(G,\A)\).  An
  ideal in~\(\SUF\) is restricted from~\(B\) if and only if it is
  invariant.  An ideal in~\(B\) is induced from~\(\SUF\) if and only
  if it is of the form \(\Cst(G,\I)\) for a Fell ideal~\(\I\)
  in~\(\A\).  There is a continuous map
  \(\varrho\colon \dual{B} \to \dual{\SUF}/{\sim}\) so that the
  induced map on open subsets,
  \[
    \varrho^{-1}\colon \Open(\dual{\SUF}/{\sim}) \to \Open(\dual{B})
    \cong \Ideals(B),
  \]
  is a bijection onto the subset \(\Ideals^{\SUF}(B)\) of induced ideals.
\end{theorem}

\begin{proof}
  The machinery in~\cite{Kwasniewski-Meyer:Stone_duality} already
  shows most of this theorem.  Only one key result is missing,
  namely, that an ideal in~\(\SUF\) is restricted if and only if it
  is invariant.  As we will show below,
  Proposition~\ref{pro:Ker-rep-invariant} implies that restricted ideals
  are invariant.  And Proposition~\ref{pro:coefficients_embed}
  implies as in~\cite{Kwasniewski-Meyer:Stone_duality} that
  invariant ideals are restricted.

  First, we show that the ideal \(r(J)\) in~\(\SUF\) for an
  ideal~\(J\) in~\(B\) is invariant.  We interpret the quotient map
  \(\pi\colon B \to B/J\) as a representation of~\(B\) on the
  Hilbert module~\(B/J\) over~\(B/J\).  By the universal property
  of~\(B\), this is associated to a representation \((\varphi,U)\)
  of the Fell bundle~\(\A\).  The construction of the universal
  representation in Section~\ref{sec:universal_rep} shows that
  \(\varphi = \pi\circ\iota\) for the canonical inclusion map
  \(\iota\colon \SUF\to \Mult(B)\).  So \(r(J) = \ker \varphi\).
  This ideal in~\(\SUF\) is invariant by
  Proposition~\ref{pro:Ker-rep-invariant}.

  Second, we show that any invariant ideal~\(I\) is restricted
  from~\(B\).  By Corollary~\ref{cor:Fell_ideal_family_ideals}, \(I\)
  corresponds to a Fell ideal~\(\I\) in~\(\A\).  There is a quotient
  Fell bundle~\(\A/\I\).
  Its restriction to units is \(\FU/I\).
  So the map \(A/I \to \Mult(\Cst(G,\A/\I))\) is faithful by
  Proposition~\ref{pro:coefficients_embed}, applied to the Fell
  bundle~\(\A/\I\).
  Let~\(J\) be the kernel of the canonical quotient map \(\Cst(G,\A)
  \to \Cst(G,\A/\I)\).
  Then \(r(J) = I\) (compare the proof of
  \cite{Kwasniewski-Meyer:Stone_duality}*{Lemma~7.14}).
\end{proof}

\section{Representations on separable Hilbert spaces}
\label{sec:Renault}

Now we specialise to the case where \(D=\C\), so that~\(\F\) is a
Hilbert space.  As before, let \((G,\alpha)\) be a locally compact
groupoid with Haar system and let~\(\A\) be a Fell bundle
over~\(G\).  Let~\(\FU \defeq \A|_{G^0}\) be the restriction of~\(\A\) to units,
which is a field of \(\Cst\)\nb-algebras over~\(G^0\).  We are going
to define measurable Fell-bundle representations on measurable
fields of Hilbert spaces over~\(G^0\) and turn these into Fell-bundle
representations on Hilbert spaces as in
Definition~\ref{def:representation_Fell}.
Under separability assumptions, we are going to prove that,
conversely, any Fell-bundle representation comes from a measurable
Fell-bundle representation in this way, which is unique up to unitary
equivalence almost everywhere.
When combined with the universal property of \(\Cst(G,\A)\), this
gives Renault's Disintegration Theorem.

The equivalence between Fell-bundle representations and measurable
Fell-bundle representations as defined below relies mainly on the
well-known representation theory of commutative
\(\Cst\)\nb-algebras.  In addition, we use ideas of Ramsay
from~\cites{Ramsay:Virtual_groups, Ramsay:Topologies} to change the
exceptional null sets that occur in the construction.

Let~\(\nu\) be a measure on~\(G^0\) and let
\(\Hils = (\Hils_x)_{x\in G^0}\) be a \(\nu\)\nb-measurable field of
Hilbert spaces over~\(G^0\) (we use the definition in
\cite{Dixmier:Von_Neumann}*{Definition~1 on p.~164}, which is repeated
in \cite{Dixmier:Cstar-algebras}*{A69}).
In particular, all Hilbert spaces~\(\Hils_x\) are assumed separable.
This will be a standard assumption throughout this section.
We assume that \(\Hils_x \neq 0\) for \(\nu\)\nb-almost all \(x\in
G^0\).
If this fails, then we restrict~\(\nu\) to the set of points \(x\in
G^0\) with \(\Hils_x\neq0\).
This leaves the Hilbert space \(\Lt^2(G^0,\Hils,\nu)\) unchanged.

\begin{definition}
  \label{def:Fell_bundle_pre-rep}
  A \emph{measurable Fell-bundle pre-representation} of~\(\A\)
  on~\(\Hils\) consists of linear maps
  \(S_g\colon \A_g \to \Bound(\Hils_{\s(g)}, \Hils_{\rg(g)})\) for all
  \(g\in G\) with the following properties:
  \begin{enumerate}
  \item \label{def:Fell_bundle_Pre-rep_1}%
    if \((g,h)\in G^2\), \(a_1\in \A_g\), \(a_2\in \A_h\), then
    \(S_g(a_1) S_h(a_2) = S_{g h}(a_1\cdot a_2)\);
  \item \label{def:Fell_bundle_Pre-rep_2}%
    if \(g\in G\), \(a\in \A_g\), then
    \(S_g(a)^* = S_{g^{-1}}(a^*)\);
  \item \label{def:Fell_bundle_Pre-rep_3}%
    if \(x\in G^0\), the linear span of \(S_{u(x)}(a)\xi\) for
    \(a\in \FU_x\), \(\xi\in \Hils_x\) is dense in~\(\Hils_x\);
  \item \label{def:Fell_bundle_Pre-rep_4}%
    for each continuous section \(f\in \Contc(G,\A)\) and measurable
    section~\(\xi\) of~\(\s^*(\Hils)\), the map
    \(g\mapsto S_g(f(g))\xi(g)\) is a measurable section
    of~\(\rg^*(\Hils)\).
  \end{enumerate}
\end{definition}

Conditions
\ref{def:Fell_bundle_Pre-rep_1}--\ref{def:Fell_bundle_Pre-rep_3} in
Definition~\ref{def:Fell_bundle_pre-rep} imply
that~\(S_{u(x)}\) -- which we also denote by \(S_x\) -- is a nondegenerate
\Star{}representation of~\(\FU_x\) on~\(\Hils_x\) for all \(x\in G^0\).

In order to define Fell-bundle representations, we need a condition
on the measure~\(\nu\) that replaces quasi-invariance.  Let
\(\nu\circ\alpha\) denote the measure on~\(G\) whose integration map
is the composite of the integration maps for \(\nu\) and~\(\alpha\).
That is,
\[
  \int_G f(g) \,\diff(\nu\circ\alpha)(g) = \int_{G^0} \int_{G^x} f(g)
  \,\diff \alpha^x(g)\,\diff \nu(x)
\]
for all bounded Borel functions \(f\colon G\to\C\) with compact support.
Define \(\nu\circ\tilde\alpha\) similarly.

\begin{definition}
  \label{def:A-quasi-invariant}
  A measure~\(\mu\) on~\(G\) is called \emph{symmetric} if it is
  equivalent to \(\inv_*\mu\), its pushforward along the inversion
  map.  A measure~\(\nu\) on~\(G^0\) is called
  \emph{quasi-invariant} if \(\nu\circ\alpha\) is symmetric.
\end{definition}

The definition of~\(\tilde\alpha\) implies
\(\inv_*(\nu\circ\alpha) = \nu\circ\tilde\alpha\).  Thus~\(\nu\) is
quasi-invariant if and only if the measures \(\nu\circ\alpha\)
and~\(\nu\circ\tilde\alpha\) on~\(G\) are equivalent.

The disintegration of representations of groupoids without Fell
bundles involves only quasi-invariant measures
(see~\cite{Renault:Representations}).  This remains true for saturated
Fell bundles.  But it becomes wrong for Fell bundles that are not
saturated:

\begin{example}
  Let~\(G\) be a locally compact Hausdorff étale groupoid and
  let~\(\FU\) be an upper semicontinuous field of \(\Cst\)\nb-algebras
  over~\(G^0\).  Since~\(G\) is étale, there is an upper
  semicontinuous field of Banach spaces~\(\A\) over~\(G\) with
  \(u^*\A = \FU\) and \(\A|_{G\setminus u(G^0)} = 0\).  This becomes a
  Fell bundle over~\(G\) with the product and involution from~\(\FU\).
  A representation of this Fell bundle is the same as a representation
  of~\(\Cont_0(G^0,\FU)\).  When we disintegrate such a
  representation, then \emph{any} measure~\(\nu\) on~\(G^0\) may
  occur.  It need not be \(G\)\nb-quasi-invariant because the Fell
  bundle~\(\A\) only sees~\(G^0\) and not~\(G\).
\end{example}

This example suggests to consider the quasi-invariance of the measure
\(\nu\circ \alpha\) on an appropriate subset of~\(G\).
The following lemma helps to make this precise:

\begin{lemma}
  \label{lem:measurable-rep-U}
  Consider a measurable Fell-bundle representation and let \(g\in G\).
  There is a unitary operator
  \[
    U_g\colon \A_g \otimes_{\FU_{\s(g)}} \Hils_{\s(g)} \congto
    S_{u\rg(g)}(\A_g\A_g^*) \Hils_{\rg(g)},\qquad
    a\otimes \xi\mapsto S_g(a)(\xi).
  \]
  In particular, \(\A_g \otimes_{\FU_{\s(g)}} \Hils_{\s(g)}\neq0\) if
  and only if
  \(\A_g \A_g^* \otimes_{\FU_{\rg(g)}} \Hils_{\rg(g)}\neq0\).  The
  unitary operators~\(U_g\) for \(g\in G\) form an isomorphism
  \(\A\,\s^*\Hils \congto \A\A^*\,\rg^*\Hils\) of measurable fields of
  Hilbert spaces over~\(G\).
\end{lemma}

\begin{proof}
  Since \(\A_g\A_g^*\A_g=\A_g\) and the maps~\(S_g\) are compatible
  with the multiplication and involution, we get a cyclic chain of
  inclusions
  \begin{multline}
    \label{eq:measurable_rep_inclusions}
    S_g(\A_g) \Hils_{\s(g)}
    = S_g(\A_g \A_g^*\A_g) \Hils_{\s(g)}
    = S_{u\rg(g)}(\A_g\A_g^*) S_g(\A_g) \Hils_{\s(g)}
    \\\subseteq S_{u\rg(g)} (\A_g \A_g^*) \Hils_{\rg(g)}
    = S_g(\A_g)S_g(\A_g)^* \Hils_{\rg(g)}
    \subseteq S_g(\A_g) \Hils_{\s(g)}.
  \end{multline}
  It follows that the inclusions
  in~\eqref{eq:measurable_rep_inclusions} are equalities and that
  the range of~\(U_g\) is dense in
  \(S_{u\rg(g)}(\A_g\A_g^*) \Hils_{\rg(g)}\).  If
  \(a_1,a_2\in \A_g\), \(\xi_1,\xi_2\in \Hils_{\s(g)}\), then we
  compute
  \begin{multline*}
    \braket{U_g(a_1 \otimes \xi_1)}{U_g(a_2\otimes \xi_2)} =
    \braket{S_g(a_1)\xi_1}{S_g(a_2)\xi_2} =
    \braket{\xi_1}{S_g(a_1)^*S_g(a_2)\xi_2} \\=
    \braket{\xi_1}{S_{u\s(g)}(a_1^* \cdot a_2)\xi_2} = \braket{a_1
      \otimes \xi_1}{a_2\otimes \xi_2}
  \end{multline*}
  This implies that~\(U_g\) is a well defined isometry.  The family of
  unitaries~\(U_g\) is measurable because of the measurability
  assumption on the maps~\(S_g\).  Then the adjoint~\(U_g^*\) is also
  measurable.
\end{proof}

\begin{definition}
  \label{def:A-quasi-invariant_relative}
  Let \(\Hils= (\Hils_x)_{x\in G^0}\) be a \(\nu\)\nb-measurable field
  of Hilbert spaces over~\(G^0\) and let
  \(S_g\colon \A_g \to \Bound(\Hils_{\s(g)}, \Hils_{\rg(g)})\) for
  \(g\in G\) be a \(\nu\)\nb-measurable pre-representation of the Fell
  bundle~\(\A\) on~\(\Hils\).  We define the Borel set
  \[
    \supp \A\A^*\,\rg^*\Hils
    \defeq \setgiven{g\in G}
    {S_{u\rg(g)}(\A_g\A_g^*) \Hils_{\rg(g)} \neq0}
    = \setgiven{g\in G}{S_g(\A_g)\Hils_{\s(g)}\not=0}.
  \]
  The first description shows that the support depends only
  on~\(S_{u(x)}\) for \(x\in G^0\) and not on~\(S_g\) for \(g\notin
  u(G^0)\).
  Therefore, the following concept also uses
  only~\((S_{u\rg(g)})_{g\in G}\):
  the measure~\(\nu\) is called \emph{quasi-invariant relative
    to~\((S_{u\rg(g)})_{g\in G}\)} if the restriction
  of~\(\nu\circ\alpha\) to the Borel subset
  \(\supp \A\A^*\,\rg^*\Hils\) of~\(G\), defined by
  \(\nu\circ\alpha|_{\supp \A\A^*\,\rg^*\Hils}(T) \defeq
  \nu\circ\alpha(T \cap \supp \A\A^*\,\rg^*\Hils)\), is symmetric.
\end{definition}

\begin{remark}
  We are going to prove that a Fell-bundle representation on a
  separable Hilbert space is equivalent to a measurable Fell-bundle
  representation with a measure class that is quasi-invariant relative
  to this representation.  Let us point out why a rather obvious idea
  to find a counterexample to this does not work.  As our Fell bundle,
  we take the direct sum of a saturated Fell bundle, say, over a
  group, and a Fell bundle that is living on units.  Then a
  representation is just a pair of representations of the two Fell
  bundles.  The first representation gives a measure class that is
  quasi-invariant in the usual sense and a field of Hilbert spaces for
  which \(\supp \A\A^*\,\rg^*\Hils = G\) is possible, but the second
  representation may have an arbitrary measure class.  So why should
  the direct sum representation have a measure class that is
  relatively quasi-invariant?  What saves us is that the bundle of
  Hilbert spaces and the measure classes for a direct sum of two
  Hilbert space representations are not just the naive sums of
  the measurable representations and the measure classes.
  For instance, if the two measure classes are orthogonal, that is,
  supported on disjoint measurable subsets, then we should first
  restrict our fields of Hilbert spaces to these disjoint subsets
  before adding them.  This changes the supports of both fields of
  Hilbert spaces defined above, and it turns out on inspection that
  the direct sum representation does indeed come from a measurable
  Fell-bundle representation and a measure class that is
  quasi-invariant relative to it.
\end{remark}

\begin{definition}
  \label{def:Fell_bundle_rep_disintegrated}
  A \emph{measurable Fell-bundle representation} of~\(\A\) consists of
  a triple \((\nu, \Hils, S)\) where~\(\nu\) is a measure on the
  space~\(G^0\) of units, \(\Hils\) is a measurable field of Hilbert
  spaces, \(S=(S_g)_{g\in G}\) is a Fell-bundle pre-representation
  of~\(\A\) on~\(\Hils\) and the measure~\(\nu\) is quasi-invariant
  relative to \((S_{u\rg(g)})_{g\in G}\).
\end{definition}

\subsection{Integration of a measurable Fell-bundle representation}
\label{sec:original-draft}

Now we describe how a measurable Fell-bundle representation gives a
representation of the Fell bundle on a Hilbert space as in
Definition~\ref{def:representation_Fell}.  Let
\((\nu, \Hils, (S_g)_{g\in G})\) be a measurable representation of the
Fell bundle~\(\A\).  We are going to define a Fell-bundle
representation \((\varphi,U)\) of~\(\A\) on the Hilbert space
\[
  \Hils[L] \defeq \Lt^2(G^0,\Hils,\nu).
\]
The representation
\(\varphi\colon \Cont_0(G^0,\FU) \to \Bound(\Hils[L])\) is defined
by
\[
  (\varphi(f)\xi)(x) \defeq S_{u(x)}(f(x))(\xi(x))
\]
for all \(f\in \Cont_0(G^0,\FU)\), \(\xi\in \Lt^2(G^0,\Hils,\nu)\),
\(x\in G^0\).  This is well defined because of the measurability
assumption in Definition~\ref{def:Fell_bundle_pre-rep}, and it is a
representation because each~\(S_{u(x)}\) is representation
of~\(\FU_x\).

There is a unique measurable field of Hilbert spaces denoted
\(\A\,\s^*\Hils\) over~\(G\) that has the fibres
\(\A_g \otimes_{\FU_{\s(g)}} \Hils_{\s(g)}\) at \(g\in G\) and
that is such that the space of its measurable sections is generated
by sections of the form \(g\mapsto f_1(g) \otimes f_2(\s(g))\) for
\(f_1\in \Contc(G,\A)\) and \(f_2\in \Lt^2(G^0,\Hils,\nu)\); here
\cite{Dixmier:Von_Neumann}*{Proposition~4 on p.~167} provides the
unique measurable field structure.  These fibres are isomorphic (via Lemma~\ref{lem:measurable-rep-U}) to
\(S_g(\A_g) \Hils_{\s(g)}\subseteq \Hils_{\rg(g)}\), and these
isomorphisms identify \(\A\,\s^*\Hils\) with a subfield
of~\(\rg^*(\Hils)\).

Similarly, we define a measurable field of Hilbert spaces
\(\A\A^*\,\rg^*\Hils\) over~\(G\) with fibres
\(S_{u\rg(g)}(\A_g\A_g^*) \Hils_{\rg(g)} \cong \A_g\A_g^*
\otimes_{\FU_{\rg(g)}} \Hils_{\rg(g)}\) at \(g\in G\).  By
definition, the fibre of this measurable field at \(g\in G\) is
nontrivial if and only if \(g\in \supp \A\A^*\,\rg^*\Hils\).  The
representation~\((S_g)_{g\in G}\) makes the two fields above
isomorphic.

The same proof as for
Lemma~\ref{lem:compose_corr_from_measure-family} gives natural
unitary operators
\begin{align}
  \label{eq:unitary-rep-Hilb-space}
  \Lt^2(G,\A,\s,\tilde\alpha) \otimes_\varphi \Hils[L]
  &\congto \Lt^2(G,\A\,\s^*\Hils,\nu\circ\tilde\alpha),\\
  \Lt^2(G,\A\A^*,\rg,\alpha) \otimes_\varphi \Hils[L]
  &\congto \Lt^2(G,\A\A^*\,\rg^*\Hils,\nu\circ\alpha),
\end{align}
mapping \(f_1\otimes f_2\) to the sections
\(g\mapsto f_1(g)\cdot f_2(\s(g))\) or
\(g\mapsto f_1(g)\cdot f_2(\rg(g))\), respectively.  We replace
\(\nu\circ\tilde\alpha\) and \(\nu\circ\alpha\) by their restrictions
to the support of the two isomorphic fields of
Hilbert spaces in which we take sections.  This does not change the
space of \(\Lt^2\)\nb-sections.

These restrictions of \(\nu\circ\tilde\alpha\) and \(\nu\circ\alpha\)
are equivalent measures if and only if the measure~\(\nu\) is
quasi-invariant relative to~\((S_{u\rg(g)})_{g\in G}\); this follows
from~\eqref{eq:measurable_rep_inclusions}.
As a result, the restrictions of \(\nu\circ\tilde\alpha\) and
\(\nu\circ\alpha\) to the supports of the relevant fields of Hilbert spaces are
equivalent.
The resulting Radon--Nikodym derivative
\begin{equation}
  \label{eq:modular_def}
  \Delta \defeq
  \frac{\diff \nu\circ\alpha}
  {\diff \nu\circ\tilde\alpha} \biggr|_{\supp \A\A^*\,\rg^*\Hils}
  \colon \supp \A\A^*\,\rg^*\Hils \to (0,\infty),
\end{equation}
generalises the modular function of a genuinely quasi-invariant
measure.
As in the classical case, the following defines a unitary operator:
\begin{align}
  \label{eq:Unitary-Tensor}
  U\colon \Lt^2(G,\A\,\s^*\Hils,\nu\circ\tilde\alpha)
  &\congto
    \Lt^2(G,\A\A^*\,\rg^*\Hils,\nu\circ\alpha),\\ \notag
  (U f)(g)
  &\defeq U_g(f(g)) \cdot \Delta(g)^{\frac{1}{2}}.
\end{align}
Using the canonical isomorphisms above, we turn~\(U\) into a unitary
operator
\[
  \Lt^2(G,\A,\s,\tilde\alpha) \otimes_\varphi \Hils[L] \congto
  \Lt^2(G,\A\A^*,\rg,\alpha) \otimes_\varphi \Hils[L].
\]

\begin{lemma}
  \label{lem:disintegrated_Fell_bundle_rep}
  The data \((\varphi,U)\) defined above is a Fell-bundle
  representation.
\end{lemma}

\begin{proof}
  All we need to check is that \(d_2^*(U)d_0^*(U)=d_1^*(U)\).  Let
  \(j=1,2,3\).  First, we identify the tensor product
  \(\Lt^2(G^2,\E_j, v_j, \mu_j) \otimes_\varphi \Lt\) in
  Definition~\ref{def:representation_Fell} with the Hilbert space of
  sections of a measurable bundle over~\(G^2\).  In fact, an
  argument similar to that used for getting
  Equation~\eqref{eq:unitary-rep-Hilb-space} gives us a unitary
  isomorphism
  \begin{equation}
    \label{eq:measurable-rep-integration}
    \Lt^2(G^2,\E_j, v_j, \mu_j) \otimes_\varphi \Lt \congto
    \Lt^2(G^2,\E_jv_j^*\Hils, \nu\circ \mu_j),
  \end{equation}
  which maps a section \(f_1\otimes f_2\) to the section
  \((g,h)\mapsto f_1(g,h) \cdot f_2(v_j(g,h))\); here \(v_j\)
  and~\(\mu_j\) are as defined in
  Definition~\ref{def:representation_Fell} and \(\E_j v_j^*\Hils\)
  denotes the \(\nu\circ \mu_j\)\nb-measurable field of Hilbert
  spaces over~\(G^2\) in which the fibre at \((g,h)\in G^2\) is
  \((\E_j)_{(g,h)}\otimes_{\FU_{v_j(g,h)}} \Hils_{v_j(g,h)}\).
  We are going to show that the unitaries
  \begin{align*}
    D_0^*(U)
    &\colon \Lt^2(G^2,\E_2v_2^*\Hils, \nu\circ \mu_2) \to
      \Lt^2(G^2,\E_1v_1^*\Hils, \nu\circ \mu_1),\\
    D_2^*(U)
    &\colon \Lt^2(G^2,\E_1v_1^*\Hils, \nu\circ \mu_1) \to
      \Lt^2(G^2,\E_0v_0^*\Hils, \nu\circ \mu_0)\\
    D_1^*(U)
    &\colon \Lt^2(G^2,\E_2v_2^*\Hils, \nu\circ \mu_2) \to
      \Lt^2(G^2,\E_0v_0^*\Hils, \nu\circ \mu_0)
  \end{align*}
  induced by~\(U\) satisfy \(D_2^*(U)D_0^*(U)=D_1^*(U)\).  The
  isomorphisms in~\eqref{eq:measurable-rep-integration} turn this
  into the desired relation \(d_2^*(U)d_0^*(U)=d_1^*(U)\).

  Since the isomorphisms in~\eqref{eq:measurable-rep-integration}
  and the unitaries \(d_j^*(U)\) commute with the left action of
  \(\Cont_0(G^2)\), so does~\(D_j^*(U)\).  Therefore, \(D_j^*(U)\)
  is a ``decomposable'' operator by
  \cite{Dixmier:Algebres-Operateurs}*{Theorem~1, Section~II.2}.  Up
  to a Radon--Nikodym derivative, the operator~\(D_j^*(U)\) is
  induced by a family of unitaries
  \[
    D_j^*(U)_{(g,h)}\colon
    \E_i v_i^* \Hils_{(g,h)} \to \E_k v_k^* \Hils_{(g,h)}
  \]
  between the fibres of the appropriate measurable fields of Hilbert
  spaces; here \(\{i,j,k\} = \{0,1,2\}\) and \(i>k\).  Recall the
  definition of the bundles~\(\E_j\) in~\eqref{eq:E_012}.

  Since \(d_0(g,h)=h\), the operator \(D_0^*(U)_{(g,h)}\) acts
  by~\(U_h\) combined with canonical isomorphisms:
  \begin{multline*}
    \E_2 v_2^*\Hils_{(g,h)}
    = \A_g\A_h\A_h^*\A_g^*\A_{g h} \otimes_{\FU_{\s(h)}} \Hils_{\s(h)}
    \cong \A_g\A_h\A_h^*\A_g^*\A_g\A_h \otimes_{\FU_{\s(h)}} \Hils_{\s(h)}
    \\\xrightarrow[\sim]{U_h}
    \A_g\A_h\A_h^*\A_g^*\A_g \A_h\A_h^* \otimes_{\FU_{\rg(h)}} \Hils_{\rg(h)}
    \cong \A_g \A_h\A_h^* \A_g^*\A_g \otimes_{\FU_{\rg(h)}}
    \Hils_{\rg(h)}
    \\= \E_1 v_1^*\Hils_{(g,h)},
  \end{multline*}
  where the isomorphism marked by~\(U_h\) is really
  \(\Id \otimes U_h\) and the isomorphisms that have no labels use
  Lemma~\ref{lem:equality_of_products} and that the ideals
  \(\A_g^*\A_g\) and \(\A_h\A_h^*\) in \(\FU_{\s(g)} = \FU_{\rg(h)}\)
  commute.  Similarly, \(D_2^*(U)_{(g,h)}\) is the composite
  \begin{multline*}
    \E_1 v_1^*\Hils_{(g,h)}
    = \A_g\A_h\A_h^*\A_g^*\A_g \otimes_{\FU_{\s(g)}} \Hils_{\s(g)}
    \xrightarrow[\sim]{U_g}
    \A_g\A_h\A_h^*\A_g^*\A_g\A_g^* \otimes_{\FU_{\rg(g)}} \Hils_{\rg(g)}
    \\\cong \A_g\A_h\A_h^*\A_g^* \otimes_{\FU_{\rg(g)}} \Hils_{\rg(g)}
    = \E_0 v_0^*\Hils_{(g,h)}
  \end{multline*}
  and \(D_1^*(U)_{(g,h)}\) is the composite
  \begin{multline*}
    \E_2 v_2^*\Hils_{(g,h)}
    = \A_g\A_h\A_h^*\A_g^*\A_{gh} \otimes_{\FU_{\s(h)}} \Hils_{\s(h)}
    \xrightarrow[\sim]{U_{g h}}
    \A_g\A_h\A_h^*\A_g^*\A_{gh} \A_{gh}^* \otimes_{\FU_{\rg(g)}} \Hils_{\rg(g)}
    \\\cong \A_g\A_h\A_h^*\A_g^* \otimes_{\FU_{\rg(g)}} \Hils_{\rg(g)}
    = \E_0 v_0^*\Hils_{(g,h)}.
  \end{multline*}
  Let \(a_1\in\A_g\), \(a_2\in \A_h\), \(\xi \in \Hils_{\s(h)}\).
  Then the assumption \(S_g(a_1) S_h(a_2) = S_{g h}(a_1 a_2)\)
  translates into
  \[
    U_{g h}(a_1 \cdot a_2 \otimes\xi)
    = S_{g h}(a_1 a_2)(\xi)
    = (S_g(a_1) \otimes \Id) (U_h(a_2 \otimes \xi))
    = U_g(a_1 \otimes U_h(a_2 \otimes \xi)).
  \]
  This is equivalent to
  \(D_2^*(U)_{(g,h)} D_0^*(U)_{(g,h)} = D_1^*(U)_{(g,h)}\).

  The unitary \(D_j^*(U)\) acts by \((D_j^*(U) f)(g,h) =
  \Delta(d_j(g,h))^{1/2} (D_j^*(U)_{(g,h)} f)(g,h)\) with~\(\Delta\)
  as in~\eqref{eq:modular_def}.
  Since the operators~\(D_j^*(U)\) commute with pointwise
  multiplication operators, it follows that
  \begin{align*}
    (D_2^*(U)D_0^*(U) f)(g,h)
    &= \Delta(g)^{1/2}\Delta(h)^{1/2} D_2^*(U)_{g,h} D_0^*(U)_{g,h}
      f(g,h),\\
    (D_1^*(U) f)(g,h)
    &= \Delta(g h)^{1/2} D_1^*(U)_{g,h} f(g,h).
  \end{align*}
  We have already proven \(D_2^*(U)_{(g,h)} D_0^*(U)_{(g,h)} =
  D_1^*(U)_{(g,h)}\).
  So \(D_2^*(U)D_0^*(U) = M_\xi \circ D_1^*(U)\), where~\(M_\xi\)
  denotes the pointwise multiplication operator by the function
  \(\xi(g,h) \defeq \Delta(g)^{1/2}\Delta(h)^{1/2} \Delta(g
  h)^{-1/2}\).
  On the one hand, this forces~\(M_\xi\) to be unitary.
  On the other hand, \(M_\xi\) is positive because~\(\xi\) has values
  in~\((0,\infty)\).
  So \(M_\xi\) is the identity operator and \(D_2^*(U)D_0^*(U) =
  D_1^*(U)\) as desired.
\end{proof}

Along the way, the proof above that \(\xi(g,h)=1\) shows
that~\(\Delta\) is multiplicative:

\begin{proposition}
  \label{def:Delta_multiplicative}
  Let \(\Sigma \defeq \setgiven{(g,h)\in G^2}{S_{g h}(\A_g \A_h
    \Hils_{\s(h)}) \neq 0}\).
  The restrictions of the measures \(\nu\circ \mu_j\) for \(j=0,1,2\)
  to~\(\Sigma\) are all equivalent, and
  \[
    \Delta(g h) = \Delta(g) \Delta(h)
  \]
  holds for \(\nu\circ \mu_j\)-almost  all \((g,h)\in\Sigma\).
\end{proposition}

\begin{remark}
  For a saturated Fell bundle, the support \(\supp
  \A\A^*\,\rg^*\Hils\) is conull in~\(G\) and~\(\Sigma\) is conull
  in~\(G^2\).
  So the function~\(\Delta\) is almost everywhere defined and almost
  everywhere multiplicative.
  This is what is called a weak homomorphism
  in~\cite{Ramsay:Topologies}.
  One of the main results in~\cite{Ramsay:Topologies}, Theorem~3.2,
  says that any such weak homomorphism is almost everywhere equal to a
  strict homomorphism.
  Since~\(\Delta\) is only well defined up to equality almost
  everywhere, we may assume that it is already chosen strict.
  This gives a choice of~\(\Delta\) that is defined for all \(g\in G\)
  and satisfies \(\Delta(g h) = \Delta(g) \Delta(h)\) for all \((g,h)\in
  G^2\).
  If the bundle is not saturated, such an improvement of~\(\Delta\)
  seems unlikely to exist.
  To begin with, the subset \(\supp \A\A^*\,\rg^*\Hils\) on
  which~\(\Delta\) is defined is no longer conull and need not be a
  subgroupoid.
\end{remark}

A genuine groupoid homomorphism \(G\to(0,\infty)\) always satifies
\(\Delta(g)^{-1} = \Delta(g^{-1})\).
This equation still works for the function~\(\Delta\) in the
non-saturated case:

\begin{lemma}
  Let \((g,h)\in G^2\).
  Then \((g,h) \in \Sigma\) if and only if \((g h,h^{-1}) \in
  \Sigma\).
  If \(g\in\supp \A\A^*\,\rg^*\Hils\), then \(g^{-1} \in
  \supp \A\A^*\,\rg^*\Hils\) and almost surely \(\Delta(g^{-1}) =
  \Delta(g)^{-1}\).
\end{lemma}

\begin{proof}
  Since~\(S_g\) is compatible with involutions, \(S_g(\A_g)
  \Hils_{\s(g)}\neq0\) if and only if \(S_{g^{-1}}(\A_{g^{-1}})
  \Hils_{\s(g^{-1})}\neq0\).
  That is, \(g\in\supp \A\A^*\,\rg^*\Hils\) if and only if \(g^{-1}
  \in \supp \A\A^*\,\rg^*\Hils\).
  Next, Lemma~\ref{lem:equality_of_products} and the multiplicativity
  of~\(S_g\) imply
  \begin{multline*}
    S_g(\A_g)S_h(\A_h)\Hils_{s(h)}
    = S_g(\A_g)S_h(\A_h) S_{h^{-1}}(\A_h^*) S_h(\A_h) \Hils_{s(h)}
    \\ \subseteq  S_{gh}(\A_g\A_h) S_{h^{-1}}(\A_h^*) \Hils_{\rg(h)}
    \subseteq  S_g(\A_g)S_h(\A_h)\Hils_{\s(h)}.
  \end{multline*}
  It follows that the inclusions here are equalities, so that
  \(S_g(\A_g)S_h(\A_h)\Hils_{s(h)} = S_{gh}(\A_g\A_h)
  S_{h^{-1}}(\A_h^*) \Hils_{\rg(h)}\).
  In particular, one of these spaces is nonzero if and only if the
  other is.
  This says that \((g,h) \in \Sigma\) if and only if \((g h, h^{-1})
  \in \Sigma\).
  If \(g,h\in \Sigma\), then the multiplicativity of~\(\Delta\)
  implies
  \[
    \Delta(g h) = \Delta(g) \Delta(h),\qquad
    \Delta(g) = \Delta(g h) \Delta(h^{-1})
  \]
  for almost all \((g,h)\in\Sigma\).
  Both equations together imply \(\Delta(h)^{-1} = \Delta(h^{-1})\)
  for almost all \((g,h)\in\Sigma\).
  By Fubini's Theorem, this implies \(\Delta(h)^{-1} =
  \Delta(h^{-1})\) for almost all \(h\in \supp \A\A^*\,\rg^*\Hils\).
\end{proof}

\subsection{Formulation of the main result}
\label{sec:main_theorem}

Now we are going to turn a Fell-bundle representation on a Hilbert
space into a measurable Fell-bundle representation.
This corresponds to Renault's Disintegration Theorem.
We first formulate the theorem that we are going to prove eventually:

\begin{definition}
  Let~\(Y\) be a locally compact space and~\(\nu\) a measure
  on~\(Y\).  A subset \(X\subseteq Y\) is called
  \emph{\(\nu\)\nb-conull} if it is \(\nu\)\nb-measurable and satisfies
  \(\nu(Y\setminus X) = 0\).
\end{definition}

Our universal property in Theorem~\ref{the:universal_groupoid_Haus}
together with the following theorem generalises Renault's
Disintegration Theorem from~\cite{Renault:Representations} to
nonsaturated separable Fell bundles over second countable locally
compact Hausdorff groupoids:

\begin{theorem}
  \label{the:disintegrate_Hilbert_space}
  Let~\(G\) be a second countable, locally compact, Hausdorff
  groupoid.  Let~\(\alpha\) be a Haar system on~\(G\) and let~\(\A\)
  be a Fell bundle over~\(G\) such that~\(\FU_x\) is separable for
  all \(x\in G^0\).
  \begin{enumerate}
  \item \label{the:disintegrate_Hilbert_space_1}%
    Any Fell-bundle representation of \((G,\A,\alpha)\) on a
    separable Hilbert space may be obtained from a
    measurable Fell-bundle representation.
    That is, it is the representation on \(\Lt^2(G^0,\Hils,\nu)\)
    defined by a measure~\(\nu\) on~\(G^0\), a \(\nu\)\nb-measurable
    field of Hilbert spaces~\(\Hils\) with \(\Hils_x\neq0\) for
    \(\nu\)\nb-almost all \(x\in G^0\), and a measurable
    pre-representation \(S = (S_g)_{g\in G}\) of~\(\A\) on~\(\Hils\),
    such that~\(\nu\) is quasi-invariant relative to~\(S\).
  \item \label{the:disintegrate_Hilbert_space_2}%
    Let \(\nu_j\), \(\Hils_j\), \(S_j\) for \(j=1,2\) be as
    in~\ref{the:disintegrate_Hilbert_space_1}.
    The corresponding Fell-bundle representations on
    \(\Lt^2(G^0,\Hils_j,\nu_j)\) are isomorphic if and only if
    \(\nu_1\) and~\(\nu_2\) are equivalent and there are a
    \([\nu_1]\)\nb-conull subset
    \(X\subseteq G^0\) and a measurable family of unitaries
    \(V_x\colon \Hils_{1,x} \congto \Hils_{2,x}\) for \(x\in X\)
    such that \(S_{2,g}(a) V_{\s(g)}(\xi) = V_{\rg(g)} S_{1,g}(a)\xi\)
    for all \(g\in G\) with \(\s(g),\rg(g)\in X\), \(a\in \A_g\),
    \(\xi\in\Hils_{1,\s(g)}\); here \([\nu_1]\) denotes the measure
    class of~\(\nu_1\).  Even more, any isomorphism of Fell-bundle
    representations is of the form
    \begin{equation}
      \label{eq:intertwiner_formula}
      (V \xi)(x) \defeq V_x(\xi(x)) \cdot \left(\frac{\diff
          \nu_1}{\diff\nu_2}\right)^{1/2}(x)
    \end{equation}
    for such an~\(X\) and a measurable family~\((V_x)_{x\in X}\) of
    unitary intertwiners.
  \item \label{the:disintegrate_Hilbert_space_3}%
    Two parallel isomorphisms \((V_x)_{x\in X}\) and
    \((V'_x)_{x\in X'}\) as in~\ref{the:disintegrate_Hilbert_space_2}
    are equal if and only if there is a
    \([\nu_1]\)\nb-conull subset \(X_0\subseteq X\cap X'\) with
    \(V_x = V'_x\) for all \(x\in X_0\).
  \end{enumerate}
  If the Fell bundle is saturated, then the conull subsets~\(X\)
  in~\ref{the:disintegrate_Hilbert_space_2} and~\(X_0\)
  in~\ref{the:disintegrate_Hilbert_space_3} may be assumed to be
  \(G\)\nb-invariant.
\end{theorem}

We will prove the theorem by analysing a Fell-bundle
representation~\((\varphi,U)\) of~\(\A\) on a separable Hilbert
space~\(\Hils[L]\).  First, we discuss the separability assumptions.

\begin{lemma}
  If~\(\FU_x\) is separable for all \(x\in G^0\), then~\(\A_g\) is
  separable for all \(g\in G\).
\end{lemma}

\begin{proof}
  A Hilbert module~\(\E\) over a \(\Cst\)\nb-algebra~\(B\) is
  countably generated if and only if the \(\Cst\)\nb-algebra
  \(\Comp(\E)\) is \(\sigma\)\nb-unital
  (see~\cite{Lance:Hilbert_modules}).
  This implies that any Hilbert bimodule between two separable
  \(\Cst\)\nb-algebras is countably generated as a right Hilbert module
  and hence separable.
\end{proof}

If~\(G\) is second countable and~\(\A_g\) is separable for all
\(g\in G\), then \(\Cont_0(G,\A)\) is separable.  Then
\(\Lt^2(G,\A,\s,\tilde\alpha)\) and \(\Lt^2(G,\A\A^*,\rg,\alpha)\)
are separable as well.  Similarly, \(\Cont_0(G^2,d_j^*\A)\) and
\(\Cont_0(G^2,v_j^*\FU)\) are separable, and so are the
correspondences involving~\(G^2\).  Summing up, all the Banach
spaces that occur in the definition of a Fell-bundle representation
of~\(\A\) on~\(\Hils[L]\) are separable.

\subsection{Representations of fields of
  \texorpdfstring{$\Cst$}{C*}-algebras}
\label{sec:rep_fields}

Next, we recall basic results about representations of commutative
\(\Cst\)\nb-algebras.

\begin{theorem}
  \label{the:disintegration_commutative}
  Let~\(X\) be a second countable, locally compact space and
  let~\(\Hils[L]\) be a separable Hilbert space.
  Let \(\varrho\colon \Cont_0(X) \to \Bound(\Hils[L])\) be a
  nondegenerate representation.
  Then there are a measure~\(\nu\) on~\(X\) and a
  \(\nu\)\nb-measurable field of Hilbert spaces \(\Hils =
  (\Hils_x)_{x\in X}\) over~\(X\) such that \(\Hils_x \neq0\) for
  \(\nu\)\nb-almost all \(x\in X\) and~\(\varrho\) is unitarily
  equivalent to the representation of \(\Cont_0(X)\) on the Hilbert
  space \(\Lt^2(X,\Hils,\nu)\) of \(\nu\)-square-integrable sections
  of the field~\(\Hils\) by pointwise multiplication.
\end{theorem}

\begin{proof}
  This follows from results in~\cite{Dixmier:Algebres-Operateurs}.
  First, we may arrange that the representation~\(\varrho\) is
  faithful: otherwise, its kernel is \(\Cont_0(U)\) for an open
  subset \(U\subseteq X\) and then we may replace~\(X\) by the
  closed subspace~\(X\setminus U\), which corresponds to the
  quotient \(\Cont_0(X)/\Cont_0(U)\).  Thus~\(\varrho\) identifies
  \(\Cont_0(X)\) with an Abelian \(\Cst\)\nb-subalgebra of
  \(\Bound(\Hils[L])\).  The nondegeneracy of~\(\varrho\) says
  that~\(\Id_{\Hils[L]}\) is in the weak closure of
  \(\varrho(\Cont_0(X))\).  Since~\(\Hils[L]\) is separable, this
  representation has a ``basic'' measure by
  \cite{Dixmier:Algebres-Operateurs}*{Proposition~4 in Section~I.7
    on p.~116}.  And then
  \cite{Dixmier:Algebres-Operateurs}*{Theorem~1 in Section~II.6 on
    p.~208} implies all the claims.
\end{proof}

We also need information about the uniqueness of this disintegration:

\begin{theorem}
  \label{the:disintegration_commutative_unique}
  Let~\(X\) be a second countable, locally compact space.  For
  \(j=1,2\), let~\(\nu_j\) be measures on~\(X\) and let
  \(\Hils_j = (\Hils_{j,x})_{x\in X}\) be \(\nu_j\)\nb-measurable
  fields of Hilbert spaces over~\(X\) such that \(\Hils_{j,x}\neq0\)
  \(\nu_j\)\nb-almost everywhere for \(j=1,2\).  There is a unitary
  intertwiner~\(V\) between the pointwise multiplication
  representations of~\(\Cont_0(X)\) on \(\Lt^2(X,\Hils_j,\nu_j)\) for
  \(j=1,2\) if and only if \([\nu_1]=[\nu_2]\) and the measurable
  fields of Hilbert spaces \(\Hils_1\) and~\(\Hils_2\) are isomorphic
  almost everywhere with respect to this common measure class.  If
  \[
    T\colon \Lt^2(X,\Hils_1,\nu_1) \congto \Lt^2(X,\Hils_2,\nu_2)
  \]
  is a unitary intertwiner, then there are unitary operators
  \(T_x\colon \Hils_{1,x} \congto \Hils_{2,x}\) for
  \([\nu_1]\)\nb-almost all \(x\in X\) such that
  \begin{equation}
    \label{eq:integrate_fiberwise_intertwiners}
    T(\xi)(x) = T_x(\xi(x)) \left(\frac{\diff
        \nu_1}{\diff\nu_2}\right)^{1/2}(x)
  \end{equation}
  for \([\nu_1]\)-almost all \(x\in X\) and all
  \(\xi\in \Lt^2(X,\Hils_1,\nu_1)\).
\end{theorem}

\begin{proof}
  This follows from results in~\cite{Dixmier:Algebres-Operateurs}.
  Since \(\Hils_{j,x}\neq0\) \(\nu_j\)\nb-almost everywhere, we may
  modify the field of Hilbert spaces~\(\Hils_j\) so that its fibres
  become nonzero everywhere without changing the Hilbert space
  \(\Lt^2(X,\Hils_j,\nu_j)\).  (This nonvanishing is assumed
  in~\cite{Dixmier:Algebres-Operateurs}.)  If two representations
  are unitarily
  equivalent, then they have the same kernel.  Replacing~\(X\) by a
  closed subspace, we may arrange that both representations of
  \(\Cont_0(X)\) are faithful.  Equivalently, the measures \(\nu_1\)
  and~\(\nu_2\) have full support.  Now
  \cite{Dixmier:Algebres-Operateurs}*{Theorem~3 in Section~II.6 on
    p.~211} implies all our claims.
\end{proof}

\begin{remark}
  The need to assume \(\Hils_x \neq0\) for \(\nu\)\nb-almost all
  \(x\in X\) is overlooked in~\cite{Buss-Holkar-Meyer:Universal}.
  This assumption ensures uniqueness of the measure class of~\(\nu\)
  by Theorem~\ref{the:disintegration_commutative}.
  If \(\nu\) and~\(\Hils\) do not satisfy this, then we may
  replace~\(\nu\) by its restriction to the \(\nu\)\nb-measurable subset
  \(\setgiven{x\in X}{\Hils_x \neq0}\) without changing the Hilbert
  space \(\Lt^2(X,\Hils,\nu)\).
  Then the important claim \([\nu_1]=[\nu_2]\) in
  Theorem~\ref{the:disintegration_commutative_unique} breaks down.
\end{remark}

The following theorem adds a coefficient algebra to our disintegration
theorems:

\begin{theorem}
  \label{the:disintegrate_rep_of_bundle}
  Let~\(X\) be a second countable, locally compact space and let
  \(\FU[B] = (\FU[B]_x)\) be a semicontinuous field of
  separable \(\Cst\)\nb-algebras over~\(X\).
  Let~\(\Hils[L]\) be a separable Hilbert space and let \(\varphi\colon
  \Cont_0(X,\FU[B]) \to \Bound(\Hils[L])\) be a nondegenerate
  representation.
  Then there are a measure~\(\nu\) on~\(X\), a \(\nu\)\nb-measurable field of
  Hilbert spaces \(\Hils = (\Hils_x)_{x\in X}\), a unitary
  \(V\colon \Hils[L] \congto \Lt^2(X,\Hils,\nu)\), and a family of nondegenerate
  representations \(\varphi_x\colon \FU[B]_x \to \Bound(\Hils_x)\) for
  all \(x\in X\) such that \(\Hils_x \neq0\) \(\nu\)\nb-almost
  everywhere, and
  \[
    (V\varphi(b) \xi)(x) = \varphi_x(b(x)) (V \xi)(x)
  \]
  for all \(b\in \Cont_0(X,\FU[B])\), \(\xi\in \Hils[L]\) and
  \(\nu\)\nb-almost all \(x\in X\).

  Here the measure~\(\nu\) is unique up to equivalence, and the measurable
  field of Hilbert spaces~\(\Hils\) and the
  representations~\(\varphi_x\) are unique up to unitary equivalence
  \(\nu\)\nb-almost everywhere.
  More generally, let \(T\colon \Hils[L] \to \Hils[L]'\) be a unitary
  intertwiner between two nondegenerate representations \(\varphi\colon
  \Cont_0(X,\FU[B]) \to \Bound(\Hils[L])\) and  \(\varphi'\colon
  \Cont_0(X,\FU[B]) \to \Bound(\Hils[L]')\) on separable Hilbert
  spaces, and disintegrate the two representations into
  \((\nu,\Hils,\varphi_x)\) and
  \((\nu',\Hils',\varphi_x')\) as above.
  Then \([\nu]=[\nu']\) and there are a conull subset \(X_0\subseteq
  X\) and a measurable family of unitary operators \(T_x\colon \Hils_x
  \congto \Hils_x'\) for \(x\in X_0\) that intertwine \(\varphi_x\)
  and~\(\varphi'_x\) for all \(x\in X_0\) and that
  satisfy~\eqref{eq:integrate_fiberwise_intertwiners}.
\end{theorem}

\begin{proof}
  The representation \(\varphi\colon \Cont_0(X,\FU[B]) \to \Bound(\Hils[L])\)
  is nondegenerate and hence extends to a unital representation
  \(\Mult(\Cont_0(X,\FU[B]))\to\Bound(\Hils[L])\).
  Composing with the canonical nondegenerate \Star{}homomorphism
  \(\Cont_0(X)\to\Mult(\Cont_0(X,\FU[B]))\) gives a nondegenerate
  representation \(\varphi_0\colon \Cont_0(X)\to \Bound(\Hils[L])\).
  Since~\(X\) is second countable and~\(\Hils[L]\) separable,
  Theorem~\ref{the:disintegration_commutative} applies
  to~\(\varphi_0\).
  It gives a measure class~\(\nu\) on~\(X\), a
  measurable field of Hilbert spaces~\(\Hils\) and a unitary
  \(V\colon \Hils[L] \congto \Lt^2(X,\Hils,\nu)\) such that
  \(\Hils_x \neq0\) \(\nu\)\nb-almost everywhere and~\(V\)
  intertwines~\(\varphi_0\) and the representation of \(\Cont_0(X)\)
  on \(\Lt^2(X,\Hils,\nu)\) by pointwise multiplication.

  Let \(B \subseteq \Mult(\Cont_0(X,\FU[B]))\) be the sum of the ideal
  \(\Cont_0(X,\FU[B])\) and the image of \(\Cont_0(X)\).
  Then~\(\varphi\) extends to a nondegenerate representation of~\(B\).
  If \(b\in B\), then \(V \varphi(b) V^*\)
  commutes with the pointwise multiplication action of
  \(\Cont_0(X)\) on \(\Lt^2(X,\Hils,\nu)\).
  Then \(V \varphi(b) V^*\) is a ``decomposable'' operator by
  \cite{Dixmier:Algebres-Operateurs}*{Theorem~1 in Section~II.2 on
    p.~164}.
  This verifies the assumption of
  \cite{Dixmier:Cstar-algebres}*{Lemma~8.3.1 on p.~147}.
  The latter lemma gives representations \(\bar\varphi_x\colon B
  \to\Bound(\Hils_x)\) for all \(x\in X\) such that for all \(b\in
  B\) and \(\xi\in \Lt^2(X,\Hils,\nu)\), the element \(V\varphi(b)
  V^* \xi\) of \(\Lt^2(X,\Hils,\nu)\) is equal \(\nu\)\nb-almost
  everywhere to the section \(x\mapsto \bar\varphi_x(b) (\xi(x))\).
  In particular, the latter functions are \(\nu\)\nb-measurable for
  all \(b,\xi\).
  This makes~\((\bar\varphi_x)\) a measurable family of representations
  of~\(B\).
  Ideas from the proof of \cite{Dixmier:Cstar-algebres}*{Lemma~8.3.1}
  show that the family of representations~\((\bar\varphi_x)_{x\in X}\) is
  unique up to equality outside a \(\nu\)\nb-null set.

  If \(b\in \Cont_0(X)\), \(\xi\in \Lt^2(X,\Hils,\nu)\), then
  \((V \varphi(b) V^* \xi)(x) \in \Lt^2(X,\Hils,\nu)\) is also
  represented by \(b(x)\cdot \xi(x)\) because~\(V\) intertwines the
  representations of~\(B\).
  The uniqueness result just asserted shows that there is a
  \(\nu\)\nb-null set \(N\subseteq X\) with \(\bar\varphi_x(b) = b(x)\)
  for all \(b\in \Cont_0(X)\), \(x\in X\setminus N\).
  Replace~\(\Hils_x\) by~\(0\) in~\(N\).
  This does not change the Hilbert space \(\Lt^2(X,\Hils,\nu)\), and
  it arranges that~\(\bar\varphi_x|_{\Cont_0(X)}\) is evaluation at~\(x\)
  for all \(x\in X\).
  Since~\(\bar\varphi_x\) is a \Star{}homomorphism on~\(B\), it follows
  that there are \Star{}homomorphisms \(\varphi_x\colon \FU[B]_x \to
  \Bound(\Hils_x)\) for \(x\in X\) such that \(\bar\varphi_x(b) =
  \varphi_x(b(x))\) for all \(b\in \Cont_0(X,\FU[B])\), \(x\in X\).

  The representation~\(\varphi\) is nondegenerate by assumption.
  Then the representation~\(\varphi_x\) is nondegenerate for
  \(\nu\)\nb-almost all \(x\in X\) because of
  \cite{Dixmier:Cstar-algebras}*{Remark 8.1.5}.
  Setting \(\Hils_x=0\) on a null set, we may arrange
  for~\(\varphi_x\) to be nondegenerate for all \(x\in X\).

  The uniqueness claim for \(\nu,\Hils,\varphi\) follows from the
  claim about intertwiners, applied to the identity intertwiner.
  So it remains to prove the disintegration of an intertwiner~\(T\).
  Theorem~\ref{the:disintegration_commutative_unique} says that
  \([\nu]=[\nu']\) and that there are a conull subset \(X_1\subseteq
  X\) and unitary operators \(T_x \colon \Hils_x \congto \Hils_x'\)
  for \(x\in X_1\) such
  that~\eqref{eq:integrate_fiberwise_intertwiners} holds.
  Since~\(T\) intertwines \(\varphi\) and~\(\varphi'\), the families
  of fibrewise representations \(b\mapsto \varphi_x(b)\) and \(b\mapsto
  T_x \varphi'_x(b) T_x^*\) both disintegrate the same
  representation~\(\varphi\).
  The uniqueness result used above implies that the set~\(X_0\) of all
  \(x\in X_1\) with \(\varphi_x(b) = T_x \varphi'_x(b) T_x^*\) is still
  conull.
  Restricting~\(T_x\) to \(x\in X_0\) then gives a family of
  intertwiners with the required properties.
\end{proof}

\subsection{First stage of the proof: almost-everywhere disintegration}
\label{sec:first_stage_proof}

In this section, we apply Theorem~\ref{the:disintegrate_rep_of_bundle}
to Fell-bundle representations \((\varphi,U)\) and intertwiners
between two such representations \((\varphi,U)\) and
\((\varphi',U')\).
This is the first step towards proving
Theorem~\ref{the:disintegrate_Hilbert_space}.

Theorem~\ref{the:disintegrate_rep_of_bundle} applied to the
nondegenerate representation \(\varphi\colon \Cont_0(G^0,\FU) \to
\Bound(\Hils[L])\) produces a measure class~\(\nu\) on~\(G^0\), a
measurable field of Hilbert spaces \(\Hils = (\Hils_x)_{x\in G^0}\),
and a measurable field of nondegenerate representations
\(\varphi_x\colon \FU_x \to \Bound(\Hils_x)\) such that \(\Hils[L] =
\Lt^2(G^0,\Hils,\nu)\) and \((\varphi(a)\xi)(x) =
\varphi_x(a)(\xi(x))\) for all \(a\in \Cont_0(G^0,\FU)\),
\(\xi\in\Lt^2(G^0,\Hils,\nu)\), and \(x\in G^0\).
Here we have replaced the intertwiner \(\Hils[L] \cong
\Lt^2(G^0,\Hils,\nu)\) by an identity to simplify notation.

Next, let \(T\colon \Hils[L] \to \Hils[L]'\) be a unitary
intertwiner between two Fell-bundle representations \((\varphi,U)\)
and \((\varphi',U')\).
Let \((\nu,\Hils,\varphi_x)\)
and  \((\nu',\Hils',\varphi'_x)\) be the resulting measure, measurable
field of Hilbert spaces and measurable family of representations for
\(\varphi\) and~\(\varphi'\).
Since~\(T\) intertwines the representations \(\varphi\)
and~\(\varphi'\), Theorem~\ref{the:disintegrate_rep_of_bundle} shows
that the measure classes \([\nu]\) and~\([\nu']\) are equal and that
there are a conull set \(X_0\subseteq G^0\) and a family of unitary
intertwiners \(T_x\colon \Hils_x \to \Hils'_x\) from~\(\varphi_x\)
to~\(\varphi'_x\), such that~\(T\) is given
by~\eqref{eq:integrate_fiberwise_intertwiners}.
Let \(T,T'\colon \Hils[L] \to \Hils[L]'\) be two intertwiners
from \((\varphi,U)\) and \((\varphi',U')\), and let \(T_x,T_x'\colon
\Hils_x \to \Hils'_x\) be the resulting operators on fibres; these are
defined on certain conull sets \(X_0\), \(X_0'\).
It is clear that \(T=T'\) if and only if \(T_x = T'_x\) for
\(\nu\)\nb-almost all \(x\in X_0 \cap X_0'\).

Next, we consider the ingredient~\(U\) in a representation.
We may compute the tensor products \(\Lt^2(G,\A,\s,\tilde\alpha)
\otimes_\varphi \F\) and \(\Lt^2(G,\A\A^*,\rg,\alpha) \otimes_\varphi
\F\) as in Lemma~\ref{lem:compose_corr_from_measure-family}:
\begin{align*}
  \Lt^2(G,\A,\s,\tilde\alpha) \otimes_\varphi \F
  &\cong \Lt^2(G,\A\,\s^*\Hils,\nu\circ\tilde\alpha),\\
  \Lt^2(G,\A\A^*,\rg,\alpha) \otimes_\varphi \F
  &\cong \Lt^2(G,\A\A^*\, \rg^*\Hils,\nu\circ\alpha).
\end{align*}
Here \(\A\, \s^*\Hils\) and \(\A\A^*\,\rg^*\Hils\) are measurable
fields of Hilbert spaces over~\(G\) with the fibres
\(\A_g \otimes_{\FU_{\s(g)}} \Hils_{\s(g)}\) and
\(\varphi_{\rg(g)}(\A_g\A_g^*)\Hils_{\rg(g)} \cong \A_g\A_g^*
\otimes_{\FU_{\rg(g)}} \Hils_{\rg(g)}\), respectively;
and \(\nu\circ\tilde\alpha\) and \(\nu\circ\alpha\) denote the
measures on~\(G\) specified by the integration maps
\(\nu\circ\tilde\alpha\) and~\(\nu\circ\alpha\).
The representation of \(\Cont_0(G,\A\A^*)\) or \(\Cont_0(G,\rg^*\FU)\)
on these Hilbert spaces is the canonical one by pointwise
multiplication.

By  Theorem~\ref{the:disintegrate_rep_of_bundle}, the mere existence
of the intertwiner~\(U\) between these representations already implies
that the restriction of the measure \(\nu\circ\tilde\alpha\) to the
subset \(\supp \A\A^*\,\rg^*\Hils\) is equivalent to the restriction
of the measure \(\nu\circ\alpha\) to the subset \(\supp
\A\,\s^*\Hils\), which is defined as the set of all \(g\in G\) with
\(\A_g \otimes_{\FU_{\s(g)}} \Hils_{\s(g)} \neq0\).
A computation as in~\eqref{eq:measurable_rep_inclusions} shows that
\(\A_g \otimes_{\FU_{\s(g)}} \Hils_{\s(g)} \neq0\) if and only if
\(\varphi_{\s(g)}(\A_g^*\A_g) \Hils_{\s(g)} \neq0\).
Therefore, \(g\in G\) belongs to \(\supp \A\,\s^*\Hils\) if and only if
\(g^{-1} \in \supp \A\A^*\,\rg^*\Hils\).
Hence the restriction of~\(\nu\circ\alpha\) to \(\supp
\A\A^*\,\rg^*\Hils\) is a symmetric measure on~\(G\).
This shows that~\(\nu\) is quasi-invariant with respect to the
measurable representation of the Fell bundle that we are going to
build.
This notion of quasi-invariance depends only on the representation of
\(\FU = \A|_{G^0}\), which we have already constructed.

Theorem~\ref{the:disintegrate_rep_of_bundle} also disintegrates the
intertwiner~\(U\) between the two representations of
\(\Cont_0(G,\rg^*\FU)\), giving a subset \(X_1 \subseteq G\) and
unitary operators
\[
  U_g\colon \A_g \otimes_{\FU_{\s(g)}} \Hils_{\s(g)} \congto
  \A_g\A_g^* \otimes_{\FU_{\rg(g)}} \Hils_{\rg(g)}
\]
for all \(g\in X_1\) such that~\(X_1\) is conull for both measures
\(\nu\circ\alpha\) and \(\nu\circ\tilde\alpha\), the map~\(U\)
intertwines the left action of \(\Cont_0(G,\rg^*\FU)\), and
\[
  (U\xi)(g) = U_g(\xi(g))
  \biggl(\frac{\diff \nu\circ\alpha}
    {\diff \nu\circ\tilde\alpha} \biggr|_{\supp \A\A^*\,\rg^*\Hils} \biggr)^{1/2}
\]
for all \(\xi \in \Lt^2(G,\A\,\s^*\Hils,\nu\circ\tilde\alpha)\) and
almost all \(g\in X_1\).
It is useful to define~\(U_g\) also if both \(\A_g
\otimes_{\FU_{\s(g)}} \Hils_{\s(g)}=0\) and \(\A_g\A_g^*
\otimes_{\FU_{\rg(g)}} \Hils_{\rg(g)}=0\).
Then the set of \(g\in G\) for which~\(U_g\) is defined is conull for
both measures \(\nu\circ\alpha\) and \(\nu\circ\tilde\alpha\), even if
\(\supp \A_g\A_g^*\rg^*\Hils\) is small.

Let \(g\in X_1\).
The left action~\(\varphi_{\rg(g)}\) defines an
isomorphism from the codomain
\(\A_g\A_g^* \otimes_{\FU_{\rg(g)}} \Hils_{\rg(g)}\) of~\(U_g\) onto
\(\varphi_{\rg(g)}(\A_g \A_g^*) \Hils_{\rg(g)} \subseteq
\Hils_{\rg(g)}\).
Thus~\(U_g\) becomes an isometry
\(U_g'\colon \A_g\otimes_{\FU_{\s(g)}} \Hils_{\s(g)} \hookrightarrow
\Hils_{\rg(g)}\).
This produces a bounded linear map
\begin{equation}
  \label{eq:Sg_from_Ug}
  S_g\colon \A_g \to \Bound(\Hils_{\s(g)}, \Hils_{\rg(g)}),\qquad
  a\mapsto \bigl(\xi\mapsto U_g'(a\otimes \xi)\bigr)
  \quad\text{for }g\in X_1.
\end{equation}
Since~\(U_g'\) is a well defined isometry of right Hilbert modules and
intertwines the left actions of~\(\FU_{\rg(g)}\), the map~\(S_g\)
satisfies
\[
  S_g(c\cdot a\cdot b) = \varphi_{\rg(g)}(c) S_g(a) \varphi_{\s(g)}(b)
\]
for \(c\in \FU_{\rg(g)}\), \(a\in \A_g\), \(b\in \FU_{\s(g)}\) and
\(S_g(a_1)^* S_g(a_2) = \varphi_{\s(g)}(\braket{a_1}{a_2})\) for
\(a_1,a_2\in \A_g\).  And the linear span of~\(S_g(a)\xi\) for
\(a\in \A_g\), \(\xi\in\Hils_{\s(g)}\) is
\(\varphi_{\rg(g)}(\A_g \A_g^*) \Hils_{\rg(g)}\).  Conversely, any
linear map~\(S_g\) with these properties comes from a unique
unitary~\(U_g\) as above.  Namely,
\(U'_g(a\otimes \xi) \defeq S_g(a)(\xi)\) defines an isometry from
\(\A_g\otimes_{\FU_{\s(g)}} \Hils_{\s(g)}\) onto
\(\varphi_{\rg(g)}(\A_g \A_g^*) \Hils_{\rg(g)}\).  Then we
identify its range with
\(\A_g\A_g^* \otimes_{\FU_{\rg(g)}} \Hils_{\rg(g)}\) to get~\(U_g\).

So far, we have analysed the data of a Fell-bundle representation.
In addition, there is the condition \(d_2^*(U) d_0^*(U) = d_1^*(U)\)
that defines representations.  The unitaries \(d_j^*(U)\) intertwine
the left multiplication actions of \(\Cont_0(G^2,\Bun[I])\) on the
Hilbert spaces
\(\Lt^2(G^2,\E_k,v_k,\mu_k) \otimes_{\Cont_0(G^0,\FU)}
\Lt^2(G^0,\Hils,\nu)\) for \(k=0,1,2\).  We compute these tensor
products as above:
\[
  \Lt^2(G^2,\E_k,v_k,\mu_k) \otimes_{\Cont_0(G^0,\FU)}
  \Lt^2(G^0,\Hils,\nu) \cong \Lt^2(G^2,\E_k \otimes_{\FU}
  v_k^*\Hils, \nu\circ\mu_k)
\]
Here \(\E_k \otimes_{\FU} v_k^*\Hils\) for \(k=0,1,2\) is a
measurable field of
Hilbert spaces over~\(G^2\), whose fibre at \((g,h) \in G^2\) is
isomorphic to
\(\varphi_{\rg(g)}(\Bun[I]_{g,h}) \Hils_{\rg(g)}\),
\(\Bun[I]_{g,h} \cdot \A_g \otimes_{\FU_{\s(g)}} \Hils_{\s(g)}\), or
\[
  \Bun[I]_{g,h} \cdot \A_{g h} \otimes_{\FU_{\s(g)}} \Hils_{\s(h)}
  \cong \Bun[I]_{g,h} \cdot \A_g \otimes_{\FU_{\s(g)}} \A_h
  \otimes_{\FU_{\s(g)}} \Hils_{\s(h)} \cong \A_g \otimes_{\FU_{\s(g)}}
  \A_h \otimes_{\FU_{\s(g)}} \Hils_{\s(h)},
\]
respectively.  To reduce the size of formulas, we will omit the tensor
products in the following and briefly write these as
\(\Bun[I]_{g,h} \Hils_{\rg(g)}\),
\(\Bun[I]_{g,h} \A_g \Hils_{\s(g)}\), and
\(\Bun[I]_{g,h} \A_{g h} \Hils_{\s(h)} \cong \A_g \A_h
\Hils_{\s(h)}\).  The existence of the unitaries \(d_j^*(U)\) implies
that the measure classes
\(\nu\circ\mu_k|_{\supp(\E_k \otimes_{\FU} v_k^*\Hils)}\) for
\(k=0,1,2\) are equivalent.  And the transfer of~\(d_j^*(U)\) to a
unitary between the Hilbert spaces
\(\Lt^2(G^2,\E_k \otimes_{\FU} v_k^*\Hils, \nu\circ\mu_k)\) for
\(k\in \{0,1,2\}\setminus \{j\}\) gives the operator on
\(\Lt^2\)\nb-sections associated to almost everywhere isomorphisms
\(d_j^*(U)_{g,h}\) between the underlying measurable fields.
Namely,
\begin{alignat*}{2}
  d_0^*(U)_{g,h} &= \Id_{\A_g} \otimes U_h \colon&\quad
  \A_g \A_h \Hils_{\s(h)} &\to \Bun[I]_{g,h} \A_g \Hils_{\rg(h)},\\
  d_1^*(U)_{g,h} &= U_{g h}\colon&\quad
  \Bun[I]_{g,h} \A_{g h} \Hils_{\s(h)} &\to \Bun[I]_{g,h} \Hils_{\rg(g)},\\
  d_2^*(U)_{g,h} &= U_g\colon&\quad \Bun[I]_{g,h} \A_g \Hils_{\s(g)}
  &\to \Bun[I]_{g,h} \Hils_{\rg(g)}.
\end{alignat*}
Let \(m_{g,h}\colon \A_g \A_h \hookrightarrow \A_{g h}\) denote the
product in the Fell bundle.
The condition \(d_2^*(U) d_0^*(U) = d_1^*(U)\) is equivalent to
\begin{equation}
  \label{eq:representation_through_Ugh}
  U_g(\Id_{\A_g} \otimes U_h)
  = U_{g h} (m_{g,h}\otimes \Id_{\Hils_{\s(h)}})
  \quad \text{in }
  \Bound(\Bun[I]_{g,h} \A_g \A_h \Hils_{\s(h)},
  \Bun[I]_{g,h} \Hils_{\rg(g)})
\end{equation}
for \(\nu\circ \mu_0\)\nb-almost all \((g,h)\in G^2\).  This makes
sense because~\(U_g\) is defined for \(g\in X_1\) and~\(X_1\) is
conull for \(\nu\circ\alpha\) and \(\nu\circ\tilde\alpha\).  The
condition~\eqref{eq:representation_through_Ugh} is empty if
\(\Bun[I]_{g,h} \Hils_{\rg(g)} = 0\) or
\(\Bun[I]_{g,h}\A_g \A_h \Hils_{\s(h)}=0\), and the restrictions of
the measures \(\nu\circ \mu_j\) for \(j=0,1,2\) to the set where these
Hilbert spaces are nonzero are equivalent measures.  This is why it
does not matter which of these three measures we use and why we do not
have to restrict our statements to the supports of the relevant fields
of Hilbert spaces over~\(G^2\).

The condition~\eqref{eq:representation_through_Ugh} holds if and only
if both unitaries agree on \(a_1 \otimes a_2 \otimes \xi\) for all
\(a_1\in \A_g\), \(a_2\in \A_h\), \(\xi\in \Hils_{\s(h)}\).  By
definition,
\begin{align*}
  U_g(\Id_{\A_g} \otimes U_h) (a_1 \otimes a_2 \otimes \xi)
  &= U_g(a_1 \otimes S_h(a_2)\xi)
    = S_g(a_1) S_h(a_2)\xi,\\
  U_{g h} (m_{g,h}\otimes \Id_{\Hils_{\s(h)}})
  (a_1 \otimes a_2 \otimes \xi)
  &= U_{g h}(a_1\cdot a_2 \otimes \xi)
    = S_{g h}(a_1\cdot a_2)\xi.
\end{align*}
So~\eqref{eq:representation_through_Ugh} is true if and only if
\begin{equation}
  \label{eq:S_g-multiplicative1}
  S_g(a_1) S_h(a_2)= S_{g h}(a_1 \cdot a_2)
  \quad \text{in } \Bound(\Hils_{\s(h)}, \Hils_{\rg(g)})
\end{equation}
for all \(a_1\in \A_g\), \(a_2\in \A_h\).  Our computations above
show the following:

\begin{lemma}
  \label{lem:representation_almost_everywhere}
  If~\(S_g\) comes from a representation \((\varphi,U)\),
  then~\eqref{eq:S_g-multiplicative1} holds for all \((g,h)\in X_2\)
  for a subset \(X_2\subseteq G^2\) that is conull for the
  measures \(\nu\circ \mu_j\) for \(j=0,1,2\).
\end{lemma}

Next, let \(T\colon \Hils[L] \to \Hils[L]'\) be a unitary intertwiner
between two Fell-bundle representations \((\varphi,U)\) and
\((\varphi',U')\).
We describe the representations by the data \((\nu,\Hils_x,U_g,S_g)\) and
\((\nu',\Hils'_x,U'_g,S'_g)\) as above.
We have already seen that it follows that \([\nu]=[\nu']\) and that
there are a \(\nu\)-conull subset \(X_0^T\subseteq G^0\) and unitary
intertwiners \(T_x\colon \Hils_x \to \Hils'_x\) from~\(\varphi_x\)
to~\(\varphi'_x\) for \(x\in X_0^T\)
such that~\(T\) is given
by~\eqref{eq:integrate_fiberwise_intertwiners}.
Since~\(T\) is an intertwiner, we know that the operators
\((\Id\otimes T)U\) and \(U'(\Id\otimes T)\) from
\(\Lt^2(G,\A,\s,\tilde\alpha) \otimes_{\varphi} \Lt^2(G^0,\Hils,\nu)\)
to \(\Lt^2(G,\A\A^*,\rg,\alpha) \otimes_{\varphi'}
\Lt^2(G^0,\Hils',\nu')\) are equal.
When we rewrite these tensor products as \(\Lt^2\)\nb-section spaces
of fields of Hilbert spaces over~\(G\), then we may apply
Theorem~\ref{the:disintegrate_rep_of_bundle}.
The equality of the two intertwiners becomes equivalent to
\begin{equation}
  \label{eq:T_intertwines_U}
  (\Id_{\A_g\A_g^*} \otimes T_{\rg(g)}) U_g
  = U'_g (\Id_{\A_g} \otimes T_{\s(g)})
\end{equation}
for \(\nu\circ \alpha\)-almost all \(g\in G\).
Equivalently, the set \(X_1\subseteq G\) of all \(g\in G\) for which
this holds is \(\nu\circ \alpha\)-conull; here it does not matter
whether we use \(\nu\circ \alpha\) or \(\nu'\circ \alpha\) because the
above equality is empty outside \(\supp \A\A^*\rg^*\Hils\), where these
two measures are equivalent.
When we express \(U_g\) and~\(U'_g\) through \(S_g\) and~\(S'_g\),
then~\eqref{eq:T_intertwines_U} becomes equivalent to
\begin{equation}
  \label{eq:T_intertwines_S}
  T_{\rg(g)}S_g(a) = S'_g(a) T_{\s(g)}
  \qquad \text{for all }a\in \A_g.
\end{equation}
So this equality holds for all \(g\in X_1\).

Finally, given two parallel intertwiners \(T,T'\), we have already
seen that \(T=T'\) holds if and only if \(T_x = T'_x\) for almost all
\(x\in G^0\).

So far, we have disintegrated Fell-bundle representations on a Hilbert
space and intertwiners between them.
We have, however, not quite achieved what
Theorem~\ref{the:disintegrate_Hilbert_space} claims.
It remains to address several issues.
So far, the operators~\(S_g\) are only defined for \emph{almost} all
\(g\in G\), and~\eqref{eq:S_g-multiplicative1} only holds for
\emph{almost}  all \((g,h)\in G^2\).
In a measurable Fell-bundle representation, however, we must
define~\(S_g\) for all \(g\in G\) and
verify~\eqref{eq:S_g-multiplicative1} on all of~\(G^2\).
In addition, we must arrange that \(\varphi_x = S_{u(x)}\) for all
\(x\in G^0\), and replace the condition
\(S_g(a_1)^* S_g(a_2) = \varphi_{u\s(g)}(a_1^* \cdot a_2)\)
by the simpler relation \(S_g(a)^* = S_{g^{-1}}(a^*)\).

Similarly, for the disintegration of an intertwiner, we have only
arranged for the equation \(T_{\rg(g)}S_g(a) = S'_g(a) T_{\s(g)}\) to
hold for \emph{almost} all \(g\in G\).
We would like there to be a conull subset \(X\subseteq G\) such that
this holds for all \(g\in G\) with \(\s(g),\rg(g) \in X\).
In addition, we would like to make the exceptional sets in~\(G^0\) in
our three statements \(G\)\nb-invariant, at least if the Fell bundle
is saturated.

\subsection{Second stage of the proof: reducing exceptional null sets
  to~\texorpdfstring{$G^0$}{G⁰}}
\label{sec:improve_null}

In this subsection, we will improve the conull sets in \(G\)
and~\(G^2\) where our statements work, so that they all come from
suitable conull subsets \(X\subseteq G^0\).
That is, for the representation~\(S_g\), we arrange that it is defined
whenever \(\s(g),\rg(g)\in X\) and
satisfies~\eqref{eq:S_g-multiplicative1} whenever \((g,h)\in G^2\)
are such that \(\rg(g),\s(g) = \rg(h), \s(h) \in X\).
For an intertwiner, we arrange for \(T_{\rg(g)}S_g(a) = S'_g(a)
T_{\s(g)}\) to hold for all \(g\in G\) with \(\s(g),\rg(g) \in X\) for
some conull subset \(X\subseteq G^0\) depending on~\(T\).
The arguments in this subsection are based on ideas of the proof of
\cite{Ramsay:Virtual_groups}*{Theorem~5.1} by Ramsay.

Since the argument becomes significantly simpler for saturated Fell
bundles, we restrict to this case until further notice.
Then the support of \(\A_g\A_g^*\,\rg^*\Hils\) consists of all \(g\in
G\) with \(\Hils_{\rg(g)}\neq0\), which is conull in~\(G\).
Hence, the measure~\(\nu\) is genuinely quasi-invariant, that is,
\(\nu\circ\alpha\) and \(\nu\circ\tilde\alpha\) are equivalent.
Then the measures
\(\nu\circ\mu_j\) on~\(G^2\) for \(j=0,1,2\) are equivalent as well.
In the following, whenever we say that a subset of \(G^0\), \(G\),
or~\(G^2\) is conull, or that a statement holds almost everywhere,
this is with respect to the measure classes~\([\nu]\),
\([\nu\circ\alpha]=[\nu\circ\tilde\alpha]\), and \([\nu\circ\mu_j]\),
respectively.
In the saturated case, \(\Bun[I]_{g,h} = \FU_{\rg(g)}\) for all
\((g,h)\in G^2\).
So we may omit this field of ideals in all our computations.
Since~\(\A\) is saturated, we have \(\A_g^*\A_g = \FU_{\s(g)}\).
So~\eqref{eq:representation_through_Ugh} holds if and only if~\(U_h\)
is equal to the composite unitary
\begin{multline}
  \label{eq:representation_through_Ugh_2}
  \A_h \Hils_{\s(h)}
  \cong \A_g^* \A_g \A_h \Hils_{\s(h)}
  \cong \A_g^* \A_{g h} \Hils_{\s(h)}
  \xrightarrow{\Id_{\A_g^*} \otimes_{\FU_{\rg(g)}} U_{g h}}
  \A_g^* \A_{g h} \A_{g h}^* \Hils_{\rg(g)}
  \\\cong \A_g^* \A_g \A_g^* \Hils_{\rg(g)}
  \xrightarrow{\Id_{\A_g^*} \otimes_{\FU_{\rg(g)}} U_g^*}
  \A_g^* \A_g \Hils_{\rg(h)}
  \cong \A_h \A_h^*\Hils_{\rg(h)}.
\end{multline}
Here the various unlabeled isomorphisms follow from
Lemma~\ref{lem:equality_of_products} because the Fell bundle is
saturated.

We know that the unitary in~\eqref{eq:representation_through_Ugh_2}
is~\(U_h\) for almost all \((g,h)\in G^2\).
By Fubini's Theorem, this says that there is a conull subset
\(X'_1\subseteq X_1\) such that for each \(h\in X'_1\), the unitary
in~\eqref{eq:representation_through_Ugh_2} is equal to~\(U_h\) for
almost all \(g\in G_{\rg(h)}\).
In particular, if \(h\in X_1'\), then the function~\(U'_h(g)\) that
maps \(g\in G_{\rg(h)}\) to the unitary
in~\eqref{eq:representation_through_Ugh_2} is constant
almost everywhere with value~\(U_h\).
Let \(X''_1\subseteq G\) be the set of all \(h\in G\) such that this
function~\(U'_h\) is constant \(\tilde\alpha_{\rg(h)}\)-almost
everywhere, and let~\(\bar{U}_h\) be its almost constant value.
The subset~\(X''_1\) contains~\(X'_1\).
So it is conull and, \emph{a fortiori}, measurable.
Since \(\bar{U}_h = U_h\) for \(h\in X'_1\), the
function~\(\bar{U}_h\) is equal to~\(U_h\)
\(\tilde\alpha_{\rg(h)}\)-almost everywhere.
This makes it measurable.

\begin{lemma}
  \label{lem:doubleprime_saturated}
  Assume the Fell bundle to be saturated.  Let \((h,k)\in G^2\).  If
  \(h,k\in X''_1\), then \(h\cdot k \in X''_1\) and
  \[
    \bar{U}_{h\cdot k} (m_{h,k}\otimes \Id_{\Hils_{\s(k)}}) =
    \bar{U}_h (\Id_{\A_h}\otimes\bar{U}_k).
  \]
\end{lemma}

This lemma is a key step in the improvement of~\(U\) to a measurable
Fell-bundle representation.  The rough idea of the proof is that the
active ingredients in \(U'_h(g)\) and \(U'_k(g h)\), for
\(g\in G\) with \(\s(g) = \rg(h)\), are \(U_g^*, U_{g h}\) and
\(U_{g h}^*, U_{g h k}\), respectively.  These are composed with
various natural unitaries produced from the Fell-bundle structure as
in Lemma~\ref{lem:equality_of_products}.  When we compose, the factors
\(U_{g h}\) and~\(U_{g h}^*\) cancel, leaving the active ingredients
\(U_g^*, U_{g h k}\) of~\(U'_{h k}(g)\).  We do not make the proof
more precise because, anyway, we need to generalise the result to the
nonsaturated case.  And we prefer to extend the statement first so as
to need only one technical proof.

So we return to the general (nonsaturated) case.
It is more complicated because some of the maps
in~\eqref{eq:representation_through_Ugh_2} are no longer unitary.
Therefore, the domains and codomains of~\(U'_h(g)\) now depend
on~\(g\), and we have to clarify what the ``almost constant value''
\(\bar{U}_h\) of this function of~\(g\) could mean.
This is done in Figure~\ref{fig:def_barUh},
\begin{figure}[htbp]
  \[
    \begin{tikzcd}
      \A_g^* \A_g \A_h \Hils_{\s(h)}
      \ar[r, hookrightarrow] \ar[d, "\cong"']
      \ar[dddd, start anchor=west, end anchor=west, bend right, dashed, "U'_h(g)"'] &
      \A_h \Hils_{\s(h)}
      \ar[dddd, "\cong"', dashed, "\bar{U}_h"]\\
      \A_g^* \A_{g h} \Hils_{\s(h)}
      \ar[d, "\cong"', "\Id_{\A_g^*} \otimes U_{g h}"] & \\
      \A_g^* \A_{g h} \A_{g h}^* \Hils_{\rg(g)}
      \ar[d, "\cong"'] & \\
      \A_h \A_h^* \A_g^* \A_g \A_g^* \Hils_{\rg(g)}
      \ar[d, "\cong"', "\Id_{\A_h \A_h^* \A_g^*}\otimes U_g^*"] & \\
      \A_h \A_h^* \A_g^* \A_g \Hils_{\rg(h)}
      \ar[r, hookrightarrow] &
      \A_h \A_h^* \Hils_{\rg(h)}
    \end{tikzcd}
  \]
  \caption{The first column defines a unitary \(U'_h(g)\) for all
    \((g,h)\in G^2\) with \(gh,g\in X_1\).  Here the unlabelled
    unitaries are built from
    the isomorphisms in Lemma~\ref{lem:equality_of_products} and the
    equality \(I\cdot J = J\cdot I\) for two ideals in a
    \(\Cst\)\nb-algebra.  The two horizontal isometries are induced
    by the multiplication in the Fell bundle.  The
    unitary~\(\bar{U}_h\) is required to make the whole diagram
    commute for \(\tilde\alpha_{\rg(h)}\)-almost all
    \(g \in G_{\rg(h)}\).  This condition uniquely determines
\(\bar{U}_h\) whenever it exists.}
  \label{fig:def_barUh}
\end{figure}%
which defines unitaries
\[
  U'_h(g)\colon \A_g^* \A_g \A_h \Hils_{\s(h)}
  \congto \A_h \A_h^* \A_g^* \A_g \Hils_{\rg(h)}
\]
for almost all \((g,h)\in G^2\) and a unitary
\[
  \bar{U}_h\colon \A_h \Hils_{\s(h)} \to \A_h \A_h^* \Hils_{\rg(h)}
\]
for certain \(h\in G\).  More precisely, the first vertical
isomorphism in Figure~\ref{fig:def_barUh} comes directly from
Lemma~\ref{lem:equality_of_products} and the third one is the
composite of the chain of isomorphisms
\begin{multline*}
  \A_g^* \A_{g h} \A_{g h}^*
  = \A_{g^{-1}} \A_{g h} \A_{(g h)^{-1}}
  \cong \A_{g^{-1}} \A_g \A_h \A_{h^{-1} g^{-1}}
  \\\cong \A_{g^{-1}} \A_g \A_h \A_{h^{-1}} \A_{g^{-1}}
  \cong \A_h \A_{h^{-1}}  \A_{g^{-1}} \A_g \A_{g^{-1}}
  = \A_h \A_h^* \A_g^* \A_g \A_g^*,
\end{multline*}
which all come from Lemma~\ref{lem:equality_of_products}.  The
following lemma shows that~\(\bar{U}_h\) is well defined:

\begin{lemma}
  \label{lem:bar-U-welldefined}
  There is at most one bounded operator~\(\bar{U}_h\) that makes the
  diagram in Figure~\textup{\ref{fig:def_barUh}} commute for almost all
  \(g\in G_{\rg(h)}\).  This operator is unitary if it exists.  It
  exists and is equal to~\(U_h\) for almost all \(h\in G\).  The
  function~\(h\mapsto \bar{U}_h\) is measurable as a family of
operators between the measurable fields \(\A\,\s^*\Hils\) and \(\A\A^*\,\rg^*\Hils\).
\end{lemma}

\begin{proof}
  Fix \(h\in G\) and~\(\bar{U}_h\) such that the diagram in
  Figure~\ref{fig:def_barUh} commutes for almost all
  \(g\in G_{\rg(h)}\).  Take \(\xi\in \A_h\Hils_{\s(h)}\).  By the
  Cohen--Hewitt Factorisation Theorem, we may write~\(\xi\) as
  \(a_1\cdot a_2\cdot \xi'\) with \(a_1,a_2\in \FU_{\rg(h)}\),
  \(\xi'\in \A_h\Hils_{\s(h)}\).  There are continuous sections
  \(b_1,b_2\) of~\(\A\) with \(b_j(u\rg(h)) = a_j\) for \(j=1,2\).
  Then
  \[
    b_1(g^{-1})\cdot b_2(g) \cdot \xi'
    \in \A_{g^{-1}} \A_g \A_h \Hils_{\s(h)}
    = \A_g^* \A_g \A_h \Hils_{\s(h)}.
  \]
  This converges to \(a_1\cdot a_2\cdot \xi'\) if~\(g\) converges
  to~\(u\rg(h)\) in~\(G_{\rg(h)}\).  Since the
  measure~\(\tilde\alpha_{\rg(h)}\) has full support, any conull
  subset of~\(G_{\rg(h)}\) has~\(u\rg(h)\) as an accumulation point.
  Since~\(G\) is second countable, there is a sequence~\((g_n)\)
  in~\(G_{\rg(h)}\) that converges to~\(u\rg(h)\) and such that the
  diagram in Figure~\ref{fig:def_barUh} commutes for \(g=g_n\) for
  all~\(n\).  Since~\(\bar{U}_h\) is bounded, it follows that
  \(\bar{U}_h(\xi) = \lim U'_h(g_n)(b_1(g_n^{-1})\cdot b_2(g_n) \cdot
  \xi')\).  This formula shows that~\(\bar{U}_h\) is unique.  Since
  all~\(U'_h(g_h)\) are unitary, it also follows that~\(\bar{U}_h\)
  is isometric.  The range of~\(\bar{U}_h\) contains
  \(\A_h \A_h^* \A_{g_n}^* \A_{g_n} \Hils_{\s(h)}\) for all
  \(n\in\N\).  An argument as above, starting with
  \(\Hils_{\s(h)} = \FU_{\s(h)} \Hils_{\s(h)}\), shows that the
  union of these subspaces is dense in
  \(\A_h \A_h^* \Hils_{\s(h)}\).  Thus~\(\bar{U}_h\) is unitary if it
  exists.

  We know that~\eqref{eq:representation_through_Ugh} holds for
  almost all \((g,h)\in G^2\).  By Fubini's Theorem, this says that,
  for almost all \(h\in G\), \eqref{eq:representation_through_Ugh}
  holds for almost all \(g\in G_{\rg(h)}\).  We claim that
  \(\bar{U}_h = U_h\) works for all such~\(h\).  To see this, we
  rewrite~\eqref{eq:representation_through_Ugh} as
  \(\Id_{\A_g} \otimes U_h = U_g^* U_{g h} (m_{g,h}\otimes
  \Id_{\Hils_{\s(h)}})\).  Then tensor this on the left
  with~\(\Id_{\A_g^*}\).  This gives an equality of operators on
  \(\A_g^* \Bun[I]_{g,h} \A_g \A_h\Hils_{\s(h)}\).  Here
  \(\Bun[I]_{g,h} \A_g \A_h \defeq \A_g \A_h \A_h^* \A_g^* \A_g \A_h
  = \A_g \A_h\).  And the equality is exactly the commuting diagram
  in Figure~\ref{fig:def_barUh}; most of the unlabelled isomorphisms
  in Figure~\ref{fig:def_barUh} are only implicit
  in~\eqref{eq:representation_through_Ugh}.

  Since~\(h\mapsto U_h\) is a measurable family of operators, so is~\(h\mapsto \bar{U}_h\) as
  they are almost everywhere equal.
\end{proof}

\begin{lemma}
  Let \((h,k)\in G^2\).  If \(h,k,h\cdot k\in X''_1\), then
  \[
    \bar{U}_{h\cdot k} (m_{h,k}\otimes \Id_{\Hils_{\s(k)}}) =
    \bar{U}_h (\Id_{\A_h}\otimes\bar{U}_k).
  \]
  Let \((h_n)_{n\in\N}\) be a sequence in~\(G_{\rg(k)}\) with
  \(\lim h_n = u\rg(k)\) and \(h_n^{-1}, h_n k\in X''_1\) for all
  \(n\in\N\).  Then \(k \in X''_1\).
  \label{lem:doubleprime}
\end{lemma}

\begin{proof}[Proof of Lemmas
  \textup{\ref{lem:doubleprime_saturated}}
  and~\textup{\ref{lem:doubleprime}}]
  Let \(h,k\in X_1''\).  By assumption, the diagram in
  Figure~\ref{fig:def_barUh} commutes for \((g,h)\) for a conull set
  of \(g\in G_{\rg(h)}\) and for \((x,k)\) for a conull set of
  \(x\in G_{\rg(k)}\).  Since the measure family~\(\tilde\alpha\) is
  right invariant, the set of \(g\in G_{\rg(h)}\) such that the
  diagram in Figure~\ref{fig:def_barUh} commutes for \(g\cdot h,k\) is
  also conull.  Thus the set of~\(g\) such that the diagram in
  Figure~\ref{fig:def_barUh} commutes for \(g,h\) and \(g h,k\) is
  conull.  Next, we show that the diagram in
  Figure~\ref{fig:lem_doubleprime} commutes for such \((g,h,k)\).
  \begin{figure}[htbp]
    \[
      \begin{tikzcd}[column sep=small]
        \A_{h k} \Hils_x \ar[r, leftarrow]
        \ar[ddddddd, rightarrow, dashed, "\bar{U}_{h k}"] &
        \A_g^* \A_g \A_{h k} \Hils_x \ar[r, leftarrow]
        \ar[d, "\cong"'] &
        \A_g^* \A_g \A_h\A_k \Hils_x \ar[r, rightarrow]
        \ar[d, "\cong"'] &
        \A_h\A_k \Hils_x
        \ar[dddd, rightarrow, "\bar{U}_k"] \\
        & \A_g^* \A_{g h k} \Hils_x \ar[r, leftarrow]
        \ar[d, "\cong"', "U_{g h k}"] &
        \A_h \A_{g h}^* \A_{g h k} \Hils_x
        \ar[d, "\cong"', "U_{g h k}"]  \ar[ru]&\\
        & \A_g^* \A_{g h k} \A_{g h k}^* \Hils_w
        \ar[r, leftarrow] \ar[d, "\cong"'] &
        \A_h \A_{g h}^* \A_{g h k} \A_{g h k}^* \Hils_w
        \ar[d, "\cong"'] \\
        & \A_{h k} \A_{h k}^* \A_g^* \Hils_w \ar[r, leftarrow]
        \ar[ddd, "\cong"'] &
        \A_h \A_k \A_k^* \A_h^* \A_g^* \A_{g h} \A_{g h}^* \Hils_w
        \ar[d, "U_{g h}^*", "\cong"'] \\
        & &
        \A_h \A_k \A_k^* \A_h^* \A_g^* \A_{g h} \Hils_y
        \ar[r, rightarrow] \ar[d, "\cong"', "U_{g h}"] &
        \A_h \Hils_y \ar[ddd, rightarrow, "\bar{U}_h"]\\
        & &
        \A_h \A_k \A_k^* \A_h^* \A_g^* \A_{g h} \A_{g h}^* \Hils_w
        \ar[d, "\cong"'] \\
        & \A_{h k} \A_{h k}^* \A_g^* \A_g \A_g^* \Hils_w
        \ar[r, leftarrow] \ar[d, "\cong"', "U_g^*"] &
        \A_h \A_k \A_k^* \A_h^* \A_g^* \A_g \A_g^* \Hils_w
        \ar[d, "U_g^*", "\cong"'] \\
        \Hils_z \ar[r, leftarrow] &
        \A_{h k} \A_{h k}^* \A_g^* \A_g \Hils_z \ar[r, leftarrow]&
        \A_h \A_k \A_k^* \A_h^* \A_g^* \A_g \Hils_z \ar[r, rightarrow] &
        \Hils_z
      \end{tikzcd}
    \]
    \caption{The gist of the proof of Lemma~\ref{lem:doubleprime}.
      The data is a chain of three composable arrows
      \(x\xrightarrow{k} y\xrightarrow{h} z\xrightarrow{g} w\).  The
      unlabelled arrows in the diagram are induced by the
      multiplication in the Fell bundle, tensored with the identity on
      the appropriate fibre of~\(\Hils\).  When marked
      with~``\(\cong\)'', then they are invertible by a repeated
      application of Lemma~\ref{lem:equality_of_products}.  Otherwise,
      they are only isometries.  The arrows marked with \(U_l\)
      or~\(U_l^*\) for an arrow~\(l\) in~\(G\) act by a tensor product
      of an identity map with the arrow \(U_l\) or~\(U_l^*\) acting on
      the last two or three tensor factors.}
    \label{fig:lem_doubleprime}
  \end{figure}
  The unlabelled arrows become identities when we realise any product
  of fibres \(\A_{g_1}\dotsm \A_{g_m}\) in the Fell bundle as a
  subspace of~\(\A_{g_1\dotsm g_m}\) as in
  Remark~\ref{rem:multiply_fibres}.  This is why all squares in
  Figure~\ref{fig:lem_doubleprime} that involve only these maps
  commute.  The rectangle involving \(U_{g h}\) and~\(U_{g h}^*\)
  commutes for the same reason.  The squares involving \(U_{g h k}\)
  and~\(U_g^*\) commute because the multiplication maps that give the
  horizontal arrows in these squares only act on the first factors
  where \(U_{g h k}\) and~\(U_g^*\) are identities.  This explains why
  the diagram formed by the middle two columns commutes.  The diagram
  formed by the right two columns commutes because of the choice of
  \(g,h,k\).

  In the saturated case, the horizontal maps are invertible as well
  because they are made from the multiplication in the Fell bundle.
  Then there is a unique unitary~\(\bar{U}_{h k}\) making the whole
  diagram commute, and we only need to assume \(h,k \in X''_1\) to get
  this.
  This finishes the proof of Lemma~\ref{lem:doubleprime_saturated}.

  In the nonsaturated case, we assume \(h\cdot k\in X''_1\) to ensure
  that there is~\(\bar{U}_{h k}\) making the whole diagram commute.
  The diagram in Figure~\ref{fig:def_barUh} shows that the isometries
  \(\bar{U}_{h\cdot k} (m_{h,k}\otimes \Id_{\Hils_{\s(k)}})\) and
  \(\bar{U}_h (\Id_{\A_h}\otimes\bar{U}_k)\) restrict to the same
  operator on \(\A_g^* \A_g \A_h \A_k \Hils_x\) for all~\(g\) as
  above.  Then an argument as in the proof of
  Lemma~\ref{lem:bar-U-welldefined} shows that
  \(\bar{U}_{h\cdot k} (m_{h,k}\otimes \Id_{\Hils_{\s(k)}})\) and
  \(\bar{U}_h (\Id_{\A_h}\otimes\bar{U}_k)\) are equal on
  \(\A_h \A_k \Hils_x\).

  Now we turn to the second statement in Lemma~\ref{lem:doubleprime}.
  It addresses the problem that the argument above requires
  \(h\cdot k\in X''_1\) as an assumption.  Let \(g,k,h_n\in G\) with
  \(\s(g) = \s(h_n) = \rg(k)\) for \(n\in\N\) and assume
  \(h_n^{-1}\in X''_1\) and \(h_n\cdot k \in X''_1\) for all
  \(n\in\N\).  When we substite \((g,h_n^{-1}, h_n k)\in G^2\) with
  \(\s(g) = \s(h_n) = \rg(k)\) for \((g,h,k)\), then the last three
  columns in Figure~\ref{fig:lem_doubleprime} form a commuting diagram
  for a conull set of \(g\in G_{\rg(k)}\).  Since~\(\N\) is countable,
  the set of \(g\in G_{\rg(k)}\) for which this happens for all
  \(n\in\N\) is still conull.
  The second column in Figure~\ref{fig:lem_doubleprime} depends only
  on \(g\) and \(h_n^{-1} \cdot h_n k = k\), but not on \(n\in\N\).
  The fourth column depends only on \(h_n k\) and~\(h_n^{-1}\), but
  not on~\(g\).
  If \(n\to\infty\), then \(h_n \to u(\rg(k))\).  So the domain
  \(\A_{h_n^{-1}} \A_{h_n k} \Hils_{\s(k)}\) ``converges'' towards
  \(\A_k \Hils_{\s(k)}\).
  More precisely, any element in \(\A_g^* \A_g \A_k\Hils_{\s(k)}\),
  the domain of the unitary in the second column, is a limit of a
  sequence \((a_n)\) with~\(a_n\) in the domain \(\A_{h_n^{-1}}
  \A_{h_n k} \Hils_{\s(k)}\) of the isometry in the third column.
  Since the diagram commutes, it follows that the vertical arrow in
  the second column does not depend on~\(g\), except for its domain.
  And this says that the unitary~\(\bar{U}_k\) exists, that is, \(k\in X_1''\).
\end{proof}

\begin{lemma}
  \label{lem:X0_exists}
  There is a \(\nu\)\nb-conull subset \(X_0\subseteq G^0\) with
  \[
    G|_{X_0}
    \defeq \setgiven{g\in G}{\rg(g),\s(g)\in X_0}
    \subseteq X''_1 \cap \inv(X''_1)
    = \setgiven{g\in G}{g,g^{-1} \in X''_1}.
  \]
\end{lemma}

\begin{proof}
  To explain the idea of the proof, we again assume for a moment
  that our Fell bundle is saturated.
  Then~\(\nu\) is quasi-invariant.
  Therefore, \(\inv(X''_1) \subseteq G\) inherits the property of
  being conull from~\(X''_1\).
  Hence \(X''_1 \cap \inv(X''_1)\) is conull in~\(G\).
  Since~\(X_1''\) is closed under multiplication by
  Lemma~\ref{lem:doubleprime_saturated}, the statement now follows
  from \cite{Ramsay:Virtual_groups}*{Lemma~5.2}.
  The following argument generalises Ramsay's argument, using only the
  weaker statement in Lemma~\ref{lem:doubleprime}.

  If \(\A_h \Hils_{\s(h)} = 0\), then the domain of~\(\bar{U}_h\) is
  the zero Hilbert space, so that~\(\bar{U}_h\) automatically
  exists.
  Thus the complement of the support of \(\supp \A\,\s^*\Hils\) always
  belongs to~\(X_1''\).
  Since~\(\nu\) is symmetric relative to our representation, it
  follows that \(X''_1 \cap \inv(X''_1)\) is still conull.
  Now the following lemma finishes the proof, using the conclusion of
  Lemma~\ref{lem:doubleprime}.
  We formulated it separately because to make it easier to use the
  same idea again later.
\end{proof}

\begin{lemma}
  \label{lem:Ramsay_1}
  Let \(X_1''\subseteq G\) be a subset such that
  \(X''_1 \cap \inv(X''_1)\) is still conull.
  Let~\(X_0\) be the set of all \(x\in G^0\) such that \(g\in X''_1
  \cap \inv(X''_1)\) for \(\alpha^x\)\nb-almost all \(g\in G^x\).
  This set is conull in~\(G^0\).
  Assume that an element \(k\in G\) belongs to~\(X_1''\) once there is
  a sequence~\((h_n)\) in \(G^{\rg(k)}\) with \(h_n,h_n^{-1} k\in X_1''\).
  Then \(G|_{X_0} \subseteq X_1''\).
\end{lemma}

\begin{proof}
  Since~\(X_1''\) is conull for \(\nu\circ\alpha\) and
  \(\nu\circ\tilde\alpha\), Fubini's Theorem implies that the sets of
  \(x\in G^0\) with \(g\in X_1''\) for almost all \(g\in G^x\) and
  with \(g\in X_1''\) for almost all \(g\in G_x\) are both \(\nu\)\nb-conull.
 Thus~\(X_0\) is conull.

  Let \(k\in G|_{X_0}\), that is, \(\rg(k),\s(k)\in X_0\).
  Then \(X''_1 \cap G_{\s(k)}\) is \(\tilde\alpha_{\s(k)}\)\nb-conull.
  Since the family of measures~\(\tilde\alpha\) is invariant on the
  right, the set of \(h\in G_{\rg(k)}\) with \(h k\in X''_1\) is
  \(\tilde\alpha_{\rg(k)}\)\nb-conull as well.
  So is the set of \(h\in G_{\rg(k)}\) with \(h^{-1}\in X''_1\)
  because \(\rg(k) \in X_0\).
  The intersection of these two sets remains
  \(\tilde\alpha_{\rg(k)}\)\nb-conull.
  Therefore, it is dense.
  So there is a sequence \((h_n)_{n\in\N}\) in~\(G_{\rg(k)}\) with
  \(h_n^{-1}, h_n k \in X''_1\) and \(\lim h_n = u\rg(k)\).
  Now the assumption implies \(k\in X''_1\).
\end{proof}

Equation~\eqref{eq:representation_through_Ugh_2} holds for the
unitaries~\(\bar{U}_g\) instead of~\(U_g\) for all \(g,h\in X''_1\).
As in~\eqref{eq:Sg_from_Ug}, we turn the unitaries~\(\bar{U}_g\)
into maps
\(S_g \colon \A_g \to \Bound(\Hils_{\s(g)}, \Hils_{\rg(g)})\).
These maps satisfy~\eqref{eq:S_g-multiplicative1} for all
\(g,h\in X''_1\) and hence for all \(g,h \in G|_{X_0}\) by
Lemma~\ref{lem:X0_exists}.

\begin{lemma}
  \label{lem:phi_from_S}
  If \(x\in X_0\), then \(\varphi_x = S_{u(x)}\).  If
  \(g\in G|_{X_0}\), then \(S_g(a)^* = S_{g^{-1}}(a^*)\).
\end{lemma}

\begin{proof}
  The unitaries~\(U_g\) intertwine the left actions
  of~\(\FU_{\rg(g)}\) because they are isomorphisms of
  correspondences and by Lemma~\ref{lem:nondegenerate_over_AA-star}.
  The unitaries~\(\bar{U}_g\) inherit this property.
  Then
  \[
    \varphi_{\rg(g)}(a_1) \cdot S_g(a_2) = S_g(a_1 \cdot a_2) =
    S_{u\rg(g)}(a_1) S_g(a_2)
  \]
  for all \(a_1\in \FU_{\rg(g)}\), \(a_2\in \A_g\).  If \(g=u(x)\),
  then the closed linear span of \(S_g(a_2)\Hils_x\) is equal
  to~\(\Hils_x\) because~\(\bar{U}_{u(x)}\) is unitary onto
  \(\varphi_x(\A_{u(x)}\A_{u(x)}^*)\Hils_x = \Hils_x\).  Therefore,
  the above equality implies \(\varphi_x(a_1) = S_{u(x)}(a_1)\).
  Since~\(\bar{U}_g\) is an isometry from
  \(\A_g \otimes_{\FU_{\s(g)}} \Hils_{\s(g)}\)
  into~\(\Hils_{\rg(g)}\), the map~\(S_g\) built from it satisfies
  \[
    S_g(a_1)^* S_g(a_2) = \varphi_{\s(g)}(a_1^* a_2) = S_{g^{-1}
      g}(a_1^* \cdot a_2) = S_{g^{-1}}(a_1^*) S_g(a_2)
  \]
  for all \(a_1,a_2\in \A_g\).  That is, \(S_g(a_1)^*\) and
  \(S_{g^{-1}}(a_1^*)\) are equal on the closed linear span of
  \(S_g(a_2)\Hils_{\s(g)}\).  The latter subspace is the image
  of~\(\bar{U}_g\), which is
  \(\varphi_{\s(g)}(\A_g\A_g^*)\Hils_{\s(g)}\).  Both \(S_g(a_1)^*\)
  and \(S_{g^{-1}}(a_1^*)\) vanish on the orthogonal complement of the
  latter.  Hence they are equal on all of~\(\Hils_{\s(g)}\).
\end{proof}

Now let \((\nu,\Hils_x,S_g)\) and \((\nu',\Hils'_x,S'_g)\) be as
above; in particular, we only assume that~\(S_g\) is defined for all \(g\in
G|_{X_0}\) and that~\eqref{eq:S_g-multiplicative1}
holds for all \(g,h\in G|_{X_0}\) for some conull subset
\(X_0\subseteq G\), because this is exactly what we have managed to
prove so far.
Let \(T\colon \Lt^2(G^0,\Hils,\nu) \to \Lt^2(G^0,\Hils',\nu')\) be
a unitary intertwiner between the resulting Fell-bundle
representations of~\(\A\).
We have already shown in Section~\ref{sec:first_stage_proof} that
\([\nu]=[\nu']\), and we have disintegrated~\(T\) into a family of
unitaries \(T_x\colon \Hils_x \to \Hils'_x\) for almost all \(x\in
G^0\) that intertwine the representations \(\varphi_x = S_{u\rg(x)}\)
and \(\varphi'_x = S'_{u\rg(x)}\) of~\(\FU_x\) for all \(x\in X_0\)
(if needed, we shrink the conull subset \(X_0\subseteq G^0\) here) and
such that the set \(X_1\subseteq G|_{X_0}\) of all \(g\in X_1\) for
which the intertwining condition \(T_{\rg(g)}S_g(a) = S'_g(a)
T_{\s(g)}\) holds for all \(a\in \A_g\) is conull.
We are going to shrink the subset \(X_0\subseteq G^0\) further so that
the intertwining condition holds for all \(g\in G|_{X_0}\).

Let \(g,h\in X_1\) and \(a_1\in \A_g\), \(a_2\in \A_h\).
Then
\begin{multline*}
  S'_{g h}(a_1\cdot a_2) T_{\s(h)}
  = S'_g(a_1)S'_h(a_2) T_{\s(h)}
  \\= S'_g(a_1) T_{\s(g)} S_h(a_2)
  = T_{\rg(g)} S_g(a_1)  S_h(a_2)
  = T_{\rg(g)} S_{g h}(a_1 a_2).
\end{multline*}
If the Fell bundle were saturated, this would already imply \(g h \in
X_1\).
In general, we still get the following, as in the proof of
Lemma~\ref{lem:doubleprime}: if \((h_n)_{n\in\N}\) is a sequence
in~\(G_{\rg(k)}\) converging towards \(u\rg(k)\) and \(h_n^{-1},h_n k
\in X_1\) for all \(n\in\N\), then \(k\in X_1\).
In addition, since \(S_g(a)^* = S_{g^{-1}}(a^*)\) for all \(a\in
\A_g\), we get \(g^{-1}\in X_1\) for all \(g\in X_1\).

Now Lemma~\ref{lem:Ramsay_1} shows that there
is \(X_0'\subseteq X_0\) with \(G|_{X_0'} \subseteq X_1\).
That is, the intertwining condition holds for all \(g\in G|_{X_0'}\).

\subsection{Third stage of the proof: invariant null sets
  in the saturated case}
\label{sec:third_stage}

So far, we have found a conull subset \(X_0\subseteq G^0\) and changed
the maps~\(S_g\) on a null set so that~\(S_g\) is a measurable
representation of
\(G|_{X_0} = \setgiven{g\in G}{\rg(g),\s(g)\in X_0}\).
We have also disintegrated a unitary intertwiner and shown
that the intertwining condition holds for all \(g\in G|_{X_0}\).
Of course, two unitary intertwiners are equal if and only if their
fibreweise operators are equal on a conull subset \(x\in X_0\subseteq
G^0\).

What is still missing is to replace the exceptional set~\(X_0\) by a
\(G\)\nb-invariant one.
In this subsection, we will achieve this under the assumption that the
Fell bundle is saturated.
Having done that, we may redefine \(\Hils_x=0\) outside~\(X_0\) and so
get rid of the exceptional set for~\(S\) altogether (but the
exceptional sets do not go away for intertwiners or their equality).
The key idea in this section goes back to~\cite{Ramsay:Topologies},
another article by Ramsay.
Note that he calls \(G\)\nb-invariant subsets ``saturated''.
Since~\(G\) is second countable, it is \(\sigma\)\nb-compact.
Then \cite{Ramsay:Topologies}*{Lemma~3.1}
provides a \emph{\(G\)\nb-invariant} \(\nu\)\nb-conull subset
\(Z\subseteq G^0\) and a measurable map \(\eta\colon Z\to G\) with
\(\s\circ \eta = \Id_Z\), \(\rg\circ \eta(Z) \subseteq X_0\cap Z\),
and \(\eta(x) = u(x)\) for all \(x\in X_0\cap Z\).
The maps
\begin{align*}
  x &\mapsto \rg\circ\eta(x)
  && \text {for } x\in Z, \\
  g &\mapsto \eta(\rg(g))\cdot g\cdot \eta(\s(g))^{-1}
  && \text{for } g\in G|_{Z}
\end{align*}
define a measurable groupoid homomorphism
\(\psi\colon G|_{Z} \to G|_{X_0\cap Z} \), which is the identity map
on \(G|_{X_0\cap Z}\).
Let \(\Hils'_x \defeq \Hils_{\psi(x)}\) for \(x\in Z\) and
\(\Hils'_x\defeq0\) for \(x\in G^0\setminus Z\).
Let \(S'_g \defeq S_{\psi(g)}\) for \(g\in G|_Z\) and \(S'_g=0\) for
\(g\in G\setminus G|_Z\).
Then \(((\Hils'_x)_{x\in G^0}, (S'_g)_{g\in G})\) is equal to
\((\Hils,S)\) on the \(\nu\)\nb-conull subset \(X_0\cap Z\).
It satisfies all the algebraic conditions for a measurable Fell-bundle
representation on~\(Z\).
It also satisfies them for trivial reasons on \(G^0\setminus Z\)
because~\(Z\) is \(G\)\nb-invariant.
This proves all statements in Theorem
\ref{the:disintegrate_Hilbert_space}.\ref{the:disintegrate_Hilbert_space_1}
for saturated Fell bundles.

Next, we turn to intertwiners.
Let \((\nu,\Hils_x,U_g,S_g)\) and \((\nu',\Hils'_x,U'_g,S'_g)\) be the
data for measurable Fell-bundle representations of~\(G\) (without any
exceptional null sets), with \(S'_g\) related to~\(U'_g\) as usual.
Let \(X_0\subseteq G^0\) be a conull subset and let \(T_x\colon
\Hils_x \to \Hils_x'\) be a family of unitaries for \(x\in X_0\) that
satisfy~\eqref{eq:T_intertwines_S} for all \(g\in G|_{X_0}\).
We have achieved this situation for intertwiners in
Section~\ref{sec:improve_null}.

Since we assume the Fell bundle to be saturated, we may view \(U_g\)
and~\(U'_g\) as unitaries \(\A_g\otimes_{\FU_{\s(g)}}\Hils_{\s(g)} \congto
\Hils_{\rg(g)}\) and  \(\A_g\otimes_{\FU_{\s(g)}}\Hils'_{\s(g)} \congto
\Hils'_{\rg(g)}\).
Now we rewrite~\eqref{eq:T_intertwines_S} as
\begin{equation}
  \label{eq:T_intertwines_U_2}
  T_{\rg(g)}
  = U'_g (\Id_{\A_g} \otimes T_{\s(g)}) U_g^*.
\end{equation}
This is true for all \(g,h \in G|_{X_0}\), \(a\in\A_g\), \(b\in
\A_h\).
As above, we now choose a \(G\)\nb-invariant, \(\sigma\)\nb-compact,
\(\nu\)\nb-conull subset \(X\subseteq G\) and a measurable map
\(\eta\colon X\to G\) with \(\rg\circ \eta = \Id_X\), \(\s\circ
\eta(X) \subseteq X_0\cap X\), and \(\eta(x) = u(x)\) for all \(x\in
X_0\cap X\).
(Compared to the argument above, we switched
\(\rg\) and~\(\s\), which amounts to taking inverses.)
Let~\(T'_x\) for \(x\in X\) be the composite unitary
\[
  \begin{tikzcd}[column sep=large]
    \Hils_x = \A_{\eta(x)}\A_{\eta(x)}^* \Hils_x
    \ar[r, "U_\eta(x)^*", "\cong"'] &
    \A_{\eta(x)} \Hils_{\s(\eta(x))}
    \ar[d, "\cong"', "\Id\otimes T_{\s(\eta(x))}"]\\
    \Hils'_x =
    \A_{\eta(x)}\A_{\eta(x)}^* \Hils'_x &
    \A_{\eta(x)} \Hils'_{\s(\eta(x))}
    \ar[l, "\cong","U'_{\eta(x)}"']
  \end{tikzcd}
\]
Then \(T'_x = T_x\) for all \(x\in X_0\cap X\), which is
conull; so the family of operators~\(T'_x\) is still
measurable and also induces the same unitary operator on
\(\Lt^2\)\nb-sections.
In addition, \eqref{eq:T_intertwines_U_2} continues to hold for~\(T'\)
for all arrows \(g\in G|_{X_0 \cap X}\) because this holds for~\(T\).
By construction, \(T'\) satisfies \eqref{eq:T_intertwines_U_2} or,
equivalently, \eqref{eq:T_intertwines_S} also for the arrows
\(\eta(x)\) for all \(x\in X\).
The set of arrows for which~\eqref{eq:T_intertwines_U} holds is
multiplicative because the Fell bundle is saturated.
So~\eqref{eq:T_intertwines_S} holds for all arrows in~\(G|_X\).

Next, let \((T_x)_{x\in X}\) and \((T'_x)_{x\in X'}\) be two
families of unitaries that both disintegrate the same intertwiner on
\(\Lt^2\)\nb-spaces.
Here \(X,X'\subseteq G^0\) are \(G\)\nb-invariant conull subsets
of~\(G^0\).
We assume that the intertwining condition~\eqref{eq:T_intertwines_S}
holds for all \(g\in G|_X\) or \(g\in G|_{X'}\), respectively.
We also assume that \(T_x = T'_x\) for almost all \(x\in G^0\); we
want to show that the set of \(x\in G^0\) for which \(T_x = T'_x\)
holds is also \(G\)\nb-invariant.

Since~\eqref{eq:T_intertwines_S} is equivalent
to~\eqref{eq:T_intertwines_U_2}, the latter holds for all arrows
\(g\in G\) with \(\s(g)\in X\cap X'\).
Thus \(T_{\s(g)} = T'_{\s(g)}\) implies \(T_{\rg(g)} = T'_{\rg(g)}\).
This says that the set of \(x\in X\cap X'\) with \(T_x = T'_x\) is
\(G\)\nb-invariant as asserted.
By now, we have proven all statements in
Theorem~\ref{the:disintegrate_Hilbert_space}
in the saturated case.

\subsection{Fourth stage of the proof: removing the null set in
  general}
\label{sec:fourth_stage}

From now on, we again allow Fell bundles that are not saturated.
We want to make the conull set \(X_0\subseteq G^0\) on which things
work \(G\)\nb-invariant in this case as well.
The key idea in order to generalise the argument in
Section~\ref{sec:third_stage} is to use a sequence
of local sections of the range map that together cover a fibrewise
dense subset of~\(G\).
To build these local sections, we first choose a countable basis
\((W_n)_{n\in\N_{\ge1}}\) for the topology of~\(G\).
This is possible because~\(G\) is assumed second countable.

\begin{lemma}[compare \cite{Ramsay:Topologies}*{Lemma~A}]
  \label{lem:sigma-compact_conull_exists}
  There is a \(\sigma\)\nb-compact \(\nu\)\nb-conull subset contained
  in~\(X\).
\end{lemma}

\begin{proof}
  Since~\(G^0\) is second countable, \(G^0 = \bigcup K_n\) with
  compact subsets~\(K_n\).  Since the measure~\(\nu\) is regular, for
  each \(n\in\N\) there is an increasing sequence of compact subsets
  \(L_{n,k}\subseteq K_n \cap X\) with
  \(\nu(K_n \cap X\setminus L_{n,k}) \to 0\).  Then
  \(Z \defeq \bigcup \bigcup_{n,k} L_{n,k}\) is \(\sigma\)\nb-compact,
  contained in~\(X\), and satisfies \(\nu(X \setminus Z) = 0\).
\end{proof}

In the arguments above, we found that the statements we want hold on
\(X\) or on~\(G|_X\) for some \(\nu\)\nb-conull subset
\(X\subseteq G^0\).  We replace~\(X\) by a \(\sigma\)\nb-compact
\(\nu\)\nb-conull subset as in
Lemma~\ref{lem:sigma-compact_conull_exists} to arrange for~\(X\)
itself to be \(\sigma\)\nb-compact.  Let
\(G_X \defeq \setgiven{g\in G}{\s(g)\in X}\).  This subset of~\(G\) is
\(\sigma\)\nb-compact as well because~\(G\) is countable at infinity
and~\(X\) is \(\sigma\)\nb-compact.  For each \(n\in\N\), let
\(R_n \defeq \rg(G_X\cap W_n) \subseteq G^0\); this is
\(\sigma\)\nb-compact as well.  The restriction of the range map to
\(G_X \cap W_n\) has a Borel section
\[
  \eta_n\colon R_n \to G_X \cap W_n
\]
by the Federer--Morse Lemma
\cite{Federer-Morse:Measurable}*{Theorem~5.1}.
The subset \(R_n\cap X\) is \(\sigma\)\nb-compact as well, and we may
assume without loss of generality that \(\eta_n(x) = u(x)\) for all
\(x\in R_n \cap X\), by changing~\(\eta_n\) accordingly on \(R_n \cap
X\).
It is convenient to let \(\eta_0(x) \defeq u(x)\colon W_0 \to G^0\) be
the unit section \(X\to G_X\).

Above, we have already disintegrated a Fell-bundle representation
\((\varphi,U)\) on a separable Hilbert space into a measurable
Fell-bundle representation \((\nu,\Hils,S)\) of~\(G|_X\) for a
\(\nu\)\nb-conull subset \(X\subseteq G^0\), which we may assume to be
\(\sigma\)\nb-compact as well.  For \(n\in\N\), let
\(\Hils^{(n)}_x \defeq 0\) for \(x\notin R_n\) and
\[
  \Hils^{(n)}_x \defeq
  \A_{\eta_n(x)} \otimes_{\FU_{\s(\eta_n(x))}} \Hils_{{\s(\eta_n(x))}}
\]
for \(x\in R_n\).  This is a measurable field of Hilbert spaces by
construction.  If our representation~\(S\) were defined on all
of~\(G\), then~\(S_{\eta_n(x)}\) would provide an isometry
\(\Hils^{(n)}_x \congto \A_{\eta_n(x)} \A_{\eta_n}^* \Hils_x \subseteq
\Hils_x\).  We are now going to define~\(\Hils_x\) so that it is
spanned by~\(\Hils_x^{(n)}\), but taking into account the appropriate
overlaps.  Each~\(\Hils_x^{(n)}\) carries a nondegenerate
representation of~\(\FU_x\) from the Fell-bundle structure.  For
\(m\in\N\), the family of ideals
\(\A_{\eta_m(x)} \A_{\eta_m(x)}^* \idealin \FU_x\) is measurable in
the sense that there is a set of measurable sections whose values are
dense in the fibres
\(\A_{\eta_m(x)} \A_{\eta_m(x)}^* \idealin \FU_x\); this is so because
\((\A_g\A_g^*)_{g\in G}\) is a continuous family of ideals
and~\(\eta_m\) is measurable.  Therefore, the subfield of Hilbert
spaces
\(\A_{\eta_m(x)} \A_{\eta_m(x)}^* \Hils_x^{(n)} \subseteq
\Hils_x^{(n)}\) is measurable, and then so is the field formed by the
orthogonal complements.  Let
\(P_m^{(n)}(x) \in \Bound(\Hils_x^{(n)})\) be the orthogonal
projection onto \(\A_{\eta_m(x)} \A_{\eta_m(x)}^* \Hils_x^{(n)}\) and
let \(Q_m^{(n)}(x) = 1 - P_m^{(n)}(x)\) be the complementary
projection.  All these projections~\(P_m^{(n)}\) are measurable, and
they commute with each other because they are the support projections
of ideals in~\(\FU_x\).  Let
\[
  \Hils'_x \defeq \bigoplus_{n=0}^\infty \prod_{m=0}^{n-1} Q_m^{(n)} \Hils_x^{(n)}
\]
for all \(x\in G^0\).  If \(x\in X\), then \(\Hils'_x= \Hils_x\)
because \(\Hils_x^{(0)} = \Hils_x\) and \(Q_0^{(n)}=0\) for
\(n\ge1\).  If \(x\notin \rg(G_X)\), then \(\Hils'_x = 0\).

We are going to extend the obvious inclusions of
\(\prod_{m=0}^{n-1} Q_m^{(n)} \Hils_x^{(n)} \subseteq \Hils_x^{(n)}\)
to canonical isometries \(j_n\colon \Hils_x^{(n)} \to \Hils'_x\) for
all \(n\in\N\).  We write
\begin{multline}
  \label{eq:Hils_x_n_decompositon}
  \Hils_x^{(n)} = \bigoplus_{k=0}^n \prod_{m=0}^{k-1} Q_m^{(n)} P_k^{(n)} \Hils_x^{(n)}
  \\= \bigoplus_{k=0}^n \prod_{m=0}^{k-1} Q_m^{(n)} P_k^{(n)} \A_{\eta_n(x)}
  \otimes_{\FU_{\s(\eta_n(x))}} \Hils_{{\s(\eta_n(x))}}.
\end{multline}
We look at the \(k\)th summand in~\(\Hils_x^{(n)}\) and abbreviate
\(Q'\defeq \prod_{m=0}^{k-1} Q_m^{(n)}\).  A vector in the \(k\)th
summand may be approximated by linear combinations of
\[
  Q'( S_{\eta_k(x)}(f_1) S_{\eta_k(x)}(f_2^*)) \xi,
  \qquad \text{with } f_1,f_2\in \A_{\eta_k(x)},
  \xi\in \A_{\eta_n(x)} \otimes_{\FU_{\s(\eta_n(x))}}
  \Hils_{{\s(\eta_n(x))}}.
\]
Using the representation of~\(G|_X\), we identify
\(S_{\eta_k(x)}(f_2^*) \xi\) with a vector in
\[
  \A_{\eta_k(x)^{-1}} \A_{\eta_n(x)} \Hils_{{\s(\eta_n(x))}} \subseteq
  \Hils_{{\s(\eta_k(x))}}
\]
and so
\(S_{\eta_k(x)}(f_1) S_{\eta_k(x)}(f_2^*) \xi \in \A_{\eta_k(x)}
\otimes_{\FU_{\s(\eta_k(x))}} \Hils_{{\s(\eta_k(x))}} =
\Hils_x^{(k)}\).  Thus we get an isometry from the \(k\)th summand
in~\(\Hils_x^{(n)}\) to \(Q'\Hils_x^{(k)}\), which is the \(k\)th
summand in~\(\Hils'_x\).

We claim that
\begin{equation}
  \label{eq:jn_scalar_products}
  \braket{j_n(f_1 \otimes \xi_1)}{j_m(f_2 \otimes \xi_2)}
  =  \braket{\xi_1}{S_{\eta_n(x)^{-1}\eta_m(x)}(f_1^* f_2)\xi_2}
\end{equation}
for all \(f_1\in \A_{\eta_n(x)}\), \(\xi_1\in\Hils_{\s(\eta_n(x))}\),
\(f_2\in \A_{\eta_m(x)}\), \(\xi_2\in\Hils_{\s(\eta_m(x))}\) for some
\(n,m\in\N\).
There is an ideal \(I\idealin \FU_{\s(g)}\) with \(\A_{\eta_m(x)}
\A_{\eta_m(x)}^* \A_g = \A_g I\).
This implies that there is a certain projection~\(Q''\)
on~\(\Hils_{\s(\eta_m)(x)}\) so that
\(\prod_{m=0}^{k-1} Q_m^{(n)} P_k^{(n)} (f\otimes \xi) = f \otimes
Q''(\xi)\) for all \(f\in \A_{\eta_m(x)}\), \(\xi\in
\Hils_{\s(\eta_m(x))}\).
Therefore, we may decompose
\(f_1\otimes \xi_1 = \sum_{k=0}^n f_1 \otimes \xi_{1,k}\) with
\(\xi_{1,k}\in\Hils_{\s(\eta_n(x))}\), so that
\(f_1 \otimes \xi_{1,k}\) belongs to the \(k\)th summand
in~\eqref{eq:Hils_x_n_decompositon}.
We may do the same to \(f_2 \otimes \xi_2\) and so arrange that both
belong to one of these summands.
After this simplification, \eqref{eq:jn_scalar_products} follows
quickly from the definition of the isometries~\(j_n\).

Thus we may also define~\(\Hils'_x\) as the Hausdorff completion of
the sum \(\bigoplus \Hils_x^{(n)}\) for the scalar product defined
by~\eqref{eq:jn_scalar_products}.

\begin{lemma}
  \label{lem:jg_lemma}
  Let \(g\in G_X\) and \(x \defeq \rg(g)\).  There is a unique map
  \[
    j_g\colon \A_g \otimes_{\FU_{\s(g)}} \Hils_{\s(g)} \to \Hils'_x
  \]
  such that
  \[
    \braket{j_g(f_1 \otimes \xi_1)}{j_n(f_2 \otimes \xi_2)}
    = \braket{\xi_1}{S_{g^{-1}\eta_n(x)}(f_1^* f_2) \xi_2}
  \]
  for all
  \(n\in\N\), \(f_1\in \A_g\), \(\xi_1\in \Hils_{\s(g)}\),
  \(f_2\in \A_{\eta_n(x)}\), \(\xi_1\in \Hils_{\s(\eta_n(x))}\).
  If \(g,h\in G_X\) satisfy \(\rg(g) = \rg(h) = x\), then
  \[
    \braket{j_g(f_1 \otimes \xi_1)}{j_h(f_2 \otimes \xi_2)}
    = \braket{\xi_1}{S_{g^{-1}h}(f_1^* f_2) \xi_2}
  \]
  for all \(f_1\in \A_g\), \(\xi_1\in \Hils_{\s(g)}\), \(f_2\in
  \A_h\), \(\xi_1\in \Hils_{\s(h)}\).
\end{lemma}

\begin{proof}
  Since the subspaces \(j_n(\Hils_x^{(n)})\) span~\(\Hils'_x\), the
  first formula in the lemma determines \(j_g(f_1 \otimes \xi_1)\)
  uniquely if it exists, and then~\(j_g\) is determined uniquely.
  So there is at most one map with all these properties.
  Let \(f_1 \in \A_g\), \(\xi_1\in \Hils_{\s(g)}\).
  We may write \(\xi_1 = f_3(u(\s(g))) f_3(u(\s(g)))^* \xi_3\) with
  \(f_3(u(\s(g)))\in \FU_{\s(g)}\), \(\xi_3\in \Hils_{\s(g)}\) by the
  Cohen--Hewitt Factorisation Theorem.
  Choose a continuous section~\(f_3\) of the bundle~\(\A\) with the
  value \(f_3(u\s(g))\) at \(u\s(g)\).
  So \(f_3(k_k) f_3(k_n^{-1})^*\xi_3\) converges to~\(\xi_1\) in norm
  for any sequence~\((k_n)\) converging to~\(u(\s(g))\).
  We claim that there is a sequence of \(n_j\in\N\) with \(\lim
  \eta_{n_j}(x) = g\).
  To choose~\(n_j\), we start with a countable neighbourhood
  basis~\((U_j)\) of~\(g\).
  By assumption, \(U_j\) must be a union of the basis
  neighbourhoods~\(W_n\) chosen above.
  So there is \(n_j\in \N\) with \(g\in W_{n_j} \subseteq U_j\).
  Now \(x = \rg(g)\in \rg(G_X\cap W_{n,j})\) and so \(\eta_{n_j}(x)\in
  W_{n_j} \in U_j\) by construction.
  Thus \(\lim \eta_{n_j}(x) = g\) as desired.
  Then
  \[
    f_1\cdot f_3(g^{-1} \eta_{n_j}(x)) \otimes
    S_{g^{-1} \eta_{n_j}(x)} f_3(g^{-1} \eta_{n_j}(x))^* \xi_1
    \in \A_{\eta_{n_j}(x)} \otimes \Hils_{\s(\eta_{n_j})}
    =  \Hils_x^{(n_j)}
  \]
  has an image \(\tau_j\in \Hils'_x\).
  Now~\eqref{eq:jn_scalar_products} allows to compute
  \(\braket{\tau_j - \tau_i}{\tau_j-\tau_i}\) and check that the
  sequence~\(\tau_j\) is a Cauchy sequence.  We let
  \(j_g(f_1\otimes \xi_1)\) be the limit of this Cauchy sequence.
  Equation~\eqref{eq:jn_scalar_products} also allows to compute
  \(\braket{\tau_j}{f_2 \otimes \xi_2}\) for
  \(f_2\otimes \xi_2\in \Hils_x^{(n)}\) for some \(n\in\N\), and this
  shows that \(j_g(f_1\otimes \xi_1)\) satisfies
  \(\braket{j_g(f_1 \otimes \xi_1)}{j_n(f_2 \otimes \xi_2)} =
  \braket{\xi_1}{S_{g^{-1}\eta_n(x)}(f_1^* f_2) \xi_2}\) as required.
  Next, doing the same also for~\(h\) instead of~\(g\), we may check
  \(\braket{j_g(f_1 \otimes \xi_1)}{j_h(f_2 \otimes \xi_2)} =
  \braket{\xi_1}{S_{g^{-1}h}(f_1^* f_2) \xi_2}\).  When \(g=h\),
  this implies that~\(j_g\) extends to an isometry on the linear span
  of elementary tensors.
\end{proof}

\begin{remark}
  The Hilbert space~\(\Hils'_x\) can also be realised as the Hausdorff
  completion of the algebraic sum
  \[
    \bigoplus_{g\in G_X^x} \A_g \otimes \Hils_{\s(g)},
  \]
  endowed with the indefinite sesquilinear form described in
  Lemma~\ref{lem:jg_lemma}; here
  \(
  G_X^x \defeq \setgiven{g\in G}{\rg(g)=x,\ \s(g)\in X}.
  \)
  This description does not depend on any further choices and is in this
  sense canonical.  However, when viewed in isolation, it is not
  a priori clear why this construction should yield a separable Hilbert
  space.  Separability follows only after identifying this completion
  with the Hilbert space~\(\Hils'_x\) constructed above, which is built
  as a countable direct sum of separable Hilbert spaces.
\end{remark}

Now we may extend the measurable Fell-bundle
representation~\((S_g)_{g\in G|_X}\) to all of~\(G\) on the measurable
field of Hilbert spaces~\(\Hils'\):

\begin{lemma}
  \label{lem:Sk_from_jg}
  Let \(k\in G\).  There is a unique isometry
  \[
    U_k\colon \A_k \otimes_{\FU_{\s(k)}} \Hils'_{\s(k)}
    \to \Hils'_{\rg(k)}
  \]
  such that
  \(
    U_k\bigl(f_1 \otimes j_g(f_2 \otimes \xi)\bigr)
    = j_{k g}(f_1 f_2 \otimes \xi)
  \)  for all \(g\in G_X\), \(f_1\in \A_k\), \(f_2\in \A_g\),
  \(\xi\in \Hils_{\s(g)}\) with \(\s(k)=\rg(g)\).
  Let \(S_g(a)(\xi) \defeq U_g(a\otimes \xi)\) for \(a\in \A_g\),
  \(g\in G\).
  The family \((S_k)_{k\in G}\) is a measurable Fell-bundle
  representation.
\end{lemma}

\begin{proof}
  Fix \(k\in G\) and set \(x\defeq \rg(k)\), \(y\defeq \s(k)\).
  Consider the subspace \(D_y\subseteq \Hils'_y\) spanned by vectors
  of the form \(j_g(f\otimes \xi)\) with \(g\in G_X\cap G^y\),
  \(f\in \A_g\), \(\xi\in \Hils_{\s(g)}\).
  By construction of~\(\Hils'_y\), this subspace is dense.  Define a
  linear map on elementary tensors by
  \[
    T_k\colon \A_k \odot D_y \to \Hils'_x,\qquad
    T_k\bigl(f_1 \otimes j_g(f_2\otimes \xi)\bigr)
    \defeq j_{k g}(f_1 f_2 \otimes \xi),
  \]
  where \(g\in G_X\) satisfies \(\rg(g)=y\).
  On the one hand, Lemma~\ref{lem:jg_lemma} implies
  \begin{align*}
    &\braket{T_k(f_1 \otimes j_{g_1}(f_2\otimes \xi_1))}
           {T_k(f_3 \otimes j_{g_2}(f_4\otimes \xi_2))} \\
    &\qquad =
    \braket{j_{k g_1}(f_1 f_2 \otimes \xi_1)}
           {j_{k g_2}(f_3 f_4 \otimes \xi_2)} \\
    &\qquad =
    \braket{\xi_1}
    {S_{(k g_1)^{-1}(k g_2)}\bigl((f_1 f_2)^*(f_3 f_4)\bigr)\xi_2} \\
    &\qquad =
    \braket{\xi_1}
    {S_{g_1^{-1}g_2}\bigl(f_2^* f_1^* f_3 f_4\bigr)\xi_2}.
  \end{align*}
  On the other hand,
  \begin{align*}
    &\braket{f_1 \otimes j_{g_1}(f_2\otimes \xi_1)}
            {f_3 \otimes j_{g_2}(f_4\otimes \xi_2)} \\
    &\qquad =
    \braket{j_{g_1}(f_2\otimes \xi_1)}
           {S_{k^{-1}k}(f_1^* f_3)\, j_{g_2}(f_4\otimes \xi_2)} \\
    &\qquad =
    \braket{j_{g_1}(f_2\otimes \xi_1)}
           {j_{g_2}(f_1^* f_3 f_4\otimes \xi_2)} \\
    &\qquad =
    \braket{\xi_1}
    {S_{g_1^{-1}g_2}\bigl(f_2^*(f_1^* f_3)f_4\bigr)\xi_2}.
  \end{align*}
  These expressions agree, showing that~\(T_k\) preserves inner
  products on elementary tensors.
  Hence~\(T_k\) extends uniquely to an isometry
  \[
    U_k\colon \A_k \otimes_{\FU_y} \Hils'_y \to \Hils'_x.
  \]
  Uniqueness follows because the vectors
  \(j_g(f\otimes \xi)\) span a dense subspace of~\(\Hils'_y\).

  Recall that the maps \(S_k\) and~\(U_k\) are related by \(S_k(a)\xi
  = U_k(a\otimes\xi)\).
  Since~\(U_k\) is well defined and isometric on the balanced tensor
  product and a left module map, \(S_k\) satisfies
  \[
    S_k(c\cdot a \cdot b)
    = \varphi_x(c)\, S_k(a)\, \varphi_y(b),
    \qquad
    S_k(a)^* S_k(b)=\varphi_y(a^*b),
  \]
  The multiplicativity relation \(S_k(a_1)S_h(a_2)=S_{k h}(a_1 a_2)\)
  holds on elementary tensors by construction and hence everywhere.

  Finally, the family \((S_k)_{k\in G}\) is measurable because the
  maps \(k\mapsto j_{k g}\) are measurable families of operators
  between measurable fields, and all constructions above are algebraic
  and compatible with the measurable structure.

  Therefore, \((S_k)_{k\in G}\) is a measurable Fell-bundle
  representation.
\end{proof}

This completes the construction of a measurable Fell-bundle
representation associated to a given Fell-bundle representation on a
separable Hilbert space.
This finishes the proof of
Theorem~\ref{the:disintegrate_Hilbert_space}.\ref{the:disintegrate_Hilbert_space_1}.
The proofs of Theorem
\ref{the:disintegrate_Hilbert_space}.\ref{the:disintegrate_Hilbert_space_2}
and
\ref{the:disintegrate_Hilbert_space}.\ref{the:disintegrate_Hilbert_space_3}
were already finished in Section~\ref{sec:improve_null} because these
statements do not assert that we may make the conull subset in~\(G^0\)
\(G\)\nb-invariant in general.
The theorem only asserts that this may be done in the saturated case,
and this has also been shown in Section~\ref{sec:third_stage}.
So we have finished the proof of
Theorem~\ref{the:disintegrate_Hilbert_space}.
We do not know whether the conull sets in~\(G^0\) in Theorem
\ref{the:disintegrate_Hilbert_space}.\ref{the:disintegrate_Hilbert_space_2}
and
\ref{the:disintegrate_Hilbert_space}.\ref{the:disintegrate_Hilbert_space_3}
can always be chosen \(G\)\nb-invariant, also for nonsaturated Fell
bundles.

\section{Twisted partial actions}
\label{sec:twisted_partial_special}

In this section, we consider a Fell bundle~\(\A\) that is
associated to a twisted partial action.
The latter consists of a continuous family of ideals
\(\Bun[D]_g\subseteq \FU_{\rg(g)}\), a continuous family of
\Star{}isomorphisms \(\act_g\colon \Bun[D]_{g^{-1}} \congto
\Bun[D]_g\) for \(g\in G\), and a continuous family of unitary
multipliers~\(w(g,h)\) of \(\act_h(\Bun[D]_g \cap \Bun[D]_{h^{-1}})\)
for \((g,h)\in G^2\), which satisfy the conditions in
Definition~\ref{def:twisted-partial-action-groupoid}.
The resulting Fell bundle is built in
Definition~\ref{def:Fell_twisted_partial_action}.
We first identify measurable Hilbert space representations of~\(\A\)
with measurable covariant representations.
Second, we prove an analogous result for representations on Hilbert
modules, which does not need separability.
This only works, however, if the action is global, not partial, and
the twist is scalar-valued.

\begin{definition}
  \label{def:measurable_covariant_representation}
  A \emph{measurable covariant representation} of a twisted partial
  action has the data in
  \ref{def:measurable_covariant_representation_d1}--\ref{def:measurable_covariant_representation_d4}
  subject to the conditions
  \ref{def:measurable_covariant_representation_1}--\ref{def:measurable_covariant_representation_4}:
  \begin{enumerate}
  \item \label{def:measurable_covariant_representation_d1}%
    \(\nu\) is a measure class on~\(G^0\);
  \item \label{def:measurable_covariant_representation_d2}%
    \(\Hils = (\Hils_x)_{x\in G^0}\) is a measurable field of Hilbert
    spaces with \(\Hils_x\neq0\) for \(\nu\)\nb-almost all
    \(x\in G^0\);
  \item \label{def:measurable_covariant_representation_d3}%
    \(\varphi_x\colon \FU_x \to \Bound(\Hils_x)\) is a measurable
    family of nondegenerate representations for all \(x\in G^0\);
  \item \label{def:measurable_covariant_representation_d4}%
    \(U_g\colon \varphi_{\s(g)}(\Bun[D]_{g^{-1}}) \Hils_{\s(g)}
    \congto \varphi_{\rg(g)}(\Bun[D]_g) \Hils_{\rg(g)}\) is a unitary
    operator for all \(g\in G\);
  \item \label{def:measurable_covariant_representation_1}%
    \(\nu\)~is quasi-invariant relative to the family of
    representations~\((\varphi_x)_{x\in G^0}\);
  \item \label{def:measurable_covariant_representation_2}%
    \(U_g \varphi_{\s(g)}(a) U_g^* = \varphi_{\rg(g)} (\act_g(a))\)
    for all \(g\in G\) and \(a\in\Bun[D]_{g^{-1}}\);
  \item \label{def:measurable_covariant_representation_3}%
    if \((g,h)\in G^2\), then~\(U_g U_h\) is a restriction of
    \(\varphi_{\rg(g)}(w(g,h)) U_{g h}\);
  \item \label{def:measurable_covariant_representation_4}%
    \((\varphi_x)_{x\in G^0}\) and \((U_g)_{g\in G}\) are measurable
    in the sense that pointwise application of \(\varphi_x(a_x)\)
    for \(a\in \SUF\) and of~\(U_g\) map measurable sections to
    measurable sections.
  \end{enumerate}
\end{definition}

\begin{remark}
  The domain of~\(U_g U_h\) is equal to
  \(\varphi_{\s(h)}(\Bun[D]_h^{-1} \cap \Bun[D]_{g h}^{-1})
  \Hils_{\s(h)}\), the intersection of the domains of \(U_{g h}\)
  and~\(U_h\).
\end{remark}

\begin{remark}
  In a measurable covariant representation, \(U_{u(x)}=1\) for all
  \(x\in G^0\).  First, \(w(u(x),u(x))=1\) implies
  \(U_{u(x)} U_{u(x)} = U_{u(x)}\).  Since~\(U_{u(x)}\) is unitary,
  this implies \(U_{u(x)}=1\).
\end{remark}

\begin{theorem}
  \label{the:covariant_rep_vs_measurable_Fell_rep}
  Measurable covariant representations of a twisted partial action
  are equivalent to measurable Fell-bundle representations of the
  associated Fell bundle.
\end{theorem}

\begin{proof}
  Let \(\nu\), \(\Hils = (\Hils_x)_{x\in G^0}\),
  \(\varphi_x\colon \FU_x \to \Bound(\Hils_x)\) for \(x\in G^0\) and
  \[
    U_g\colon \varphi_{\s(g)}(\Bun[D]_{g^{-1}}) \Hils_{\s(g)} \to
    \varphi_{\rg(g)}(\Bun[D]_g) \Hils_{\rg(g)}
  \]
  for \(g\in G\) be a measurable covariant representation of the
  twisted partial action.
  Extend each~\(U_g\) to a partial isometry
  \(\bar{U}_g\colon\Hils_{\s(g)} \to \Hils_{\rg(g)}\).
  If \(a\in \A_g \defeq \Bun[D]_g\), then let
  \begin{equation}
    \label{eq:Sga_definition}
    S_g(a) \defeq \varphi_{\rg(g)}(a) \bar{U}_g
    \colon \Hils_{\s(g)} \to \Hils_{\rg(g)}.
  \end{equation}
  The maps~\(S_g(a)\) are bounded with \(\norm{S_g(a)} \le \norm{a}\)
  because~\(\bar{U}_g\) is a partial isometry and~\(\varphi_{\rg(g)}\)
  is a \Star{}homomorphism.
  In addition, the map~\(S_g\) is linear.
  We claim that \(\nu\),
  \(\Hils = (\Hils_x)_{x\in G^0}\), and
  \(S_g\colon \A_g \to \Bound(\Hils_{\s(g)}, \Hils_{\rg(g)})\) is a
  measurable Fell-bundle representation.
  First, we compute
  \begin{equation}
    \label{eq:Sga_alternative}
    S_g(a) = \varphi_{\rg(g)}(a) \bar{U}_g
    = \bar{U}_g \bar{U}_g^* \varphi_{\rg(g)}(a) \bar{U}_g
    = \bar{U}_g \varphi_{\s(g)}(\act_g^{-1}(a))
  \end{equation}
  for all \(a\in \A_g\) because these formulas hold on the subspace
  \(\varphi_{\s(g)}(\Bun[D]_{g^{-1}})(\Hils_{\s(g)})\) and all three operators
  vanish on its orthogonal complement.  Since
  \(\bar{U}_{u(x)}=\Id_{\Hils_x}\), we get
  \(S_{u(x)} = \varphi_{u(x)}\).  We claim that
  \(S_g(a)S_h(b) = S_{g h}(a\cdot b)\) for all \((g,h)\in G^2\),
  \(a\in \A_g\), \(b\in \A_h\).  We use~\eqref{eq:Sga_alternative} and
  its analogues for \(S_h(b)\) and \(S_{g h}(a\cdot b)\).
  Equation~\eqref{eq:Fell_twisted_partial_action_multiply} implies
  that
  \[
    \act_{g h}^{-1}(a\cdot b)
    \defeq
    \act_{g h}^{-1}\bigl(\act_g(\act_{g^{-1}}(a) b) w(g,h)\bigr).
  \]
  This belongs to
  \[
    \act_{g h}^{-1}(\Bun[D]_g\cap \Bun[D]_{gh})
    = \Bun[D]_{(gh)^{-1}} \cap \Bun[D]_{h^{-1}}
    \subseteq \Bun[D]_{(g h)^{-1}}.
  \]
  Then
  \(S_{g h}(a\cdot b) = U_{g h} \varphi_{\s(h)}(\act_{g h}^{-1}
  (\act_g(\act_{g^{-1}}(a) b) w(g,h)))\).  This vanishes on the
  orthogonal complement of
  \(\Hils[K] \defeq \varphi_{\s(h)}(\Bun[D]_{(gh)^{-1}} \cap
  \Bun[D]_{h^{-1}})\), and so does the composite \(S_g(a) S_h(b)\).
  Therefore, we may restrict attention to~\(\Hils[K]\).  This is
  contained in the domain of \(U_h\) and~\(U_{g h}\).  We compute
  \begin{multline*}
    S_{g h}(a\cdot b)|_{\Hils[K]}
    = \varphi_{\rg(g)}(\act_g(\act_{g^{-1}}(a) b) w(g,h))
    U_{g h}|_{\Hils[K]}
    = \varphi_{\rg(g)}(\act_g(\act_{g^{-1}}(a) b) U_g U_h|_{\Hils[K]}\\
    = U_g \varphi_{\s(g)}(\act_{g^{-1}}(a) b) U_h|_{\Hils[K]}
    = U_g \varphi_{\s(g)}(\act_{g^{-1}}(a)) \varphi_{\rg(h)}(b)
    U_h|_{\Hils[K]}
    = S_g(a) S_h(b)|_{\Hils[K]}.
  \end{multline*}
  This finishes the proof that
  \(S_g(a) S_h(b) = S_{g h}(a b)\) for all \(a\in\A_g\),
  \(b\in\A_h\).

  Next, we check \(S_{g^{-1}}(a^*) = S_g(a)^*\) for \(a\in\A_g\),
  \(g\in G\).
  Since~\(\bar{U}_g\) is a partial unitary and
  \(\bar{U}_{g^{-1}} \bar{U}_g\) is a restriction of
  \(\varphi_{\s(g)}(w(g^{-1},g)) \bar{U}_{g^{-1} g} =
  \varphi_{\s(g)}(w(g^{-1},g))\cdot \Id_{\Hils_{\s(g)}}\), it follows
  that
  \(\bar{U}_g^* = \varphi_{\s(g)}(w(g^{-1},g))^* \bar{U}_{g^{-1}}\).
  Equation~\eqref{eq:Sga_alternative} implies
  \[
    S_g(a)^*
    = \varphi_{\s(g)}(\act_g^{-1}(a))^* \bar{U}_g^*
    = \varphi_{\rg(g^{-1})}(\act_g^{-1}(a^*) w(g^{-1},g)^*)
    \bar{U}_{g^{-1}}.
  \]
  This is equal to \(S_{g^{-1}}(a^*)\) because \(a^* \in \A_{g^{-1}}\)
  is equal to \(\act_g^{-1}(a^*)\cdot w(g^{-1},g)^*\)
  by~\eqref{eq:Fell_twisted_partial_action_involution}.

  The nondegeneracy of the representations~\(\varphi_x\) and
  Condition~\ref{def:measurable_covariant_representation_4} in
  Definition~\ref{def:measurable_covariant_representation} imply
  Conditions \ref{def:Fell_bundle_Pre-rep_3} and
  \ref{def:Fell_bundle_Pre-rep_4} in
  Definition~\ref{def:Fell_bundle_pre-rep}.  So a
  measurable covariant representation induces a measurable Fell-bundle
  representation.

  Conversely, take a measurable Fell-bundle representation
  \((\nu,\Hils,\varphi,S)\) of the Fell bundle associated to the twisted
  partial action.  Fix \(g\in G\).  Consider the linear map
  \begin{equation}
    \label{eq:def_Ug_from_S}
    \Bun[D]_{g^{-1}} \odot \Hils_{\s(g)} \to
    \varphi_{\rg(g)}(\Bun[D]_g)\Hils_{\rg(g)},\qquad
    a\otimes \xi \mapsto S_g(\act_g(a))\,\xi.
  \end{equation}
  We claim that this map preserves the inner product on
  \(\Bun[D]_{g^{-1}} \odot \Hils_{\s(g)}\) and therefore extends to an
  isometry
  \[
    U_g\colon
    \Bun[D]_{g^{-1}} \otimes_{\varphi_{\s(g)}} \Hils_{\s(g)}
    \to \varphi_{\rg(g)}(\Bun[D]_g)\Hils_{\rg(g)}.
  \]

  Let \(a_1,a_2\in \Bun[D]_{g^{-1}}\) and
  \(\xi_1,\xi_2\in \Hils_{\s(g)}\).  Using the defining properties of
  a Fell-bundle representation, we compute
  \begin{align*}
    \braket{S_g(\act_g(a_1))\xi_1}{S_g(\act_g(a_2))\xi_2}
    &= \braket{\xi_1}{
        S_g(\act_g(a_1))^*\,S_g(\act_g(a_2))\,\xi_2} \\
    &= \braket{\xi_1}{
        S_{u(\s(g))}\bigl(\act_g(a_1)^*\cdot \act_g(a_2)\bigr)\,\xi_2}.
  \end{align*}
  In the Fell bundle associated to a twisted partial action, the product
  \(\act_g(a_1)^*\cdot \act_g(a_2)\in \A_{g^{-1}g}=\FU_{\s(g)}\) is
  precisely \(a_1^*a_2\); this is the \(\FU_{\s(g)}\)-valued inner
  product on the imprimitivity bimodule \(\A_g=\Bun[D]_g\).
  Hence
  \[
    \braket{S_g(\act_g(a_1))\xi_1}{S_g(\act_g(a_2))\xi_2}
    = \braket{\xi_1}{\varphi_{\s(g)}(a_1^*a_2)\,\xi_2},
  \]
  which is exactly the inner product defining the Hilbert space
  \(\Bun[D]_{g^{-1}} \otimes_{\varphi_{\s(g)}} \Hils_{\s(g)}\).

  Therefore, \eqref{eq:def_Ug_from_S} induces an isometry \(U_g\) as
  claimed.  Its range is dense in
  \(\varphi_{\rg(g)}(\Bun[D]_g)\Hils_{\rg(g)}\) by nondegeneracy, and
  hence \(U_g\) is unitary.

  By construction, \(U_g \varphi_{\s(g)}(\act_g^{-1}(a)) = S_g(a)\)
  for all \(a\in \Bun[D]_g\).  Then
  \begin{align*}
    \varphi_{\rg(g)}(a_1) U_g(\varphi_{\s(g)}(a_2)\xi)
    &= \varphi_{\rg(g)}(a_1) S_g(\act_g(a_2))\xi
    = S_g(a_1 \act_g(a_2))\xi
      \\&= U_g(\varphi_{\s(g)}(\act_g^{-1}(a_1)) \varphi_{\s(g)}(a_2)\xi)
  \end{align*}
  for all \(a_1\in \Bun[D]_{g}\), \(a_2\in \Bun[D]_{g^{-1}}\),
  \(\xi\in \Hils_{\s(g)}\).
  Let~\(a_2\) run through an approximate unit
  for~\(\Bun[D]_{g^{-1}}\).  Then we get
  \(\varphi_{\rg(g)}(a) U_g = U_g\varphi_{\s(g)}(\act_g^{-1}(a))\) for
  all \(a\in \Bun[D]_g\), as operators
  \(\varphi_{\s(g)}(\Bun[D]_{g^{-1}}) \Hils_{\s(g)} \to
  \varphi_{\rg(g)}(\Bun[D]_g) \Hils_{\rg(g)}\).  This is equivalent to
  \ref{def:measurable_covariant_representation}.\ref{def:measurable_covariant_representation_2}.
  The measurability of~\(U_g\) is equivalent to the measurability of
  the family~\(S_g\).  Condition
  \ref{def:measurable_covariant_representation}.\ref{def:measurable_covariant_representation_3}
  follows from the multiplicativity property of the maps~\(S_g\).  So
  \((\nu,\Hils,\varphi,U)\) is a measurable covariant representation
  of the Fell bundle.
\end{proof}

\begin{proposition}
  \label{pro:measurable_covariant_representation}
  Assume that our Fell bundle comes from a global twisted action.
  Then a measurable covariant representation has the following data:
  \begin{enumerate}
  \item a measure class~\(\nu\) on~\(G^0\);
  \item a measurable field of Hilbert spaces
    \(\Hils = (\Hils_x)_{x\in G^0}\) with \(\Hils_x\neq0\) for
    \(\nu\)\nb-almost all \(x\in G^0\);
  \item a measurable family of nondegenerate representations
    \(\varphi_x\colon \FU_x \to \Bound(\Hils_x)\) for all \(x\in G^0\);
  \item unitary operators
    \(U_g\colon \Hils_{\s(g)} \congto \Hils_{\rg(g)}\) for all
    \(g\in G\);
  \end{enumerate}
  this is subject to the following conditions:
  \begin{enumerate}[resume]
  \item \label{pro:measurable_covariant_representation_1}%
    \(\nu\)~is quasi-invariant;
  \item \label{pro:measurable_covariant_representation_2}%
    \(U_g \varphi_{\s(g)}(a) U_g^* = \varphi_{\rg(g)} (\act_g(a))\)
    for all \(g\in G\);
  \item \label{pro:measurable_covariant_representation_3}%
    if \((g,h)\in G^2\), then
    \(U_g U_h= \varphi_{\rg(g)}(w(g,h)) U_{g h}\);
  \item \label{pro:measurable_covariant_representation_4}%
    \((\varphi_x)_{x\in G^0}\) and \((U_g)_{g\in G}\) are measurable
    in the sense that pointwise application of \(\varphi_x(a_x)\)
    for \(a\in \SUF\) and of~\(U_g\) map measurable sections to
    measurable sections.
  \end{enumerate}
\end{proposition}

\begin{proof}
  The assumption that the twisted action is global means that
  \(\Bun[D]_g = \FU_{\rg(g)}\) for all \(g\in G\) and therefore the
  corresponding Fell bundle is saturated.  Therefore, the
  measure~\(\nu\) is quasi-invariant in the usual, absolute sense.  In
  addition, the unitaries~\(U_g\) in a covariant representation are
  unitaries \(\Hils_{\s(g)}\congto \Hils_{\rg(g)}\) because
  \(\varphi_{\rg(x)}(\Bun[D]_g)\Hils_{\rg(g)} = \Hils_{\rg(g)}\) for
  all \(g\in G\).  This simplifies the data in
  Definition~\ref{def:measurable_covariant_representation} to the data
  listed in this proposition.
\end{proof}

If the twist is trivial,
then~\ref{pro:measurable_covariant_representation_3} in
Proposition~\ref{pro:measurable_covariant_representation} simplifies
to \(U_g U_h= U_{g h}\), so that \(g\mapsto U_g\) becomes a groupoid
representation.

In general, we do not know how to define ``covariant representations''
of a twisted partial action on a Hilbert module because we cannot
disintegrate~\(U\) to a family of fibrewise unitaries~\(U_g\).  We can
only define covariant representations if the action is global, that
is, \(\Bun[D]_g = \FU_{\rg(g)}\) for all \(g\in G\), and the twist is
scalar-valued, that is,
\(w(g,h) \in \U(\C) \subseteq \U(\FU_{\rg(g)})\) for all
\((g,h)\in G^2\).  In this case, the action satisfies
\(\act_g \act_h = \act_{g h}\) for all \((g,h)\in G^2\), and the twist
together with the trivial action defines a twisted action on
\(\Cont_0(G^0)\).  We turn this into a Fell line bundle~\(L\)
over~\(G\).

\begin{definition}
  Let~\(\FU\) be an upper semicontinuous field of \(\Cst\)\nb-algebras
  over~\(G^0\), let \(\act\colon \s^* \FU \to \rg^*\FU\) be a
  continuous action of~\(G\) on~\(\FU\), and let
  \(w\colon G^2 \to \U(\C)\) be a scalar-valued twist.  Let~\(\A\) be
  the Fell bundle over~\(G\) defined by \((\FU,\act,w)\) and let~\(L\)
  be the Fell line bundle defined by \(\Cont_0(G^0)\) with the
  canonical action of~\(G\) and the twist~\(w\).  Let~\(\F\) be a
  Hilbert module over some \(\Cst\)\nb-algebra~\(D\).  Let
  \(\varphi\colon \Cont_0(G^0, \FU) \to \Bound(\F)\) be a
  nondegenerate \Star{}homomorphism and let
  \(\varphi'\colon \Cont_0(G^0) \to \Bound(\F)\) be the restriction to
  scalars of the extension of~\(\varphi\) to the multiplier algebra of
  \(\Cont_0(G^0,\FU)\).  Represent
  \(\Cont_0(G) \odot \Cont_0(G^0,\FU) \ni f\otimes a\) on
  \(\Lt^2(G,\s,\tilde\alpha) \otimes_{\varphi'} \F\) by
  \(M_f \otimes \varphi(a)\); this extends uniquely to a nondegenerate
  \Star{}homomorphism
  \(\varphi_{\s}\colon \Cont_0(G, \s^*\FU) \to
  \Bound(\Lt^2(G,\s,\tilde\alpha) \otimes_{\varphi'} \F)\).  Define
  another nondegenerate \Star{}homomorphism
  \(\varphi_{\rg}\colon \Cont_0(G, \rg^*\FU) \to
  \Bound(\Lt^2(G,\rg,\alpha) \otimes_{\varphi'} \F)\) in a similar
  way.  Let
  \(U\colon \Lt^2(G,\s,\tilde\alpha) \otimes_{\varphi'} \F \to
  \Lt^2(G,\rg,\alpha) \otimes_{\varphi'} \F\) be a unitary.  The pair
  \((\varphi,U)\) is called \emph{covariant representation} of
  \((\FU,\act,w)\) on~\(\F\) if \((\varphi',U)\) is a Fell-bundle
  representation of~\(L\) and
  \[
    U \varphi_{\s}(a) U^* = \varphi_{\rg}(\act(a))
  \]
  for all \(a\in \Cont_0(G,\s^*\FU)\).
\end{definition}

\begin{theorem}
  There is a natural bijection between Fell-bundle representations
  of~\(\A\) on~\(\F\) and covariant pairs of representations of
  \((\FU,\act,w)\) on~\(\F\).
\end{theorem}

\begin{proof}
  The relevant Fell bundle~\(\A\) is isomorphic to~\(\s^*\FU\) as a
  right Hilbert \(\s^*\FU\)-module, with the left action
  of~\(\rg^*\FU\) through the isomorphism
  \(\act^{-1}\colon \rg^*\FU \congto \s^*\FU\).  Therefore,
  \[
    \Lt^2(G,\s,\A,\tilde\alpha)
    \cong \Lt^2(G,\s,\s^*\FU,\tilde\alpha)
    \cong \Lt^2(G,\s,\tilde\alpha) \otimes_{\Cont_0(G^0)} \Cont_0(G^0,\FU)
  \]
  by Lemma~\ref{lem:compose_corr_from_measure-family}.  Similarly,
  \[
    \Lt^2(G,\rg,\A\A^*,\alpha)
    \cong \Lt^2(G,\rg,\rg^*\FU,\alpha)
    \cong \Lt^2(G,\rg,\alpha) \otimes_{\Cont_0(G^0)} \Cont_0(G^0,\FU).
  \]
  Therefore, there are canonical isomorphisms
  \begin{align*}
    \Lt^2(G,\s,\A,\tilde\alpha)\otimes_{\Cont_0(G^0,\FU)} \F
    &\cong \Lt^2(G,\s,\tilde\alpha)\otimes_{\Cont_0(G^0)} \F,\\
    \Lt^2(G,\rg,\A\A^*,\alpha)\otimes_{\Cont_0(G^0,\FU)} \F
    &\cong \Lt^2(G,\rg,\alpha)\otimes_{\Cont_0(G^0)} \F.
  \end{align*}
  The two Hilbert modules on the left are the domain and codomain of a
  Fell-bundle representation of~\(\A\), whereas those on the right are
  the domain and codomain for a representation of the Fell line
  bundle~\(L\).  For~\(U\) to give a Fell-bundle representation
  of~\(\A\), it must intertwine the canonical left actions of
  \(\Cont_0(G,\A\A^*) = \Cont_0(G,\rg^* \FU)\) on the two left hand
  sides.  A quick inspection shows that these representations of
  \(\Cont_0(G,\rg^* \FU)\) correspond to pointwise multiplication on
  \(\Lt^2(G,\rg,\alpha) \otimes_{\Cont_0(G^0)} \Cont_0(G^0,\FU)\) and
  pointwise multiplication composed with~\(\act^{-1}\) on
  \(\Lt^2(G,\s,\tilde\alpha) \otimes_{\Cont_0(G^0)}
  \Cont_0(G^0,\FU)\).  The covariance condition for~\(U\) in the
  definition of a covariant representation says exactly that~\(U\)
  intertwines these two left actions of \(\Cont_0(G^0,\FU)\).  To be a
  representation, we also need the equation
  \(d_1^*(U) = d_2^*(U) d_0^*(U)\).  Once again, the Hilbert modules
  on which these equations hold for representations of \(\A\)
  and~\(L\) are identified in a canonical way.  Here we use that the
  twist is scalar-valued and is the same for \(\A\) and~\(L\).  This
  is why the coherence conditions \(d_1^*(U) = d_2^*(U) d_0^*(U)\) for
  the two types of Fell-bundle representations are also equivalent to
  each other.
\end{proof}

\section{Fell line bundles as twisted groupoids}
\label{sec:Fell_line}

Next, we consider representations of a twisted groupoid
\(\Cst\)\nb-algebra.  We describe a twist of a groupoid~\(G\)
through a Fell line bundle \(L=(L_g)_{g\in G}\) over~\(G\).  This is
equivalent to another description through a circle extension of~\(G\).
Namely, let \(\Sigma\subseteq L\) be the subset of unitaries in the
Fell line bundle.  If \(y\in L_g\) for some \(g\in G\), then the
conditions \(y^* y = 1\) in \(L_{\s(g)} \cong \C\) and \(y y^* = 1\)
in \(L_{\rg(g)} \cong \C\) are equivalent.  Therefore,
\[
  \Sigma = \setgiven[\bigg]{y\in \bigsqcup_{g\in G} L_g}
  {y^* y = 1 \in L_{\s(g)} \cong \C}.
\]
The subset~\(\Sigma\) with the multiplication and the topology of the
Fell bundle and the range and source maps \(\s(y) = \s(g)\) and
\(\rg(y) = \rg(g)\) for \(y\in L_g\) becomes a topological groupoid.
The circle group~\(\T\) acts freely on~\(\Sigma\) through the scalar
multiplication in the Fell bundle, such that \(\Sigma/\T \cong G\).

\begin{theorem}
  \label{the:measurable_rep_Fell_line}
  Let~\(L\) be a Fell line bundle over~\(G\).  There is a bijection
  between measurable representations of \(L\) and \(\T\)\nb-equivariant
  measurable representations of~\(\Sigma\).
\end{theorem}

\begin{proof}
  Consider a measurable representation
  \((\nu, (\Hils_x)_{x\in G^0}, (S_g)_{g\in G})\) of the Fell line
  bundle~\(L\).  By definition, each~\(S_g\) is a bounded linear map
  \(L_g \to \Bound(\Hils_{\s(g)},\Hils_{\rg(g)})\).  Since
  \(L_g \cong \C\), such a map is determined by its restriction to the
  unitary elements in~\(L_g\).  Define \(U_y\defeq S_g(y)\) for
  \(y\in \Sigma \cap L_g\), \(g\in G\).  Since~\(y\) is unitary,
  \(S_g(y) S_g(y)^* = S_g(y)S_{g^{-1}}(y^*) = S_{\rg(g)}(y y^*) =
  \Id_{\Hils_{\rg(g)}}\) and, similarly,
  \(S_g(y)^* S_g(y) = \Id_{\Hils_{\s(g)}}\).  As a result, we get
  unitaries \(U_y\colon \Hils_{\s(y)} \to \Hils_{\rg(y)}\) for all
  \(y\in\Sigma\).  These unitaries, together with the same \(\nu\)
  and~\(\Hils\), form a measurable representation of the
  groupoid~\(\Sigma\); it is easy to see that~\(\nu\) is
  \(G\)\nb-quasi-invariant as a measure class on~\(G^0\) if and only
  if it is \(\Sigma\)\nb-quasi-invariant.  In addition, the
  unitaries~\(U_y\) satisfy the \(\T\)\nb-equivariance condition
  \(U_{z\cdot y} = z\cdot U_y\) for all \(z\in\T\), \(y\in\Sigma\).
  Conversely, a measurable representation of~\(\Sigma\) that is
  \(\T\)\nb-equivariant in this sense must come from a measurable
  representation of the Fell line bundle~\(L\).
\end{proof}

We are going to improve this statement to an isomorphism of
\(\Cst\)\nb-algebras.  This uses a circle action on \(\Cst(\Sigma)\).
For \(z\in\T\), define a linear map on the \Star{}algebra
\(\Contc(\Sigma)\) by \((\mu_z f)(x) \defeq f(z^{-1} x)\).

\begin{lemma}
  \label{lem:muz_on_Sigma}
  Each~\(\mu_z\) is a central unitary multiplier
  of~\(\Contc(\Sigma)\) and extends to a central unitary
  multiplier of~\(\Cst(\Sigma)\).  The map \(z\mapsto\mu_z\) is
  strictly continuous.
\end{lemma}

\begin{proof}
  We compute
  \[
    \mu_z (f_1 * f_2) = \mu_z(f_1) * f_2 = f_1 * \mu_z(f_2)
  \]
  because
  \(z^{-1} \cdot (x\cdot y) = (z^{-1}\cdot x)\cdot y = x\cdot
  (z^{-1} \cdot y)\) for all \(z\in\T\), \(x,y\in \Sigma\) with
  \(\s(x) = \rg(y)\).  This says that~\(\mu_z\) is a central unitary
  multiplier of~\(\Contc(\Sigma)\).  The map
  \(\mu_z\colon \Contc(\Sigma) \to \Contc(\Sigma)\) is bounded in
  the \(I\)\nb-norm and continuous in the inductive limit topology.
  If \(\norm{\blank}\colon \Contc(\Sigma) \to [0,\infty)\) is any
  \(\Cst\)\nb-norm that is bounded in the \(I\)\nb-norm, then
  \(\norm{\mu_z(f)} = \norm{\mu_z(f)^* \mu_z(f)}^{1/2} = \norm{f^*
    f}^{1/2} = \norm{f}\) for all \(f\in\Contc(\Sigma)\).  This
  implies that~\(\mu_z\) extends to the groupoid
  \(\Cst\)\nb-algebra, and the extension remains a central unitary
  multiplier.  If \(f\in\Contc(\Sigma)\), then the map
  \(\T\to\Contc(\Sigma)\), \(z\mapsto \mu_z(f)\), is continuous even
  in the inductive limit topology on~\(\Contc(\Sigma)\).  The same
  is true for \(\mu_z^*(f) = \mu_{z^{-1}}(f)\).  This implies that
  the map \(z\mapsto \mu_z\) is strictly continuous.
\end{proof}

The group homomorphism~\(\mu\) from~\(\T\) to the group of central
unitary multipliers of~\(\Cst(\Sigma)\) integrates to a
\Star{}homomorphism from \(\Cont_0(\Z) \cong \Cst(\T)\) to the
centre of the multiplier algebra of~\(\Cst(\Sigma)\).  As a result,
\(\Cst(\Sigma)\) decomposes as a \(\Cont_0\)\nb-direct sum of the
\(\Cst\)\nb-subalgebras
\[
  \Cst(\Sigma)_n \defeq \setgiven{f\in \Cst(\Sigma)}
  {\mu_z(f) = z^n f \text{ for all }z\in\T}
\]
for \(n\in\Z\).  Define \(\Contc(\Sigma)_n\) for \(n\in\Z\) similarly.

\begin{proposition}
  \label{pro:Cst_Sigma_n}
  If \(n\in\Z\), then \(\Cst(\Sigma)_n\) is the completion of
  \(\Contc(\Sigma)_n\) in the maximal \(\Cst\)\nb-seminorm that is
  bounded in the restriction of the \(I\)\nb-norm.
\end{proposition}

\begin{proof}
  The canonical projection
  \(E_n\colon \Cst(\Sigma) \to \Cst(\Sigma)_n\),
  \(f\mapsto \int_\T \mu_z(f) z^{-n} \,\diff z\), maps
  \(\Contc(\Sigma)\) into \(\Contc(\Sigma)_n\).  This implies that
  \(\Contc(\Sigma)_n\) is a dense \Star{}subalgebra
  of~\(\Cst(\Sigma)_n\).  So \(\Cst(\Sigma)_n\) is the completion of
  \(\Contc(\Sigma)_n\) in the restriction of the norm of
  \(\Cst(\Sigma)\).  This restriction is a \(\Cst\)\nb-seminorm that
  is bounded in the restricted \(I\)\nb-norm.  It remains to show
  that it is the largest such seminorm.  Since the map~\(E_n\) is a
  \Star{}homomorphism and bounded for the \(I\)\nb-norm, any
  \(\Cst\)\nb-seminorm on~\(\Contc(\Sigma)_n\) that is bounded in the
  restricted \(I\)\nb-norm extends to a \(\Cst\)\nb-seminorm
  on~\(\Contc(\Sigma)\) that is still bounded in the \(I\)\nb-norm,
  and hence it is bounded by the norm on \(\Cst(\Sigma)\).
\end{proof}

A section \(f\in\Contc(G,L)\) defines a scalar-valued map
\(\tilde{f}\colon \Sigma \to \C\) by
\(\tilde{f}(x) = x^* \cdot f(g)\) for all \(g\in G\),
\(x\in L_g \cap \Sigma\).  The map~\(\tilde{f}\) is continuous and
has compact support, and it belongs to~\(\Cst(\Sigma)_1\) because
\[
  (\mu_z\tilde{f})(x)
  = \tilde{f}(z^{-1}\cdot x)
  = (z^{-1}\cdot x)^* \cdot f(g)
  = z\cdot x^* \cdot f(g)
  = z\cdot \tilde{f}(x).
\]
Thus we get a canonical map \(\Contc(G,L) \to \Contc(\Sigma)_1\).

\begin{theorem}
  \label{the:rep_line_bundle}
  The map \(\Contc(G,L) \to \Contc(\Sigma)_1\) is a
  \Star{}isomorphism.  It extends to a \Star{}isomorphism
  \(\Cst(G,L) \cong \Cst(\Sigma)_1\).
\end{theorem}

\begin{proof}
  It is easy to check that the map
  \(\Contc(G,L) \to \Contc(\Sigma)_1\) is a \Star{}homomorphism.  To
  see that it is bijective, take a function
  \(\tilde{f}\in\Contc(\Sigma)_1\).  Let \(g\in G\).  Then
  \(z\cdot x\cdot \tilde{f}(z\cdot x) = x\cdot \tilde{f}(x)\) for all
  \(x\in L_g \cap \Sigma\), that is,
  \(f(g) \defeq x\cdot \tilde{f}(x) \in L_g\) is well defined.  This
  defines an element \(f\in \Contc(G,L)\) that maps to~\(\tilde{f}\)
  under the map above.  The two constructions are clearly inverse to
  each other, that is, we have a \Star{}isomorphism
  \(\Contc(G,L) \to \Contc(\Sigma)_1\).  Now
  Proposition~\ref{pro:Cst_Sigma_n} implies the \Star{}isomorphism
  \(\Cst(G,L) \to \Cst(\Sigma)_1\).
\end{proof}

\begin{remark}
  The same argument shows
  \(\Cst(\Sigma)_n \cong \Cst(G,L^{\otimes n})\).
\end{remark}

\begin{remark}
  Let~\(\F\) be a Hilbert module over some \(\Cst\)\nb-algebra with a
  nondegenerate \Star{}homomorphism
  \(\varphi\colon \Cont_0(\Sigma^0) = \Cont_0(G^0) \to \Bound(\F)\).
  A representation of the groupoid~\(\Sigma\) is a unitary operator
  \(U_\sigma\colon \Lt^2(\Sigma, \s,\tilde\alpha)
  \otimes_{\Cont_0(\Sigma^0)} \F \to \Lt^2(\Sigma, \rg,\alpha)
  \otimes_{\Cont_0(\Sigma^0)} \F\) satisfying the extra condition
  \(d_2^*(U) d_0^*(U) = d_1^*(U)\).  Define unitary representations
  \(\mu\colon \T\to \U(\Lt^2(\Sigma, \s, \tilde\alpha))\) and
  \(\mu\colon \T\to \U(\Lt^2(\Sigma, \rg, \alpha))\) by
  \(\mu_z(h)(x) \defeq h(z^{-1}\cdot x) \) for \(z\in \T\),
  \(x\in \Sigma\) and a suitable \(\Lt^2\)-function~\(h\).  We call a
  representation of~\(\Sigma\) \emph{homogeneous of degree~\(n\)} if
  \(\mu_z U \mu_z^* = z^n U\) for all \(z\in \T\).  A computation
  shows that a representation of \(\Cst(\Sigma)\) factors through
  \(\Cst(\Sigma)_n\) if and only if the corresponding groupoid
  representation \((\varphi,U)\) of~\(\Sigma\) is homogeneous of
  degree~\(n\).  Together with Theorem~\ref{the:rep_line_bundle}, it
  then follows that Fell-bundle representations of~\(L\) on Hilbert
  modules are equivalent to homogeneous representations of the
  groupoid~\(\Sigma\) of degree~\(1\).  The same statement is easy to
  check for measurable Fell-bundle representations on Hilbert spaces.
\end{remark}

\section{Fell bundles over groups}
\label{sec:Fell_over_groups}

This section studies Fell bundles and their section \(\Cst\)\nb-algebras
over groups.  Let~\(G\) be a locally compact group and let~\(\A\) be
a Fell bundle over~\(G\).  The Haar measure on~\(G\) is unique up to
a scalar factor.  So the section \(\Cst\)\nb-algebra of~\(\A\) does
not depend on it, and we denote it by \(\Cst(G,\A)\).  The usual
definition of \(\Cst(G,\A)\) for a group~\(G\)
in~\cite{Doran-Fell:Representations_2} differs
in two ways from ours.  First, it uses the modular function
\(\Delta\colon G\to (0,\infty)\) to define the adjoint.  We
define~\(\Delta\) so that
\(\diff (g^{-1}) = \Delta(g^{-1}) \,\diff g\).  The usual formula
for the involution on \(\Contc(G,\A)\) is
\(f^\times(g) \defeq f(g^{-1})^* \Delta(g)^{-1}\) instead of our
formula \(f^*(g) \defeq f(g^{-1})^*\).  The map
\[
  \mu\colon \Contc(G,\A) \to \Contc(G,\A),\qquad
  f\mapsto f\cdot \Delta^{1/2},
\]
is a \Star{}algebra isomorphism from \((\Contc(G,\A),*,^\times)\) to
\((\Contc(G,\A),*,^*)\), that is,
\[
  \mu(f_1* f_2) = \mu(f_1) *\mu(f_2),\qquad
  (\mu f)^* = \mu(f^\times)
\]
for all \(f_1,f_2,f\in \Contc(G,\A)\).
The second difference regards the \(L^1\)-norm that we use to bound
representations on group \(\Cst\)\nb-algebras, or more generally group
crossed products or Fell bundles.
The usual \(L^1\)\nb-norm and the \(I\)\nb-norm are only related
through an inequality \(\norm{f}_1 \le \norm{\mu(f)}_I\), which
follows from the Cauchy--Schwarz inequality:
\begin{align*}
  \norm{f}_1
  &= \int_G\norm{f(g)}\,\diff g
    = \int_G
    (\norm{f(g)}^{1/2}\Delta(g)^{1/4})(\norm{f(g)}^{1/2}\Delta(g)^{-1/4}) \,\diff g\\
  &\le \left(\int_G\norm{f(g)}\Delta(g)^{1/2} \,\diff g\right)^{1/2}
    \left(\int_G\norm{f(g)}\Delta(g)^{-1/2}\,\diff g\right)^{1/2}\\
  &\leq \norm{\mu(f)}_I^{1/2} \norm{\mu(f)}_I^{1/2}
    = \norm{\mu(f)}_I.
\end{align*}
Nevertheless, our definition of \(\Cst(G,\A)\) gives the usual Fell
bundle section \(\Cst\)\nb-algebra:

\begin{proposition}
  \label{pro:bounded_by_L1-norm}
  Let~\(D\) be a \(\Cst\)\nb-algebra and~\(\F\) a Hilbert
  \(D\)\nb-module.  Equip~\(\F\) with a pre-representation
  \(L\colon \Contc(G,\A)\times \F_0 \to \F\) of\/ \(\Contc(G,\A)\) for
  a vector space~\(\F_0\) and a linear map \(\iota\colon \F_0\to\F\)
  with dense image \textup{(}see
  Definition~\textup{\ref{def:pre-representation})}.  Then
  \(L\circ \mu\) extends to an ordinary \Star{}representation of\/
  \((\Contc(G,\A),*,^{\times})\) on~\(\F\) that is bounded with
  respect to the usual \(L^1\)\nb-norm
  \(\norm{f}_1 \defeq \int_G \norm{f(g)} \,\diff g\).
\end{proposition}

\begin{proof}
  Corollary~\ref{cor:prerep} shows that~\(L\) extends to a
  representation~\(L'\) of~\(\Contc(G,\A)\) that is bounded in the
  \(I\)\nb-norm.  Composing with the \Star{}isomorphism~\(\mu\)
  gives a \Star{}representation of
  \((\Contc(G,\A),*,^\times)\).  It remains to improve the norm
  bound.  To do this, we
  change the factorisation \(f = f_1 \cdot f_2\) produced by
  Lemma~\ref{lem:factor_f} and use
  \(f = f_1'\cdot f_2'\) with \(f_1'\defeq f_1\Delta^{-1/4}\) and
  \(f_2'\defeq f_2\Delta^{+1/4}\) instead.
  As in the proof of Lemma~\ref{lem:Lf_well-defined} we get the norm bound
  \(\norm{L'(f)}\le \norm{f_1'}_{\Lt^2(G,\A\A^*,\rg,\alpha)} \cdot
  \norm{f_2'}_{\Lt^2(G,\A,\s,\tilde\alpha)}\).  Computations as above
  show that
  \(\norm{f_1'}_{\Lt^2(G,\A\A^*,\rg,\alpha)} \le
  \norm{\mu^{-1}(f)}_1^{1/2}\) and
  \(\norm{f_2'}_{\Lt^2(G,\A,\s,\tilde\alpha)} \le
  \norm{\mu^{-1}(f)}_1^{1/2}\).  So
  \(\norm{L'(f)} \le \norm{\mu^{-1}(f)}_1\).
\end{proof}

We also learn from the previous result that any representation of
\((\Contc(G,\A),*,^\times)\) that is continuous in the inductive
limit topology is already bounded for the usual \(L^1\)\nb-norm.
This is an analogue of Corollary~\ref{cor:continuity-rep} using the
\(L^1\)\nb-norm instead of the \(I\)\nb-norm.  For ordinary crossed
products by locally compact group actions, this was proved by Green
in \cite{Green:Local_twisted}*{Corollary to Proposition~3 on
  p.~203}.  This is extended to Fell bundles over locally compact
groups in \cite{Doran-Fell:Representations_2}*{Theorem~VIII.13.8}.

In the case of groups, there is an alternative definition of
representations of Fell bundles using maps \(\A\to \Bound(\F)\) that
respect product and involution and are continuous for the strict
topology.  More precisely, this is defined as follows:

\begin{definition}[Fell]
  \label{def:continuous_rep_Fell_group}
  A \emph{continuous representation} of~\(\A\) on a Hilbert
  module~\(\F\) is a map \(\pi\colon \A\to \Bound(\F)\) which is linear
  on the fibres of~\(\A\) and satisfies
  \begin{enumerate}
  \item \label{def:continuous_rep_Fell_group_1}%
    \(\pi(a\cdot b)=\pi(a)\pi(b)\) and \(\pi(a^*)=\pi(a)^*\) for all
    \(a,b\in \A\);
  \item \label{def:continuous_rep_Fell_group_1b}%
    the representation~\(\pi|_{\A_1}\) of~\(\A_1\) is nondegenerate;
  \item \label{def:continuous_rep_Fell_group_2}%
    the map \(a\mapsto \pi(a)\) is continuous from~\(\A\) to
    \(\Bound(\F)=\Mult(\Comp(\F))\) with respect to the strict topology.
  \end{enumerate}
\end{definition}

Recall that the strict topology on \(\Bound(\F)\) is equivalent to the
\Star{}strong topology on bounded sets.  Using the compatibility with
the involution in~\ref{def:continuous_rep_Fell_group_1} it follows
that the continuity of \(\pi\) in~\ref{def:continuous_rep_Fell_group_2}
is equivalent to norm continuity of the map \(a\mapsto \pi(a)\xi\) for
all \(\xi\in \F\), that is, strong continuity of~\(\pi\).

Continuous representations of Fell bundles were extensively studied by
Fell, see for instance~\cite{Doran-Fell:Representations_2} and
references there in.
Usually, only Hilbert space
representations are considered in the literature, but similar
techniques apply to Hilbert module representations as above.
One of
the main results in this direction states that continuous
representations correspond bijectively to representations of
\(\Cst(G,\A)\).
Combined with our results, it follows that the
continuous representations in the sense of Fell must be equivalent to
our representations.
We are going to prove this directly and deduce an automatic continuity
result for measurable representations on Hilbert spaces.

The key idea is that there is a universal continuous representation
\[m\colon \A\to\Mult(\Cst(G,\A))\] of the Fell bundle by multipliers
of~\(\Cst(G,\A)\).  For \(g\in G\) and \(a\in \A_g\), we define
\(m_a\colon \Contc(G,\A)\to \Contc(G,\A)\) by
\begin{equation}
  \label{eq:mult-Fell-bundle-univ}
  m_a(f)(x) = \Delta(g)^{1/2} a f(g^{-1}x)
\end{equation}
for \(f\in \Contc(G,\A)\) and \(x\in G\); here~\(\Delta\) is the
modular function of~\(G\).  The isomorphism~\(\mu\) defined above
transfers this by \(m_a f = \mu(\tilde{m}_a \mu^{-1}(f))\) to the more
obvious multiplier \((\tilde{m}_a f)(x) = a f(g^{-1}x)\) of
\((\Contc(G,\A),*,^\times)\), which already appears in
\cite{Doran-Fell:Representations_2}*{Section~VIII.12.3}.

\begin{lemma}
  \label{lem:universal_continuous_rep_group}
  The maps \(a\mapsto m_a\) above form a continuous Fell-bundle
  representation of~\(\A\) on \(\Cst(G,\A)\).
\end{lemma}

\begin{proof}
  It is easy to see that \(m_a (m_b(f)) = m_{a b}(f)\) for all
  \(a,b\in\A\) , \(f\in\Contc(G,\A)\).  We also compute
  \begin{align*}
    f_1^* * (m_a(f_2)) (x)
    &= \int_G f_1(h^{-1})^* a f_2(g^{-1} h^{-1} x) \Delta(g)^{1/2} \,\diff h\\
    &= \int_G f_1(gh^{-1})^* a f_2(h^{-1} x) \Delta(g)^{-1/2} \,\diff h\\
    &= \int_G (a^* f_1(g h^{-1}))^* f_2(h^{-1} x) \Delta(g)^{-1/2} \,\diff h
      = (m_{a^*}(f_1)^* * f_2) (x).
  \end{align*}
  In the second equality above, we made the change of variables
  \(h\mapsto hg^{-1}\) with
  \(\diff (hg^{-1}) = \Delta(g)^{-1}\diff h\).  Using a trick as in
  Lemma~\ref{lem:extension-multiplier}, we show that~\(m_a\) extends
  to a bounded multiplier of \(\Cst(G,\A)\) with
  \(\norm{m_a} \le \norm{a}\).  For \(f\in \Contc(G,\A)\), the map
  \(a \mapsto m_a(f)\) is continuous for the inductive limit topology
  on \(\Contc(G,\A)\).  Then the map
  \(m\colon \A \to \Mult(\Cst(G,\A))\) is continuous for the strict
  topology on \(\Mult(\Cst(G,\A))\).
\end{proof}

\begin{theorem}
  \label{the:Fell_over_group_continuous_rep}
  Let~\(G\) be a locally compact group and let~\(\A\) be a Fell bundle
  over~\(G\).  Let~\(D\) be a \(\Cst\)\nb-algebra and let~\(\F\) be a
  Hilbert \(D\)\nb-module.  Then there is a natural bijection between
  Fell-bundle representations of~\(\A\) on~\(\F\) and continuous
  representations of~\(\A\) on~\(\F\).
\end{theorem}

\begin{proof}
  Theorem~\ref{the:universal_groupoid_Haus} gives a natural bijection
  between Fell-bundle representations of~\(\A\) on~\(\F\) and
  nondegenerate representations of \(\Cst(G,\A)\).
  Any nondegenerate representations of \(\Cst(G,\A)\) extends to
  multipliers, and composing with the universal continuous
  representation \(m\colon \A\to\Mult(\Cst(G,\A))\) of~\(\A\), it
  induces a continuous representation of~\(\A\) on~\(\F\).
  Conversely, let \(\pi\colon \A \to \Bound(\F)\) be a continuous
  representation.
  Then it integrates to a representation of the \(L^1\)-cross-section
  \Star{}algebra and hence also to its enveloping \(\Cst\)\nb-algebra
  \(\Cst(G,\A)\) (see~\cite{Doran-Fell:Representations_2}).
  The latter corresponds to a Fell-bundle representation of~\(\A\).
  In fact, we may directly describe the latter using~\(\pi\).
  Let \(\varphi\defeq \pi|_{\A_1}\) and define \(U\colon \Contc(G,\A)\odot \Cont(G,\F) \to\Contc(G,\F)\) by
  \[
    (U (a\otimes f))(g) \defeq \pi(a(g)) f(g)
  \]
  for \(a\in \Contc(G,\A)\), \(f\in \Contc(G,\F)\), \(g\in G\).  We
  compute that~\(U\) is isometric and thus extends to an isometry
  \(U\colon \Lt^2(G,\A) \otimes_{\A_1} \F \to\F\).  This operator
  intertwines the canonical actions of~\(\A_1\).  Therefore, its image
  is contained in \(\Lt^2(G,\A\A^*)\otimes_{\A_1} \F\).  The
  multiplicativity of~\(U\) implies by an easy computation that~\(U\)
  is a Fell-bundle representation.
\end{proof}

Now we turn to measurable Fell-bundle representations on a Hilbert
space.  Here the theory in Section~\ref{sec:Renault} simplifies
somewhat.  The space~\(G^0\) has just one point and so disintegrations
over~\(G^0\) become empty.  In particular, the measure class~\(\nu\)
lives on a one-point space, so that it becomes irrelevant, and the
measurable field of Hilbert spaces~\(\Hils\) becomes just the
underlying Hilbert space of our representation.  So a measurable
representation on~\(\Hils\) becomes just a map
\(S\colon \A \to \Bound(\Hils)\) that preserves the multiplication and
the involution and is measurable in a suitable sense.

\begin{theorem}
  \label{the:automatic_continuity_group}
  Let~\(G\) be a second countable, locally compact group and
  let~\(\A\) be a separable Fell bundle over~\(G\).
  Let~\(\Hils\) be a separable Hilbert space.
  Any measurable Fell-bundle representation of~\(\A\) on~\(\Hils\) is
  continuous.
\end{theorem}

\begin{proof}
  By Theorem~\ref{the:disintegrate_Hilbert_space}, a measurable
  Fell-bundle representation \(S=(S_g)_{g\in G}\) of~\(\A\)
  on~\(\Hils\) is equivalent to a Fell-bundle representation as in
  Definition~\ref{def:representation_Fell}.
  And this is equivalent to a representation \(\sigma\colon \Cst(G,\A)
  \to \Bound(\Hils)\).
  Let \(\bar\sigma\colon \Mult(\Cst(G,\A)) \to \Bound(\Hils)\) be the
  unique extension of~\(\sigma\) to multipliers.
  Composing this with the universal continuous representation
  \(m\colon \A\to\Mult(\Cst(G,\A))\) defined
  in~\eqref{eq:mult-Fell-bundle-univ} gives a continuous
  representation \(S'\defeq \bar\sigma \circ m\colon \A \to
  \Bound(\Hils)\) of~\(\A\) on~\(\Hils\).
  When we integrate this continuous representation, we get~\(\sigma\).
  That is,
  \[
    \int_G S_g(a(g)) \,\diff g
    = \sigma(a)
    = \int_G \bar\sigma (m_{a(g)}) \,\diff g
  \]
  for all \(a\in \Contc(G,\A)\).
  Both \(S\) and~\(S'\) are measurable Fell-bundle representations,
  and the identity map on~\(\Hils\) is a unitary intertwiner between
  the representations of~\(\Cst(G,\A)\) that they induce.
  Then it is also an intertwiner between the Fell-bundle
  representations induced by \(S\) and~\(S'\) by
  Theorem~\ref{the:universal_groupoid_Haus}.
  Therefore, the identity map must be as in Theorem
  \ref{the:disintegrate_Hilbert_space}.\ref{the:disintegrate_Hilbert_space_2}.
  This says that \(S_g = S'_g\) holds for all \(g\in G\) and not just
  for almost all~\(g\).
\end{proof}

\section{Transformation groupoids}
\label{sec:trafo_groupoids}

Let a group or groupoid~\(G\) act (partially) on a locally compact
space~\(X\).  Let \(G\ltimes X\) be the resulting transformation
groupoid.  As we shall see, a Fell bundle over \(G\ltimes X\) yields a
Fell bundle over~\(G\), and these two Fell bundles have naturally
isomorphic representations and hence isomorphic section
\(\Cst\)\nb-algebras.  In particular, this reproves that the groupoid
\(\Cst\)\nb-algebra \(\Cst(G\ltimes X)\) is isomorphic to the crossed
product \(G\ltimes \Cont_0(X)\).

The transformation groupoid \(G\ltimes X\) has object space~\(X\), and
its arrow space is the space of triples \((x,g,y)\) with \(x,y\in X\),
\(g\in G\) such that \(\s(g) = \rg(y)\) and \(x=g\cdot y\); in
particular, \(y\) is in the domain of the action of~\(g\).
This is an extra condition in the case of a partial action, which is
allowed in this section.
We do not give names to the domains, however, and write ``\(g\cdot x\)
is defined'' instead.
The multiplication is defined by \((x,g,y) \cdot (y,h,z) = (x,g\cdot
h,z)\).
Let~\(\alpha\) be a Haar system on~\(G\).
This induces a Haar system~\(\beta\) on \(G\ltimes X\) by
\[
  \int_{(G\ltimes X)^x} f(x,g,g^{-1}\cdot x) \,\diff
  \beta^x(x,g,g^{-1}\cdot x)
  \defeq
  \int f(x,g,g^{-1}\cdot x) \,\diff \alpha^{\rg(x)}(g),
\]
where the second integral is over the (open) set \(\setgiven{g\in
  G^{\rg(x)}}{g^{-1}\cdot x \text{ is defined}}\).

Let~\(\B\) be a Fell bundle over \(G\ltimes X\).
It induces a Fell bundle~\(\A\) over~\(G\) as follows.
Let
\[
  X_g \defeq \setgiven{(x,g,g^{-1} x)\in G\ltimes X}{x \in X,\ \rg(x)
    = \rg(g),\ g^{-1} x \text{ is defined}}.
\]
The fibre of~\(\A\) at \(g\in G\) is \(\A_g \defeq \Cont_0(X_g,\B|_{X_g})\),
and the multiplication and involution on~\(\A\) are defined by
applying the multiplication and involution on~\(\B\) pointwise.
The topology is defined so that the \(\Cont_0\)\nb-sections of~\(\A\)
are equivalent to \(\Cont_0\)\nb-sections of~\(\B\).
More precisely, if~\(f\) is a \(\Cont_0\)\nb-section of~\(\B\), then
the corresponding continuous section of~\(\A\) maps \(g\in G\) to the
restriction of~\(f\) to~\(X_g\).
This defines a Fell bundle over~\(G\).

\begin{proposition}
  \label{pro:represent_trafo_groupoid}
  Let~\(G\) be a locally compact groupoid with a Haar
  system~\(\alpha\).  Let~\(X\) be a locally compact space with a
  \textup{(}partial\textup{)} action of~\(G\).  Let~\(\B\) be a Fell
  bundle over~\(G\ltimes X\).  Define a Haar measure~\(\beta\)
  on~\(G\ltimes X\) and a Fell bundle~\(\A\) over~\(G\) as above.
  Then
  \[
    \Cst(G,\A) \cong \Cst(G\ltimes X,\B,\beta).
  \]
\end{proposition}

\begin{proof}
  By definition, there is a canonical isomorphism
  \(\Cont_0(G,\A) \cong \Cont_0(G\ltimes X,\B)\).  It restricts to a
  canonical isomorphism between \(\SUF \defeq \Cont_0(G^0,\FU)\)
  and \(\SUF[B]\defeq \Cont_0(X,\FU[B])\).  It also follows that
  \(\Cont_0(G,\A\A^*) \cong\Cont_0(G\ltimes X,\B\B^*)\), and similarly
  for the spaces of \(\Cont_0\)\nb-sections on \(G^2\) and
  \((G\ltimes X)^2\) that we need.  The Haar measures are related in
  such a way that
  \(\Lt^2(G, \A,\s,\tilde\alpha) \cong \Lt^2(G\ltimes X,
  \B,\s,\tilde\beta)\) and
  \(\Lt^2(G, \A\A^*,\rg,\alpha) \cong \Lt^2(G\ltimes X,
  \B\B^*,\rg,\beta)\) as correspondences.  And the correspondences
  \(\Lt^2(G^2, \E_k,v_k,\mu_k)\) for \(k=0,1,2\) that occur in the
  representation conditions over \(G\) and~\(G\ltimes X\) are
  canonically isomorphic as well.  Using these isomorphisms, we
  identify the data and the conditions for Fell-bundle representations
  of \(\A\) and~\(\B\) on any Hilbert module~\(\F\).  This induces a
  natural isomorphism between \(\Cst(G,\A)\) and
  \(\Cst(G\ltimes X,\B,\beta)\) because of their universal properties.
\end{proof}

\begin{example}
  When~\(\B\) is the trivial Fell bundle, then
  \(\Cst(G\ltimes X,\B,\beta) = \Cst(G\ltimes X)\) is the groupoid
  \(\Cst\)\nb-algebra of \(G\ltimes X\) for our choice of Haar system.
  Proposition~\ref{pro:represent_trafo_groupoid} identifies this with
  the section \(\Cst\)\nb-algebra of a Fell bundle over~\(G\).  In
  fact, this Fell bundle comes from the (partial) action of~\(G\) on
  \(\Cont_0(X)\) induced by its (partial) action on~\(X\) as explained
  after Definition~\ref{def:twisted-partial-action-groupoid}.  So it
  is the (partial) crossed product \(\Cont_0(X)\rtimes G\).
\end{example}

\begin{remark}
  Without our universal property,
  Proposition~\ref{pro:represent_trafo_groupoid} is not completely
  trivial.  While the \(\Cont_0\)\nb-sections of the two bundles are
  canonically isomorphic, the \(\Contc\)\nb-sections of the Fell
  bundles \(\A\) and~\(\B\) are different unless the map \(X\to G^0\)
  is proper.  This is visible already for the trivial Fell bundle and
  when~\(G\) is a group acting on a noncompact space~\(X\).  The
  \(\Cst\)\nb-algebra \(\Cst(G\ltimes X,\B,\beta)\) is defined as a
  completion of \(\Contc(G\ltimes X)\), whereas \(\Cst(G,\A)\)
  is defined as a completion of \(\Contc(G,\Cont_0(X))\).  So we
  partially completed the function space in the direction of~\(X\).
  Proposition~\ref{pro:represent_trafo_groupoid} implies that any
  representation of the \Star{}algebra \(\Contc(G\ltimes X)\) that is
  continuous in the inductive limit topology over compact subsets of
  \(G\ltimes X\) extends uniquely to a representation of
  \(\Contc(G,\Cont_0(X))\) that is continuous in the inductive limit
  topology over compact subsets of~\(G\).  If \(G=\{1\}\), we have to
  extend representations of \(\Contc(X)\) to \(\Cont_0(X)\).  This is
  done in Lemma~\ref{lem:extension-multiplier}.
\end{remark}

\begin{remark}
  Results closely related to
  Proposition~\ref{pro:represent_trafo_groupoid} have appeared earlier
  in several special cases.
  In particular, decomposition results for crossed products by transformation groupoids are discussed in detail, together with many further examples and references, in~\cite{Buss-Meyer:Groupoid_fibrations}.

  The general framework developed there is that of \emph{groupoid
  fibrations} and actions of groupoids on groupoids by equivalences,
  which in particular covers global actions and leads to saturated
  Fell bundles.
  From this point of view, transformation groupoids arising from
  global actions fit naturally into this framework.

  In the present section, however, we also allow \emph{partial}
  actions, and the associated transformation groupoids do not, in
  general, define groupoid fibrations as
  in~\cite{Buss-Meyer:Groupoid_fibrations}.
  Proposition~\ref{pro:represent_trafo_groupoid} may therefore be
  viewed as an analogue of the decomposition results for groupoid
  fibrations, extended to the setting of (possibly) partial actions,
  even though no general notion of a ``partial groupoid fibration'' is
  developed here.
\end{remark}

\end{document}